\documentclass[oneside,12pt]{article}

\usepackage{graphicx, subcaption}
\graphicspath{{figures/}}
\usepackage{microtype}
\usepackage[top=1in,bottom=1in,left=1in,right=1in]{geometry}
\usepackage{amsfonts, amsmath, amssymb, amsthm, bbm, bm, MnSymbol, mathtools}
\usepackage{appendix}
\RequirePackage[numbers]{natbib}
\renewcommand\leq\leqslant
\renewcommand\geq\geqslant

\usepackage{algorithm2e}
\RestyleAlgo{ruled}

\usepackage{xcolor}
\definecolor{ChadDarkBlue}{rgb}{.1,0,.2}  
\definecolor{ChadBlue}{rgb}{.1,.1,.5}  
\definecolor{ChadRoyal}{rgb}{.2,.2,.8}  
\definecolor{ChadGreen}{rgb}{0,.4,0}    
\definecolor{ChadRed}{rgb}{.5,0,.5} 
\definecolor{QueenBlue}{RGB}{8, 109, 133}
\definecolor{NiceGreen}{RGB}{25, 150, 100}
\definecolor{PaulGreen}{rgb}{0.0,.5,.0}
\definecolor{CP1}{HTML}{0F4C75}
\definecolor{CP2}{HTML}{3282B8}
\definecolor{CP3}{HTML}{097969}
\usepackage{url}
\usepackage[colorlinks,breaklinks,   
bookmarks=false,
pdfstartview=Fit,  
pdfview=Fit,       
linkcolor=CP3,
urlcolor=ChadGreen,
citecolor=CP2,
hyperfootnotes=false
]{hyperref}
\usepackage[figure,table]{hypcap} 
\makeatletter
\newcommand{\mytag}[2]{%
	\text{#1}%
	\@bsphack
	\begingroup
	\@onelevel@sanitize\@currentlabelname
	\edef\@currentlabelname{%
		\expandafter\strip@period\@currentlabelname\relax.\relax\@@@%
	}%
	\protected@write\@auxout{}{%
		\string\newlabel{#2}{%
			{#1}%
			{\thepage}%
			{\@currentlabelname}%
			{\@currentHref}{}%
		}%
	}%
	\endgroup
	\@esphack
}
\makeatother

\theoremstyle{plain}
\newtheorem{thm}{Theorem}[section]
\newtheorem{prp}[thm]{Proposition}
\newtheorem{lem}[thm]{Lemma}
\newtheorem{cor}[thm]{Corollary}

\theoremstyle{remark}
\newtheorem{dfn}[thm]{Definition}
\newtheorem{rem}{Remark}

\def\beq{\begin{equation}} 
\def\eeq{\end{equation}}
\def\beqn{\begin{eqnarray*}}
\def\eeqn{\end{eqnarray*}}
\def\Bal{\begin{align}}
\def\Eal{\end{align}}
\def\Bitem{\begin{itemize}\setlength{\itemsep}{.2in}}
\def\bitem{\begin{itemize}\setlength{\itemsep}{.05in}}
	\def\eitem{\end{itemize}}
\def\blatin{\begin{enumerate}\setlength{\itemsep}{.05in}\renewcommand{\labelenumi}{\roman{enumi}.}}
	\def\elatin{\end{enumerate}}
\def\Benum{\begin{enumerate}\setlength{\itemsep}{.2in}}
	\def\benum{\begin{enumerate}\setlength{\itemsep}{.05in}}
		\def\eenum{\end{enumerate}}
	\def\bmult{\begin{multline*}}
		\def\emult{\end{multline*}}
	\def\bcenter{\begin{center}}
		\def\ecenter{\end{center}}
	\def\bframe{\begin{frame}}
		\def\eframe{\end{frame}}

\newcommand{\thmref}[1]{Theorem~\ref{thm:#1}}
\newcommand{\prpref}[1]{Proposition~\ref{prp:#1}}
\newcommand{\corref}[1]{Corollary~\ref{cor:#1}}
\newcommand{\lemref}[1]{Lemma~\ref{lem:#1}}
\newcommand{\secref}[1]{Section~\ref{sec:#1}}
\newcommand{\subsecref}[1]{Subsection~\ref{subsec:#1}}
\newcommand{\figref}[1]{Figure~\ref{fig:#1}}

\newcommand{\dfnref}[1]{Definition~\ref{dfn:#1}}

\newcommand{\appref}[1]{Appendix~\ref{app:#1}}

\DeclareMathOperator{\diam}{diam}

\DeclareMathOperator{\tr}{tr}
\DeclareMathOperator{\diag}{diag}

\def\cA{\mathcal{A}}
\def\cB{\mathcal{B}}
\def\cC{\mathcal{C}}
\def\cD{\mathcal{D}}
\def\cE{\mathcal{E}}
\def\cF{\mathcal{F}}
\def\cG{\mathcal{G}}
\def\cH{\mathcal{H}}
\def\cI{\mathcal{I}}
\def\cJ{\mathcal{J}}

\def\cL{\mathcal{L}}

\def\cO{\mathcal{O}}
\def\cP{\mathcal{P}}
\def\cQ{\mathcal{Q}}

\def\cS{\mathcal{S}}
\def\cT{\mathcal{T}}
\def\cU{\mathcal{U}}
\def\cV{\mathcal{V}}
\def\cW{\mathcal{W}}
\def\cX{\mathcal{X}}

\def\cZ{\mathcal{Z}}

\newcommand{\bbeta}{{\boldsymbol\beta}}

\def\bbE{\mathbb{E}}

\def\bbN{\mathbb{N}}

\def\bbP{\mathbb{P}}

\def\bbR{\mathbb{R}}

\def\ind{\mathbbm{1}}

\newcommand{\<}{\langle}
\renewcommand{\>}{\rangle}
\let\lac\{
\let\rac\}
\renewcommand{\{}{\left\lac}
\renewcommand{\}}{\right\rac}
\newcommand{\inner}[2]{\langle #1, #2 \rangle}

\newcommand{\ceil}[1]{\lceil #1 \rceil}

\def\implies{\ \Rightarrow \ }
\def\iff{\ \Leftrightarrow \ }

\def\1{\mathbbm{1}}

\def\({\left(}
\def\){\right)}

\DeclareMathOperator{\op}{op} 
\DeclareMathOperator{\Card}{Card}
\DeclareMathOperator{\Leb}{Leb}

\DeclareMathOperator{\Id}{Id}
\DeclareMathOperator{\Vect}{Vect}
\DeclareMathOperator{\unif}{Unif}

\DeclareMathOperator{\pr}{pr}

\DeclareMathOperator{\supp}{supp}
\DeclareMathOperator{\vol}{vol}

\newcommand{\ve}{\varepsilon}

\newcommand*\diff{\mathop{}\!\mathrm{d}}
\newcommand*\Diff{\mathop{}\!\mathrm{D}}
\newcommand\wt{\widetilde}
\newcommand\wh{\widehat}

\DeclareMathOperator{\an}{an}
\DeclareMathOperator{\iso}{iso}
\newcommand{\hbm}{\cH^{\beta_0,\beta_\bot}_\delta(M,L)}
\newcommand{\hbu}{\cH_{\an}^{\bm\beta}(\cU,L,\zeta)}
\DeclareMathOperator{\hel}{d_H}
\DeclareMathOperator{\dhe}{d_H}
\DeclareMathOperator{\dhed}{d_H^2}
\newcommand*\ball{\mathsf{B}}
\DeclareMathOperator{\DP}{DP}
\DeclareMathOperator{\Invg}{Inv\Gamma}

\DeclareMathOperator{\BMF}{BMF}
\DeclareMathOperator{\Exp}{Exp}
\DeclareMathOperator{\Gibbs}{Gibbs}
\DeclareMathOperator{\pk}{pk}
\DeclareMathOperator{\Beta}{Beta}

\usepackage{xcolor}

\newcommand{\newclem}[1]{{\textcolor{black}{#1}}}

\newcommand{\newju}[1]{\textcolor{black}{#1}}

\newcommand{\newpaul}[1]{{\textcolor{black}{#1}}}

\begin{document}
	
		\title{Estimating a density near an unknown manifold: a Bayesian nonparametric approach}
		
		\author{
		Cl\'ement Berenfeld%
		\footnote{Institut für Mathematik, Universität Potsdam, Germany, \texttt{\url{berenfeld@uni-potsdam.de}}}
		\and
		Paul Rosa%
		\footnote{University of Oxford, Department of Statistics, Oxford, UK, \url{paul.rosa@jesus.ox.ac.uk}}
		\and
		Judith Rousseau%
		\footnote{University of Oxford, Department of Statistics, Oxford, UK, \url{judith.rousseau@stats.ox.ac.uk}}
	}
		\date{}

		\maketitle

		\begin{abstract}
			We study the Bayesian density estimation of data living in the offset of an unknown submanifold of the Euclidean space. In this perspective, we introduce a new notion of anisotropic H\"older for the underlying density and obtain posterior rates that are minimax optimal and adaptive to the regularity of the density, to the intrinsic dimension of the manifold, and to the size of the offset, provided that the latter is not too small --- while still allowed to go to zero. Our Bayesian procedure, based on location-scale mixtures of Gaussians, appears to be convenient to implement and yields good practical results, even for quite singular data. 
		\end{abstract}


	\section{Introduction}
	
	%
	
	\subsection{Manifold density estimation}
	
In many high dimensional statistical problems it is common to consider that the data has an intrinsic low dimensional structure. More precisely, statistics and computer sciences have seen a growing interest in the so-called \emph{manifold hypothesis} where the data is believed to be supported (or near supported) on a low dimensional submanifold $M$ of an ambient space (see \cite{ma2012manifold} for an introduction).  

There are good intuitive reasons to believe that real world data (such as natural images, sounds,  texts, etc) belong to the vicinity of a low dimensional submanifold, often due to physical constraints, see for instance \cite{lee2007nonlinear} or \cite{fefferman2016testing}. Empirical evidence has also been shown in a number of important cases such as texts data sets \cite{belkin2001laplacian}, sounds \cite{klein1970vowel,belkin2001laplacian}, images and videos \cite{weinberger2006unsupervised,VAEBayes} or more recently in Covid data \cite{MengersenCovid}. Analysing such data sets is often called manifold learning (see \cite{ma2012manifold} for an introduction). Manifold learning deal with either nonlinear dimension reduction techniques, manifold estimation or the construction of generative models and the estimation of the distribution on or near an unknown manifold. These  problems are strongly connected. Dimension reduction  consists in finding low dimensional representations of the data. This is typically done by constructing mappings as in Kernel PCA \cite{scholkopf1998nonlinear} or graph based methods such as Isomap \cite{tenenbaum2000global},  Locally Linear Embeddings \cite{roweis2000nonlinear} or Laplacian Eigenmaps \cite{LaplacianEigenmaps}.  Instead of estimating an embedding, the problem of  reconstructing the manifold is another popular aspect of  manifold learning, see   \cite{Genovese_2012, aamari2019nonasymptotic, Divol_minimax, Puchkin2022}   or \cite{dunson2021inferring} among others. Finally the estimation of the distributions on or near manifolds  and the construction of generative models have received recent wide interests in the statistics and machine learning communities, specially with the developments of deep learning algorithms.  There is a growing literature on generative models under the manifold hypothesis with many methodological developments around variational autoencoders (see \cite[Sec 14.6]{DeepLearning} or \cite{VAEBayes}), Generative adversarial networks (see \cite{GANs,GansMethods,WGANs} among others) or recent versions of normalizing flows (see \cite{horvat2021density}). The theoretical results associated to these approaches control the error between the true generative process and the estimated generative models typically under adversarial losses such as the Wasserstein distance,  as in \cite{tang2022minimax}, since the focus is more on generating interesting samples than on estimating the distribution per se. 
\\
	
In this paper we study the estimation of the density in the vicinity of an unknown submanifold $M$ \newclem{of the ambient space $\mathbb R^D$, of dimension $1 \leq d \leq D-1$}. Density estimation is a canonical  problem in statistics and machine learning and in addition to being of interest in itself can be used as an intermediate steps in many tasks of unsupervised or supervised learning such as clustering, prediction and classification, dimensionality reduction  or in ridge estimation, see for instance \cite{genovese2014nonparametric, chen2015asymptotic, mukhopadhyay2020estimating} among many others.
 Density or distribution estimation under the exact manifold hypothesis (assuming that the data belong to a submanifold) has been studied theoretically  for instance in \cite{ozakin2009submanifold} or \cite{divol2021reconstructing} under Wasserstein losses and in \cite{berenfeld2021density} under the pointwise loss.
 Assuming that the data belong exactly to a smooth manifold may be too restrictive since signals are often corrupted by noise. Hence in this paper we assume that the data belong to a neighbourhood $M^\delta = \{ x \in \mathbb R^D; d(x, M) \leq \delta \}$, where neither $M$ nor $\delta$ or the density $f$ are  known. This problem is studied in \cite{chae2021likelihood} in the  special case of data corrupted with Gaussian noise and \cite{mukhopadhyay2020estimating} proposes a Bayesian nonparametric method to estimate a density on $M^\delta$ based on mixtures of Fisher - Gaussian distributions for which they prove consistency under the assumption that the width $\delta$ of the tube $M^\delta$ is fixed.

 \subsection{Our approach and contributions}
 
As far as we are aware there is no theoretical results on convergence rates - either from a frequentist or a Bayesian approach - for estimating a density on tubes $M^\delta$ when $M$  and $\delta$ are unknown and $\delta$ is possibly small. In this paper we bridge this gap and we propose a  Bayesian nonparametric method based on specific families of location-scale  mixtures of Gaussian distributions. We study the posterior concentration rates associated to these priors, i.e. the smallest possible $\ve_n$ such that 
$$ \Pi( d(f_0, f) \leq \ve_n | X_1, \dots , X_n)\rightarrow 1,$$
in probability when the data $X_1, \dots, X_n$ are a $n$ sample from $f_0$ and  where $\Pi( \cdot  | X_1, \dots , X_n)$ denotes the posterior distribution, see \cite{ghosal2000convergence}. As is well known, when the distance $d(.,.)$ is the Hellinger or the $L_1$ metric,  this posterior concentration  rates induces  also a convergence rate $\ve_n$ for the posterior mean $\hat f$, see for instance \cite{ghosal2000convergence}. Typcally the rate $\ve_n$ depends on regularity properties of the density $f_0$ and on the prior. 

To do so we first define a general mathematical  framework describing regularity properties of densities defined on possibly small neighbourhoods $M^\delta$ of submanifolds $M$, with the idea that the density has a given smoothness $\beta_0$ along the manifold $M$ and another smoothness $\beta_\perp$ along the normal to the manifold. This manifold driven anisotropic smoothness is defined in Section \ref{subsec:manihold} and is an extension to anisotropic H\"older functions along coordinate axes. \newju{Because the corruption of a signal by some small noise can be modelled in many different ways, with additive homoscedastic noise being the simplest and most restrictive one, our framework only assumes that the density lives on $M^\delta$ and does not model the noise.} As shown in Section \ref{sec:model}, our framework encompasses common models of noisy data. 

Building on that we show that the posterior concentration rate depends on $\beta_0, \beta_\perp$ together \newclem{with the dimension $d$ of $M$} and the width $\delta$ of the tube. Interestingly the prior $\Pi$ does not need to depend on $\beta_0, \beta_\perp, d, \delta$ or $M$ which makes the approach fully adaptive and the rate we obtain, at least when $\delta$ is not too small is of order (up to a polylog term)
$$n^{-\gamma}~~~\text{with}~~~\gamma = \frac{\beta_0}{2\beta_0 + d + (D-d)\beta_0/\beta_{\perp}},$$ 
\newclem{where we recall that $D$ is the dimension of the ambient space,} and is minimax. 
\\

	Nonparametric location mixtures of Gaussians are known to be flexible models for densities, and adaptive minimax rates of convergence on H\"older types spaces have been obtained using Bayesian or frequentist estimation procedures based on location mixtures of normals, see \cite{kruijer2010adaptive, shen2013adaptive}, and \cite{ghosal2007posterior} for Bayesian methods and \cite{maugis2013adaptive} for a penalized likelihood approach. However, location mixtures are not versatile enough since the covariance matrix remains fixed across the components, so we instead take advantage of the flexibility of location-scale mixtures of Gaussians. In   \cite{canale2017posterior} the authors derive a suboptimal posterior concentration rate for isotropic positive H\"older  densities on $\mathbb R^D$, while \cite{maugis2013adaptive} obtained minimax convergence rates for penalized maximum likelihood methods based on the same type of location-scale mixtures and  \cite{naulet2017posterior} obtained also  minimax  posterior concentration rates using a hybrid location-scale mixture prior in the regression model. These results thus indicate that one has to be careful in designing the prior in nonparametric location-scale mixtures of Gaussians. The priors we consider in this paper are variants of location-scale mixture priors, see Section \ref{sec:prior}, which are flexible  enough to adapt to the non linear or manifold driven smoothness of the class of densities studied here. This prior construction can also be seen as tiling the manifold by low-rank \emph{Gaussian pancakes}, a method that is similar to mixtures of factors analyzers \cite{ghahramani1996algorithm,chen2010compressive} or manifold Parzen windows \cite{vincent2002manifold} where, however, no theoretical guarantees on the estimation of the density were proven.
\\

Hence our contributions are both methodological and theoretical.  From a methodological point of view, we provide with a family of versatile priors (see Section \ref{sec:model}) that are shown empirically and theoretically to perform very well in modelling data that are singularly supported near submanifolds. In particular we show empirically that these variants of location-scale mixtures of Gaussian priors behave much better than the standard \textit{conjuguate} location-scale mixture of Gaussian prior, see Section \ref{sec:num}. From a theoretical point of view, we introduce a new notion of H\"older smoothness along a submanifold (see Section \ref{subsec:manihold}) which is proving to be adequate for the study of such almost-degenerate densities, and we derive posterior concentration rates for this new model  (Section \ref{subsec:main}). The rates are optimal if the data do not collapse too quickly towards the manifold. These results rely on an intermediate result in approximation theory which has interests in its own right and is provided in Section \ref{subsec:approx}. 

\subsection{Organisation of the paper}

In Section \ref{sec:model}, we define  manifold-anisotropic H\"older function, together  with the families of priors we consider in the paper. Section \ref{sec:theoretical} contains the main theoretical  results and  \secref{num} the empirical, numerical results. We provide in Section \ref{sec:mainproof} proofs of the main results, namely the contraction rate Theorem \ref{thm:main} and the approximation Theorem \ref{thm:approx}. Some useful facts on manifolds are presented in Appendix \ref{app:A}. The rest of the Appendices contains additional proofs and lemmata, as well as details on the numerical setting of \secref{num}. 
	
	\subsection{Notations}
	For a multi-index $k = (k_1,\dots,k_D) \in \bbN^D$, we set $|k| = k_1 + \dots + k_D$ and $k! = k_1! \dots  k_D!$. For $x \in \bbR^D$, we write $x^k = x_1^{k_1} \dots x_D^{k_D} \in \bbR$ and $x_{\max}$ (resp. $x_{\min}$) to be the maximal (resp. minimal) value of its entries. For any two indices $i,j \in \{1,\dots,D\}$ with $i \leq j$, we set $x_{i:j} = (x_i,\dots,x_j) \in \bbR^{j-i+1}$. Finally, for a sufficiently regular function $f : \bbR^D \to \bbR$, we define its $k$-th partial derivative as
	$$
	\Diff^k f(x) = \frac{\partial^{|k|}f}{\partial x_1^{k_1}\dots \partial x_D^{k_D}} (x).
	$$ 
	
	When $M \subset \bbR^D$ is a measurable subset with Hausdorff dimension $d$,  $\mu_M$ denotes the Borel measure $
	\mu_M = \cH^d(\cdot \cap M)$ where $\cH^d$ is the $d$-dimensional Hausdorff measure on $\bbR^D$. For $r > 0$ and $x \in \bbR^D$, let $\ball_M(x,r) = \ball(x,r)\cap M$ where $\ball(x,r)$ is the usual Euclidean ball of $\bbR^D$. If $M$ is closed, then $\pr_M$ defines the (possible multi-valued) orthonormal projection from $\bbR^D$ to $M$. 
	\\
	
	We will denote by $\|\cdot\|$ the usual Euclidean norm of $\bbR^k$ for any $k \in \bbN^*$. When $\cL$ is a linear map between such spaces, we write $\|\cL\|_{\op}$ for the operator norm associated with the Euclidean norms. The notation $\|\cdot\|_1$ (resp. $\|\cdot\|_\infty$) will refer to both the $L^1$-norm (resp $\sup$-norm) for vectors of $\bbR^k$ for any $k \in \bbN^*$, and to the $L^1$-norm (resp $\sup$-norm) for measurable functions from $\bbR^k$ to $\bbR$ for any $k \in \bbN^*$. The brackets $\inner{\cdot}{\cdot}$ will be used to denote the usual Euclidean product in $\bbR^k$ for any $k\in\bbN^*$. For any matrix $A \in \bbR^{k\times k}$, the notation $\|\cdot\|_A^2$ will refer to the quadratic form over $\bbR^k$ defined by $x\mapsto \inner{Ax}{x}$, which is a squared norm if $A$ is positive definite. The set of orthogonal transform of $\bbR^D$ will be denoted by $\cO(D,\bbR)$, or sometimes simply $\cO(D)$.
	\\
	
	For two positive functions $f,g : \bbR^D \to \bbR$ we write the Hellinger distance as 
	$$
	\hel(f,g) = \{\int_{\bbR^D} (\sqrt{f(x)} - \sqrt{g(x)} )^2 dx \}^{1/2}.
	$$
	
	In this paper, $M$ will designate a closed submanifold of $\bbR^D$ of dimension $1 \leq d \leq D-1$. For any point $x \in M$, the tangent and normal spaces of $M$ at $x$ will be denoted $T_x M$ and $N_x M$, and the corresponding bundles $TM$ and $NM$. We write $\exp_{x} : (T_x M, 0) \to (M,x)$ for the exponential map of $M$ at point $x$. We let $d_{M}(x,y)$ denote the intrinsic distance between $x$ and $y$ in $M$.	
	\\
	
	Finally we will use throughout the symbols $\simeq$, $\lesssim$ and $\gtrsim$ to denote equalities or inequalities up to a constant, when the constant is not important. 
	
		\section{Model : distributions concentrated near manifolds}  \label{sec:model}
	We assume that we observe $X_1, \dots, X_n$ independent and identically distributed from $P_0$ on $\mathbb R^D$ with density $f_0$ with respect to Lebesgue measure. We assume that there is a low dimensional structure underlying our observations, i.e. that $f_0$ has support concentrated near a low dimensional manifold $M$ which is unknown. More precisely there exists $\delta>0$ unknown and typically small such that $P_0(M^\delta)=1$, where $M^\delta$ is the $\delta$-offset of $M$: it is the set of points that are at distance less than $\delta$ from $M$,
	$$
	M^\delta := \bigcup_{x\in M} \ball(x,\delta) = \{z \in \bbR^D~|~d(z,M) \leq \delta\}.
	$$
A typical example is when the observations are noisy versions of data whose support is $M$: $X=Y + Z$ with $Y \in M$ and $|Z|\leq \delta$ almost surely. 
	\newclem{We would like to underline here that $\delta$ is to be understood as $\delta = \delta_n$ which can either stay constant (and smaller than the reach of $M$, defined in Appendix \ref{app:A}) or goes to $0$ as $n$ goes to infinity, in which case $f_0 = f_{0,n}$ and $X_j = X_{j,n}$ is a triangular array of data. This dependence on $n$ for $f_{0,n}$ does not impact our results and  in order to make the notation lighter, we intentionally drop the extra $n$ from the indices, and will recall this dependency at any time when it might create confusion.}  
	\newju{ When $\delta$ is fixed, the density is not degenerate but its support has a non trivial geometry. Allowing $\delta = \delta_n = o(1)$ 
 is a way to model very concentrated densities (for a given $n$) and is similar in spirit to high dimensional regression where the number of covariates $p$ is modelled as a function of $n$. In other words it is a way to take into account the impact of $\delta$ in the concentration properties of the posterior distribution. }
	When the noise $Z$ has a density  smoother than the density of $Y$ on $M$ (with respect to the Hausdorff measure),  the density of $X$ is anisotropic with a smoothness along the manifold $M$ smaller  than that along the normal directions. In this paper we thus aim at constructing priors which are flexible enough  to lead to \textit{good} estimation of $f_0$ in situations where the density has a complex  anisotropic structure in that it has an unknown smoothness $\beta_0$ along an unknown manifold $M$ and a different (larger) smoothness $\beta_\perp$, also unknown, along the normal spaces of the manifold. In this context, since the anisotropy varies spatially, it is therefore important to consider priors which adapt spatially to such \textit{non linear smoothness}. In Section \ref{sec:prior} below we consider two families of location-scale mixtures with a new and careful modelling of the prior on the variance of the components and we show in Section \ref{sec:theoretical} that these priors are well behaved for densities  with manifold driven smoothness.
	\\
	

	To begin with, we define what we think is a new notion of anisotropic H\"older spaces on the Euclidean space $\bbR^D$, and which happens to be a natural extension of the usual notion of (isotropic) H\"older smoothness. We are aware that there exist various notions of anisotropic smoothness, see for instance \cite{kerkyacharian2001nonlinear, hoffman2002random, comte2013anisotropic, goldenshluger2011bandwidth, goldenshluger2014adaptive}, with most of them stemming from the anisotropic smoothness as defined in \cite{nikol2012approximation}. In all the aforementionned references, the anisotropy was consistently defined as a control of the variations of the partial derivatives along each axis separately, with no control of the cross-derivatives  (and no guarantee that they, in fact, exist). While this is enough in a Euclidean framework, we argue that, to the best of our effort, we could not make such assumptions sufficient in our non-linear setting, as the proofs presented in \secref{mainproof} or in the Appendices might highlight. Instead, we come out with a new notion of H\"older anisotropy, in the footsteps of what \cite{shen2013adaptive} already sketched in their paper, that handles cross-derivatives in the same way that the usual notion of (isotropic) H\"older smoothness does, and which in fact coincides with the latter when the anisotropy vector is isotropic. This new class is defined in the following subsection  and its main properties reviewed in Section \ref{app:auxhold}. We would also like to mention the notion of mixed smoothness as introduced in \cite{cleanthous2019minimax}, which is a stronger notion of regularity than ours: the classes defined below can be seen as nested between the ones of \cite{nikol2012approximation} and \cite{cleanthous2019minimax}.
	
	\subsection{General anisotropic H\"older functions} \label{subsec:anis}
	An anisotropic H\"older functions $f : \bbR^D \to \bbR$ is, informally, a function whose smoothness is different along each axis of $\bbR^D$. Letting $\bm{\beta} = (\beta_1,\dots,\beta_D) \in (\bbR_+^*)^D$, which will represent the regularity indices along each axis, we define
	\beq \label{eq:defalpha}
	\bm\alpha = (\alpha_1,\dots,\alpha_D)~~~\text{where}~~~\alpha_i = \beta/\beta_i \in [0,D]~~~\text{and}~~~ \beta^{-1}= \frac1D \sum_i \beta_i^{-1}.
	\eeq
	The coefficient $\beta$ acts as the effective smoothness of the function $f$. Notice that $\alpha_1 + \dots + \alpha_D = D$. In this section, we define the spaces of anisotropic functions over bounded open subset of $\bbR^D$. We defer to Section \ref{app:auxhold} the introduction of the same class over general open subsets. We let $\cU \subset \bbR^D$ be a bounded open subset and $L : \cU \to \bbR_+$ be any non-negative function. 
	\begin{dfn}\label{dfn:anihold}
		The anisotropic H\"older spaces $\cH_{\an}^{\bm\beta}(\cU,L)$ is the set of all functions $f : \cU \to \bbR$ 
		 satisfying:
		\bitem
		\item[i)] For any multi-index $k \in \bbN^D$ such that $\inner{k}{\bm\alpha} < \beta$, the partial derivative $\Diff^k f$ is well defined on $\cU$ and
		$|\Diff^k f(x)| \leq L(x)$ for all $x \in \cU$; 
		\item[ii)] For any multi-index $k \in \bbN^D$ such that $\beta - \alpha_{\max} \leq \inner{k}{\bm\alpha} < \beta$, there holds
		\beq
		\label{eq:hold} 
		|\Diff^k f(y) - \Diff^k f(x)| \leq L(x) \sum_{i=1}^D |y_i - x_i|^{\frac{\beta- \inner{k}{\bm\alpha}}{\alpha_i} \wedge 1}~~~ \forall x,y \in \cU.
		\eeq
		\eitem
	\end{dfn}
	See \figref{beta} for a graphical representation of the quantities at stake. The function $L$ acts as an upper-bound for the localized and anisotropic version of the usual H\"older-norm:
	$$
	\{\max_{\inner{k}{\bm\alpha} < \beta} |\Diff^k f(x)|\} \vee \max_{\beta - \alpha_{\max} \leq\inner{k}{\bm\alpha} < \beta} \sup_{y \in \cU} \frac{|\Diff^k f(y) - \Diff^k f(x)|}{\displaystyle\sum_{i=1}^D |y_i - x_i|^{\frac{\beta- \inner{k}{\bm\alpha}}{\alpha_i}\wedge 1}} . 
	$$
	\begin{figure}[!h]
		\centering
		\includegraphics[width = 6cm]{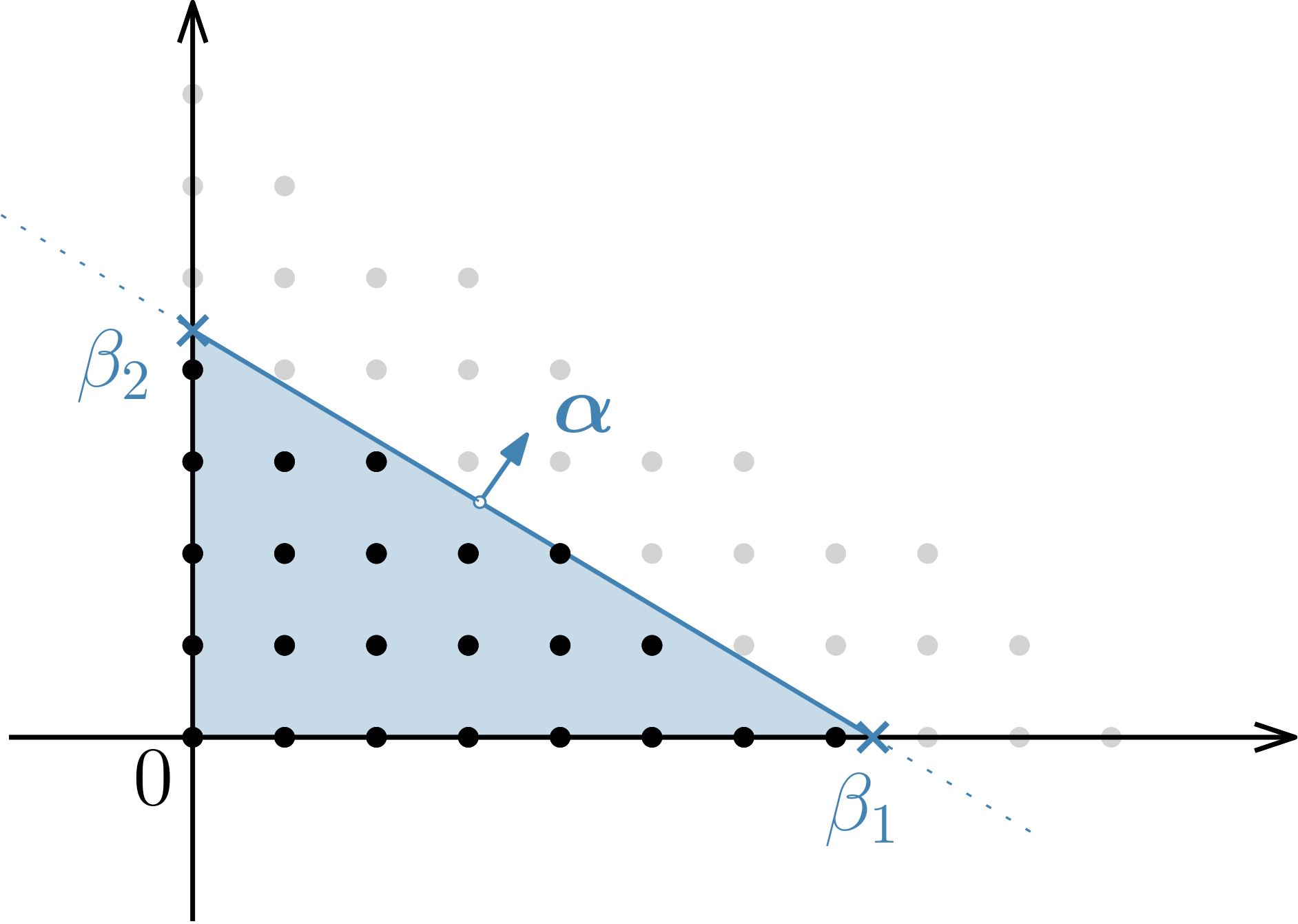}  \hspace{20pt}
		\includegraphics[width = 6cm]{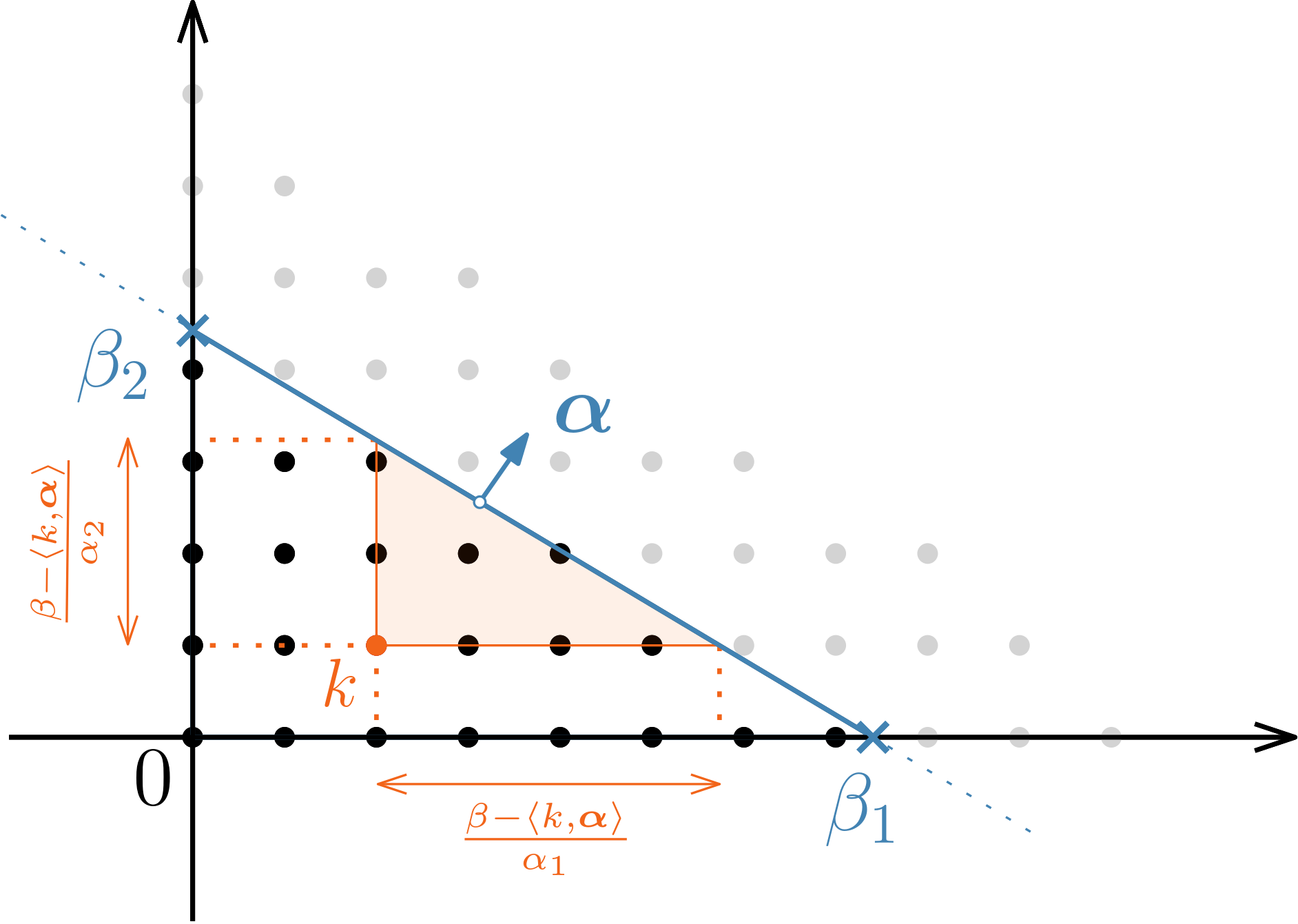} 
		\caption{An exemple in dimension $D = 2$.  The vector $\bm\alpha$ is the only vector of $1$-norm $D$ which has positive coordinates and which is orthogonal to the simplex of vertices $\{\beta_i e_i\}_{1 \leq i \leq D}$. In black are the points $k$ of $\bbN^2$ such that $\inner{k}{\bm\alpha} < \beta$.}
		\label{fig:beta}
	\end{figure}
	
	Note that  constraint (i) on the intermediate derivatives $\Diff^k f(x)$ for $0 < \inner{k}{\bm\alpha} < \beta$ may seem superfluous since some Kolmogorov-Landau type inequalities would yield some bounds on these derivatives, but we add them nonetheless to our functional class to simplify some notations. We list and prove in Section \ref{app:auxhold} various useful properties of functions in the anisotropic H\"older class. Also the function $L$ in the definition of $\cH^{\bm\beta}_{\an}(\cU,L)$ can be constant, in which case we will typically denote it $C$, to make it more explicit  (leading to $\cH^{\bm\beta}_{\an}(\cU,C)$).

\begin{rem}The usual isotropic H\"older spaces are  special cases of our definition of $\cH^{\bm\beta}_{\an}(\cU,L)$ corresponding to $\bm\beta = (\beta,\dots,\beta)$ with $\beta > 0$. In this case we write
	$$
	\cH^{\beta}_{\iso}(\cU,L) :=  \cH^{\bm\beta}_{\an}(\cU,L)~~~\text{for}~~~\bm\beta = (\beta,\dots,\beta).
	$$
\end{rem}

		As a final remark, we will use the same notations for the spaces of multivalued functions when their coordinate functions are all in the corresponding space. For instance, if $\Psi : \cU \to \bbR^{D_1}$, for $D_1 >1$,  then
	$$
	\Psi = \(\Psi_1,\dots,\Psi_{D_1}\right) \in \cH^{\bm\beta}_{\an}(\cU,L) ~~\underset{\text{def}}{\iff} ~~ \Psi_i \in \cH^{\bm \beta}_{\an}(\cU,L)~\text{for all}~i \in \{1,\dots,D_1\},
	$$ 
	and the same holds for the other spaces defined in this subsection.
	
	\subsection{Manifold anisotropic H\"older functions} \label{subsec:manihold}
	We now consider functions whose smoothness directions at point $x \in \bbR^D$ are dependent on the position of $x$ with respect to a given submanifold $M \subset \bbR^D$ of dimension $1 \leq d \leq D-1$. More specifically, we extend the above notions of anisotropy to functions with a given  regularity in the tangential directions of $M$, and of another regularity in the normal directions of $M$.  We call such functions \emph{manifold-anisotropic H\"older}, or sometimes simply \emph{M-anisotropic}. To define such a class of functions, we assume that $M$ is a closed submanifold with reach bounded from below by $\tau > 0$ (see \appref{A} for definition and properties of the reach) and we consider local parametrizations at any $x_0 \in M$
	$$
	\Psi_{x_0} : \cV_{x_0} \to M,
	$$
	where $\cV_{x_0}$ is a neighborhood of $0$ in $T_{x_0} M$. The maps $\Psi_{x_0}$ can be taken in a wide class of parametrizations of $M$. For instance, one could consider taking $\Psi_{x_0}$ to be (close to) the inverse projection over $M \to T_{x_0} M$ where $T_{x_0} M$ is seen as an affine subspace of $\bbR^D$ going through $x_0$, see for instance \cite{aamari2019nonasymptotic} or \cite{divol2021reconstructing}. For purely practical matter, we choose $\Psi_{x_0}$ to be the exponential map $\exp_{x_0}$, although the results in this paper could be carried out with other well-behaved parametrizations, such as the one mentioned above. 
	In particular,  in the case of the exponential maps,  we can define the domain of $\Psi_{x_0}$ to be $\ball_{T_{x_0} M}(0,\pi\tau)$, see \appref{A}. In the rest of this paper, we set
	$$
	\cV_{x_0} := \ball_{T_{x_0} M}(0,\tau/8),
	$$
	with factor $1/8$ being there for technical reasons. If all the maps $\Psi_{x_0}$ are of regularity $\beta_M > 1$, meaning that there exists  a constant $C_M > 0$ such that 
	\beq \label{betam}
	\Psi_{x_0} \in \cH_{\iso}^{\beta_M}(\cV_{x_0}, C_M),~~\forall x_0 \in M
	\eeq
	(in particular $M$ is at least $\cC^{k}$ with $k =\ceil{\beta_M-1}$), then one can construct a map 
	$$
	\bar \Psi_{x_0} : \begin{cases} \cV_{x_0} \times N_{x_0} M &\to \bbR^D \\
		~~~~(v,\eta) &\mapsto \Psi_{x_0}(v) + N_{x_0}(v,\eta).
	\end{cases}
	$$
	where $N_{x_0}(v,\cdot)$ is an isometry from $N_{x_0} M$ to $N_{\Psi_{x_0}(v)} M$ and where
	$
	v \mapsto N_{x_0}(v,\cdot) \in \cH_{\iso}^{\beta_M -1}(\cV_{x_0}, C_M^\perp)
	$
	for some other constant $C_M^\perp$ depending on $C_M$, $\tau$ and $\beta_M$. We refer to \appref{A} for further details concerning the construction of $\bar\Psi_{x_0}$ and the proof of its regularity. When restricting the latter map, one gets a local parametrization of the offset $M^{\tau/2}$ around $x_0$
	$$
	\bar\Psi_{x_0} : \cV_{x_0} \times \ball_{N_{x_0} M}(0,\tau/2) \to M^{\tau/2}
	$$
	as shown in Lemma \ref{lem:defNB}. This parametrization is such that $\pr_M ( \bar\Psi_{x_0}(v,\eta)) = \Psi_{x_0}(v)$ for any $(v,\eta) \in \cV_{x_0} \times \ball_{N_{x_0} M}(0,\tau/2)$ and $\bar\Psi_{x_0}$ is  a diffeomorphism from  $\cV_{x_0} \times \ball_{N_{x_0} M}(0,\tau/2) $ to its image which satisfies 
	\beq \label{barpsi}\bar\Psi_{x_0} \in \cH^{\beta_M-1}_{\iso}( \cV_{x_0} \times \ball_{N_{x_0} M}(0,\tau/2), C^*_M)
	\eeq 
	for some $C^*_M > 0$ depending on $C_M$, $\tau$ and $\beta_M$. See \figref{psix0} for a visual interpretation of this parametrizations.
	\\
	
	\begin{figure}[t!]
		\centering
		\includegraphics[width=12cm]{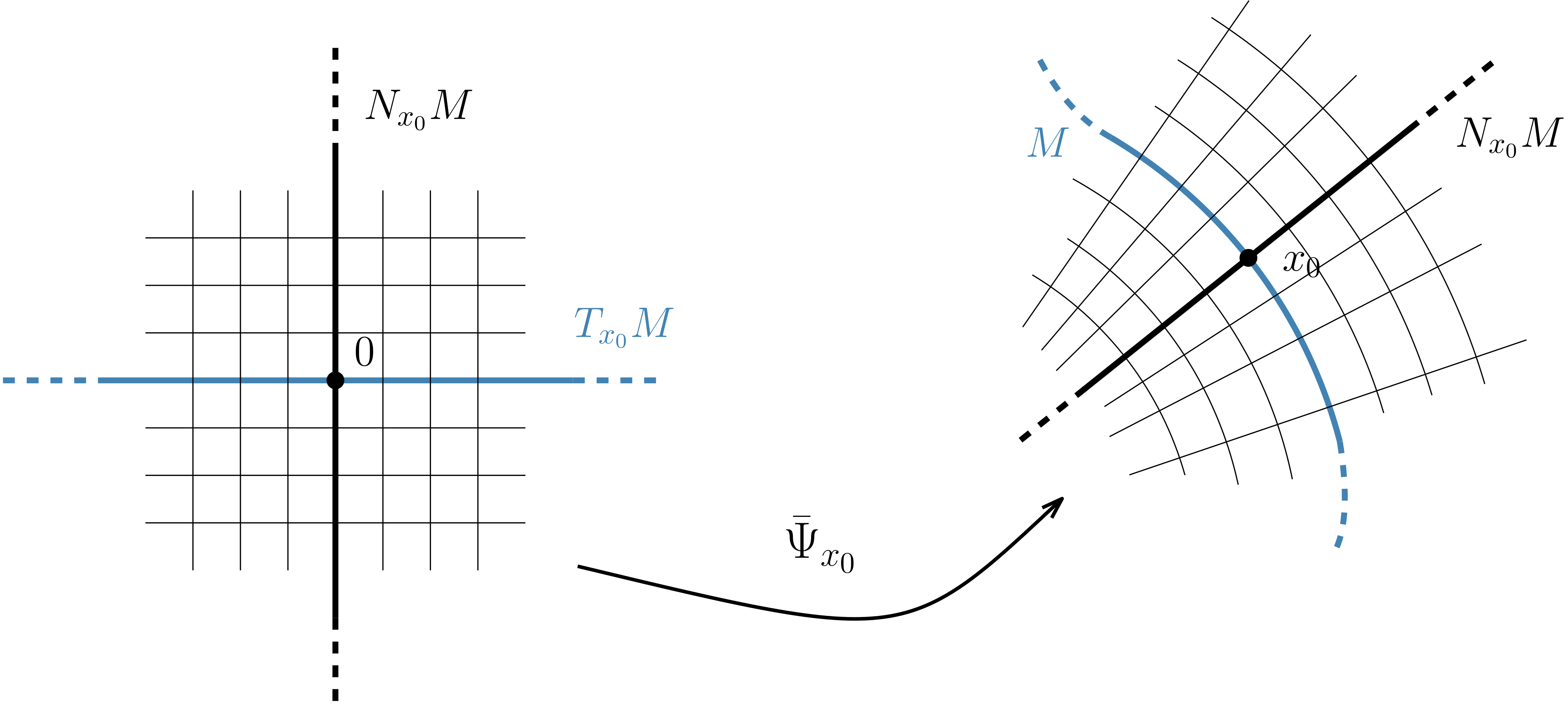} 
		\caption{A visual interpretation of the parametrization $\bar\Psi_{x_0}$.}
		\label{fig:psix0}
	\end{figure}
	
	For any $\delta > 0$, we define $\bar\Psi_{x_0,\delta}(v,\eta) := \bar\Psi_{x_0}(v,\delta \eta)$ to be the rescaled version of $\bar\Psi_{x_0}$ in the normal directions. It is a well defined parametrization of $M^{\tau/2}$ on the set $\cW_{x_0,\delta} := \cV_{x_0} \times \ball_{N_{x_0} M}(0,\tau/2\delta)$. We let $\beta_0, \beta_\bot$ be two positive real numbers, and define the vector 
	$$
	\bm\beta_{0,\perp} = (\underbrace{\beta_0,\dots,\beta_0}_{d},\underbrace{\beta_\bot,\dots,\beta_\bot}_{D-d}) \in \bbR^D. 
	$$
	\newclem{
	and define $\bm\alpha$ to be the counterpart of $\bm\beta_{0,\perp}$ as in \eqref{eq:defalpha}:
	\beq \label{eq:defalpha0}
	\begin{split}
	\bm\alpha = (\alpha_0,\dots,\alpha_0,\alpha_\bot,\dots,\alpha_\bot) ~~~~&\text{with}~~~\alpha_0 = \beta/\beta_0,~~\alpha_\perp = \beta/\beta_\perp, \\
	~~&\text{and}~~~\beta^{-1} = \frac{d}D \beta^{-1}_0+\frac{D-d}{D}\beta^{-1}_\perp.
	\end{split}
	\eeq
	}
	Now for any function $L : \bbR^D \to \bbR_+$, we define:
	\begin{dfn}\label{def:defmanihold}  Let $L : \bbR^D \to \bbR_+$ be a function; the class
		$
		\cH^{\beta_0,\beta_\bot}_\delta(M,L)
		$
		is the set of all functions $f : \bbR^D \to \bbR$ which satisfy:
		\benum
		\item[i)] $f$ is supported on $M^\delta$;
		\item[ii)] For any $x_0 \in M$, set $ \bar f_{x_0,\delta} := \delta^{D-d} f \circ \bar\Psi_{x_0,\delta}$ and $L_{x_0,\delta} :=  \delta^{D-d} L \circ \bar\Psi_{x_0,\delta}$,  then 
		\beq\bar f_{x_0,\delta} \in \cH^{\bm\beta_{0,\perp}}_{\an}(\cW_{x_0,\delta},  L_{x_0,\delta}). \label{eq:mholder}
		\eeq
		\eenum
	\end{dfn}
Informally, such a function  is  $\beta_0$-H\"older along the manifold $M$, and $\beta_\perp$-H\"older normal to the manifold $M$. The normalization $\delta^{D-d}$ accounts for the scaling $\eta \mapsto \delta \eta$ along the normal spaces (which are of dimension $D-d$) in the definition of $\bar\Psi_{x_0,\delta}$. Its presence is natural and can be understood as follows: when $f$ is a density supported on $M^\delta$, the typical magnitude of its values is of order $1/\delta^{D-d}$, and the absence of normalization would whence make the above functional class irrelevant to describe the regularity of such densities. 
\newju{\begin{rem}\label{stabi:def:main}
Strictly speaking  $\cH^{\beta_0,\beta_\bot}_\delta(M,L)$ depends on the choice of the parametrizations $\bar \Psi$ and should be written as  $\cH^{\beta_0,\beta_\bot}_\delta(M,L, \bar \Psi)$.  We show however in Proposition \ref{prop:stabilityPsi} that 
$$\cH^{\beta_0,\beta_\bot}_\delta(M,L, \bar \Psi) \subset \cH^{\beta_0,\beta_\bot}_\delta(M,C_1L,  \bar \Psi') \subset \cH^{\beta_0,\beta_\bot}_\delta(M,C_2L, \bar \Psi)$$ 
for some constant $C_1, C_2>0$ independant of $\Psi$ and $\Psi'$, as soon as $\bar \Psi'$ satisfies \eqref{barpsi} also. Hence throughout the paper we denote this H\"older space $\cH^{\beta_0,\beta_\bot}_\delta(M,L)$ .
\end{rem}
}


	M-anisotropic functions happen to be a convenient way to describe the regularity of a number of densities that are naturally supported around $M$. To illustrate this, take $f_\ast : M \to \bbR$ to be a $\beta_0$-H\"older density, meaning that there exists $L_0 : M \to \bbR$ such that for any $x_0 \in M$, 
	$$
	f_\ast \circ\Psi_{x_0} \in \cH^{\beta_0}_{\iso}(\cV_{x_0}, L_0\circ\Psi_{x_0}).
	$$
	Now take $K : \bbR^D \to \bbR$ to be a normalized positive smooth isotropic kernel supported on $\ball(0,1)$. We introduce 
	$
	c_\bot^{-1} = \int K(\ve) \diff \mu_E(\ve) 
	$ 
	where $E$ is any (through isometry) $(D-d)$-dimensional subspaces of $\bbR^D$. We also assume that 
	$K \in \cH^{\beta_\perp}_{\iso}(\bbR^D,L_\perp) $
	for some function $L_\perp$ which is also rotationally invariant. 
	\begin{prp} \label{prp:model}
		Let $f$ be the density of a random variable $Z = X + \delta \cE$ where $X \sim f_*(x) \mu_M(\diff x)$ and $0 < \delta < \tau$. Then,
		\benum
		\item \emph{(Orthonormal noise)} If $\beta_\perp \leq \beta_M - 1$, and if
		$\cE | X \sim c_{\bot} K(\ve) \mu_{N_X M}(\diff \ve),$
		then $$f \in \hbm, \quad  \text{with }, \quad 
		L (x) := C \delta^{-(D-d)} L_0(\pr_M x)  \times L_\perp(x-\pr_M x), \quad C>0.$$
		\item \emph{(Isotropic noise)} If $\delta < \tau/32$ and $\beta_0 \leq \beta_\perp \leq \beta_M - 1$, and if 
		$\cE \sim K(\ve) \diff \ve$, independently of $X,$
		then $$f \in \hbm,\quad \text{ with }\quad 
		L(x) := C \delta^{-D} \int_{M}  L_\perp\(\frac{x-y}{\delta}\right) L_0(y) \mu_M(\diff y),
		\quad C>0.$$

		\eenum
		In both cases  $C$ depends on $C_M$, $\tau$, $\beta_M$, $\beta_{0}$, $\beta_{\bot}$.
	\end{prp}
	See Section \ref{app:manihold} for a proof of this result.
	\begin{rem} Throughout the paper we assume that the true density $f_0 \in  \hbm$, which implies that $P_0(M^\delta)=1$. However it is enough to assume that  $P_0(M^\delta)\geq 1- o(1/n)$ where $n$ is the number of observations. This makes no difference in terms of the  results presented in  \secref{theoretical}.  The weaker assumption $P_0(M^\delta)\geq 1- o(1/n)$ is for instance fulfilled in the additive noise model (see \prpref{model})  with $Z = X + \frac{ \delta }{ \sqrt{ C\log n } } \cE$ with the Gaussian kernel $K(x) := (2\pi)^{D/2} \exp(-\|x\|^2/2)$.
	\end{rem}
\begin{rem} \label{rem:Holderconstant-delta}\newju{Note that under the model  $Z = X + \delta \cE$ with isotropic noise and $K$ the distribution of the noise belonging to $C^{\beta_\perp}$ with $\|K\|_{C^{\beta_\perp}} \lesssim 1$, then $f \in \hbm \cap C^\beta$, with $\|f\|_{C^{\beta_\perp}} \lesssim \delta^{-\beta-D+d}$. Hence for H\"older constant is strongly impacted by $\delta $ and $D$. The geometric framework of $M$ anisotropy,  allows to handle this dependence on $\delta$ and is therefore useful for such models when $\delta$ is small.}  
		\end{rem}
	In the following section we describe the family of priors which we use to estimate the above family of densities. 
	
	\subsection{Location-scale mixtures of normal priors}  \label{sec:prior}
	
	We model the manifold-anisotropic H\"older densities using location-scale mixtures of normals. We parametrize the covariances of the components by $\Sigma = O^T\Lambda O$ where $O$ is a unitary matrix and $\Lambda = \diag(\lambda_1, \cdots, \lambda_D)$ is diagonal. Location-scale mixtures can then be written as: 
	\begin{equation} \label{loc-scale-mix}
		f_{P}(x) = \int_{\bbR^D} \varphi_{O^T\Lambda O}(x-\mu)\diff P(\mu, O, \Lambda),\quad  P = \sum_{k=1}^K p_k \delta_{(\mu_k, O_k, \Lambda_k)}, \quad K \in \mathbb N \cup \{ +\infty\},
	\end{equation}
	where, for any positive definite matrix $\Sigma$, 
	$$
	\varphi_\Sigma(z) := \frac{1}{\det^{1/2} \(2\pi\Sigma\)} \exp\{-\frac{1}{2} \|z\|^2_{\Sigma^{-1}}\},
	$$
	is the density of a centered Gaussian with covariance matrix $\Sigma$. The two most well known families of priors on $P$ are Dirichlet process priors and mixtures with random number of components, also known as mixtures of finite mixtures. Recall that if $P$ follows a Dirichlet process priors with parameters $A$ and $H$ where $A > 0$ and $H$ is a probability measure on some measurable space $\Theta$ , then
	\begin{equation*} 
		P = \sum_{k=1}^\infty p_k \delta_{ \theta_k}~~~\text{with}~~~ p_k = V_k\prod_{i<k}(1-V_i), \quad V_i\stackrel{iid}{\sim} \Beta(1, A) \quad \text{and} \quad \theta_k \stackrel{ iid}{ \sim} H .
	\end{equation*}
	If $P$ follows a mixture of finite mixtures prior of parameters $\alpha_K$ and $\pi_K$ where $\alpha_K > 0$ and $\pi_K$ is a probability measure on $\bbN$, then 
	$$P = \sum_{k=1}^K p_k \delta_{ \theta_k}~~~\text{with}~~~K \sim \pi_K,\quad (p_1, \cdots, p_K) ~| K \sim \cD( \alpha_K, \dots , \alpha_K) \quad \text{and} \quad  \theta_k  \stackrel{iid}{\sim } H.$$
	In both cases (Dirichlet process and mixture of finite mixtures) we call $H$ the base probability measure. 
	Obviously in the case of mixtures of finite mixtures the conditional prior on $(p_1, \dots, p_K)$ and $\theta_1, \dots, \theta_K$ could be different but we  consider this setup for the sake of simplicity. \\
	
	As shown empirically  by \cite{mukhopadhyay2020estimating} location-scale Dirichlet process mixtures with base measure constructed from the conjuguate prior of the Gaussian model are not well adapted to the problem at hand. We show however that if  particular care is given to the choice of $H$, the posterior on manifold-anisotropic density is well behaved. In particular we consider the two following types of location-scale mixtures:
	\begin{itemize}
		\item \emph{Partial location-scale mixtures}: The eigenvalues $\Lambda$ of the covariance of the Gaussians are common accross components, 
		\begin{equation}\label{partialprior}
			f_{\Lambda,P}(x) = \int_{\bbR^D} \varphi_{O^T\Lambda O}(x-\mu)\diff P(\mu, O) ,\quad  P = \sum_{k=1}^K p_k \delta_{(\mu_k, O_k)}, \quad K \in \bbN \cup \{ +\infty\}
		\end{equation}
		where $P$ is a probability distribution on $\mathbb R^D \times \mathcal O(D)$ (where $\mathcal O(D)$ is the set of unitary matrices in $\mathbb R^D$) and is either a Dirichlet process prior or a mixture of finite mixtures. 
		\item \emph{Hybrid location-scale mixtures}:
		The density $f_P$ is written as \eqref{loc-scale-mix} where $P$ conditionally on a probability $Q_2$ on $\mathbb R_+^D$ follows a Dirichlet process mixture or a mixture of finite mixtures with base measure $H_0 (\diff\mu, \diff O, \diff\lambda)= H_1(\diff\mu, \diff O) \otimes Q_2(\diff\lambda)$, and $Q_2$ follows a distribution $\wt{\Pi}_\Lambda$. 
	\end{itemize}

	We denote by  $\Pi$ the prior on the parameter and we consider the following assumptions on $\Pi$. These conditions differs wether $\Pi$ is assumed to come from a partial location-scale mixture prior or a hybrid location-scale mixture prior.
	\\
	
	\noindent
	\textbf{Conditions on the partial location-scale mixtures.} $f$ is modelled as in \eqref{partialprior} and $P$ follows either a Dirichlet process with base measure $H$ or a mixture of finite mixtures with base measure $H$ and prior on $K$ satisfying
	 \beq\label{eq:finitemixt}
	 -\log \Pi_K( K = x ) \simeq x (\log x)^r , \quad r =0,1.\eeq 
	 Here $r=0$ corresponds for instance to the geometric  and $r=1$ to the Poisson priors on $K$.
	The base measure $H(\diff\mu, \diff O) = h(\mu, O)\diff\mu \diff O$ where $\diff\mu$ designates the Lebesgue measure on $\mathbb R^D$ and $\diff O$ the Haar measure on $\mathcal O(D)$, and we further assume that there exist $c_1,b_1 > 0$ and $b_2 > 2D-1$  such that 
	\beq\label{condH:partial}
	e^{-c_1\|\mu\|^{b_1}} \lesssim  h(\mu,O) \lesssim  (1 + \|\mu\|)^{-b_2} \quad \text{with} \quad  \quad \forall \mu, O.
	\eeq
	We also assume that $\Lambda$ is drawn from a probability measure $\Pi_\Lambda$ that has a density $\pi_\Lambda$ with respect to Lebesgue measure on $\mathbb R^D$, and that this density satisfies: there exist $c_2, c_3, b_3>0$ and $b_4 > D(D-1)/2$ such that 
	\beq\label{cond:piL}
	\begin{split}
		&e^{-c_2 \sum_{i=1}^D \lambda_i^{-d/2}} \lesssim \pi_\Lambda(\lambda_1, \cdots , \lambda_D ) \quad \text{for small}~\lambda_1,\dots,\lambda_D \in (\bbR_+^*)^D, \\
		&\Pi_\Lambda\( \min_{1\leq i\leq D} \lambda_i < x \) \lesssim e^{- c_3 x^{-b_3}}\quad~~ \text{for small}~x > 0,\\	\text{and}~~~~~~~~~~&\Pi_\Lambda\( \max_{1\leq i\leq D} \lambda_i > x \) \lesssim x^{-b_4} \quad \quad \quad \text{for large}~x > 0. \\
	\end{split}
	\eeq
	Condition \eqref{condH:partial} is weak and is  for instance satisfied as soon as $\mu$ and $O$ are independent under $H$ with positive and continuous density for $O$ and positive density for $\mu$ with weak tail assumptions.  Condition \eqref{cond:piL} is also weak and common in the case of location Gaussian mixtures and is verified in particular if the  $\sqrt{\lambda_i}$'s are independent inverse Gammas under $\Pi_\Lambda$, or if the $\lambda_i$'s are independent inverse Gammas and $d \geq 2$. 
	\\
	
	\noindent
	\textbf{Conditions on the hybrid  location-scale mixtures.} 
	$H_1$ satisfies \eqref{condH:partial} and $Q_2$ is random with distribution $\tilde \Pi_\lambda$ which satisfies: for all $b>0$ there exists $B_0, c_2>0$ such that  for \newclem{$x_1^2\leq  x_2$} both small, 
	\begin{equation}\label{cond:piH:LB}
		\begin{split}
			\tilde\Pi_\Lambda\left[Q_2\left( [x_1,x_1(1+x_1^b)]^d \times [x_2,x_2(1+x_1^b)]^{D-d} \right) \geq x_1^{B_0} \right] &\gtrsim e^{-c_2 x_1^{-d/2}}.
		\end{split}
	\end{equation}
	Moreover we assume that for some positive constant $c_3, b_3, c_4, b_4>0$ such that, 
	\begin{equation}\label{cond:piH:UB}
		\begin{split}
			\bbE_{\tilde \Pi_\Lambda}\left[Q_2\left(\min_{1\leq i\leq D} \lambda_i \leq x \right) \right] &\lesssim e^{-c_3  x^{-b_3}}\quad \text{for small}~x,\\
			\bbE_{\tilde \Pi_\Lambda}\left[Q_2\left(\max_{1\leq i\leq D} \lambda_i > x \right) \right] &\lesssim e^{-c_4 x^{b_4}} \quad\quad\quad \text{for large}~x.\\
		\end{split}
	\end{equation}
	\begin{rem}
		One can view the partial location-scale mixture as a special instance of the hybrid location-scale mixture defined above: take $Q_2 $ to be the Dirac mass at a value $\Lambda$ where $\Lambda \sim \Pi_\Lambda$. 	
	\end{rem}
	
	Conditions  \eqref{cond:piH:LB} and \eqref{cond:piH:UB} are in particular satisfied if 
	$ Q_2$ comes from a Dirichlet process. More precisely, if $Q_2$ is of the form
	\begin{align*}
		Q_2(\diff\lambda) = \prod_{i=1}^D  Q_0(\diff\lambda_i) ~~~\text{with}~~~Q_0 \sim \DP( BH_\lambda),
	\end{align*}
	with $B > 0$, then \eqref{cond:piH:LB} and \eqref{cond:piH:UB} are satisfied for reasonable choices of probability distributions $H_\lambda$ on $\bbR_+$. We show in the next proposition that this is in particular true when $H_\lambda$ is a square-root- inverse Gamma. 
	\begin{prp}\label{prp:hybridDP}
		Assume that $ Q_2 = Q_0^{\otimes D}$ where $Q_0 \sim \DP( B H_\lambda),$ where  $B>0$ and  $H_\lambda$ has density $h_\lambda$ verifying
		\newclem{
		$$ e^{- c_2 \lambda^{-1/4}}\1_{\lambda \leq 1} \lesssim h_\lambda(\lambda)\lesssim e^{- c_3 \lambda^{-b_3}}e^{-c_4 \lambda^{b_4} },$$}
		 then conditions \eqref{cond:piH:LB} and \eqref{cond:piH:UB} are satisfied. 
	\end{prp}
	A proof of \prpref{hybridDP} can be found in \appref{prior}. For instance if \newclem{$\lambda^{1/4}$} follows an inverse Gamma truncated on $[0, R]$ for arbitrarily large $R$ under $H_\lambda$, then Proposition \ref{prp:hybridDP} holds. 
	
	\begin{rem}\label{rem:priord}
		Although the conditions on the prior, \eqref{cond:piL} and \eqref{cond:piH:LB}, depend on $d$, they are satisfies for all $d$ by setting $d=1$ and in particular they are verified for all $d\geq 1$ if \newclem{$\lambda_i^{1/4}$} follow an inverse Gamma under the base measure, which is agnostic to $d$. 
	\end{rem}
\begin{table}[t!]  
\centering
\begin{tabular}{c|c|c||c|c}
         & MFM & DPM & Partial & Hybrid \\ \hline
         Conditions & \eqref{eq:finitemixt}+\eqref{condH:partial} &\eqref{condH:partial}  & \eqref{cond:piL} & \eqref{cond:piH:LB}+\eqref{cond:piH:UB}
\end{tabular}
\caption{Summary table of the required conditions depending on the type of mixture and the type of scale sampling.}
\label{tab:condtab}
\end{table}

	\section{Main results} \label{sec:theoretical}
	
	\subsection{Posterior contraction rates}  \label{subsec:main}
	
	Recall that   $X_1,\dots,X_n$ is an $n$ sample drawn from a distribution $P_0$ with  density $f_0$. \newclem{Recall that $f_0 = f_{0,n}$ might depend on $n$, but that all the constants appearing in the subsequent conditions do not.} This density is concentrated around a submanifold $M$, with a  a given smoothness $\beta_0$ along the manifold and a typically much larger smoothness $\beta_\perp$ along the normal spaces. More precisely, we will assume:
	\bitem
	\item \emph{Conditions on $M$}: the submanifold $M$ is of dimension $d$ and has a reach greater than $\tau > 0$. Furthermore, there exists $\beta_M > 4$ and $C_M > 0$ such that $\Psi_{x_0} \in \cH^{\beta_M}_{\iso}(\cV_{x_0},C_M),~\forall x_0 \in M$. In particular, $M$ also satisfies \eqref{barpsi}.
	\item \emph{Conditions on $f_0$}: the density $f_0$ is in $\hbm$ \newclem{with $\delta = \delta_n \leq \tau/2$}. Furthermore, there exists 
	$c_5, \kappa > 0$, 
	\beq \label{expdec}
	f_0(x) \lesssim e^{- c_5 \|x \|^\kappa}~~~~\forall x\in \bbR^D,
	\eeq
	and  for some $\omega > 6 \beta$ and $ C_0<\infty$, 
	\beq \label{assint}
	\int_{\cW_{x_0,\delta}} \left|\frac{\Diff^k \bar f_{x_0,\delta}}{\bar f_{x_0,\delta}}\right|^{\omega/\inner{k}{\bm\alpha}} \bar f_{\delta,x_0} \leq  C_0 ~~~\text{and}~~~ \int_{\cW_{x_0,\delta}} \left|\frac{L_{x_0,\delta}}{\bar f_{x_0,\delta}}\right|^{\omega/\beta} \bar f_{\delta,x_0}  \leq  C_0. 
	\eeq
	for all $\delta$ small, $x_0 \in M$ and all $0 \leq \inner{k}{\bm\alpha} < \beta$. \newclem{Recall that the quantities $\bm\alpha$ and $\beta$ are defined through $\beta_0$ and $\beta_\perp$ through \eqref{eq:defalpha0}.} 
	\eitem  
	

	In the rest of this paper the symbols $\simeq$, $\lesssim$ and $\gtrsim$  denote equalities or inequalities up to a constant depending on $D$, $d$, $\tau$, $\beta_M$, $C_M$, $\beta_0$, $\beta_\perp$ and all the other constants appearing in conditions  \eqref{eq:finitemixt} to \eqref{assint}. 
	
	\begin{thm} \label{thm:main} 
		Let $X_1, \dots, X_n$ be a $n$-sample from $f_0$ satisfying \eqref{expdec} and \eqref{assint} with $\omega>6 \beta + (2\beta+D)(D-d)\log(1/\delta)/\log n$  
		and that \newclem{$\beta_0 \leq \beta_\perp \leq \beta_M - 4$}. 
Consider the partial location scale mixture or the hybrid location scale mixture prior [either the mixture of finite mixtures  or the Dirichlet process mixture] satisfying the conditions displayed in Table \ref{tab:condtab}.

		Then, under the conditions stated above,
		$$
		\mathbb E_0^n \left[ \Pi\left(\hel(f_P,f_0) \geq \ve_n~|~\cX_n\right) \right] = o(1) 
		$$
		where \newclem{
		\beq \label{eq:ratemain}
		\ve_n \simeq \log^p n \times \{\frac{1}{\sqrt{n \delta^{\frac{D}{2\alpha_0-\alpha_\perp}}}} \vee n^{-\frac{\beta}{2\beta + D}}\},
		\eeq}
		with $p > 0$ depending on $D$, $\kappa$ and $\beta$, \newclem{and where we recall that $\beta, \alpha_0$ and $\alpha_\perp$ are defined through \eqref{eq:defalpha0}}. 
	\end{thm}
	
	\newclem{A few remarks are in order.} 
	
\begin{rem} Note that the conditions regarding the prior $\Pi$ do not involve $M$, $\delta$, $L$, $\beta$, or $\tau$: they are regarded as unknown in this framework. \newclem{Also, while the conditions \eqref{cond:piL} or \eqref{cond:piH:LB} involve the intrinsic dimension $d$, one can always choose $\Pi$ independent of $d$ and satisfying these two conditions for all $d \geq 1$, as noted in Remark \ref{rem:priord}. Hence the Bayesian procedure is fully adaptive.} 
	\end{rem}	
\begin{rem} \label{minimax}
\newju{The rates obtained in Theorem \ref{thm:main} are minimax (up to a $
\log n$ term)  as soon as $\delta \gtrsim n^{ -( 2\alpha_0 - \alpha_\perp)/(2\beta+D)}:= \delta_n(\beta_0,\beta_\perp)$. To see this, note that 
the case of densities concentrating near the linear subspace spanned by the $d$ first vectors of the canonical basis of $\bbR^D$, described as 
$$\mathcal F_{\text{lin}}(\delta) :=\{ f_\delta(x) = f(x_1, x_	2/\delta)/\delta^{D-d}, f \in \mathcal H_{\text{an}}^{\bm{\beta}}(\mathcal U, \mathbb R) \}$$ 
where $x_1 \in \mathbb R^d, x_2 \in \mathbb R^{D-d}$ are restricted to a compact subset $\mathcal U $ of $\mathbb R^D$, is a special case of the assumption \eqref{expdec}. Even if $\delta = \delta_n = o(1)$, the minimax rate for estimating,  in $L_1$ norm,  densities belonging to $\mathcal F_{\text{lin}}(\delta)$ is bounded from below by $n^{-\beta/(2\beta + D)}$. This is a consequence of fact that for two densities $f_{\delta,1}$ and $f_{\delta,2}$, 
$\|f_{\delta,1} - f_{\delta,2}\|_1 = \|f_{1,1} - f_{1,2}\|_1$ is independent of $\delta$  and a look at the construction of the lower bound in \cite{goldenshluger2014adaptive} --- proof of Theorem 4 (ii) for $p=1$,  $r = \left(\infty,\dots,\infty\right)$, $s = \infty$ and $\theta = 1$ (tail dominance from \eqref{expdec}) --- shows that the lower bound is not impacted by $\delta$. Obviously the result would be different if other metrics,  such as the $L_2$ norm  were  considered. However these distances, being non  transformation invariant,  are less natural in the context of density estimation. Hence, at least when $\delta \gtrsim \delta_n(\beta_0,\beta_\perp)$  the rate $\ve_n$ is the minimax estimation rate (up to a $\log n$ term).}
	\end{rem}
	
	\begin{rem}[Relations to existing results] The result of Theorem \ref{thm:main} is in line with the many existing results regarding Bayesian multivariate density estimation, as in \cite{ghosal2001entropies,ghosal2007posterior, kruijer2010adaptive,shen2013adaptive,canale2017posterior}, and extends them in this new \emph{near manifold} framework. We illustrate these extensions with simple cases and examples below.
	\bitem
	\item \textbf{Linear case}: As explained above, the case where $M$ is a linear subspace spanned by vectors of the canonical basis of $\bbR^D$ and where $\delta$ is constant (not going to zero) recovers the classical Euclidean anisotropic setting and our result thus encompasses in particular those of  \cite{shen2013adaptive} regarding the Bayesian estimation of multivariate, anisotropic H\"older densities. The rates we obtain are the same and are minimax up to logarithmic factors. The technical assumptions we do are roughly identical to theirs --- although in our context, the regularity vectors can only take two values $\beta_0$ and $\beta_\perp$ because of the geometry of the problem. \textcolor{black}{ The linear anisotropic case with regularity taking more than 2 values can be treated similarly to what we have done in Theorem \ref{thm:approx},  thus extending the results of \cite{shen2013adaptive} to linear anisotropy  along different axes than the canonical basis.}
	\item \textbf{Isotropic case} When $\beta_0 = \beta_\perp = \beta$, the density $f_0$ is globally $\beta$-H\"older on $\bbR^D$, and the general results of \cite{shen2013adaptive} would also apply in that case. However, because $f_0$ has magnitude $\delta^{-(D-d)}$ in a tubular neighborhood of size $\delta$, its $\beta$-H\"older norm is of order $\delta^{-(D-d+\beta)}$. This yields a rate of order
	$$
	\frac{1}{\delta^{D-d+\beta}} n^{-\frac{\beta}{2\beta+D}}.
	$$
	The rate we find in \thmref{main} is:
	$$
	\frac{1}{\sqrt{n\delta^D}} \vee n^{-\frac{\beta}{2\beta+D}}
	$$
	which is always better if the density is sufficiently smooth, ie $\beta \geq (d-D/2)_+$, or for $\delta$ not too small otherwise, ie $\delta \gg n^{-D/((2\beta+D)(2d-D-2\beta))}$. That is because our procedure takes advantage of the geometry of the problem in this context.
	\item \textbf{Fixed width case}: The case where $\delta > 0$ is fixed and do not go to zero has been investigated by \cite{mukhopadhyay2020estimating} using Fisher-Gaussian mixtures. However, little theorical results or rates are derived as they were not interested in providing a general model for such densities. This case encompasses many natural situations such as for instance the ones considered in \cite{capitao2023deconvolution} where the data can be naturally described using spherical coordinates with an unknown center point.
	\item  \textbf{Diverging width case}: The case where $\delta \to \infty$ implies in particular that the reach of $M$ becomes infinite, which can only be the case if $M$ is an affine subspace of $\bbR^D$, as the reach lower-bounds the minimal radius of curvature of $M$. We thus fall back in the Linear case described at the beginning of this remark. \textcolor{black}{ In this case  there is no need to rescale with $\delta$ to define the smoothness class and we can consider the alternative smoothness class $\mathcal H^{\beta_0, \beta_\perp}(M,L)$ of function $f: \mathbb R^D \rightarrow \mathbb R$ verifying $f \circ \bar \Psi \in  \cH_{\an}^{\bm\beta_{0,\perp}}(M,L)$ where $\bar \Psi$ is the rotation aligning with $M, M^\perp$. }
	 Notice however that our functional classes, as defined in Definition \ref{def:defmanihold},  becomes trivial for $\delta = \infty$, and  one would need to threshold at $\delta = \log^t n$ to use this definition, where $t$ is driven by the exponential tail of the density, thus losing logarithmic terms in the final contraction rate. \textcolor{black}{It is therefore more interesting in this case to use the alternative class $\mathcal H^{\beta_0, \beta_\perp}(M,L)$.}
	
	\item \textbf{Decaying width case}: Finally, the case where $\delta = \delta_n \to 0$ is, to our knowledge, new and the rate we obtain, in its dependence to $\delta_n$, does not relate to any rate that we know of.
	\eitem
	\end{rem}

		\begin{rem} \label{rem:union}
		Because the approximation results of \subsecref{approx} are stable under finite mixtures, so do the results of \thmref{main}. In particular, the support of $f_0$ can be a finite union of vicinity of submanifolds $M_i$ with arbitrary intersections and with each $M_i$ fulfilling \eqref{betam},   see \figref{Mi} for a diagram of such a situation. Hence the assumption on the lower bound on the reach can be significantly weakened, \textcolor{black}{ 
as shown in the following proposition
\begin{prp}\label{prop:mixture}
Assume that there exist $M_j, j\leq J$ smooth manifolds with dimensions $d_j, j \leq J$ and verifying the \textit{conditions on $M$}; assume that there exist  probability densities $f_{0,j}$ $j\leq J$  verifying assumptions \eqref{expdec} and \eqref{assint} with respect to $(\beta_{0j}, \beta_{\perp,j})$ and $p_{0j}>0, j\leq J$ such that $\sum_j p_{0j} =1$.  Then under the hybrid location scale mixure prior satisfying the conditions diplayed in Table \ref{tab:condtab},
 \newclem{
 $$
		\mathbb E_0^n \left[ \Pi\left(\hel(f_P,f_0) \geq \ve_n~|~\cX_n\right) \right] = o(1) , \quad \text{where}
		$$
		$$
		\ve_n \simeq \log^p n \times \max_j \{\frac{1}{\sqrt{n \delta_j^{\frac{D}{2\alpha_{0,j}-\alpha_{\perp,j}}}}} \vee n^{-\frac{\beta_j}{2\beta_j + D}}\},
		$$}
If in addition for all $j$, $(d_j,\beta_{0j}, \beta_{\perp,j}, \delta_j) = (d,\beta_{0}, \beta_{\perp}, \delta)$
then the above result holds also true under the partial location scale mixture prior. 
\end{prp}
 Proposition \ref{prop:mixture} is a direct consequence of the proof of Theorem \ref{app:pr:thmain} and of  the fact that with $\text{KL}(f_0, f) = P_0( \log (f_0/f))$ and $f = \sum_{j=1}^J p_{0j}f_j$,  
 $$ KL(f_0, f) \leq \sum_{j=1}^Jp_{0j} KL(f_{0j},f_j),$$ see Section \ref{sec:mainproof} for more details.\\
The case of collections of manifolds has also been considered in \cite{kim2019Kernel}, who bound the variance term   of a Kernel estimator, in sup norm, under a condition on the support which include unions of manifolds.  }

We consider such an example in the simulations of Section \ref{sec:num}.
Interestingly, to allow for different $\beta_{0i}$ or different $d_i$ accross $i=1, \cdots, N$, we need to consider the hybrid location - scale prior, since we need to allow for different  $\Delta_{\sigma_i,\delta_i}$ , for $i=1, \cdots, N$ in the approximation theory presented in Theorem \ref{thm:approx}. 
	We stress out here that the accuracy of the density estimation does not decrease near the intersections. \textcolor{black}{It is possible that locally better rates can be achieved (depending on $\beta_j, d_j, \delta_j$ in particular, however deriving local rates is much harder than global Hellinger or $L_1$ rates, see \cite{rockova:rousseau:2023} and is beyond the scope of this paper. }
	\end{rem}
	\begin{figure}[h!]
		\centering
		\includegraphics[width = 5cm]{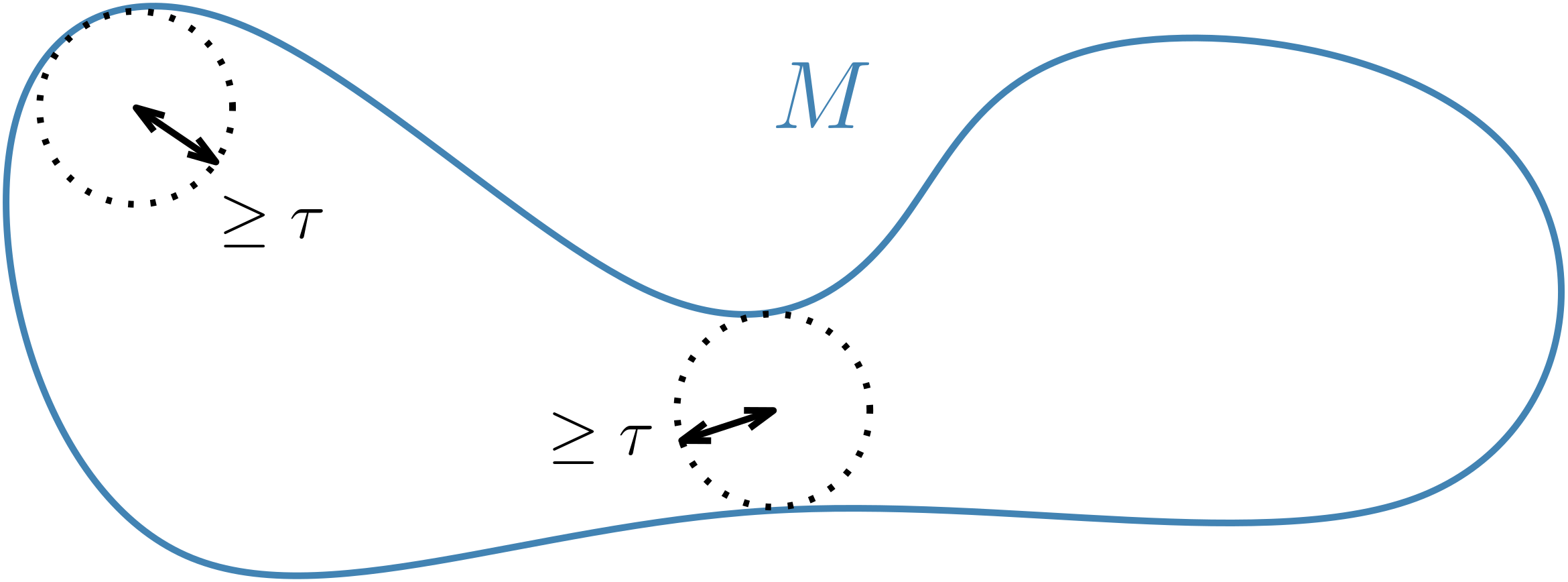}~~~~
		\includegraphics[width = 5cm]{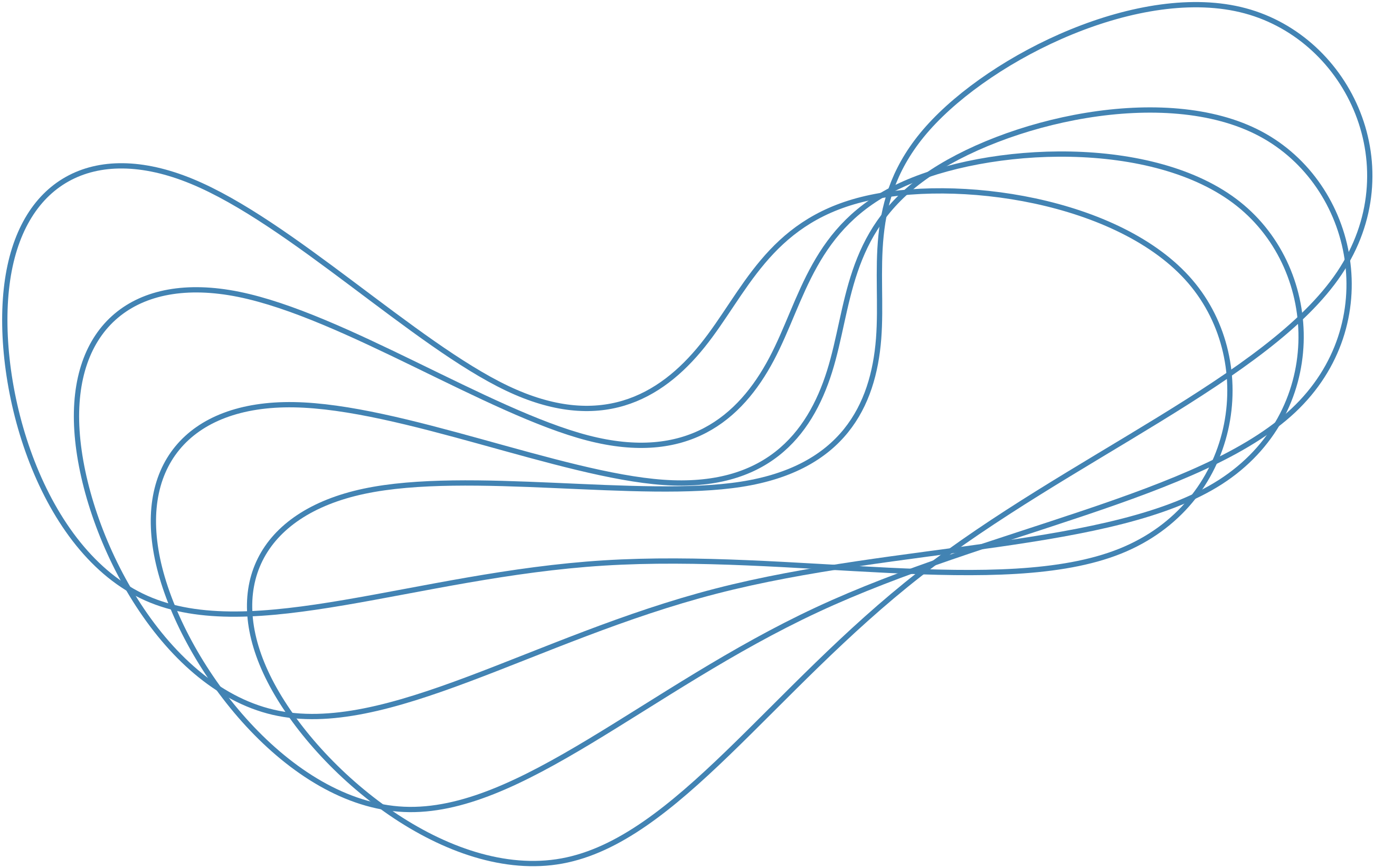}
		\caption{(Left) An example of smooth submanifolds with a reach constrained to be greater than $\tau$ and (Right) a finite union of such manifolds. Both subsets are admissible as (near) supports for a density $f_0$ satisfying the contraction rates displayed Proposition \ref{prop:mixture}.}
		\label{fig:Mi}
	\end{figure}

	\begin{rem} 
		The case where $\beta_\perp$ is infinite is particularly of interest.  Then, $\beta \to D\beta_0/d$, $\alpha_0 \to D/d$, $\alpha_\perp\to 0$ and the rate $\ve_n$ becomes, up to $\log n$ terms
		$$
		\frac{ 1}{\sqrt{n \delta^{d/2}}} \vee  n^{-\frac{\beta_0}{2\beta_0 + d}},
		$$
		which is, when $\delta$ is not too small (i.e  $\delta \gtrsim n^{-2/(2\beta_0+d)}$), the minimax rate for estimating a $\beta_0$ H\"older density in $\bbR^d$. Here, again, the strength of our result lies in that  the manifold (and thus the support of $f_0$)  is unknown and the prior depends neither on $\mathbf{\beta}$ nor on $\delta$ or $d$ (or $M$). 
	\end{rem}

	Assumptions \eqref{expdec} and \eqref{assint} are common assumptions in density estimation based on mixtures of Gaussians, see for instance \cite{kruijer2010adaptive} or \cite{shen2013adaptive} for the Bayesian approaches and \cite{maugis2013adaptive} for the frequentist  approaches.  They are rather weak asssumptions. The difficulty with \eqref{assint} is that it is expressed on $\bar f_{x_0,\delta}$, which is natural in our context  since the smoothness assumption on $f_0$ is expressed in terms of   $\bar f_{x_0,\delta}$, but is not so intuitive. However, a careful examination of \eqref{assint} shows that this assumption is for instance implied by the stronger, but chart-independent assumption: 
		$$
		\int_{\bbR^D} \{\frac{L(x)}{f_0(x)}\}^{\omega_*} f_0(x) \diff x \leq C_0,
		$$
		for some large $\omega^*$. Moreover, to understand better what \eqref{assint}   means in terms of $f_0$, we  illustrate it in  our archetypal model where $f_0$ is the density of $X = Y +\delta \mathcal E$ as in Proposition \ref{prp:model}. We then have the following result:  
		
	\begin{lem}\label{lem:moments}  Under the conditions of Proposition \ref{prp:model} and if
		$$
		\int_{M} \{\frac{L_0(x)}{f_*(x)} \}^{\omega_* } f_*(x) \diff\mu_{M}(x) < \infty ~~~\text{and}~~~ \int_{\ball(0,1)} \{\frac{L_\perp(\eta)}{K(\eta)} \}^{\omega_*}K(\eta) \diff \eta < \infty,
		$$
		for some $\omega_*$ large. Then \eqref{assint} is satisfied in both the orthonormal noise model and in the isotropic noise model.
	\end{lem}
	
	The proof of \lemref{moments} can be found in Section \ref{app:prpmodel}. Note that the conditions in \lemref{moments} are fulfilled by a number of natural kernels or densities such as $K(\eta) \approx (1-\|\eta\|^2)_+^p$ and $L_\perp \approx (1-\|\eta\|^2)_+^{p-\beta_\perp}$ with large $p$ or $K(\eta) \approx \exp(-(1-\|x\|^2)^{-1}_+)$ and $L_\perp(\eta) \approx (1-\|\eta\|^2)^{-\beta_M} K(\eta)$, see \lemref{chixj} for further details.
	
	\subsection{Approximating M-anisotropic densities}  \label{subsec:approx}
	
	Theorem \ref{thm:main} is proved using the approach of \cite{ghosal2007posterior}, with a control on the prior mass of Kullback-Leibler neighbourhoods of $f_0$ and on the entropy of the support of the prior. The main difficulty in our setup is in proving the Kullback-Leibler prior mass condition. To do so, we need to construct an \textit{efficient} approximation of $f_0$ by mixtures of Gaussian. This construction is of interest in its own as it sheds light on the behaviour of such mixtures and on the geometry of M-anisotropic densities. \\

	To explain the construction, we denote, for any $x \in M^\tau$, $T_x = T_{\pr_M(x)} M$ and $N_x = N_{\pr_M(x)} M$. \newclem{We set $\cB^0_x$ an orthonormal basis of $T_x$, $\cB^\perp_x$ an orthonormal basis of $N_x$, and let $O_x$ be the matrix of the linear map $z \mapsto (\pr_{T_x} z, \pr_{N_x} z)$ from the canonical basis to the concatenated basis $(\cB^0_x,\cB^\perp_x)$}. We finally write $\Sigma(x) = O_x^\top \Delta_{\sigma,\delta}^2 O_x$ where
	$$
	\Delta_{\sigma,\delta} = \begin{pmatrix} \sigma^{\alpha_0} \Id_{d} & 0 \\ 0 & \delta \sigma^{\alpha_\perp} \Id_{D-d} \end{pmatrix}.
	$$
	Note that $\Sigma(x)$ does not depends on the choice of the bases $\cB_x^0$ and $\cB_x^\perp$ since for any orthonormal basis change $P$ that preserves $T_x$ (whence $N_x$), one has $P^\top  \Delta_{\sigma,\delta} P =  \Delta_{\sigma,\delta}$.
	For any function $f : \bbR^D \to \bbR$, we define
	\begin{align*} 
		K_\Sigma f(x) &:= \int_{M^\tau} \varphi_{\Sigma(y)} (x-y) f(y) \diff y,
	\end{align*} 
	where we recall that $\varphi_{\Sigma(y)}$ is the density of a centered Gaussian with variance $\Sigma(y)$. The idea behind the construction of the approximation of $f_0$ by a mixture of gaussians is to show that $K_\Sigma f_0(x) $ is close to $f_0$ and then to define a perturbation $f_1$ of $f_0$ such that $K_\Sigma f_1(x) -f_0(x) = O(L(x) \sigma^{\beta})$. Compared to the construction proposed by \cite{kruijer2010adaptive} in the univariate case  or \cite{shen2013adaptive} in the multivariate case where $K_\Sigma f = \varphi_\Sigma \ast f$, in our construction $\Sigma$ varies with the location $y$. Note that in particular $\int \varphi_{\Sigma(y)}(x-y) \diff y$ may be different from 1.  This dependence in $y$ is crucial to adapt to the geometry of the manifold but considerably complicates the proof as the underlying kernel integral operator can no longer be written as a convolution. We first show that $f_0$ can be efficiently approximated pointwise.
	\begin{thm}\label{thm:approx} Assume that $\sigma,\delta \leq 1$, that $f_0 \in \hbm$ satisfies \eqref{expdec} and \eqref{assint}, that $\sigma^{2\alpha_0-\alpha_\perp} =o(\delta)$ and that the manifold satisfies \eqref{betam} with $\beta_0 \leq \beta_\perp \leq \beta_M-4$.
		Then there exists a function $g : \bbR^D \to \bbR$ such that, for any $H > 0$,
		$$
		| K_\Sigma g(x) - f_0(x) | \lesssim \sigma^\beta L(x)\1_{M^\tau} + (H\log(1/\sigma))^{D/\kappa} \sigma^H \| L \|_\infty ~~~\forall x \in \bbR^D.
		$$
		The function $g$ has the form 
		$$g(x) =d_0(x,\sigma,\delta)f_0(x) + \frac{ 1}{ \delta^{D-d} }\sum_{0 < \inner{k}{\bm\alpha} < \beta} \sigma^{\inner{k}{\bm\alpha}}\sum_{j=1}^J d_{j,k}(x,\sigma,\delta) \Diff_z^{k} \overline{(\chi_j f_0)}_{x_j,\delta}(z_{j,x}),$$
		where $z_{j,x} := \Delta_{1,\delta}^{-1} \bar\Psi_{x_j}^{-1}(x)$, where $(\chi_j)_{j\leq J}$ is a partition of unity, defined in Section \ref{app:thmapprox2}, of the set $M^\tau \cap \ball( 0, R_0(\log (1/\sigma))^{1/\kappa})$ associated with a $\tau/64$-packing $(x_j)_{j\leq J}$ of $M^\tau \cap \ball( 0, R_0(\log (1/\sigma))^{1/\kappa})$ and where $d_{j,k}(x,\sigma,\delta) $ are smooth and bounded functions depending on $\chi_j$ and $M$.
	\end{thm}
	We then establish that the previous bound translates to a control in terms of Hellinger distance. 
	\begin{cor} \label{cor:approx}
		In the context of \thmref{approx}, if $f_0$ also satisfies \eqref{assint} and if $\sigma^{\omega - 6\beta} = o(\delta^{D-d})$
		and  $\delta \sigma^{\alpha_\perp} = o(|\log \sigma|^{-1/2})$,
		the probability density $\tilde h \propto g \1_{g\geq f_0/2} +  f_0/2 \1_{g<f_0/2}$ verifies 
		\begin{equation}\label{eq:approx}
			\hel(K_\Sigma \tilde h,g)^2 \lesssim \sigma^{2\beta} |\log \sigma |^{16\beta + D/\kappa}.
		\end{equation}
	\end{cor}
	
	Theorem \ref{thm:approx} is proven in Section \ref{app:thmapprox1} while the proof of \corref{approx}, which appears to be a non-trivial consequence of \thmref{approx}, is delayed to Section \ref{app:thmapprox2}.

	\section{Numerical experiments} \label{sec:num}
	In this section we present some implementations and visual results of the location-scale mixtures of Subsection \ref{sec:prior}. 
	The aim of this section is to illustrate the ability of the aforementioned priors to describe in a relevant way data that can be very singular. 
	More precisely, we consider the hierarchical model:
	\beq 
	\label{prior:1}
	\begin{split}
			(y_i)_{i=1}^n ~|~ (\mu_i,O_i)_{i=1}^n, \Lambda &\sim \bigotimes_{i=1}^n \mathcal{N}(\cdot|\mu_i,O_i \Lambda O_i^T)~~~\text{with}~~~\Lambda = \diag(\lambda_1,\dots,\lambda_D),\\
			(\mu_i,O_i)_{i=1}^n ~|~  P &\sim  P^{\otimes n}, \quad 
			P \sim \DP(\alpha P_0) , \, P_0 = \mathcal{N}(\cdot|\mu_0,\Sigma_0) \otimes \unif(\cO(D))\\
			 (\lambda_j)_{j=1}^D |( b_j)_{j=1}^D &\sim  \bigotimes_{j=1}^D \Invg(a_j,b_j),   \quad
			(b_j)_{j=1}^D \sim \bigotimes_{j=1}^D \Exp(\kappa_j)
	\end{split}
	\eeq
	We set in our experiments the value of the hyperparameters as $\alpha=1$, $a_j = \kappa_j =1$ for all $1 \leq j \leq D$ and $\mu_0 = 0$, $\Sigma_0 = I_D$. Alternatively, we also consider the previous model with the $b_j$'s being fixed ahead.\\
	
	We present two different inference approaches : Gibbs sampling (using \cite[Algo 8]{neal2000markov} and example 5.1 in \cite{Ghosal_van_der_Vaart_2017}) and stochastic approximation  of the  Maximum A Posteriori estimator (MAP) using the {\tt python} package {\tt pyro} \cite{bingham2019pyro}. The former is an asymptotically exact MCMC sampler while the latter is an approximate (but way easier to compute) point estimator. We refer the reader to the Section \ref{app:num} for further details of the implementation.\\
	
	The manifolds we use for the experiments are a 2D spiral, two intersecting circles, a 3D spiral and a torus, see Figures \ref{fig:2Dspiral} to \ref{fig:torus}. Precise parametrizations are given in Section \ref{app:numdetail}.
As for the data generating distributions, we use the additive isotropic noise model (as described in \prpref{model}), with noise kernel 
	$
	K_{\beta_\perp}(x) \propto (1 - \|x\|^2)_+^{\beta_\perp},
	$
	and base densities chosen as follows: we take the uniform distribution for the two-circles and for the torus, whereas for the 2D and 3D spiral, we push-forward through the embeddings (described in Section \ref{app:numdetail}) the probability distribution $g_{\beta_0}$ on $[0,1]$ given by
	$
	g_{\beta_0}(t) \propto (1-(1-2t)^{\beta_0}) \ind_{t \in [0,1/2]} + (1-(2t-1)^{{\beta_0}+1})\ind_{t \in [1/2,1]}. 
	$
	In the experiments, the regularities were set to $\beta_0 = 2$ and $\beta_\perp = 6$.\\
	
	We test different values for $\delta$ and challenge the model \eqref{prior:1} for various choices of $b_j$'s against the conjugate Normal Inverse-Wishart prior on these datasets (for which we choose to use the {\tt R} package {\tt dirichletprocess} \cite{dirichletprocess}). Since the Gibbs sampling algorithm is computationally very intensive, we compare this algorithm with the much lighter approximate MAP on the 2-dimensional datasets, with a sample of size $n=300$, see Figures \ref{fig:2Dspiral} and \ref{fig:circles}.\\

	As seen in these figures, the hierarchical prior \eqref{prior:1} seem to perform much better than the prior with fixed $b_j$ (unless these are chosen very carefully) and much better than the Normal Inverse-Wishart, which is also consistent with the empirical results of \cite{mukhopadhyay2020estimating}. 
	Our theoretical results together with our empirical results strongly suggest that this conjuguate DP mixture is suboptimal is this case. \\
	
	Thanks to the low computational cost of approximating the MAP, we were able to conduct experiments on 3D datasets, the 3D-spiral and the torus. See Figures \ref{fig:3Dspiral} and \ref{fig:torus}.

\begin{figure}[t!]
\centering
\includegraphics[width = 2cm]{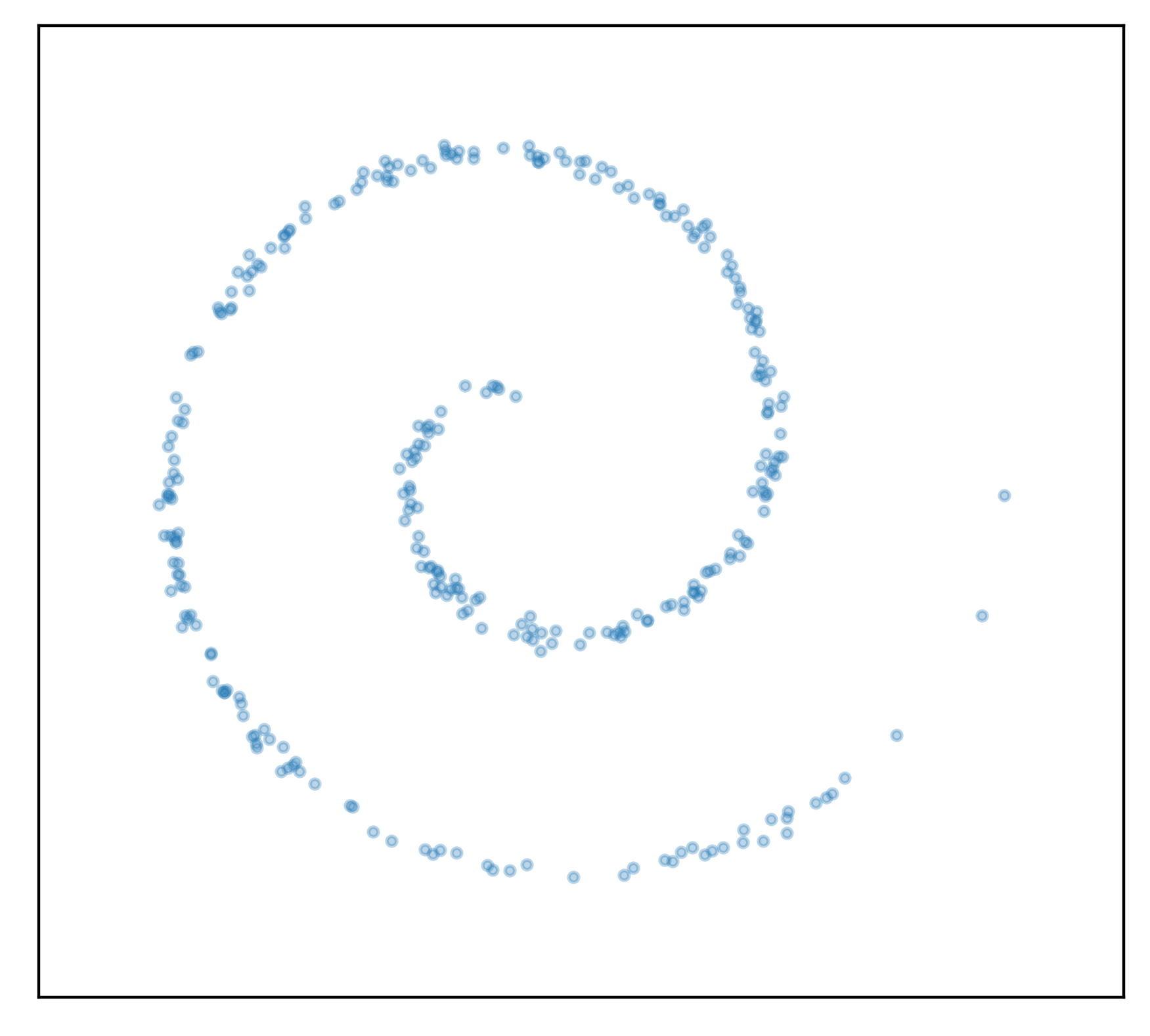}
\includegraphics[width = 2cm]{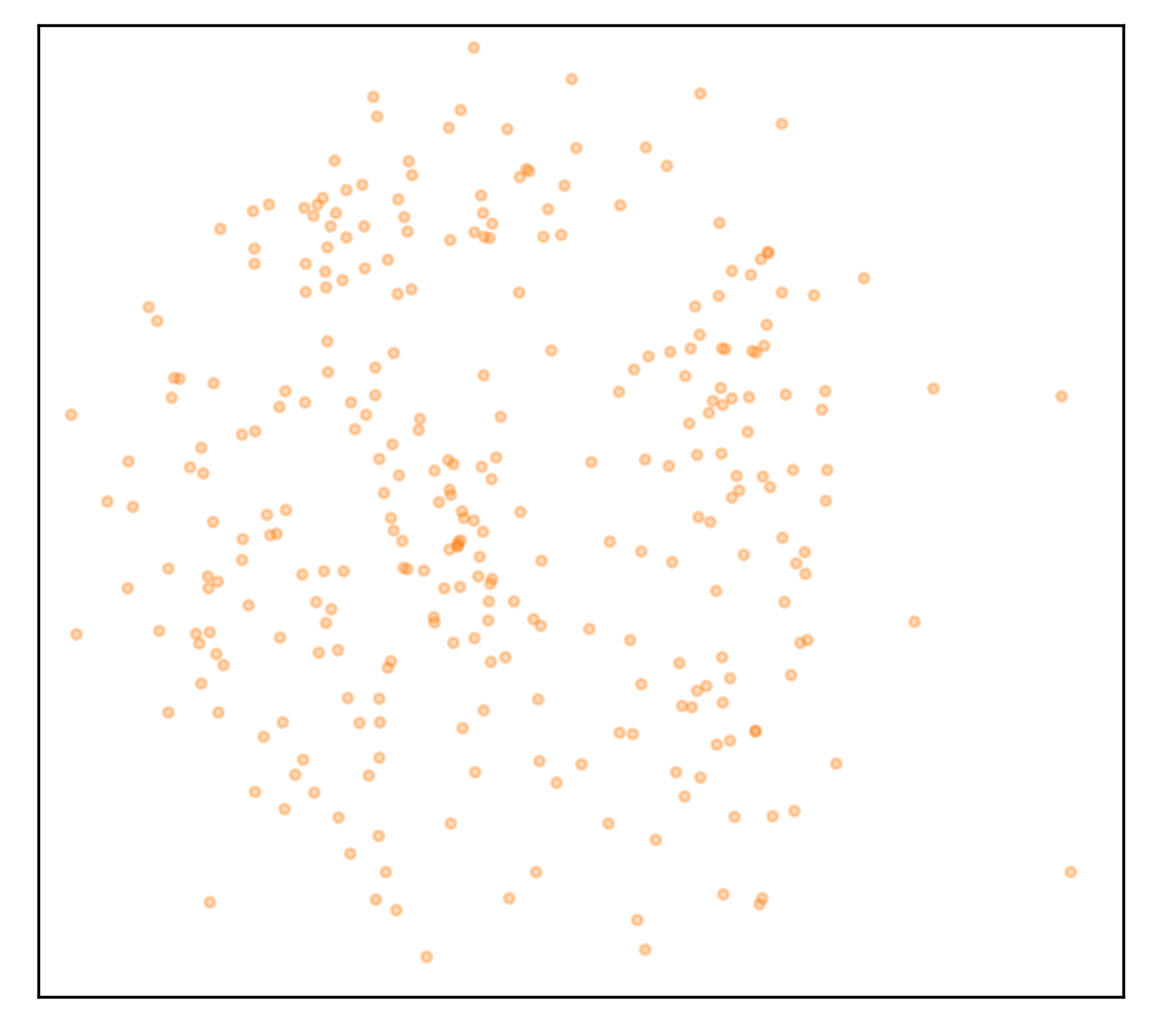}
\includegraphics[width = 2cm]{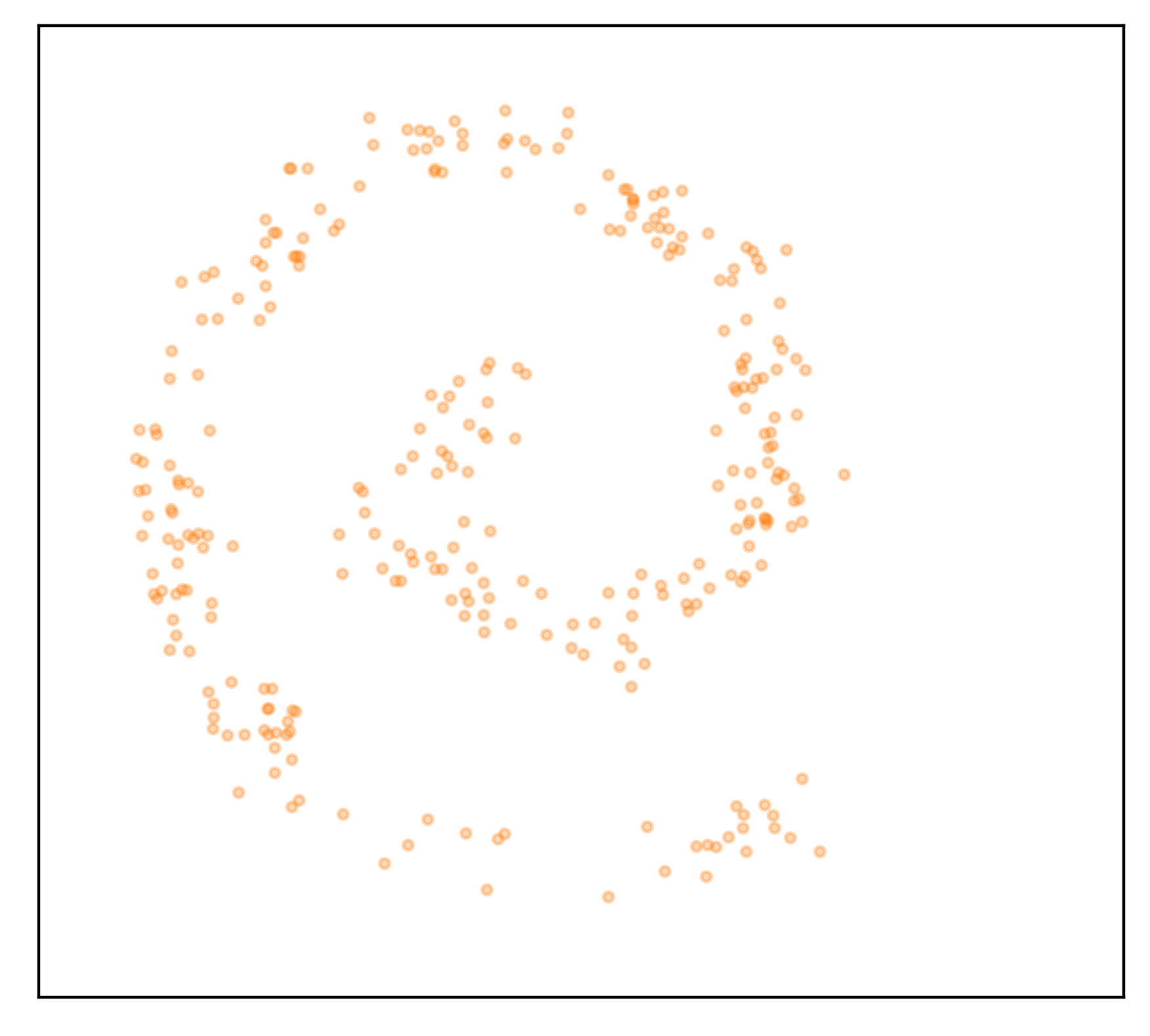}
\includegraphics[width = 2cm]{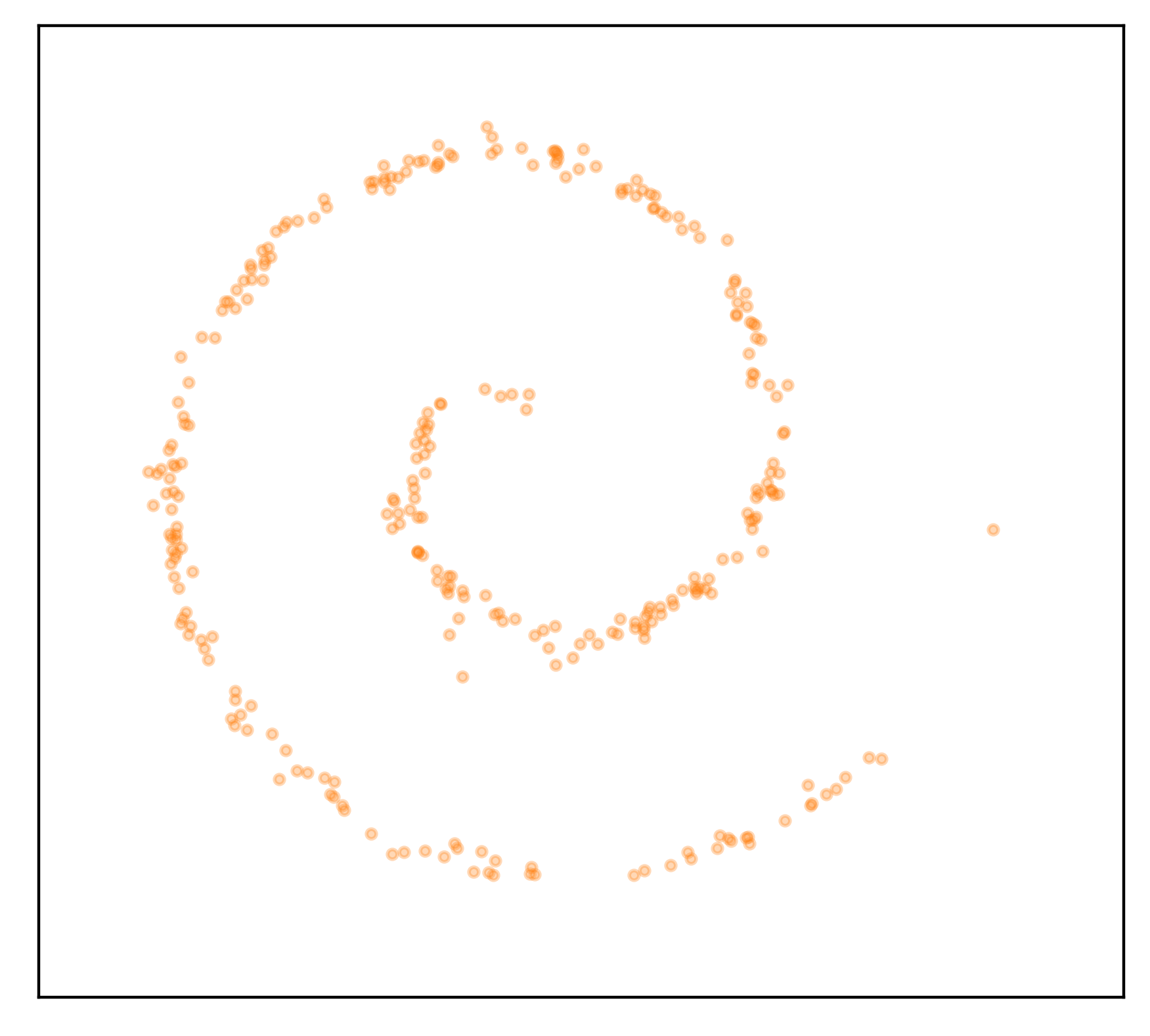}
\includegraphics[width = 2cm]{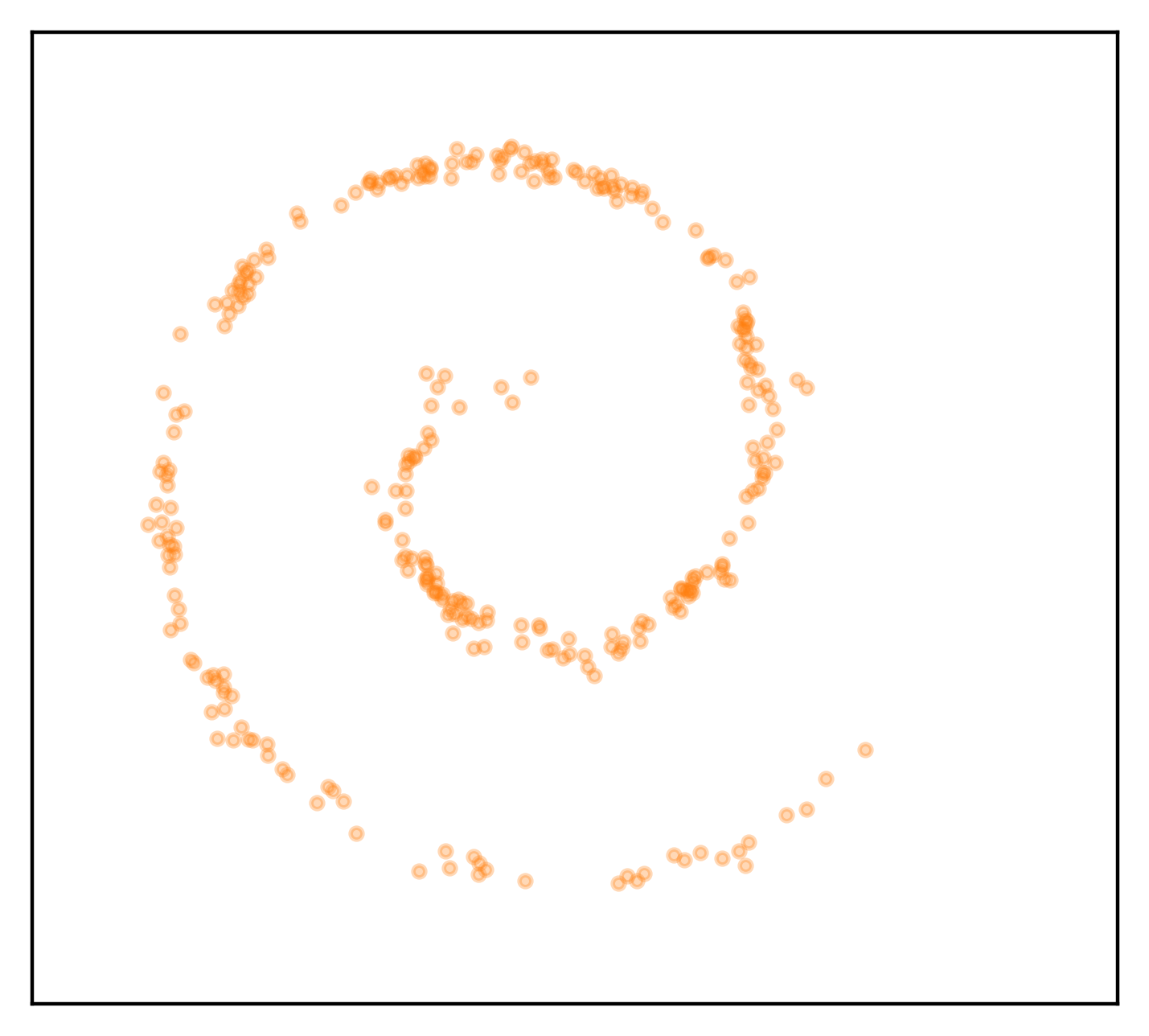}
\includegraphics[width = 2cm]{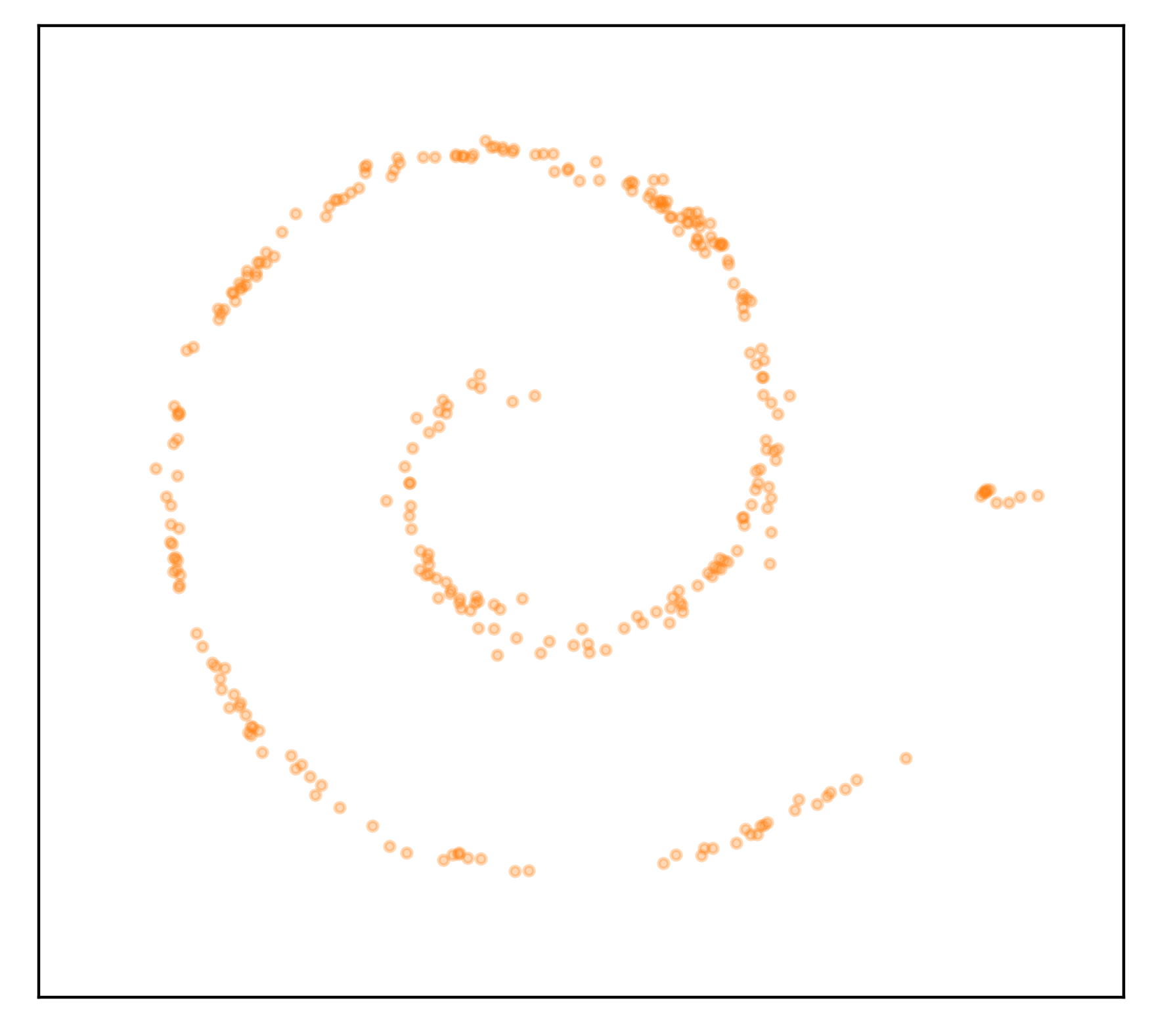} \\
\includegraphics[width = 2cm]{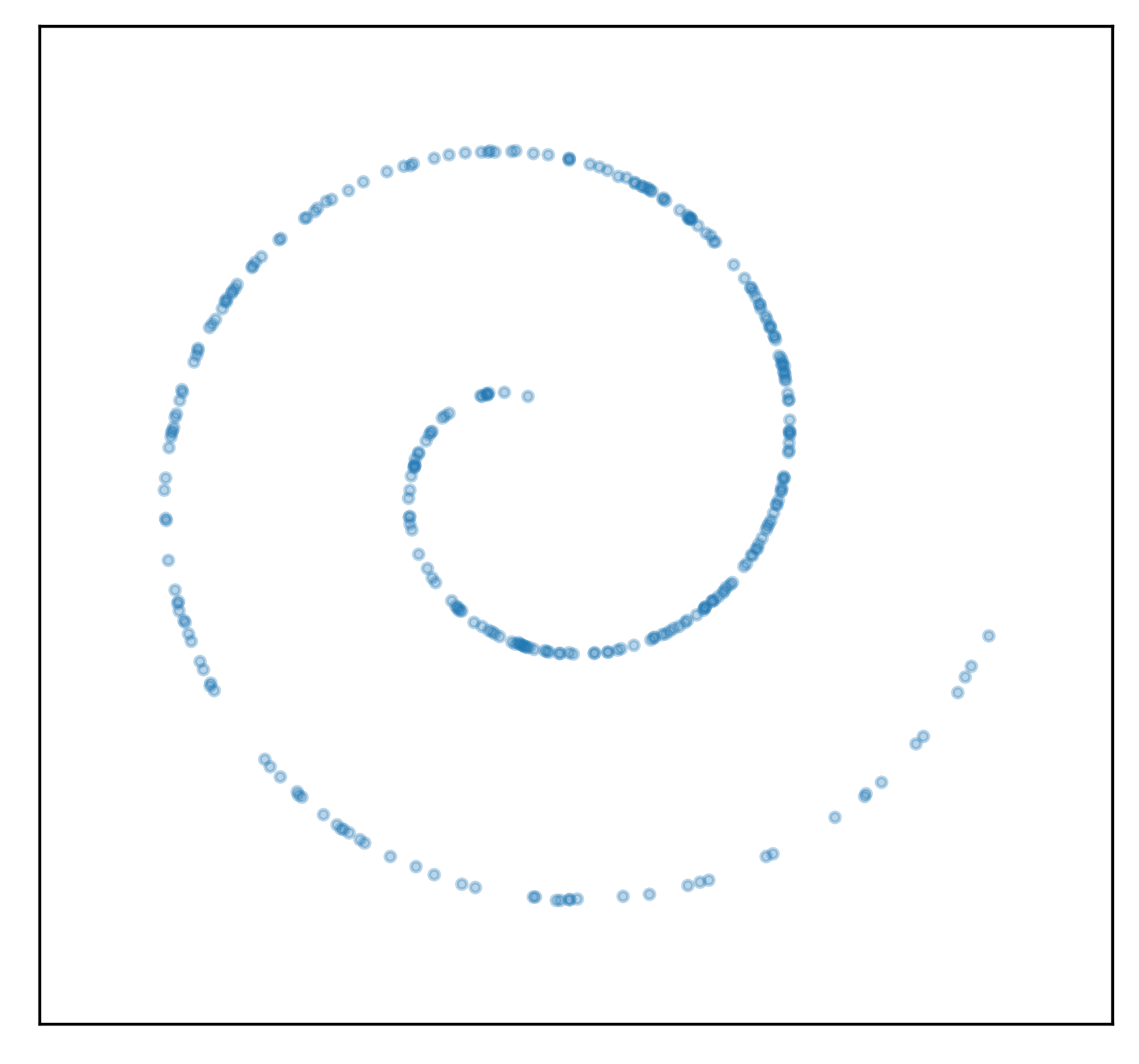}
\includegraphics[width = 2cm]{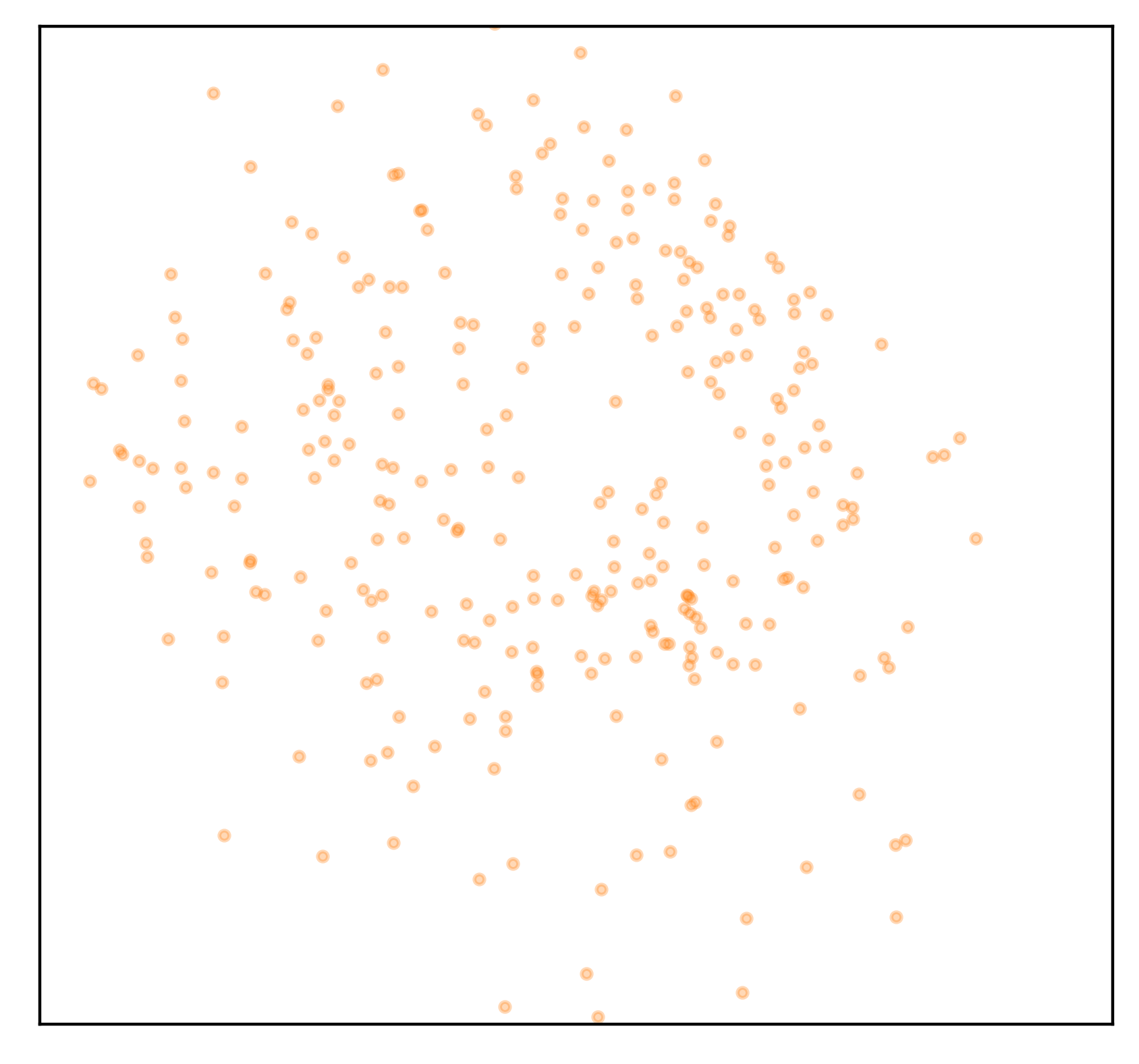}
\includegraphics[width = 2cm]{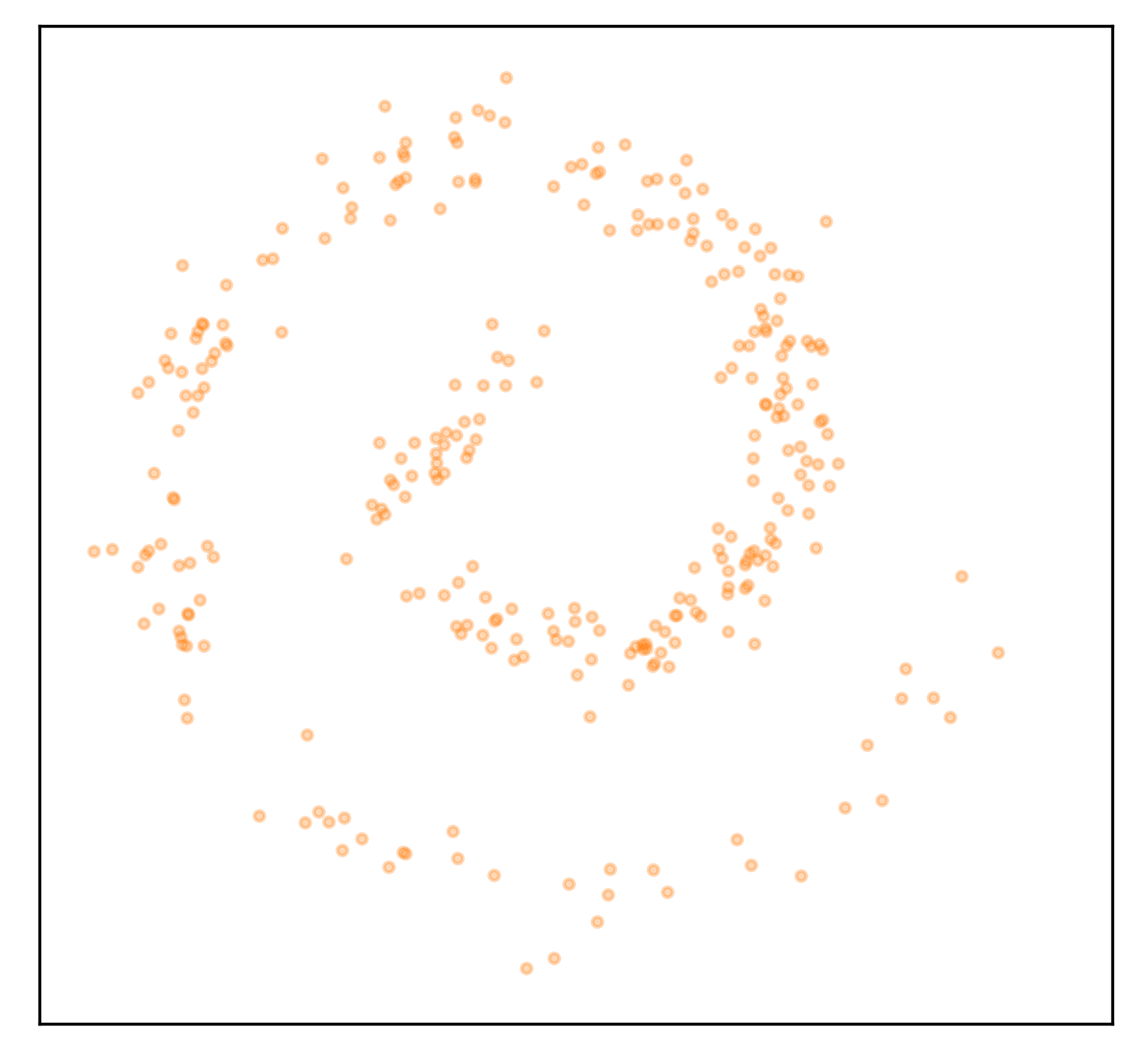}
\includegraphics[width = 2cm]{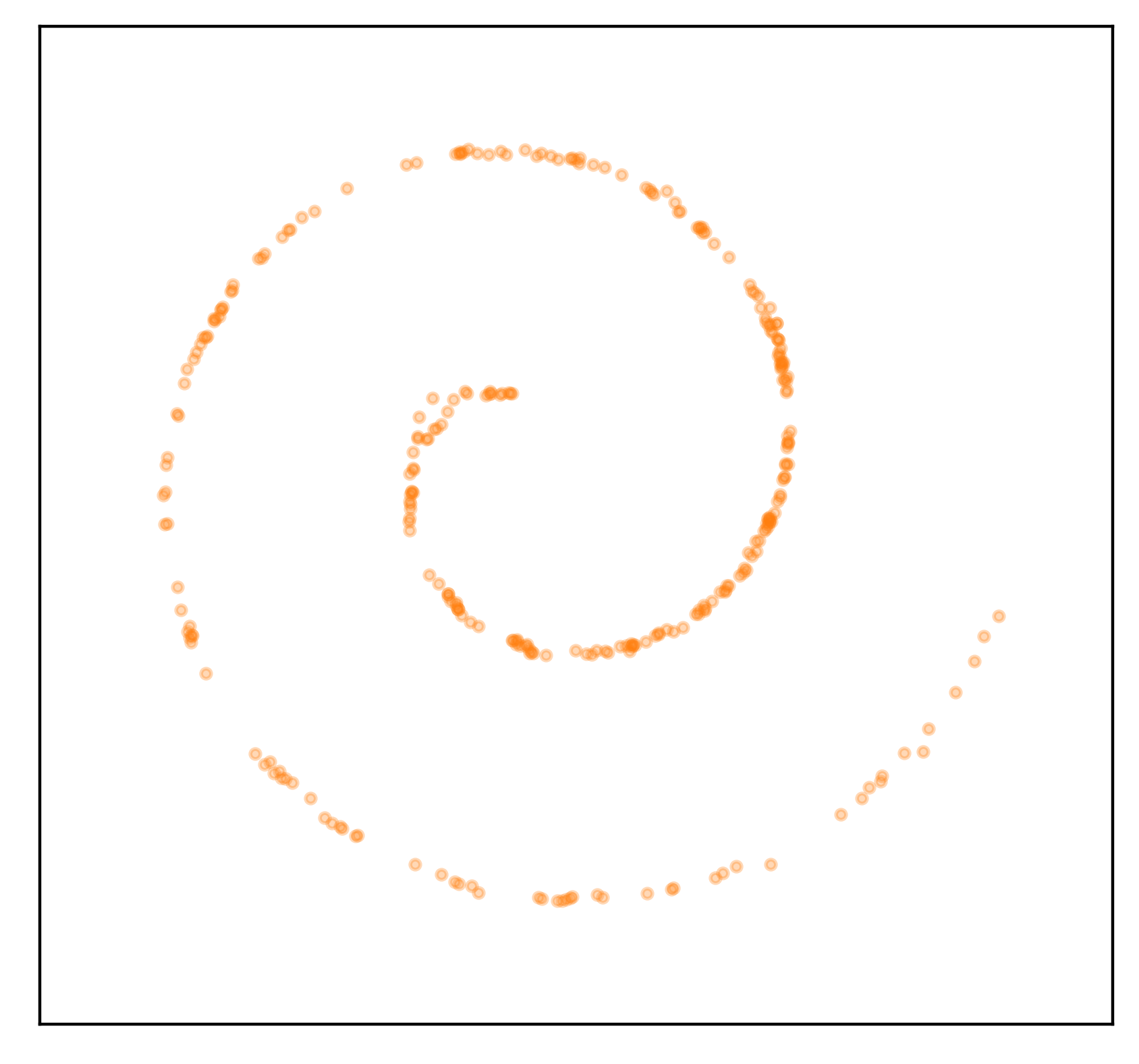}
\includegraphics[width = 2cm]{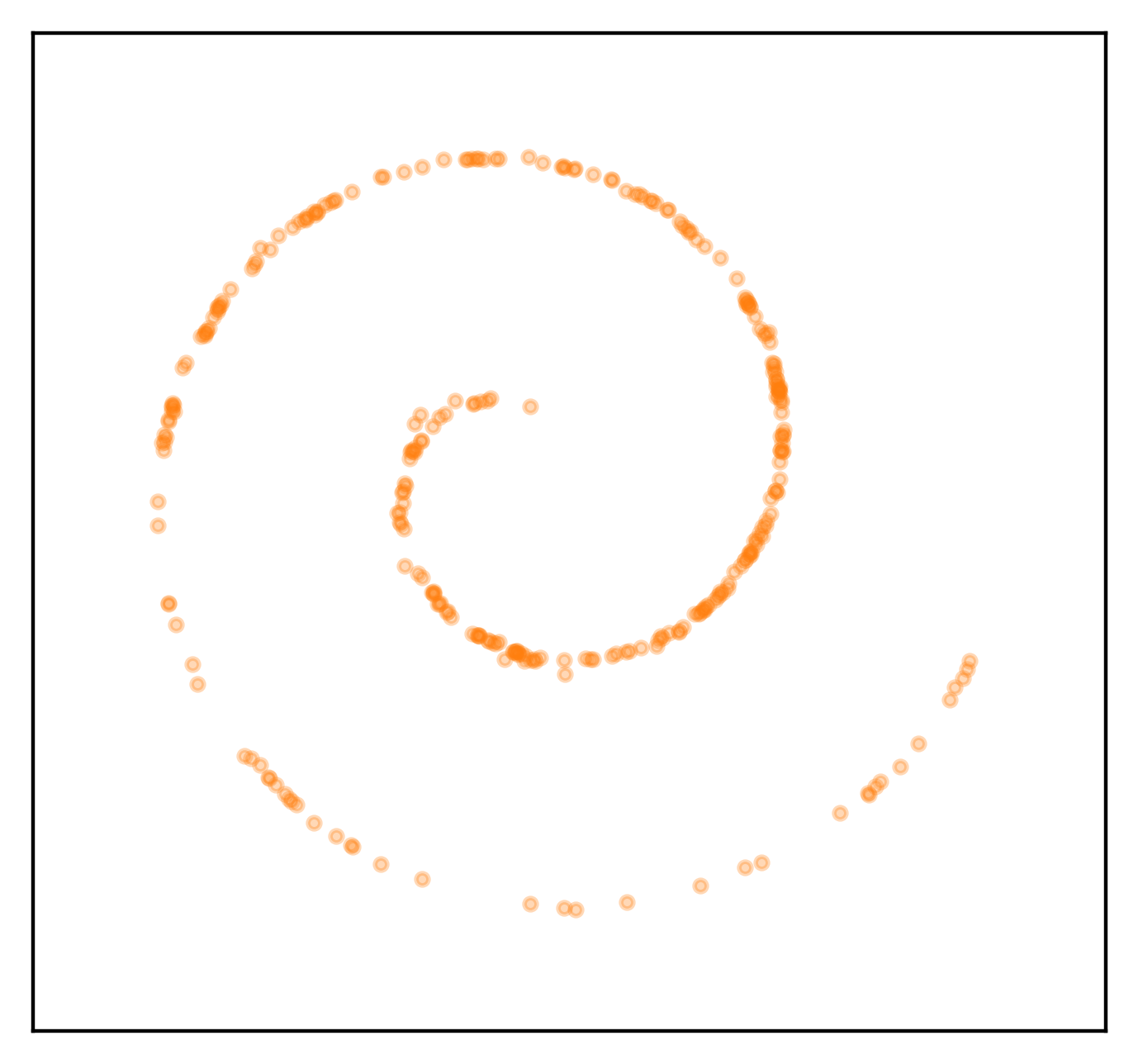}
\includegraphics[width = 2cm]{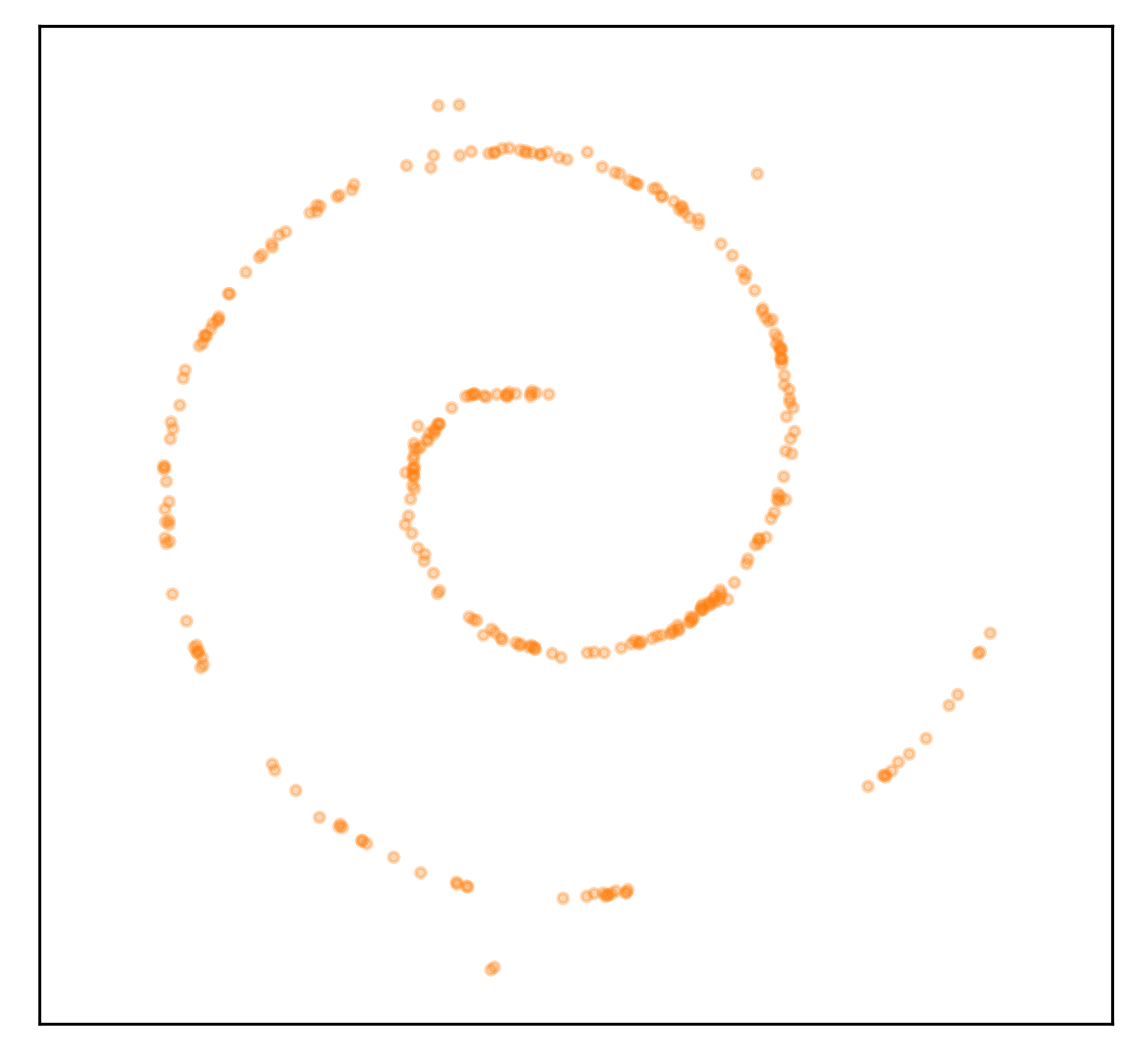}
\caption{(First column) We observe $n = 300$ points drawn under $P_0$ for the 2D-spiral with $\delta = 0.1$ (Top row) and $\delta = 0.01$ (Bottom row). We then sample the same amount of points using the Gibbs sampler with (Second column) Inverse-Wishart prior, (Third column) model \eqref{prior:1} with $b_j = 1$,  (Forth column) model \eqref{prior:1} with $b_j = 0.001$ and (Fifth column) model \eqref{prior:1} with the exponential hyperprior on $b_j$. Finally, we used in (Sixth column) the approximate MAP distribution.} 
\label{fig:2Dspiral}
\end{figure} 
\begin{figure}[h!]
\centering
\includegraphics[width = 2cm]{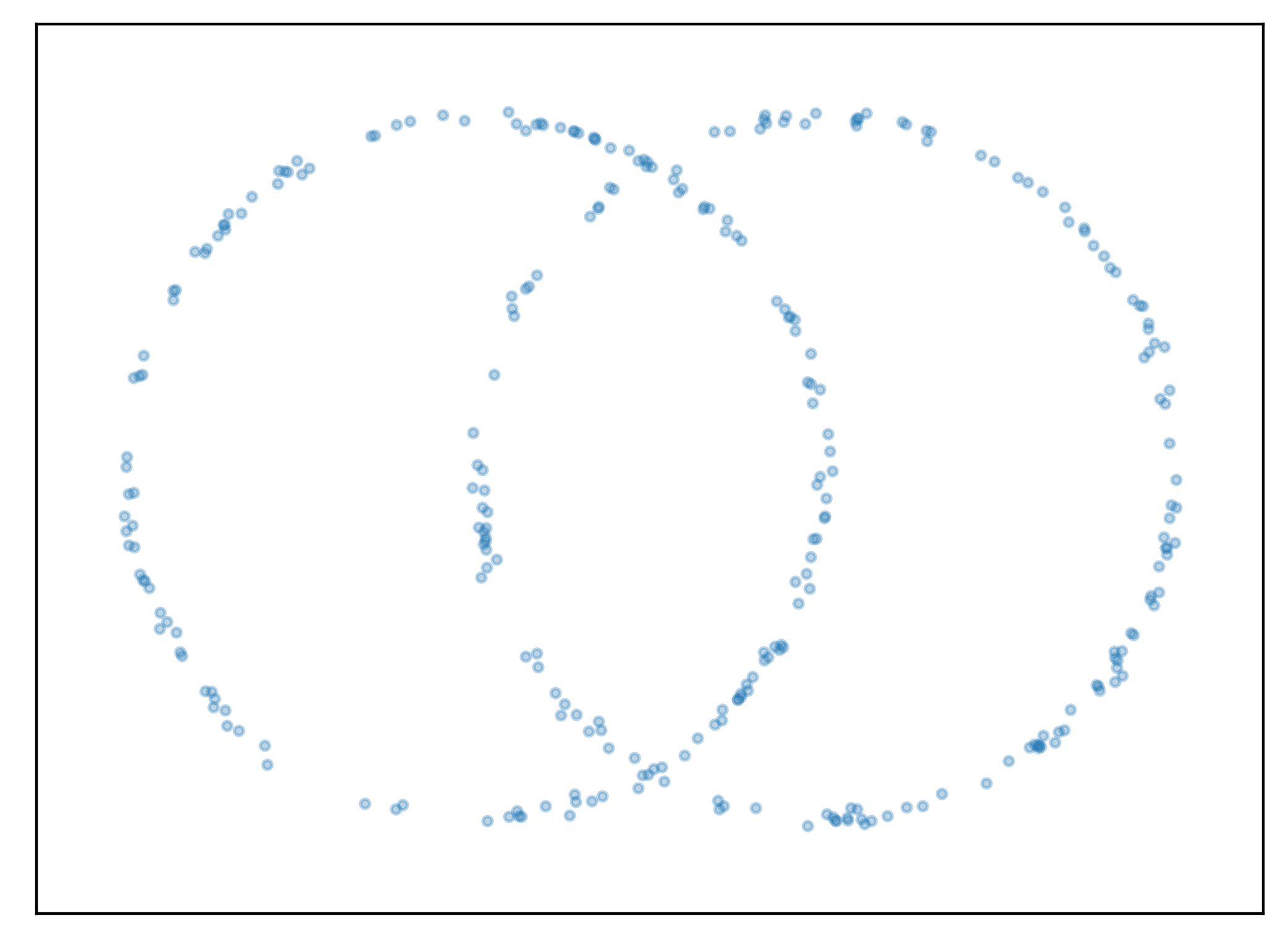}
\includegraphics[width = 2cm]{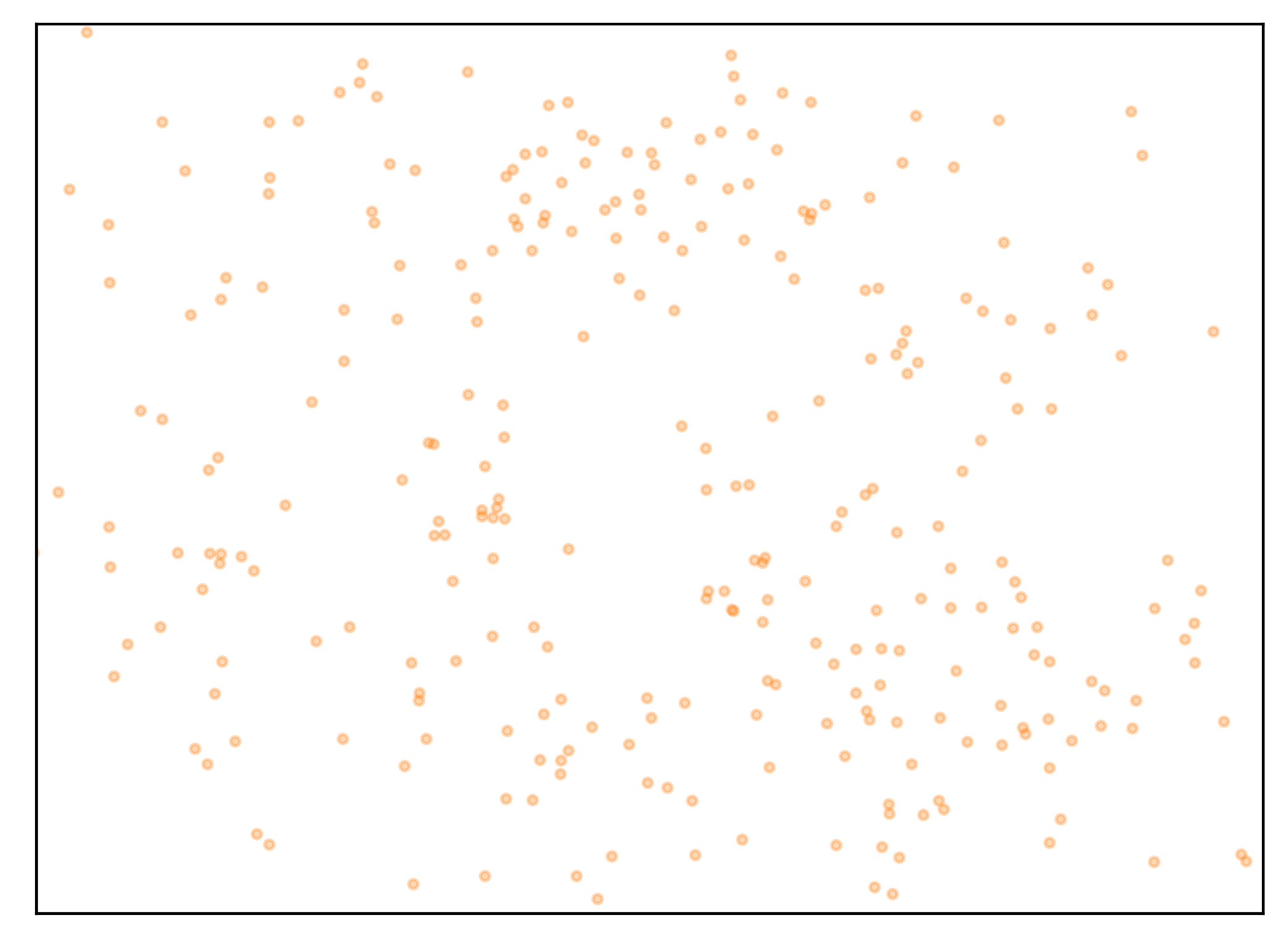}
\includegraphics[width = 2cm]{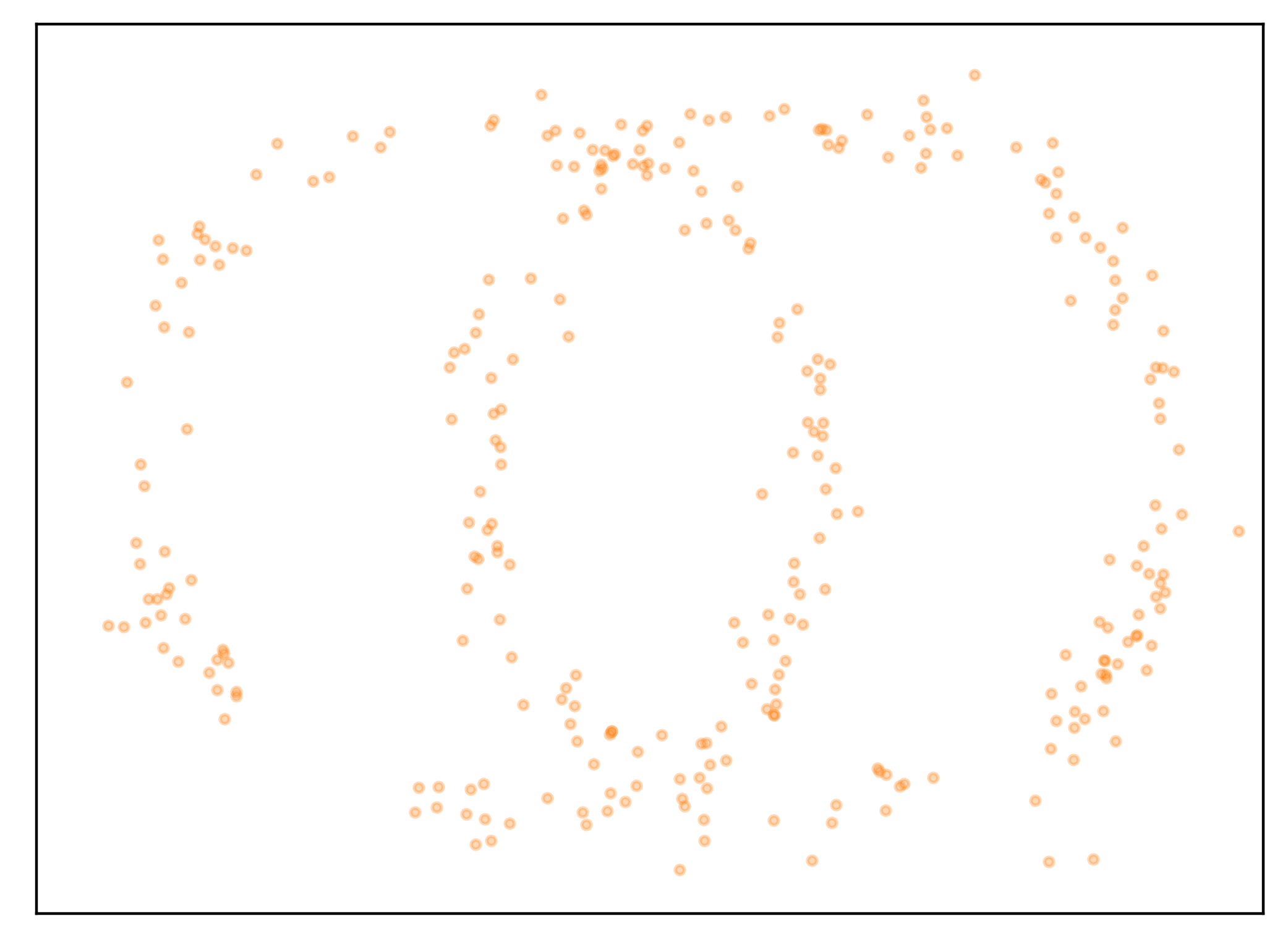}
\includegraphics[width = 2cm]{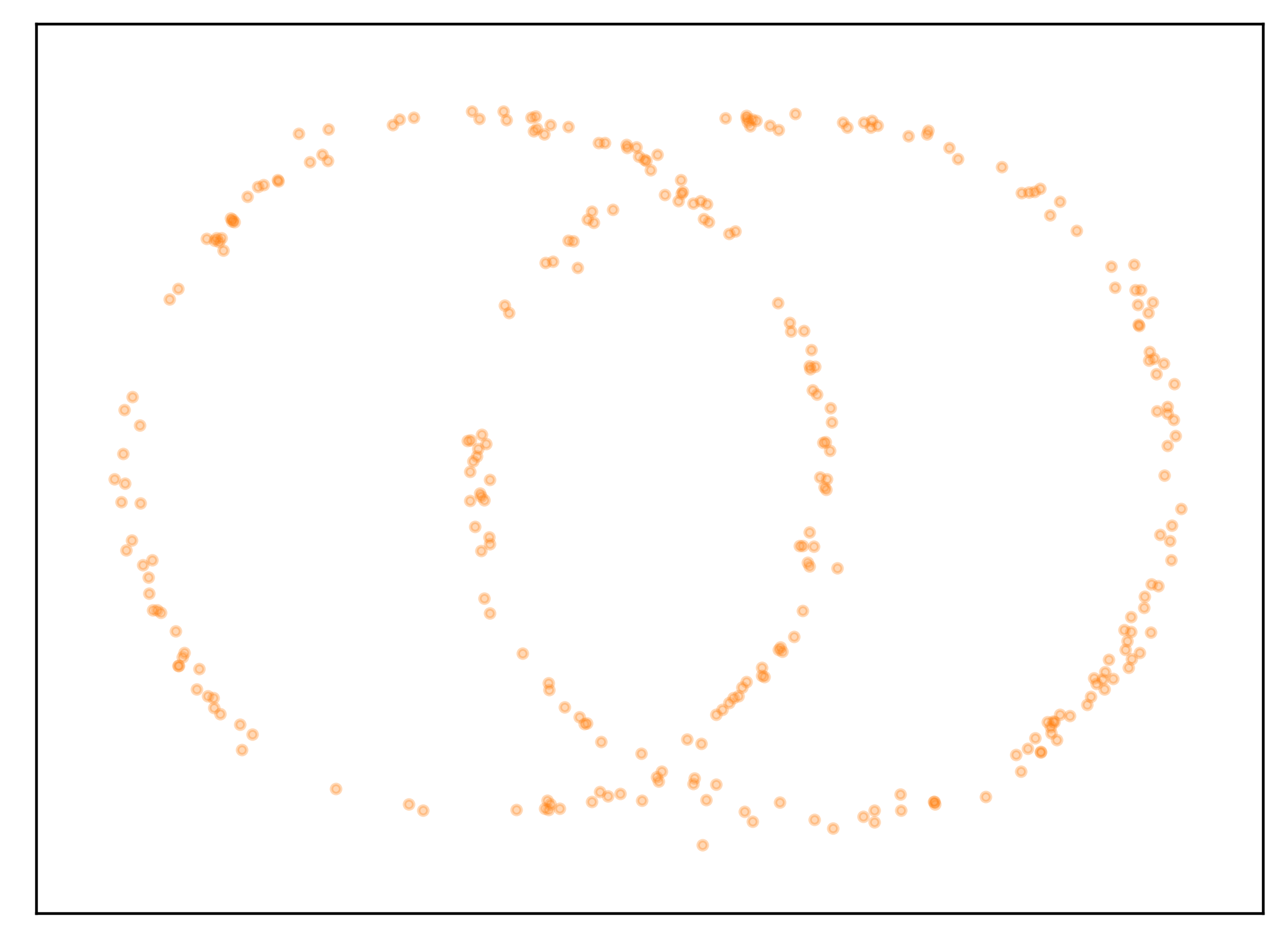}
\includegraphics[width = 2cm]{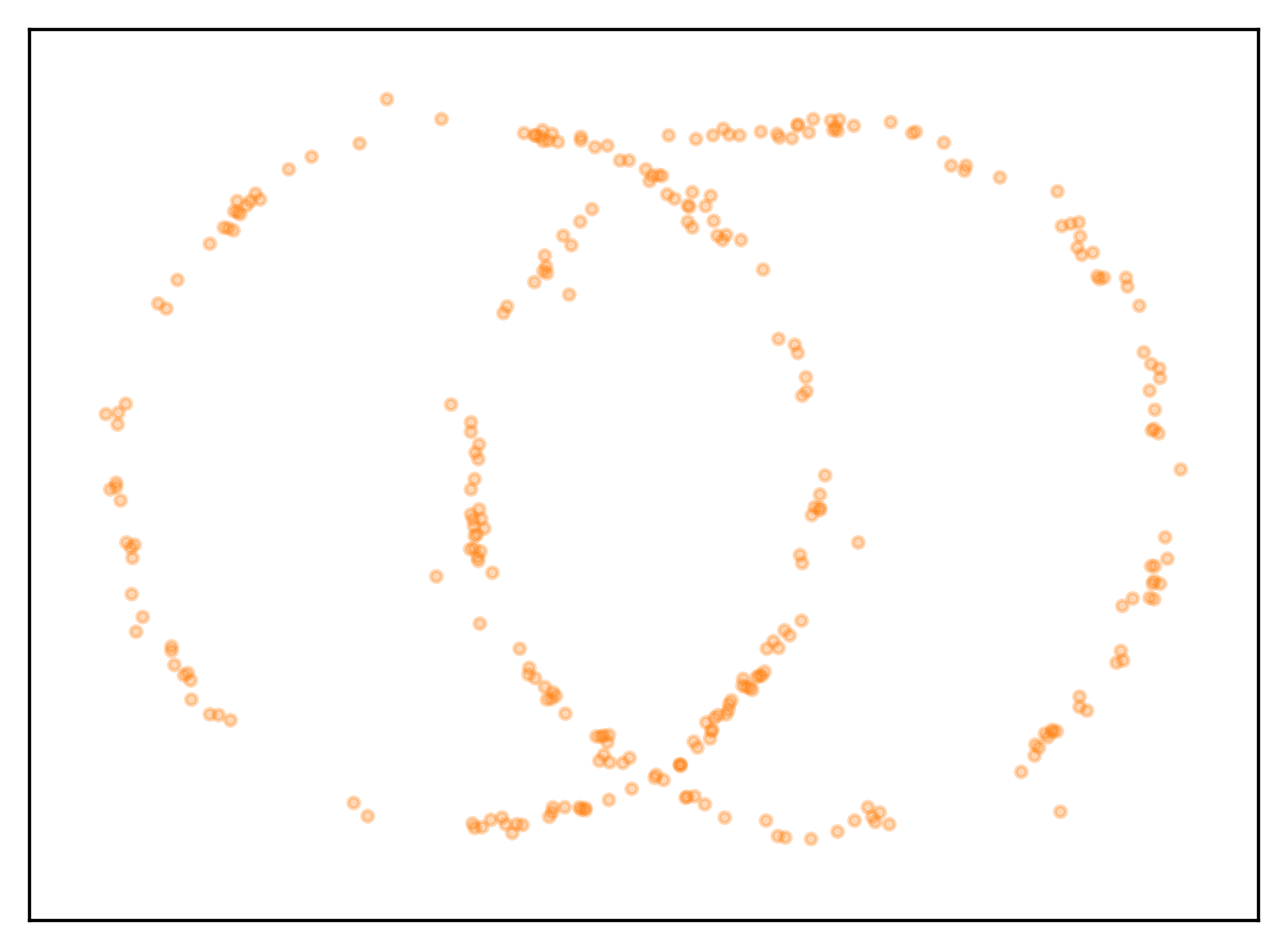}
\includegraphics[width = 2cm]{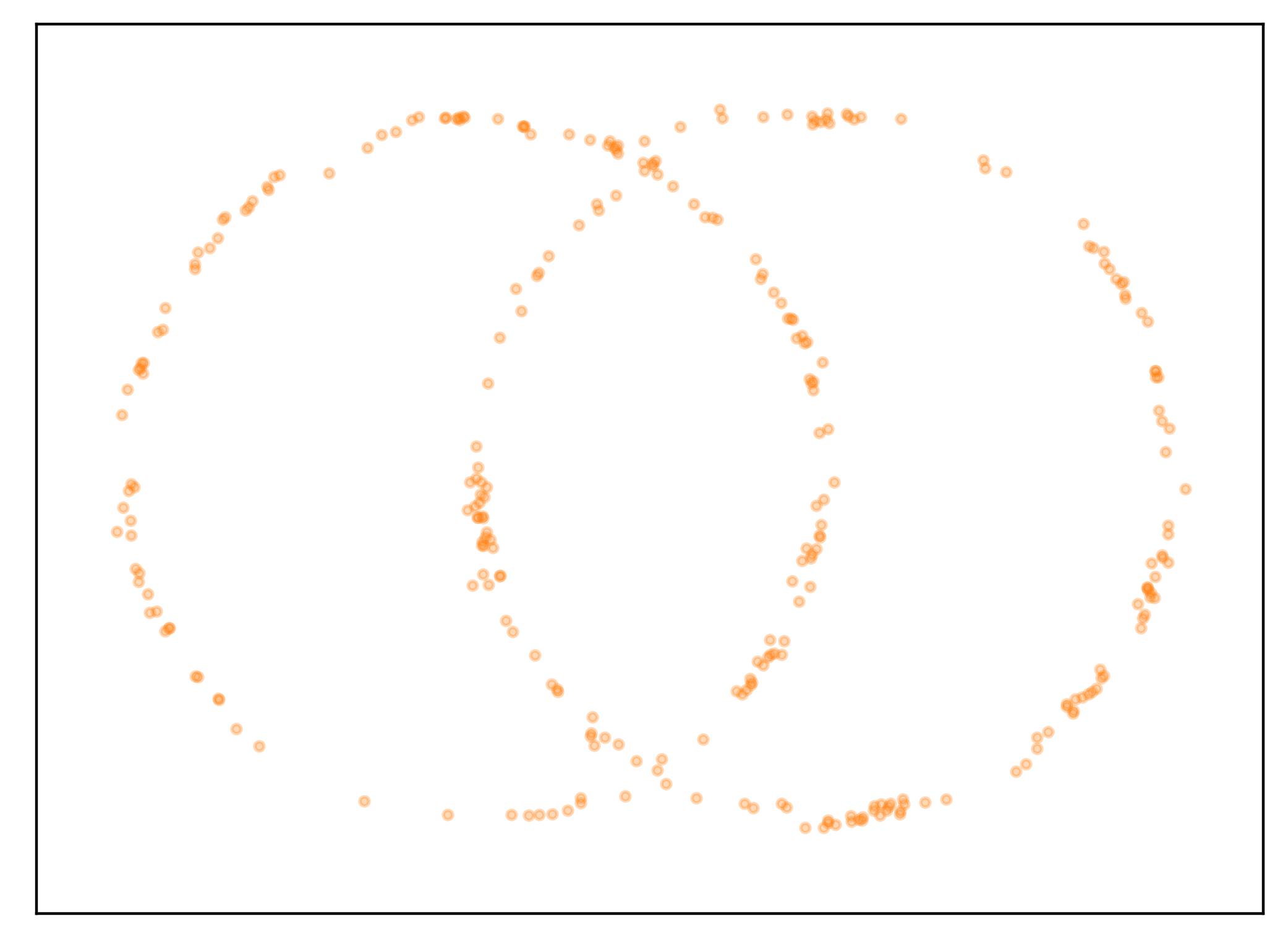} \\
\includegraphics[width = 2cm]{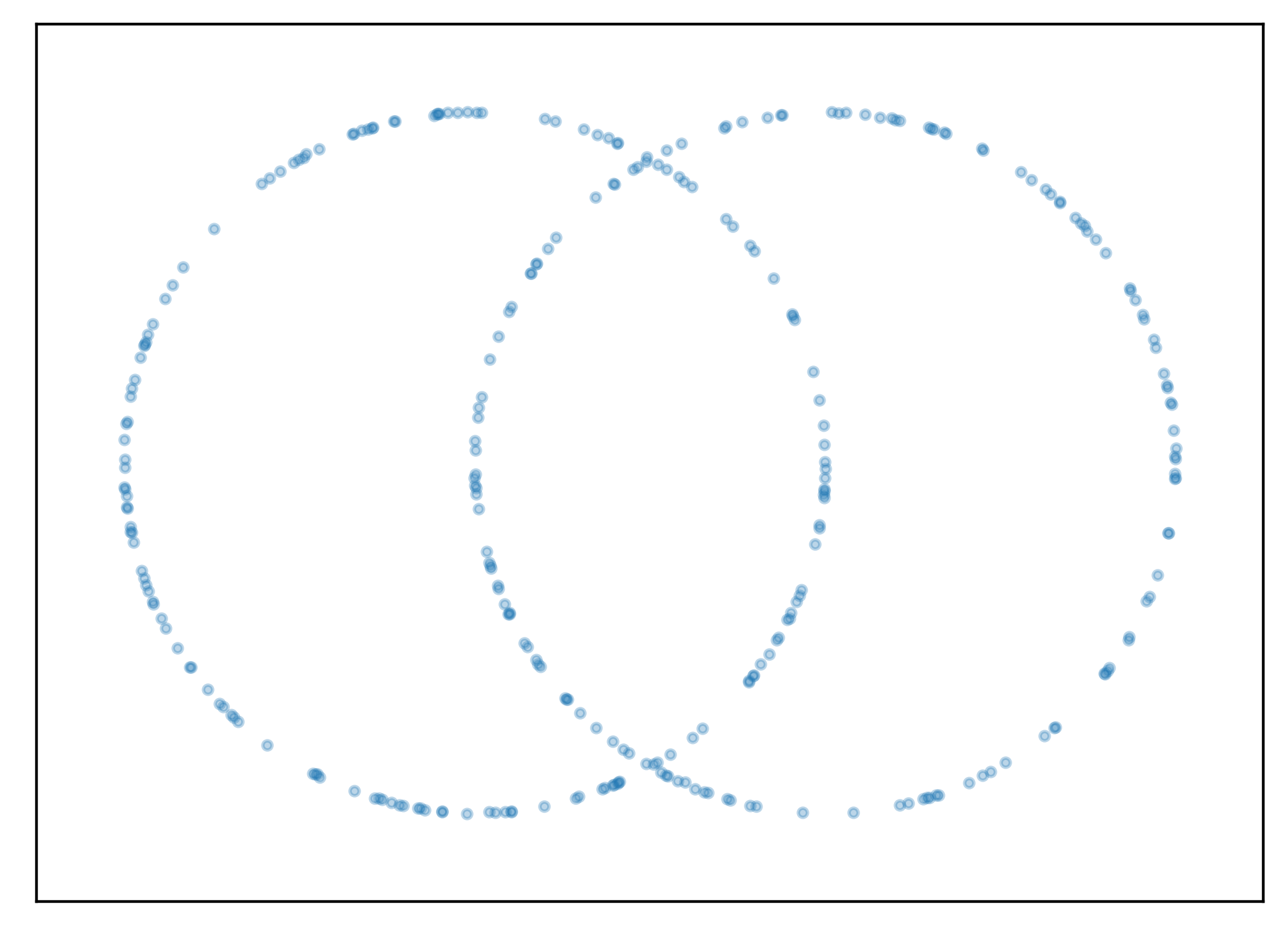}
\includegraphics[width = 2cm]{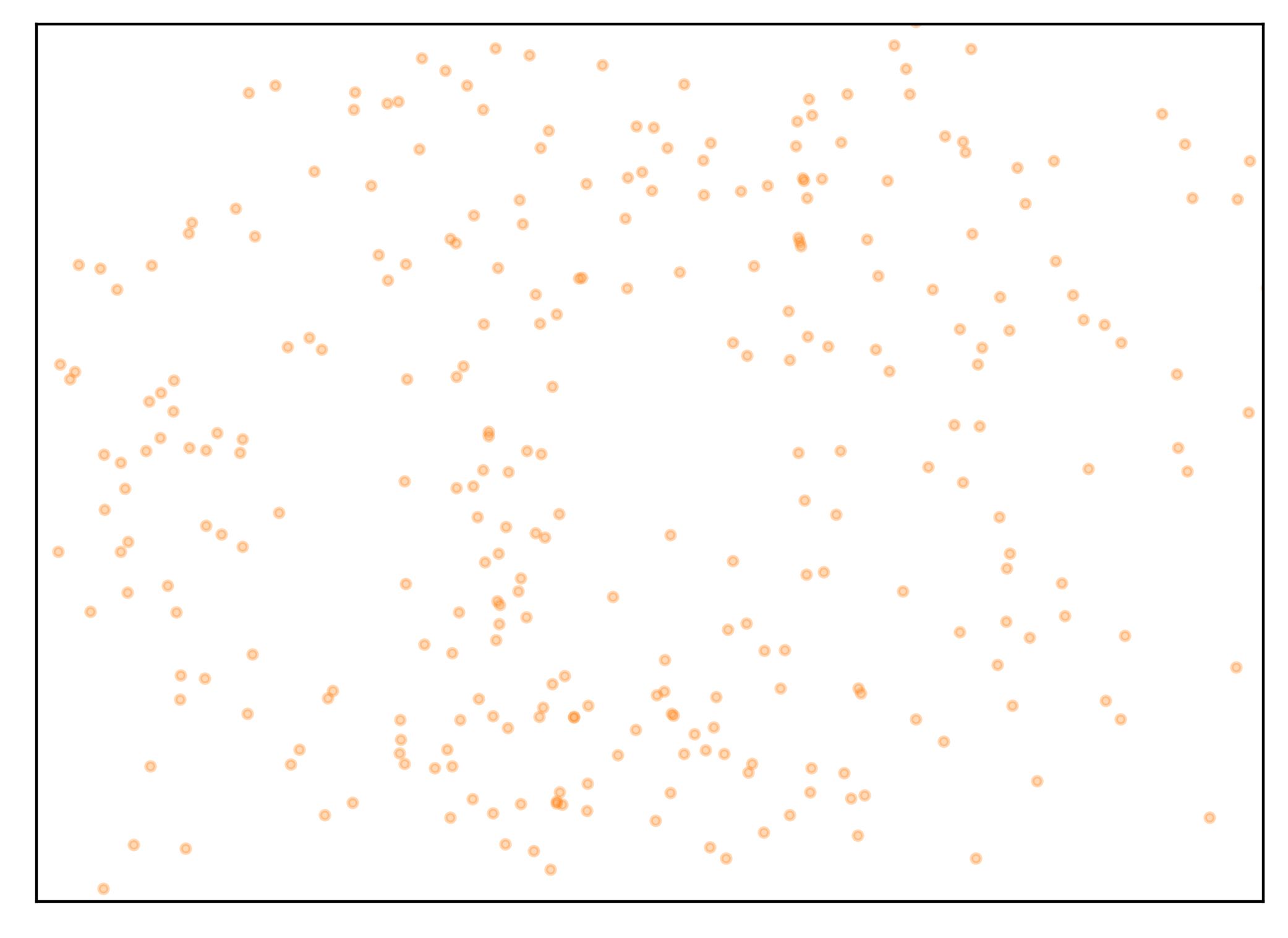}
\includegraphics[width = 2cm]{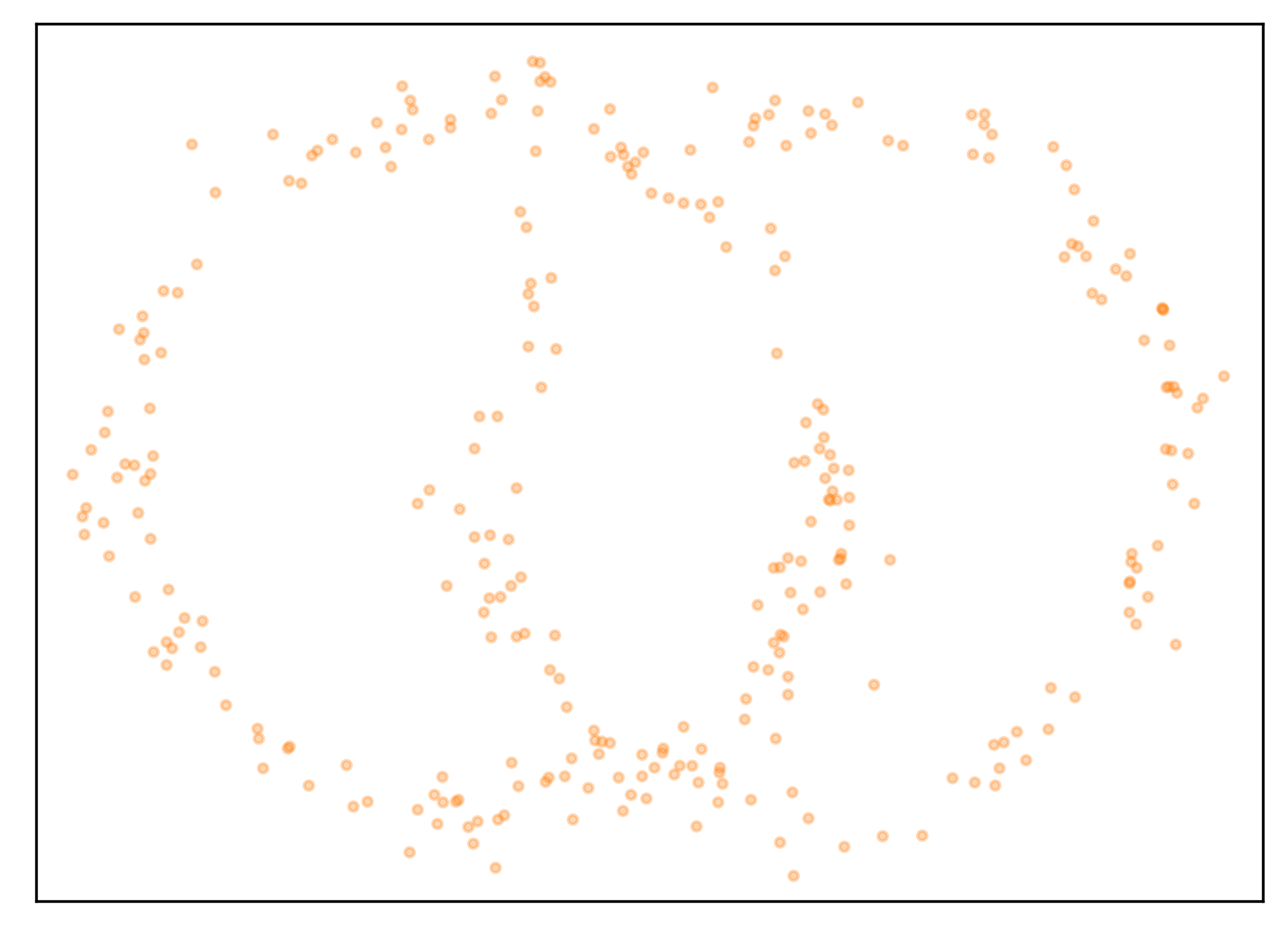}
\includegraphics[width = 2cm]{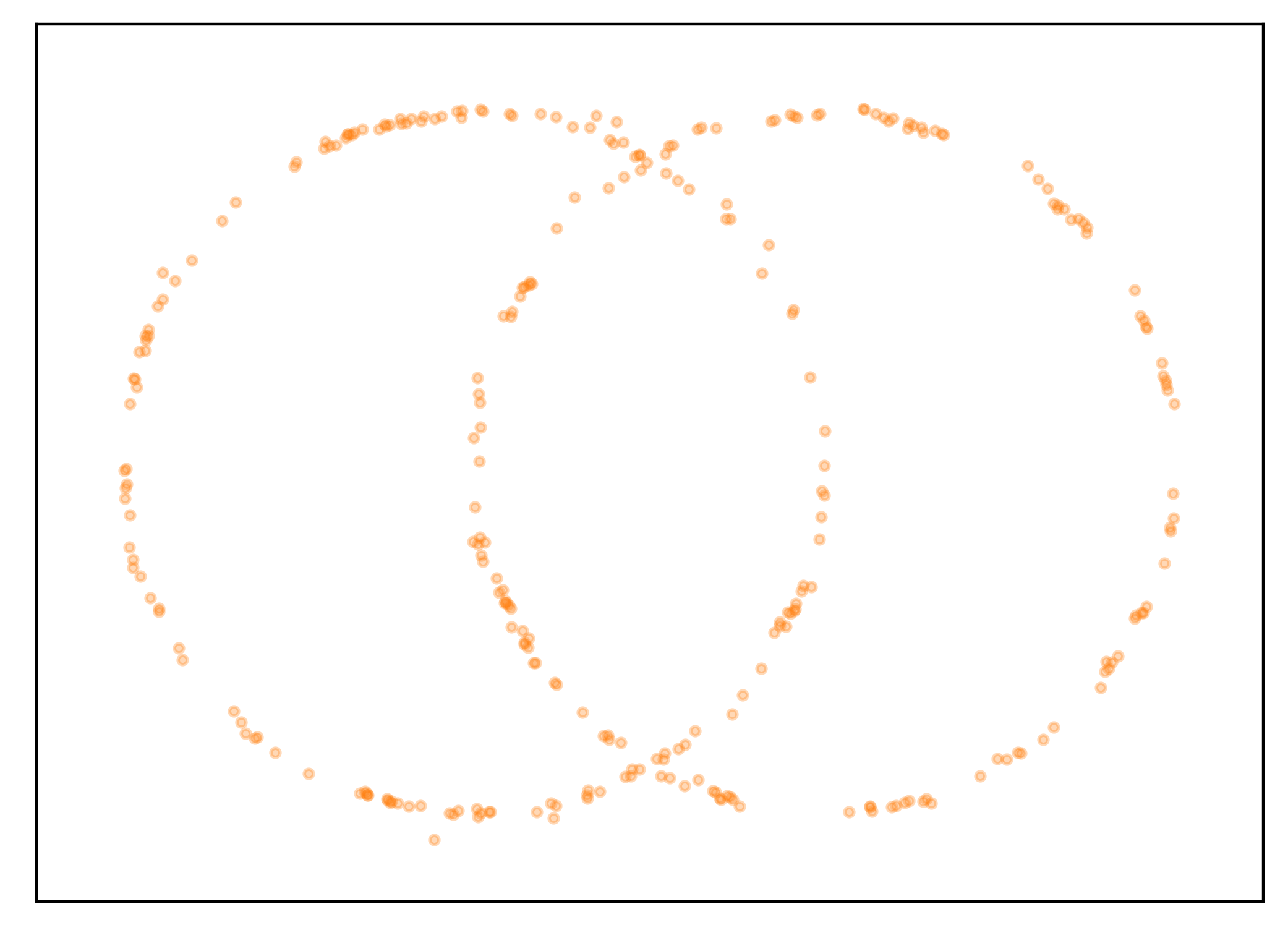}
\includegraphics[width = 2cm]{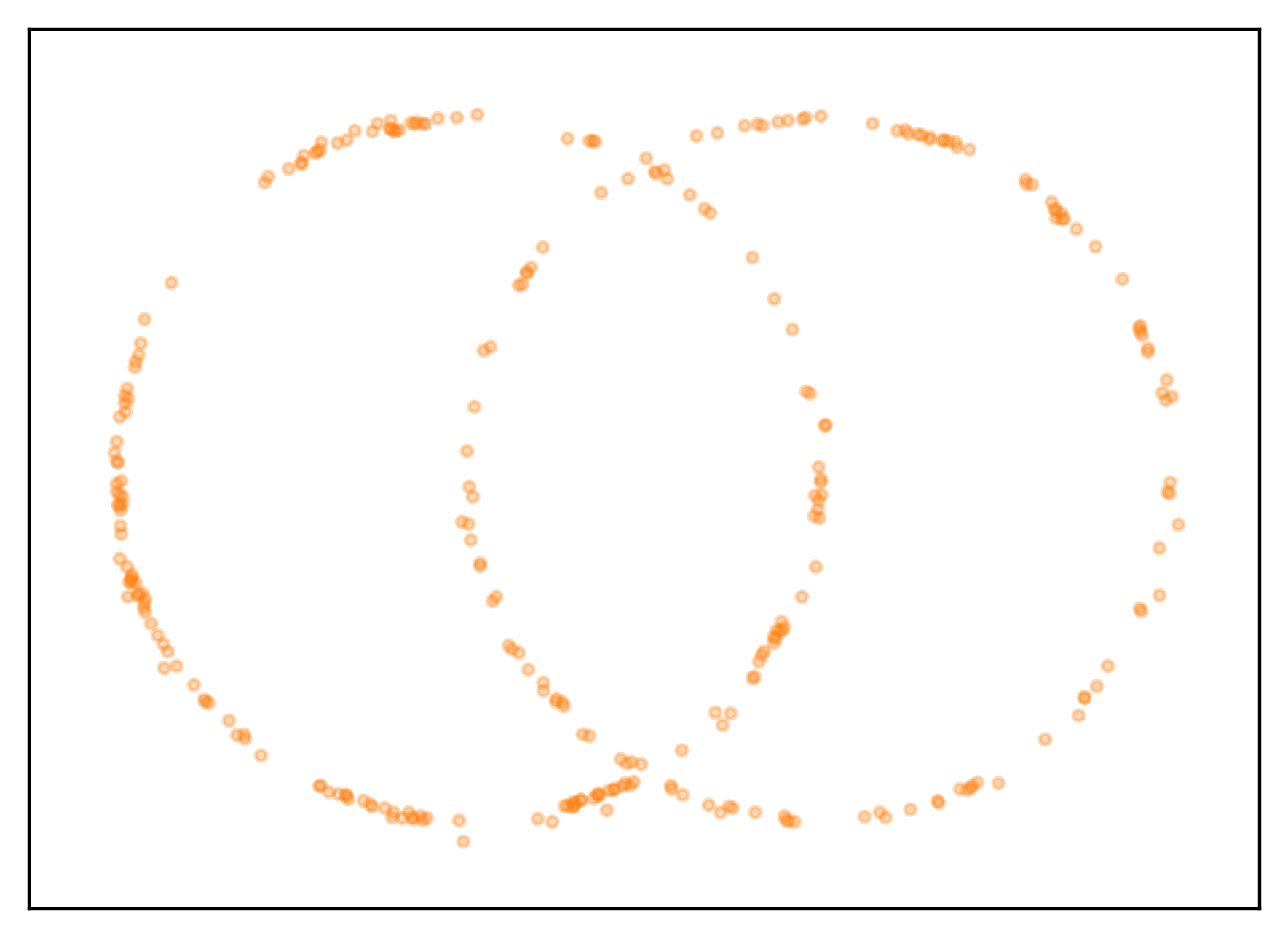}
\includegraphics[width = 2cm]{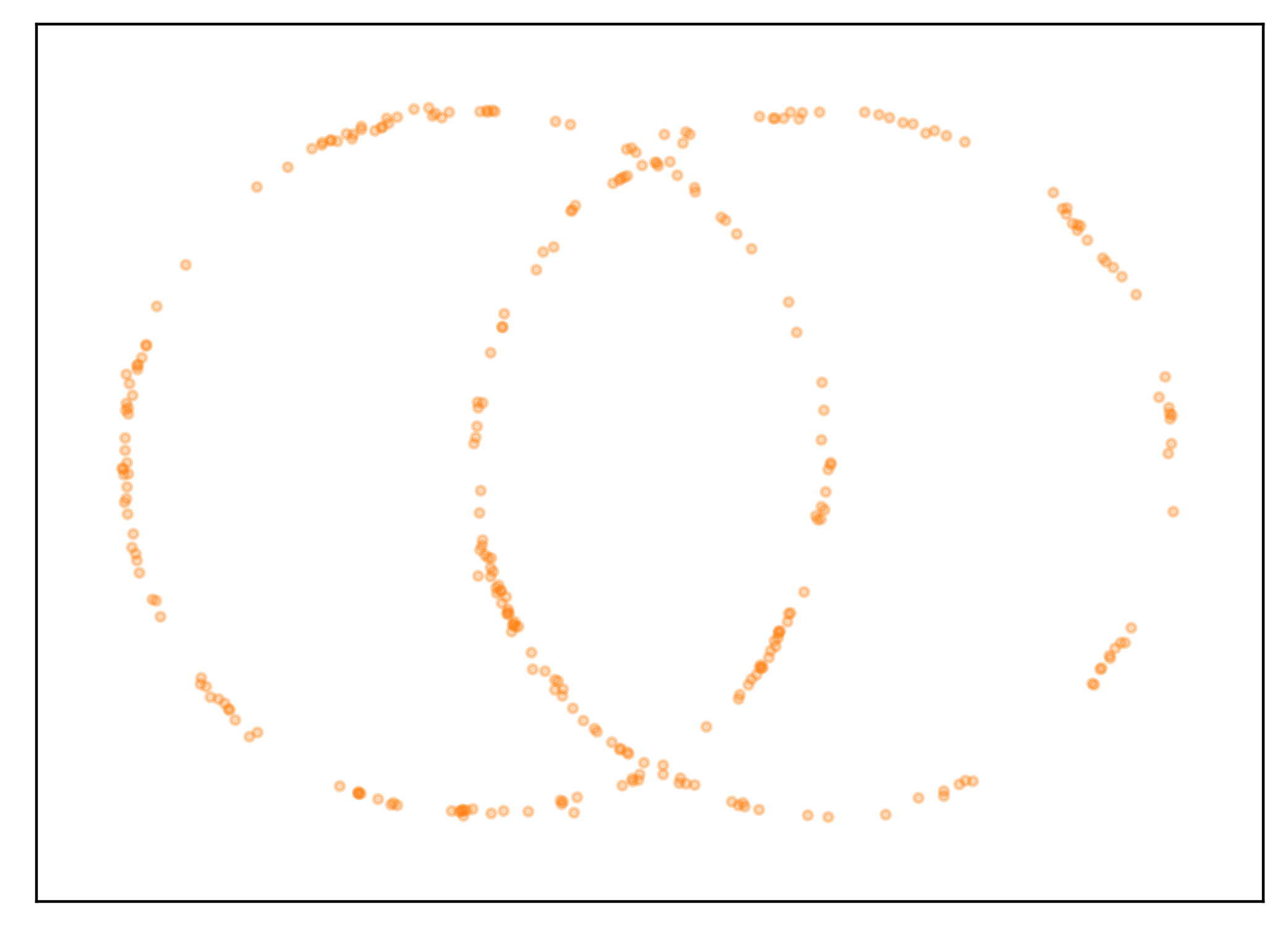}
\caption{Same experiment as in \figref{2Dspiral} but for the two overlapping circles.}
\label{fig:circles}
\end{figure}

\begin{figure}[h!]
\centering
\includegraphics[width = 3cm]{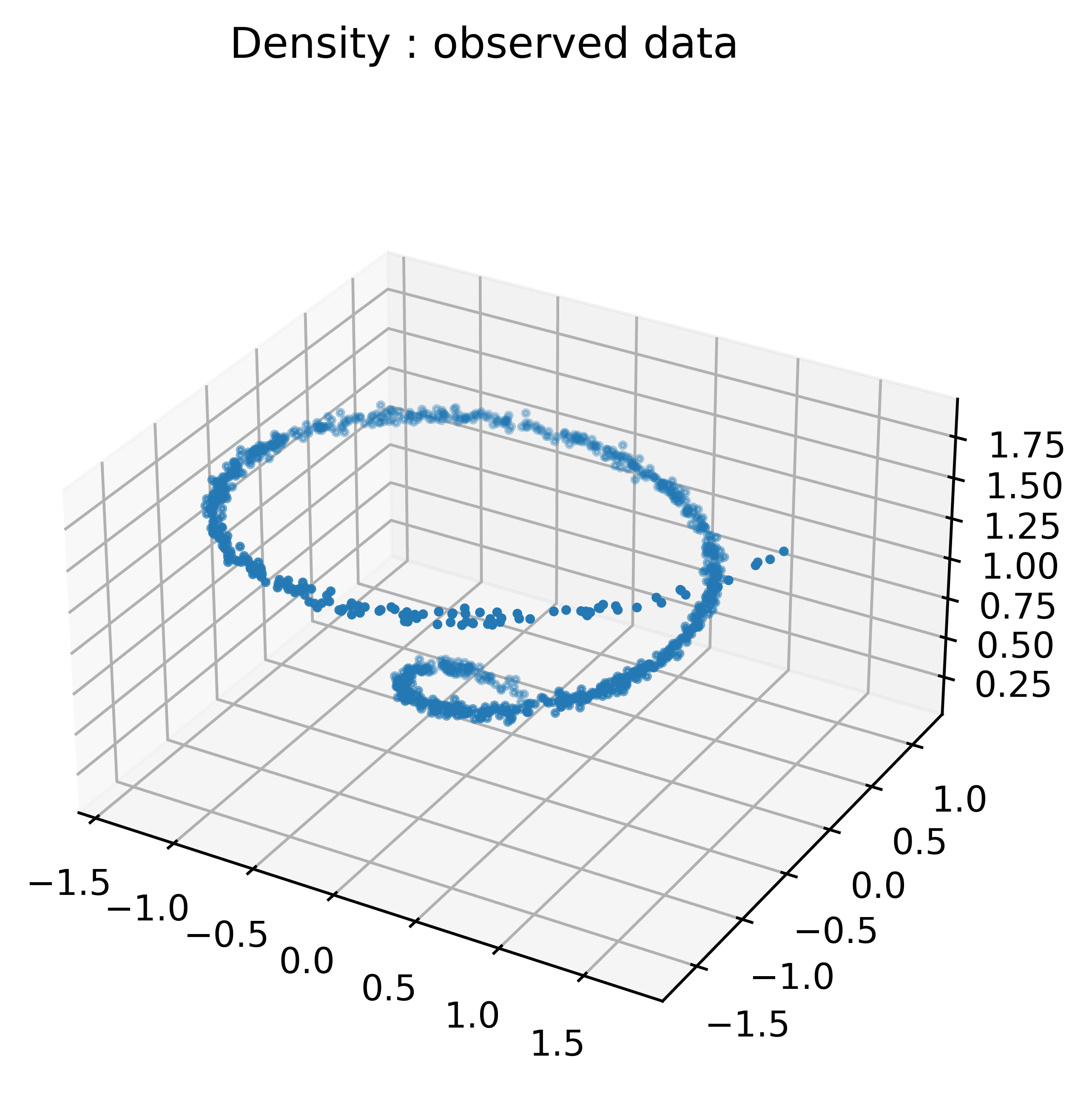}
\includegraphics[width = 3cm]{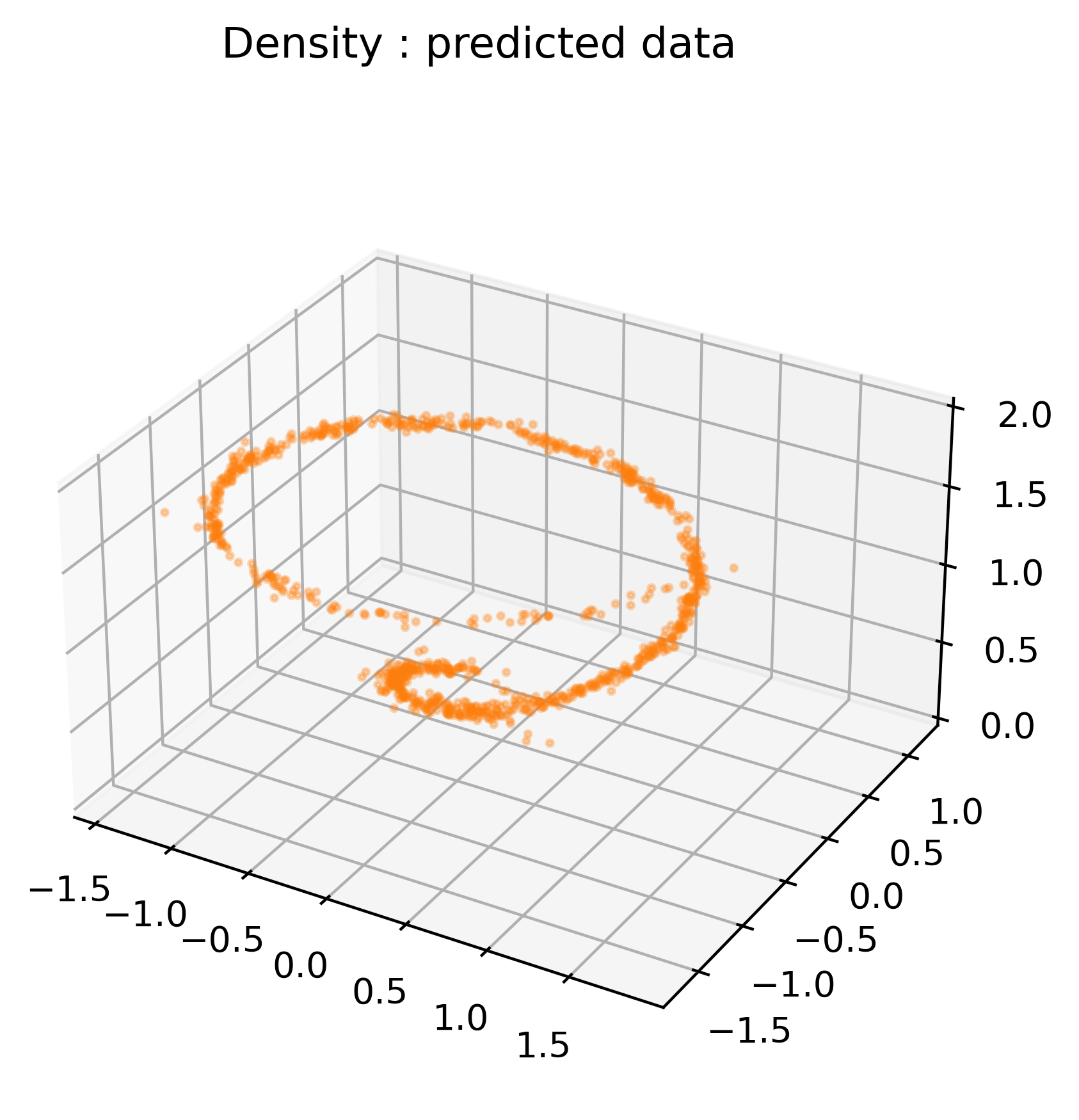}
\includegraphics[width = 3cm]{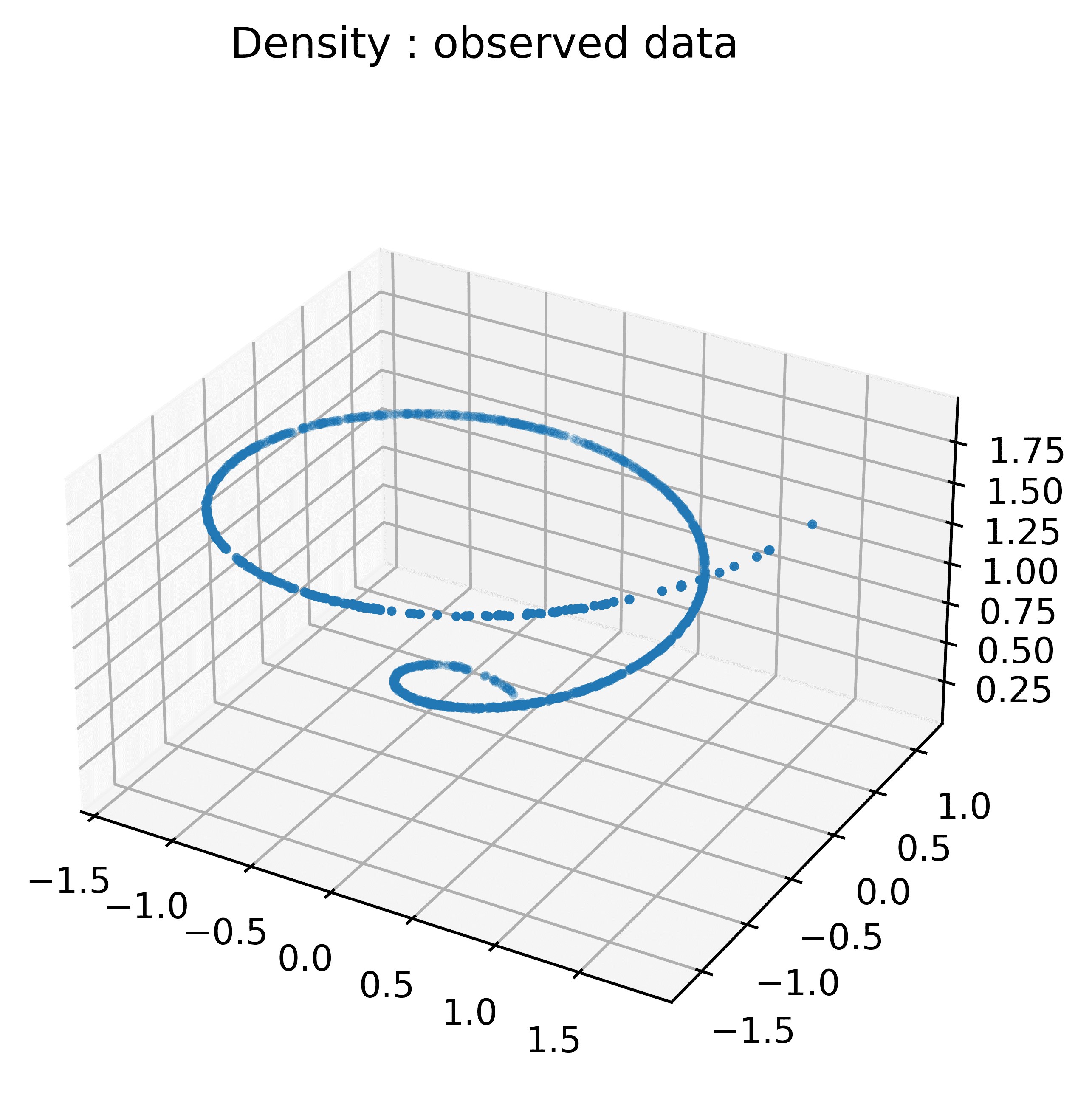}
\includegraphics[width = 3cm]{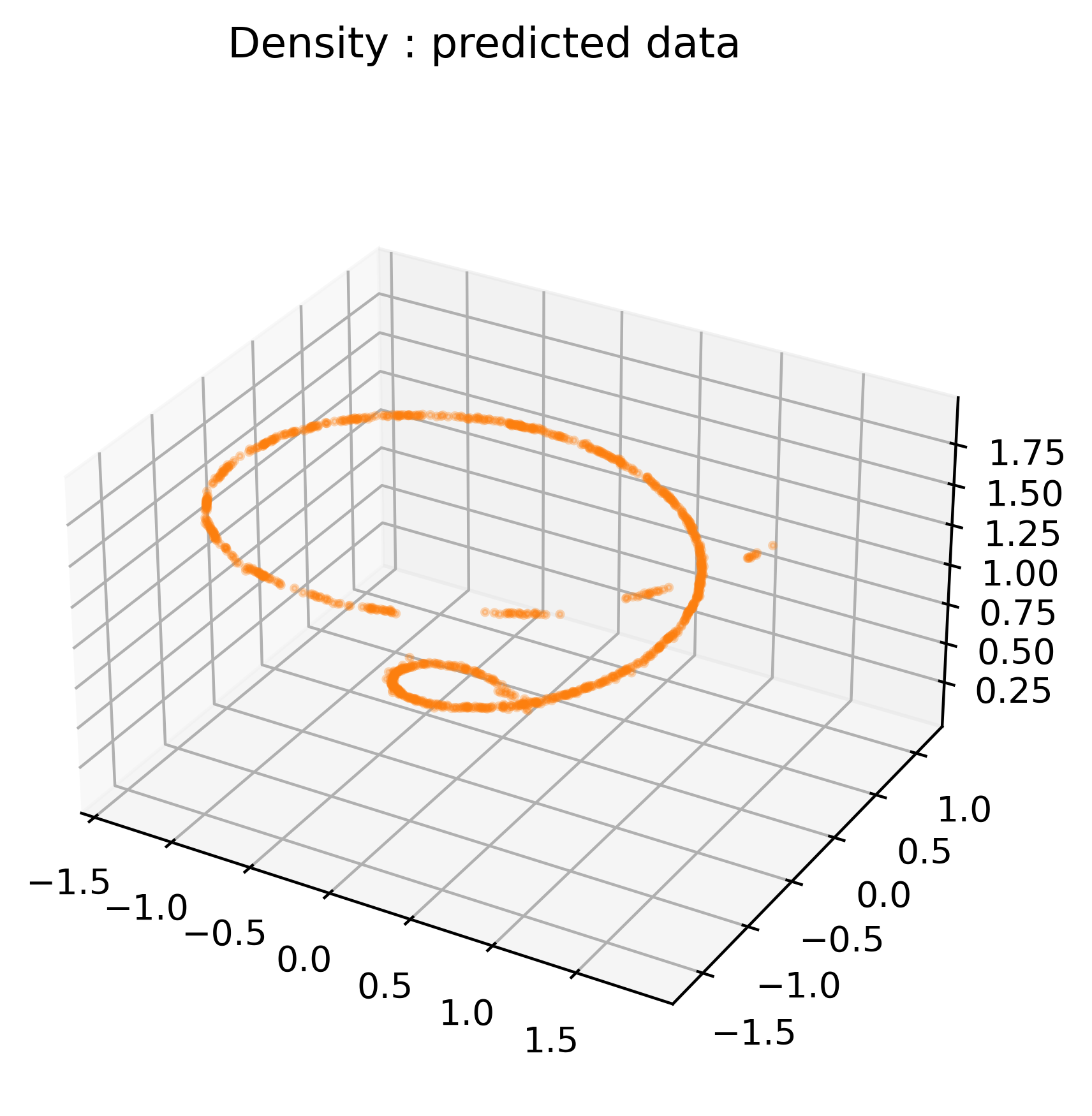}
\caption{We observe $n = 1000$ points (blue plots) and sample the same amount through the approximate MAP distribution (orange plots). The width $\delta$ is set to (Left two plots) $0.1$ and (Right two plots) to $0.01$.}
\label{fig:3Dspiral}
\end{figure}

\begin{figure}[h!]
\centering
\includegraphics[width = 3cm]{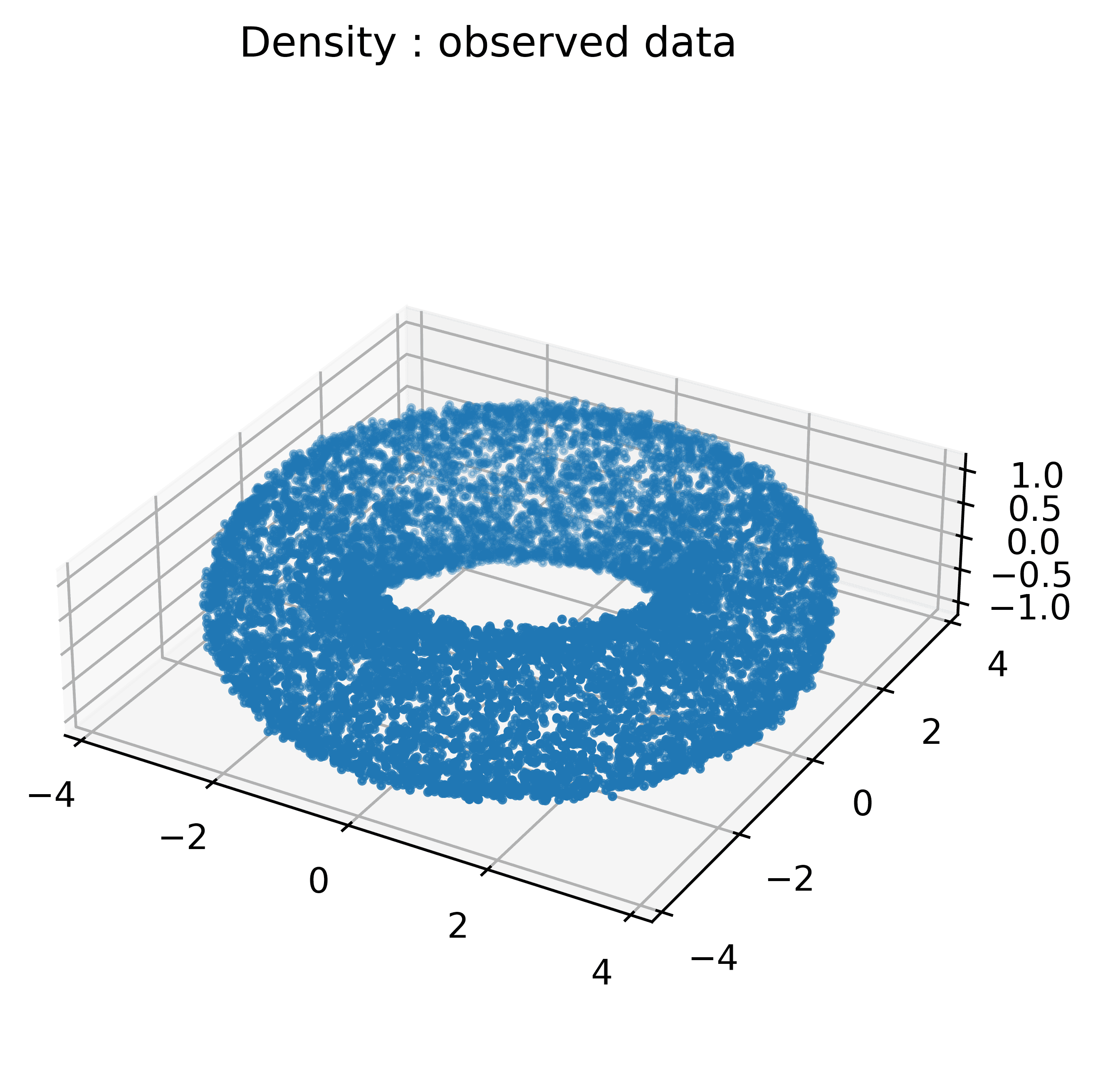}
\includegraphics[width = 3cm]{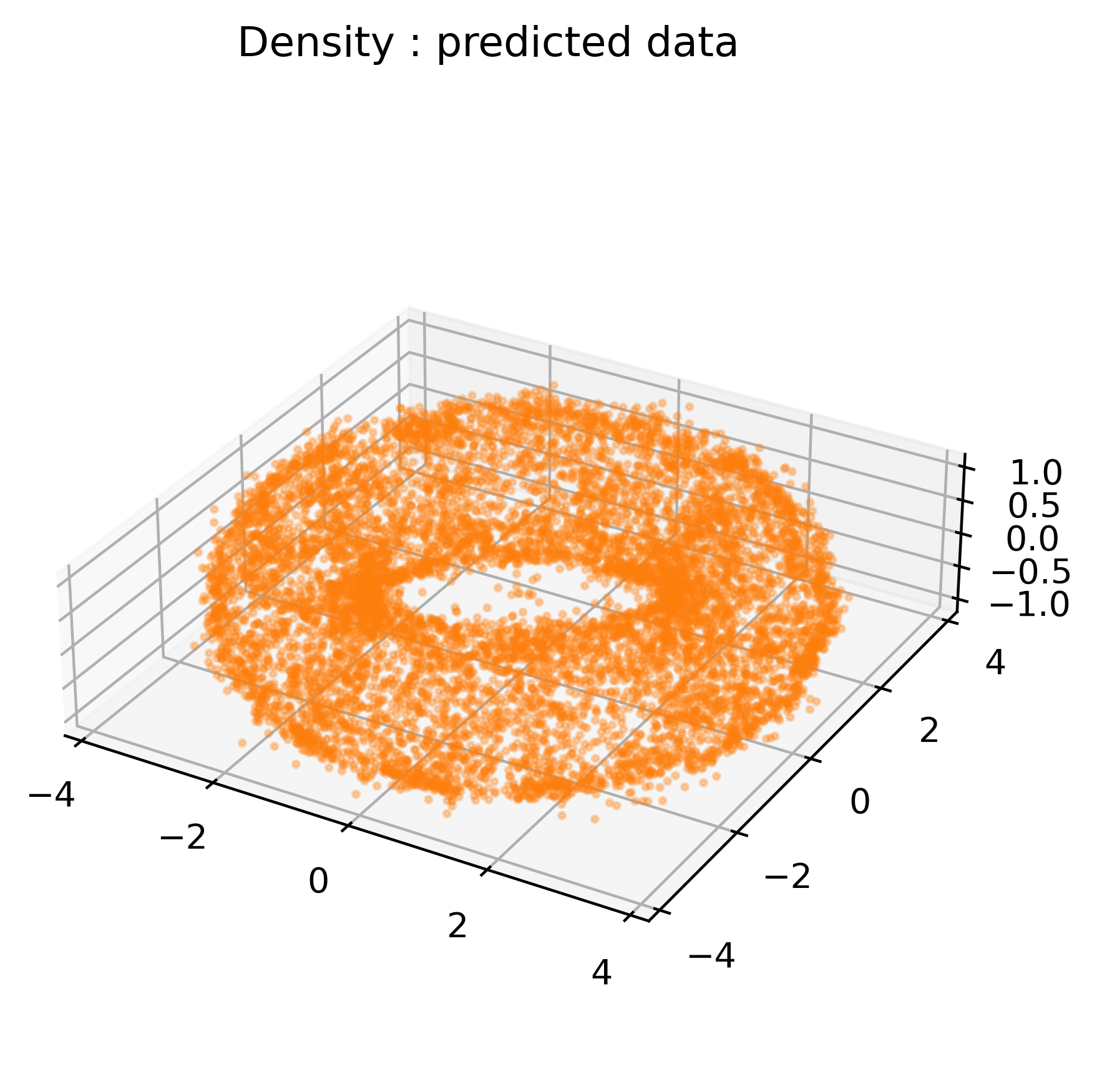}
\includegraphics[width = 3cm]{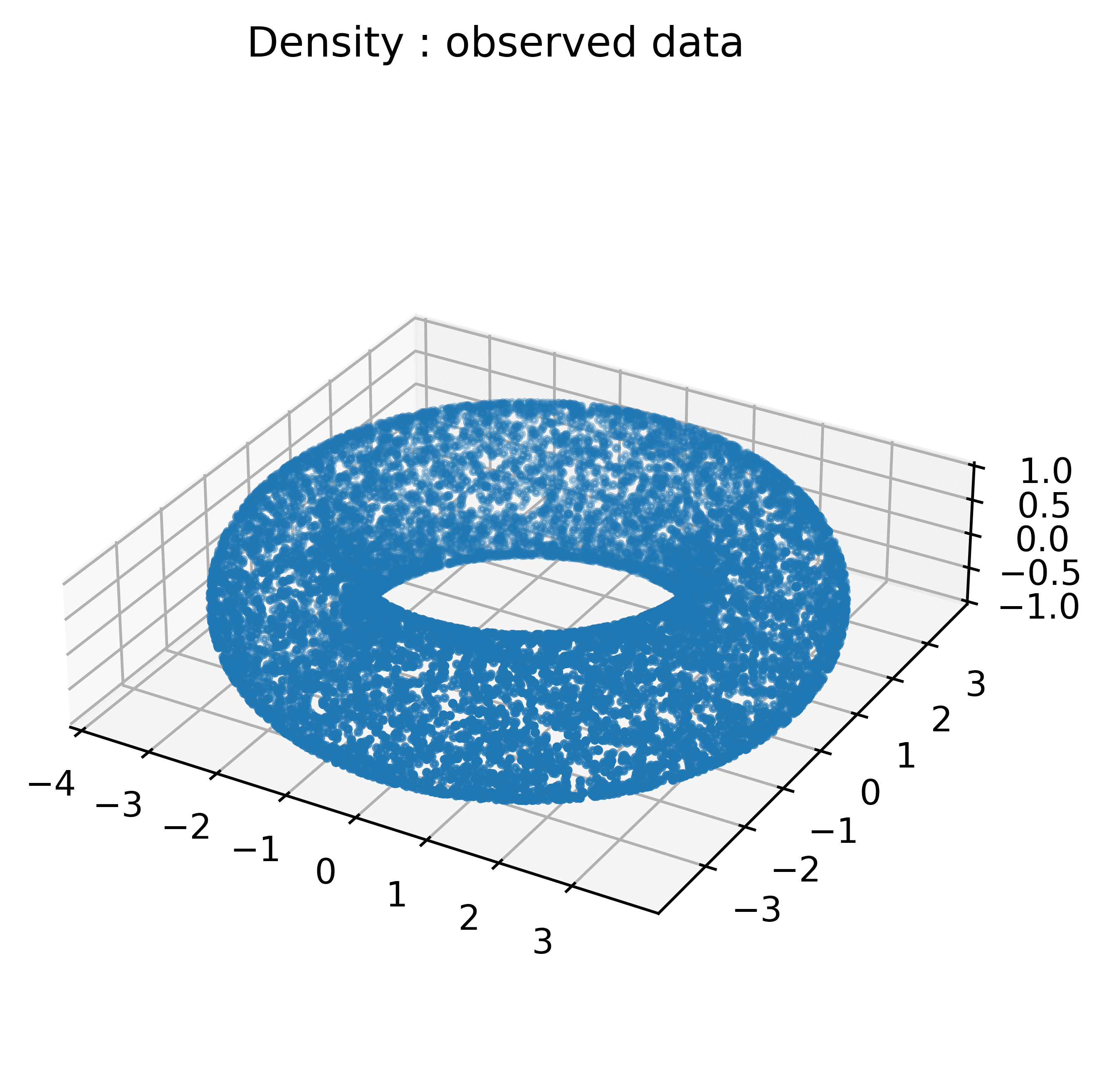}
\includegraphics[width = 3cm]{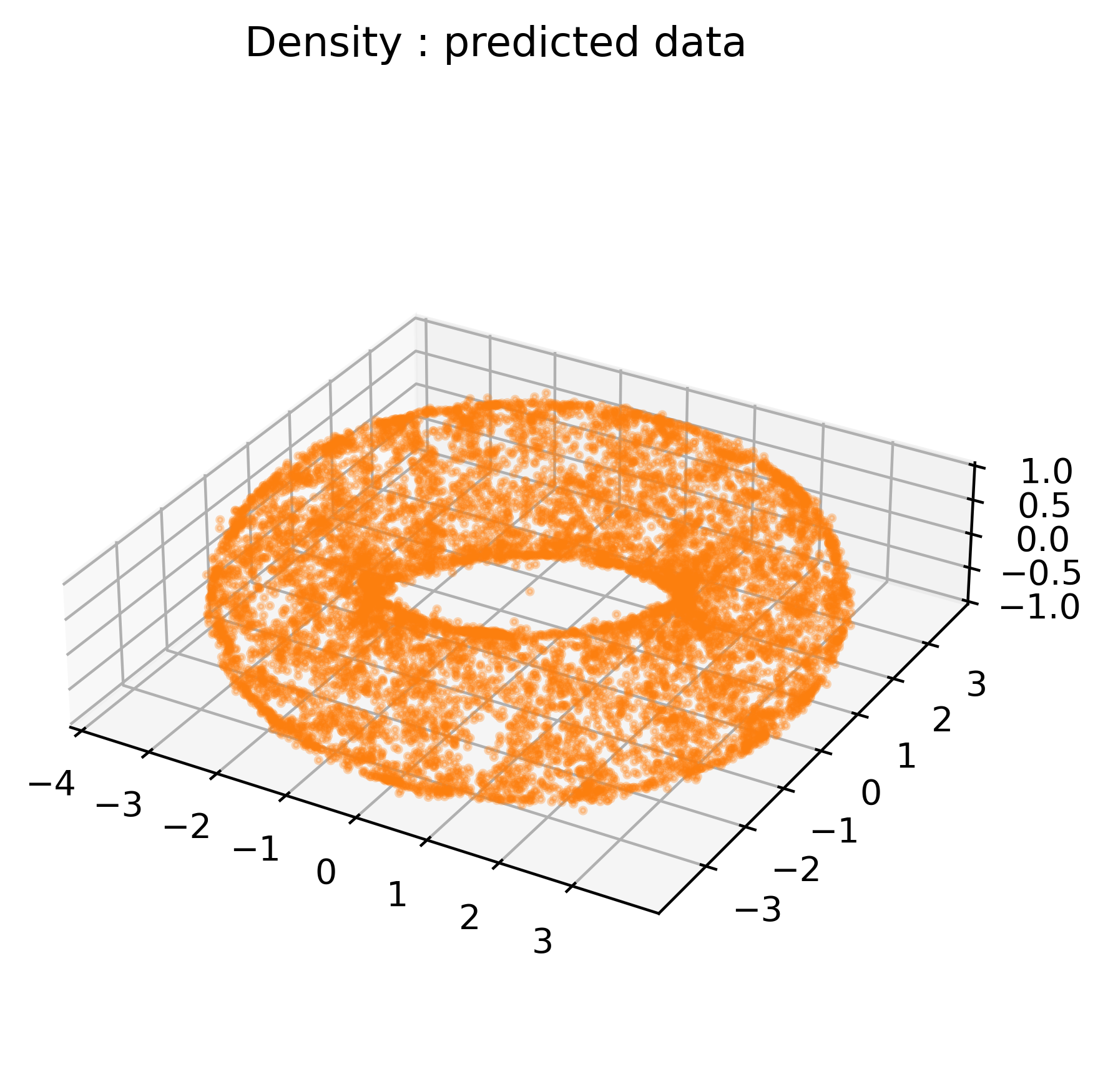}
\caption{Same experiment as in Figure \ref{fig:3Dspiral} for the torus with $n=10000$ points observed and sampled and with respectively (Left two plots) $\delta = 0.5$ and (Right two plots) $\delta = 0.05$.}
\label{fig:torus}
\end{figure}


\section{Proofs of the main results} \label{sec:mainproof}
		
		\subsection{Proof of \thmref{main}} \label{app:pr:thmain}
	Theorem \ref{thm:main} is proved using \cite[Thm 5]{ghosal2007posterior}, which relies on two things: making sure that the prior probability distribution puts enough mass around the true density $f_0$ (Kullback-Leibler condition), and ensuring that most of its probability mass is concentrated on a subset of manageable entropy (entropy condition). \newclem{Notice that \cite[Thm 5]{ghosal2007posterior} is expressed for a fixed true distribution $P_0$, but a careful look at the statement and at the proof of the theorem shows that this results extends straighforwardly to a true distribution $P_{0,n}$ depending on $n$, as their result is non-asymptotic in nature.}

\begin{proof}[Proof of Theorem \ref{thm:main}] 

The most challenging part of the proof is the Kullback-Leibler condition. To verify it we define for any $\ve > 0$,  
\beq \label{eq:defball}
\ball(f_0,\ve)= \{f : \bbR^D \to \bbR~\middle|~\bbP_0 \left(\log \frac{f_0}{f}\right) \leq \ve^2~~~\text{and}~~~\bbP_0\left(\log^2 \frac{f_0}f\right) \leq \ve^2\},
\eeq
and introduce
		\begin{align*} 
			\tilde\ve_n := \delta^{\frac{\beta}{2\alpha_0 - \alpha_\perp}} \wedge n^{-\frac{\beta}{2\beta+D}} ~~~~\text{and}~~~~
			\ve_n := \{\frac{C^{1/2}}{\sqrt{n}} \tilde\ve_n^{-D/2\beta} \log^{t/2}(1/\tilde\ve_n)\} \vee \{\tilde\ve_n \log^s(1/\tilde\ve_n)\} 
		\end{align*} 
		where $s$, $t$ and $C$ are introduced in \lemref{thickness}.
		The sequence $\tilde \ve_n$ goes to $0$ and is such that $\tilde \ve_n^{2\alpha_0} \leq \delta^\beta \tilde \ve_n^{\alpha_\perp}$ so that	$\ball(f_0,\tilde\ve_n \log^{s}(1/\tilde \ve_n)) \subset  \ball(f_0,\ve_n)$. The control of $\Pi\left(f_{P} \in \ball(f_0,\ve_n)\right)$ relies first on Theorem \ref{thm:approx} which constructs explicitely an approximation of $f_0$ by a continuous location scale mixture of Gaussians $K_\Sigma f_1$ with $f_1$ close to $f_0$ and then on a discrete approximation of  $K_\Sigma f_1$; it is formally stated in  \lemref{thickness} below and implies that  
		$	\Pi\left(f_{P} \in \ball(f_0,\tilde\ve_n \log^{s}(1/\tilde \ve_n)\right) \gtrsim \exp\(- C \tilde\ve_n^{-D/\beta} \log^{t}(1/\tilde\ve_n)\right)
		\gtrsim \exp(-n\ve_n^2).
		$.

To verify the entropy condition of \cite[Thm 5]{ghosal2007posterior}, we define, for any sequence $\ve_n$ going to $0$,  
	$\mathcal F_n = \mathcal F_n(\ve_n,R_0,H_0,\sigma_0,\sigma_1)$ to be the set of all probability density function $f_P$ with $P = \sum_{h \geq 1} \pi_h \delta_{\mu_h,U_h,\Lambda_h}$ such that
	\begin{align*} 
		\sum_{h > H_n} \pi_h \leq \ve_n,~~~~\forall h \leq H_n, \mu_h \in B(0,R_n)~~~~\text{and}~~~~ \forall h \leq H_n, \Lambda_h \in \mathcal Q_n= [\underline \sigma_n^2, \bar \sigma_n^2]^D, U_h \in \cO(D)
	\end{align*} 
	where $R_n =  \exp(R_0 n\ve_n^2 ) $, $H_n =  \lfloor H_0 (n\ve_n^2)/\log n \rfloor$,  $\underline{\sigma}_n^2 = 	\sigma_0^2 (n\ve_n^2)^{-1/b_3}$ for some positive constant $R_0,H_0$ and $\sigma_0$ and where, for some $\sigma_1 > 0$,
	\begin{itemize}
		\item $\bar \sigma_n^2 = \exp( \sigma_1^2 n\ve_n^2 ) $ in the case of the Partial location-scale mixture
		\item $\bar \sigma_n^2 = \sigma_1^2 (n\ve_n^2)^{1/b_4} $ in the case of the Hybrid location-scale mixture.
	\end{itemize} 
	We then show in \lemref{sieve} below that $\Pi(\mathcal F_n^c) \lesssim \exp\left(- c_1n\ve_n^2\right)$ as soon as $n \ve_n^2 \geq n^\omega$ for some $\omega > 0$.

 We now define, for  $\mathbf j = (j_h, h\leq H_n) \in \bbN^{H_n}$,
	$$ \mathcal F_{n, \mathbf j}  := \{f_P \in \cF_n~|~  \forall h\leq h_n,~  j_h  \sqrt{n} < \|\mu_h\| \leq ( j_h+1)\sqrt{n}\} $$
	along with the following refinement
	\begin{align*} 
	\mathcal F_{n, \mathbf j, 0}  &:= \{f_P \in \mathcal F_{n, \mathbf j}  ~\middle|~ \frac{ \max_i\lambda_{i} }{ \min_i\lambda_{i}} \leq n \}\\
	\text{and}~~~\forall \ell \geq 1,~ \mathcal F_{n, \mathbf j, \ell}  &:= \{f_P \in \mathcal F_{n, \mathbf j}  ~\middle|~ \, n^{ 2^{\ell-1} } < \frac{ \max_i\lambda_{i} }{ \min_i\lambda_{i}} \leq n^{ 2^{\ell}} \}. 
	\end{align*} 	
In Lemma \ref{lem:sieve} below we show that for the partial location scale prior
$$\sum_{\mathbf j, \ell} \sqrt{\Pi(\mathcal F_{n, \mathbf j, \ell} )N(\ve_n, \mathcal F_{n, \mathbf j,  \ell}, \| \cdot \|_1 ) } e^{-M_0 n \ve_n^2 } = o(1);$$ 		
		while for the hybrid one the same holds true with $\mathcal F_{n, \mathbf j}$.

Hence \cite[Thm 5]{ghosal2007posterior} implies that  the posterior contracts at rate $\ve_n$, so long as $n\ve_n^2 \geq n^{t_0}$ for some $t_0 > 0$. We now distinguish three cases. Denoting $p = s\vee t/2$, there holds:
		\bitem
		\item[1.] If first $\delta^{\frac{D}{2\alpha_0-\alpha_\perp}} \leq n^{-1} $ then the results of \thmref{main} is trivial and there is nothing to show (because the contraction rate goes to $\infty$ instead of $0$); 
		\item[2.] If $\delta^{\frac{D}{2\alpha_0-\alpha_\perp}} \geq n^{-\frac{D}{2\beta+D}}$, then $\tilde\ve_n = n^{-\frac{\beta}{2\beta+D}}$ and one easily gets that $\ve_n \simeq n^{-\frac{\beta}{2\beta+D}} \log^p n$. In particular, $n\ve_n^2 \gg n^{t_0}$ for $t_0=  D/(2\beta+D)$.
		\item[3.] If finally $n^{-1} < \delta^{\frac{D}{2\alpha_0-\alpha_\perp}}  < n^{-\frac{D}{2\beta+D}}$, then $\tilde\ve_n = \delta^{\frac{\beta}{2\alpha_0 - \alpha_\perp}}$ and $\log(1/\delta) \gtrsim \log n$ so that
		$$n \ve_n^2 \geq C \tilde\ve_n^{-D/\beta} \log^{t}(1/\tilde\ve_n) \gtrsim  \delta^{-\frac{D}{2\alpha_0 - \alpha_\perp}} \log^t(n) \gg  n^{\frac{D}{2\beta+D}}.
		$$
Finally to understand how $\ve_n$ depends on $\delta$ in this last case, note that  by assumption,  $\log(1/\tilde\ve_n) \lesssim \log(1/\delta) \lesssim \log(n)$ and $\tilde\ve_n < n^{-\beta/(2\beta+D)}$, so that
		$$
		\ve_n \lesssim \log^p n \times \{\frac{\tilde\ve_n^{-D/2\beta}}{\sqrt{n}} \vee \tilde\ve_n\} = \log^p n \times \frac1{\sqrt{n \delta^{D/(2\alpha_0 - \alpha_\perp)}}}.
		$$
		\eitem
\end{proof}	

We end this section with the presentations of the technical lemmata that are keys to the proof of \thmref{main}.

	\begin{lem}\label{lem:sieve} Under assumptions of Table \ref{tab:condtab}, for any sequence $\ve_n \to 0$ such that $n\ve^2_n \geq n^{t_0}$, for some $\omega >0$, and for all $c_1>0$, if $R_0, H_0, \sigma_1^2 $ are large enough and $\sigma_0^2$ is small enough,  there exists $M_0>0$ such that, 
		\bitem 	
		\item[i)] $\Pi(\mathcal F_n^c) \lesssim \exp\left(- c_1n\ve_n^2\right)$; 
		\item[ii)] In the case of the partial location-scale prior 
		$$\sum_{\mathbf j, \ell} \sqrt{\Pi(\mathcal F_{n, \mathbf j, \ell} )N(\ve_n, \mathcal F_{n, \mathbf j,  \ell}, \| \cdot \|_1 ) } e^{-M_0 n \ve_n^2 } = o(1);$$ 		
		\item[iii)] In the case of the hybrid location-scale prior 
		$$\sum_{\mathbf j} \sqrt{\Pi(\mathcal F_{n, \mathbf j} )N( \ve_n, \mathcal F_{n, \mathbf j}, \| \cdot \|_1 ) } e^{-M_0 n \ve_n^2 } = o(1).$$ 		
		\eitem
	\end{lem}

The proof of \lemref{sieve} can be found in Section \ref{app:sieve}. The last elementary brick in the proof of \thmref{main} is the control of the probability of small balls around $\bbP_0$, which is stated below.

	\begin{lem} \label{lem:thickness} Let $\ve > 0$ and assume that it is small enough so that $\ve^{2\alpha_0} \leq \delta^\beta \ve^{\alpha_\perp}$ and that $\ve^{\alpha_\perp} \ll \log^{-1}(1/\ve)$. Then, in the context of \thmref{main}, there holds
		$$
		\Pi\left(f_{P} \in \ball(f_0,\tilde\ve)\right) \gtrsim \exp\{- C \ve^{-D/\beta} \log^{t}\left(1/\ve\right)\},
		$$
		where $\ball(f_0,\cdot)$ is defined in \eqref{eq:defball}, where the constant $C$ depends on the parameters, where $\tilde\ve \simeq  \ve  \log^s (1/\ve)$ and with $s,t > 0$ depending on $D$, $\beta$ and $\kappa$.
	\end{lem}

	The complete proof is given in Section \ref{app:thickness} and is sketched as follow: the first step is the Hellinger approximation of $f_0$  by $K_\Sigma \tilde h$ as expressed in Corollary \ref{cor:approx}. Then we exhibit an $\ve$-approximation of $K_\Sigma \tilde h$ by a discrete location scale mixture with a controlled number of atoms through the use of Lemma \ref{lem:disc} below. The result then follows from similar arguments as in \cite{shen2013adaptive,kruijer2010adaptive} or \cite{naulet2017posterior}.
	
	
	\begin{lem} \label{lem:disc} Let $\ve > 0$ such that $\delta\sigma^{\alpha_\perp}\log(1/\ve) \ll 1$. For any density $g$ on $M^\delta$ satisfying \eqref{expdec}, there exists a discrete probability measure $G$ on $\bbR^D$ with at most $N\simeq \sigma^{-D} \log^{D}(1/\ve)$ atoms such that 
			$$\|K_\Sigma G - K_\Sigma g \|_\infty \lesssim \frac{\ve}{\sigma^D \delta^{D-d}}~~\text{and}~~\|K_\Sigma G - K_\Sigma g \|_1 \lesssim \ve \log^{D/2}(1/\ve).
		$$  
	The atoms of $G$ are in $M^\delta$ and are $\sigma^{2\alpha_0}\ve$-apart.
	\end{lem}
	The proof of \lemref{disc} can be found in Section \ref{app:disc}. We underline that although it uses similar ideas to \cite[Lem 3.1]{ghosal2001entropies}, it is not a straightforward adaptation of it, since in $K_\Sigma$ the covariances depend  on the locations of the mixture in a complicated way.

 		\subsection{Proof of \thmref{approx}} \label{app:thmapprox1}
		
As explained in Section \ref{subsec:approx} a key ingredient of the proof of \thmref{main} is the pointwise approximation of $f_0$ by $K_\Sigma \tilde g$ where $g$ is close to $f_0$ and is explicited in the proof of \thmref{approx} below.

\begin{proof}[Proof of \thmref{approx}]
		
Let  $x_0 \in M$ and define 
	$$
		\cW_{x_0}^j := \ball_{T_{x_0}}\(0,\frac{2+j}{16}\tau\) \times \ball_{N_{x_0}}\(0,\frac{6+j}{8}\tau\)~~~\text{for}~~j \in\{0,1,2\},
		$$
		and $\cO_{x_0}^j = \bar\Psi_{x_0}(\cO_{x_0}^j)$ for $j \in \{0,1,2\}$. We have
		$
		\cW_{x_0}^0 \subset \cW_{x_0}^1 \subset \cW_{x_0}^2$ and$\cO_{x_0}^0 \subset \cO_{x_0}^1 \subset \cO^2_{x_0}.
		$
		Furthermore, the sets $\cO_{x_0}^0$ for $x_0 \in M$ forms a covering of $M^{3\tau/4}$, see Section \ref{sec:partition} for more details. 
	
	We now drop the $x_0$ from the notation. Let $f : \bbR^D \to \bbR$ be in $\hbm$ and supported on $\cO^0 $ --- it is to be thought of as $f_0$ multiplied by a smooth function supported on $\cO^0$.  Take $x \in \cO^1$ and compute
			$$
			K_\Sigma f(x) := \int_{\bbR^D} \varphi_{\Sigma(u)} (x - u) f(u) \diff u = \int_{M^\delta \cap \cO^0} \varphi_{\Sigma(u)} (x - u) f(u) \diff u .
			$$
			We first prove that we can construct a function $g$ such that  $K_\Sigma g$ is close to $f$ and we then apply this result to $f = f_0 \chi_j$ with $\chi_j$ the partition of unity defined in Lemma \ref{lem:chixj}.

			The idea is to use the fact that $\Sigma(u) = o(1)$ and the smoothness of $u \mapsto \Sigma(u)$ so that
			$$
			K_\Sigma f(x) \approx  \int_{M^\delta \cap \cO^0} \varphi_{\Sigma(x)} (x - u) f(u) \diff u \approx f(x). 
			$$
			We now write down the approximation rigorously and quantify the error, taking into accound the geometry of the manifold $M$. In all that follows, we use the notation $z = (v,\eta)$ for points belonging to $ \ball_{T_{x_0} M}(0,\tau/16) \times \ball_{N_{x_0 M}}(0,\tau/2)$, while throughout $w = (v, \delta \eta) \in  \ball_{T_{x_0} M}(0,\tau/16) \times \ball_{N_{x_0} M}(0,\delta \tau/2)$.
			We first make the change of variable $w = \bar\Psi^{-1}(u)$, yielding
			\begin{align*} 
				K_\Sigma f(x) &= \int_{\bar\Psi^{-1}(M^\delta \cap \cO^{0})} \varphi_{\Sigma(\bar \Psi(w))} (x - \bar \Psi(w)) f(\bar \Psi(w)) |\det \diff \bar \Psi(w)| \diff w.
			\end{align*} 
			Then, denoting by $w_x = \bar\Psi^{-1}(x)$, we write
			$$w = \Delta_{\sigma, \delta} z + w_x =  \Delta_{\sigma, \delta} z  +  \Delta_{1, \delta}z_x, \quad w_x = (v_x, \delta \eta_x), \quad z_x =  (v_x,  \eta_x),$$
			in the integral above, giving
			\begin{align} \label{Ksigma:decomp1}
				K_\Sigma f(x) =  &\frac{1}{(2\pi)^{D/2}\delta^{D-d}} 
				\int_{\Delta^{-1}_{\sigma,1}(\cW^0-z_x)} e^{- B_\sigma(z_x,z)} \bar f_\delta(\Delta_{\sigma,1} z + z_x) \zeta(\Delta_{\sigma,\delta}z + w_x)  \diff z,
			\end{align} 
			with
			\begin{align*} 
			B_\sigma(z_x,z) &:= \frac12 \| x - \bar \Psi(\Delta_{\sigma,\delta}z + w_x)\|^2_{\Sigma^{-1}(\bar \Psi(\Delta_{\sigma,\delta}z + w_x))} \\
			\text{and}~~~\zeta(\Delta_{\sigma,\delta}z + w_x) &:=  |\det \diff  \bar \Psi(\Delta_{\sigma,\delta}z + w_x)|,  
			\end{align*} 
			and where we used the fact that $|\det \Sigma(u)|$ is constantly $\sigma^D \delta^{D-d}$ for $u \in M^\tau$. Notice that $\bar f_\delta$ is zero outside of $\Delta^{-1}_{\sigma,1}(\cW^0-z_x)$ so we can replace the latter by $T_{x_0} M \times N_{x_0} M \simeq \bbR^D$. We denote 
			\begin{align*}
				\bar K_\Sigma  \bar f_\delta(z_x)  =  &\frac{1}{(2\pi)^{D/2} } 
				\int e^{- B_\sigma(z_x,z)} \bar f_\delta(\Delta_{\sigma,1} z + z_x) \zeta(\Delta_{\sigma,\delta}z + w_x)  \diff z,
			\end{align*} 
			We now develop each term separately. First 
			\beq \label{devf}
			\bar f_\delta(\Delta_{\sigma,1}z+z_x) =  \underbrace{\bar f_\delta(z_x)}_{= \delta^{D-d} f(x)} + \sum_{0 < \inner{k}{\alpha} < \beta} \frac{z^{k}}{k!} \sigma^{\inner{k}{\alpha}} \Diff^k \bar f_\delta(z_x)  + R^f_\sigma(x,z)  
			\eeq 
			with $R^f_\sigma(z_x,z) \leq D L_{\delta}(z_x) \sigma^\beta \|1\vee z \|_1^{\beta_{\max}} = D \delta^{D-d} L(x) \sigma^\beta \|1\vee z \|_1^{\beta_{\max}} $ in application of \corref{taylor}. \newclem{We can thus write
\beq \label{exp:Ksigma}
\bar K_\Sigma \bar f_\delta(z_x) = \bar f_\delta(z_x) I_0(z_x) + \sum_{0 < \inner{k}{\alpha} < \beta} \frac{\Diff^k \bar f_\delta(z_x)}{k!} \sigma^{\inner{k}{\alpha}} I_k(z_x)  + R^K_\sigma(z_x),  
\eeq
with 
\begin{align*} 
I_k(z_x) &= \frac{1}{(2\pi)^{D/2}}\int e^{- B_\sigma(z_x,z)} \zeta(\Delta_{\sigma,\delta}z + w_x) z^k  \diff z\\
\text{and}~~~R^K_\sigma(z_x) &= \frac{1}{(2\pi)^{D/2} }\int e^{- B_\sigma(z_x,z)} \zeta(\Delta_{\sigma,\delta}z + w_x) R_\sigma^f(z_x,z)  \diff z.
\end{align*} 
We will show that
\bitem
\item[(i)] $R_\sigma^K(z_x) = O(\sigma^\beta \delta^{D-d}L(x))$ uniformly;
\item[(ii)] For all $k$, $I_k(z_x) \in \mathcal \cH_{\iso}^{\beta_M-3}(\cW_{x_0}^0, C)$ \newju{for some constant $C>0$}; 
\item[(iii)] $I_0(z_x) = 1 + o(1)$ uniformly.
\eitem
If these three assertions holds, then we can recursively define a function $g$ such that
$$
\bar K_\Sigma \bar g(x) = \bar f_\delta(x) + O(\sigma^\beta L_\delta(x))
$$
uniformly on $x$. The recursion goes as follows: we let $\cQ = \{0<\inner{k}{\alpha}~|~\inner{k}{\alpha} < \beta\} = \{ q_1 < \dots < q_N\}$ and let $q_0=0$ and define
\begin{align*} 
\bar g_0(z_x) &= \bar f_\delta (z_x)/I_0(z_x), \\
\bar g_{1}(z_x) &= \bar g_0(z_x) - \frac{\sigma^{q_1}}{ I_0(z_x)}\sum_{\inner{k}{\alpha} = q_1} \frac{1}{k!} \Diff^k(\bar g_0)(z_x) I_k(z_x),
\end{align*} 
then using Proposition \ref{prp:der}, for all $\inner{k}{\alpha}=q_1$, $\Diff^k(\bar g_0) \in \cH_{\an}^{\bbeta(1-q_1/\beta)}(\tilde C_kL)$ for some $\tilde C_k>0$ and we apply  \eqref{exp:Ksigma} to $\bar g_1$ which leads to 
\begin{align*}
\bar K_\Sigma \bar g_1(z_x) &= \bar f_\delta (z_x)- \sigma^{q_1}\sum_{\inner{k}{\alpha} = q_1} \frac{1}{k!} \Diff^k(\bar g_0)(z_x) I_k(z_x) +\sum_{q_1 \leq \inner{k}{\alpha} < \beta} \sigma^{\inner{k}{\alpha}} \frac{\Diff^k \bar g_0(z_x)}{k!} I_k(z_x)  \\
& - \sigma^{q_1}\sum_{\inner{k}{\alpha} = q_1} \frac{1}{k!}\sum_{\inner{k'}{\alpha}<\beta-q_1}\frac{\Diff^{k'} (\Diff^k(\bar g_0) I_k) (z_x)) }{k'!}  I_{k'}(z_x) + \tilde R^K_\sigma(z_x),  \\
& = \bar f_\delta (z_x) +\sum_{q_2 \leq \inner{k}{\alpha} < \beta} \sigma^{\inner{k}{\alpha}} \frac{\Diff^k \bar g_0(z_x)}{k!} I_k(z_x)  \\
& - \sigma^{q_1}\sum_{\inner{k}{\alpha} = q_1} \frac{1}{k!}\sum_{\inner{k'}{\alpha}<\beta-q_1}\sigma^{\inner{k'}{\alpha}}  \frac{\Diff^{k'} (\Diff^k(\bar g_0) I_k) (z_x)) }{k'!}  I_{k'}(z_x) + \tilde R^K_\sigma(z_x)  \\
& = \bar f_\delta (z_x) +\sum_{q_2 \leq \inner{k}{\alpha} < \beta} \sigma^{\inner{k}{\alpha}}  h_k(z_x)  I_k(z_x)  + \tilde R^K_\sigma(z_x) 
\end{align*}
where  $ \tilde R^K_\sigma(z_x)$ also satisfies (i) above and $h_k \in   \cH_{\an}^{\bbeta(1-\beta_k/\beta)}( CL)$, $\beta_k =\inner{k}{\alpha}$ and $C>0$.  
 All in all, $\bar g$ will be of the form
$$
\bar g(z_x) = d_0(z_x,\sigma,\delta)\bar f_\delta(z_x) + \frac1{\delta^{D-d}}\sum_{0 < \inner{k}{\bm\alpha} < \beta} \sigma^{\inner{k}{\bm\alpha}} d_k(z_x,\sigma,\delta) \Diff^k \bar f_\delta(z_x)
$$
with all $d_k$ uniformly bounded and $d_0(z_x,\sigma,\delta)  = 1 + O( \sigma^{2\alpha_0 - \alpha_{\perp}}/\delta)$. Let us now prove (i), (ii) and (iii). There holds
\begin{align*} 
\bar \Psi(\Delta_{\sigma,\delta}z + w_x) &= \Psi(v_x + \sigma^{\alpha_0} v) + \delta N(v_x + \sigma^{\alpha_0} v)[\eta_x+ \sigma^{\alpha_\perp} \eta] \\
&= x + \sigma^{\alpha_0} \diff \Psi(v_x)[v] + \delta \sigma^{\alpha_\perp} N(v_x)\eta + \delta \sigma^{\alpha_0} \diff N(v_x)[v]\eta_x + R^{\bar \Psi}_\sigma(z_x,z).
\end{align*} 
where, with \newju{for $s\in (0,1)$ $v_s = v_x + s\sigma^{\alpha_0}v$, 
\begin{equation*}
\begin{split}
R^{\bar \Psi}_\sigma(z_x,z) &=
\sigma^{2\alpha_0} \int_0^1 (1-s) [\diff^2 \Psi(v_s)[v^{\otimes 2}] + \delta \diff^2_v N(v_s)[v^{\otimes 2}](\eta_x+\sigma^{\alpha_\perp}\eta)]\diff s \\
&\quad \quad + \delta\sigma^{\alpha_0+ \alpha_\perp}\diff N(v_x)[v]\eta \\
&= O(r_\sigma) \quad \text{with} \quad  r_\sigma := \sigma^{2\alpha_0} + \delta \sigma^{\alpha_0+\alpha_\perp}.
\end{split}
\end{equation*}
The above equation shows that $r_\sigma^{-1}  R^{\bar \Psi}_\sigma(z_x,z)\in \cH_{\iso}^{\beta_M-3}(\cW^0,C)$ as a function of $z_x$, where $C$ depends on $\|\bar\Psi\|_{\cC^{\beta_M-1}}$, uniformly on $z$.} Now notice that in appropriate orthogonal basis, the differential of $\bar\Psi$ in $w_x$ writes
$$
\begin{pmatrix}
\diff \Psi(v_x) + \pr_{T_x} \diff N(v_x)[\cdot](\delta \eta_x) & 0 \\
 \pr_{N_x} \diff N(v_x)[\cdot](\delta \eta_x) & N(v_x)
\end{pmatrix}
$$
so that in particular, $|\det \diff \Psi(w_x)| = |\det \{\diff \Psi(v_x) + \pr_{T_x} \diff N(v_x)[\cdot](\delta \eta_x)\}|$.
Let now $u = \bar \Psi(\Delta_{\sigma,\delta}z + w_x)$. The projection $\pr_{T_u}$ only depends smoothly on $\diff \Psi(v_x + \sigma^{\alpha_0}v)$ so that there $R_\sigma^T(z_x,v)$ such that
$$
\pr_{T_u}[y] = \pr_{T_x}[y] + R_\sigma^T(z_x,v)(y)
$$
with
\begin{equation*}
\begin{split}
R_\sigma^T(z_x,v) (y)  = \sigma^{\alpha_0} \int_0^1(1-s)\diff_v\pr_{T_{v_s}}[v](y) \diff s \lesssim \sigma^{\alpha_0} \|y\|,
 \end{split}
 \end{equation*}
  and that for all $y\neq 0$, $\|y\|^{-1}\sigma^{-\alpha_0}R_\sigma^T(\cdot,v)(y) \in \cH_{\iso}^{\beta_M-2}(\cW^0,C)$. This also implies that $\pr_{N_u}[y] = \pr_{N_x}[y] - R_\sigma^T(x,v)[y]$. As a results, we find that
$$
\pr_{T_u}(u-x) = \sigma^{\alpha_0} \diff \Psi(v_x) v + \delta \sigma^{\alpha_0} \pr_{T_x} \diff N(v_x)[v]\eta_x + \wt R_\sigma^T(z_x,z)
$$
with $r_\sigma^{-1} \wt R_\sigma^T(x,z) \in \cH_{\iso}^{\beta_M-3}(\cW^0,C)$ as a function of $z_x$, uniformly in $z$. Likewise, 
$$
\pr_{N_u}(u-x) = \delta \sigma^{\alpha_0} \pr_{N_x} \diff N(v_x)[v]\eta_x + \delta \sigma^{\alpha_\perp} N(v_x) \eta + \wt R_\sigma^N(z_x,z)
$$
with $r_\sigma\wt R_\sigma^N(x,z) \in \cH_{\iso}^{\beta_M-3}(\cW^0,C)$ as a function of $z_x$, uniformly in $z$. All in all, we get that
\begin{align*} 
B_\sigma(z_x,z) &= \frac1{2\sigma^{2\alpha_0}}\|\pr_{T_u}(u-x)\|^2 + \frac1{2\delta^2\sigma^{2\alpha_\perp}}\|\pr_{N_u}(u-x)\|^2  \\
&= \frac12 \|A_\sigma(z_x,z)\|^2 + R^B_\sigma(z_x,z)
\end{align*} 
where 
\beq\label{eq:As}A_\sigma(z_x,z) = \diff \Psi(v_x) v + \delta \pr_{T_x} \diff N(v_x)[v]\eta_x + \sigma^{\alpha_0-\alpha_\perp} \pr_{N_x} \diff N(v_x)[v]\eta_x + N(v_x)\eta,
\eeq
and where $R^B_\sigma(z_x,z) =O( \tilde r_\sigma) $ with $\tilde r_\sigma = \sigma^{\alpha_0} + \delta\sigma^{\alpha_\perp} + \sigma^{2\alpha_0}/(\delta \sigma^{\alpha_\perp}) = o(1)$, together with  $\tilde r_\sigma^{-1} R^B_\sigma(z_x,z) \in \cH_{\iso}^{\beta_M-3}(\cW^0,C)$ as a function of $z_x$, uniformly in $z$. Finally, let us write that 
$$
\zeta(w_x + \Delta_{\sigma,\delta} z) = |\det \diff\bar\Psi(w_x)| + R^\zeta_\sigma(z_x,z)
$$
 and where $R^\zeta_\sigma(z_x,z)$ is $(\beta_M-2)$-H\"older as a function of $z_x$, uniformly in $z$. 
Let us now show that (i) and (ii) holds. We show that if $|F(z_x,z)| \leq A \|z\|^r$, then
\beq
\int e^{- \frac{1}{2} \|A_\sigma(z_x,z)\|^2} |F(z_x,z)| \diff z \leq A C_r \label{eq:IF}
\eeq
for some constant $C_r > 0$ depending on $r$ and $C_M$. To do so, let us notice that $A_\sigma(x,\cdot)$ is an invertible linear map such that $\mathrm{spec}(A_\sigma(z_x,\cdot)) = \mathrm{spec}(\diff \bar \Psi(w_x))$ and thus
\begin{align*} 
\int e^{- \frac{1}{2} \|A_\sigma(z_x,z)\|^2} |F(x,z)| \diff z 
&= A |\det A_\sigma(z_x,\cdot)|^{-1} \int_{\bbR^D} e^{- \frac{1}{2} \|z\|^2} \|A_\sigma(z_x,\cdot)^{-1} z\|^r \diff z \\
&= A |\det \diff \bar \Psi(w_x)|^{-1} \|\diff \bar \Psi(w_x)^{-1}\|^r_{\op} \int_{\bbR^D} e^{- \frac{1}{2} \|z\|^2} \|z\|^r \diff z \\
&\leq A C_r,
\end{align*}
where $C_r$ depends only on $r$ and $C_M$. 
Using this newly proven \eqref{eq:IF}, we have immediately that (i) and (ii) holds. For (iii), notice that, 
\begin{align*} 
I_0(x) &= \frac{1}{(2\pi)^{D/2}} \int e^{- \frac{1}{2} \|A_\sigma(z_x,z)\|^2} e^{- \frac{1}{2} R^B_\sigma(z_x,z)} (|\det \diff\bar\Psi(w_x)| + R^\zeta_\sigma(z_x,z))  \diff z \\
&= \frac{|\det \diff\bar\Psi(w_x)|}{(2\pi)^{D/2}} \int e^{- \frac12 \|A_\sigma(z_x,z)\|^2}  \diff z + \frac{1}{(2\pi)^{D/2}} \int e^{- \frac12 \|A_\sigma(z_x,z)\|^2} R^I_\sigma(z_x,z)  \diff z
\end{align*} 
where $R^I_\sigma(z_x,z) = \(e^{- \frac{1}{2} R^B_\sigma(z_x,z)}-1\) |\det \diff\bar\Psi(w_x)| +e^{- \frac{1}{2} R^B_\sigma(z_x,z)} R^\zeta_\sigma(z_x,z) = o(1)$ uniformly. Now there holds
\begin{align*} 
\int e^{- \frac12 \|A_\sigma(z_x,z)\|^2}  \diff z = (2\pi)^{D/2} |\det(\diff \bar \Psi(w_x))|^{-1} + o(1),
\end{align*} 
and using again \eqref{eq:IF}
$$
 \frac{1}{(2\pi)^{D/2}} \int e^{- \frac12 \|A_\sigma(z_x,z)\|^2} R^I_\sigma(z_x,z)  \diff z = o(1)
$$
so that in total $I_0(x) = 1 + o(1)$ as expected.}

			We are now ready to prove \thmref{approx}. We let 
			$$R = \{ H \log(1/\sigma) \}^{1/\kappa}$$
			and $\chi_1,\dots,\chi_J$ be the functions defined at \lemref{chixj} from a $\tau/64$-packing of $M \cap \ball(0,R)$. Recall that we can always chose $J$ of order less than $R^D$. In light of the point (iv) of \lemref{chixj}, the function $f_j := \chi_j f_0$ is still in $\cH_\delta^{\beta_0,\beta_\perp}(M,CL)$ for some constant $C$ depending on $\tau$ and $\beta_\perp$. Since $\supp f_j \subset \cO_{x_j}^0$ (point (i) of \lemref{chixj}), the first part of this proof yields that there exists some functions $g_j$ supported on $\cO_{x_j}^0$, such that
			\beq
			|K_\Sigma g_j(x) - f_j(x)| \lesssim L(x) \sigma^\beta \label{gj}
			\eeq
			uniformly on $\cO_{x_j}^1$. Now notice that for $x$ outside of $\cO_{x_j}^1$, we have $d(x,\cO_{x_j}^0) > (\sigma^{\alpha_0} \vee \delta\sigma^{\alpha_\perp})  \sqrt{(H+D)\log(\sigma)}$ so that 
			$$
			|K_\Sigma g_j(x)| \leq \int_{\cO_{x_j}^0} |\varphi_{\Sigma(u)}(x-u) g_j(u)| \diff u \leq \frac{\sigma^H}{(2\pi)^{D/2}\delta^{D-d}} \int_{\cO_{x_j}^0} |g_j(u)|\diff u\lesssim \sigma^H \sup_{\cO^0_{x_j}} L
			$$ 
			and the equality \eqref{gj} extends to the whole set $\bbR^D$ with the bound $\sigma^H\| L \|_\infty $ on $\bbR^D \setminus \cO_{x_j}^1$.
			Using the linearity of $K_\Sigma$, we thus find that for $g = \sum_{j=1}^J g_j$, and for any $x \in \bbR^D$, there holds,
			\begin{align*} 
				| K_{\Sigma} g(x) -  f_0(x)| &\leq \sum_{j=1}^J |K_\Sigma g_j(x) - f_j(x) | \leq \sum_{j \in J(x)} |K_\Sigma g_j(x) - f_j(x) |  + \sum_{j \notin J(x)} |K_\Sigma g_j(x)| \\
				&\lesssim |J(x)| \times L(x) \sigma^\beta + (J - |J(x)|)\times \sigma^H\| L \|_\infty  \\
				&\lesssim  \sigma^{\beta} L(x) + \{ H \log(1/\sigma) \}^{D/\kappa} \sigma^H\| L \|_\infty\
			\end{align*} 
			where we denoted $J(x) = \{1 \leq j \leq J~|~x \in \cO_{x_j}^1\}$, and used the fact that $J(x)$ is bounded from above by something depending on $D$  and $\tau$ only, ending the proof.
	\end{proof}

		\section{Discussion} \label{sec:discussion}

	With the aim of developing Bayesian procedures in the framework of manifold learning, we exhibited a new family of priors based on location-scales of Dirichlet mixture of Gaussians, and described a general setting for studying density supported near a submanifold. The latter relies on two things: first, a parametrization of the offset of the manifold and second, an anisotropic class of H\"older functions. In this model, we obtained concentration rates in \thmref{main} for the associated posterior distribution that are adaptive to the regularity of the underlying density while being totally agnostic of the underlying submanifolds and their main features.  Our procedure is therefore fully adaptive. \\
	
	An interesting feature of our theoretical framework is that it allows to express the rate in terms of the smoothnesses of the density and the manifold  together with the thickness $\delta$ of the support around the manifold. When $\delta$ is fixed, our results can be viewed as an extension  of minimax rates for regular anisotropic densities to manifold driven anisotropic densities. But we are also considering the regime where $\delta = o(1)$ which corresponds to the manifold learning problem. The rates obtained in Theorem \ref{thm:main} have two regimes: one when $\delta$ is not too small with rate $n^{-\beta/(2\beta+D)} $ (up to $\log n $ terms) and the concentrated regime where $\delta$ is very small ($\delta = o( n^{-2/(2\beta+D)})$ ) where we obtain the rate $(n \delta^{D/(2\alpha_0-\alpha_\perp)})^{-1/2}$. It is not clear if this latter rate is optimal or not.  The case $\delta=0$ would correspond to the observations belonging to the manifold $M$ (and for which we would expect the rate $n^{-\beta_0/(2\beta_0+d)}$) but cannot be thought as a limiting case of our problem since then the distribution has density with respect to the Hausdorff measure on $M$ and not with respect to the Lebesgue measure on $\bbR^D$. When $M$ is unknown the model is not dominated and our approach is not applicable. The problem of posterior contraction rates  when the distribution lives on an unknown manifold remains open, although some interesting ideas in  \cite{tang2022minimax} or \cite{camerlenghi2022wasserstein} could be used to address it. 
	\\
	
\newju{An important special case of our framework is the deconvolution model as defined in (ii) of Proposition \ref{prp:model}. For such a model, our results address the problem of the predictive density of the observations, which is the direct problem. Other quantities are also of interest in deconvolution models, such as the distribution of the unobserved signal or the underlying manifold $M$. These are important question and we believe that our framework can be used as a first step to answer these questions, possibly using inversion inequalities as in \cite{rousseau:scricciolo23}} \\
	
Another interesting output of our results is that if nonparametric mixtures of normal densities define a versatile and flexible model for smooth densities, the structure of the prior on the mixing distribution is crucial. In this paper we propose two classes of priors which we believe enjoy many strong theoretical properties while remaining reasonnably simple to implement. Moreover variational Bayes algorithms can be easily implemented and for which the same theoretical guarantees hold. 
Moreover,  in order to make the MCMC algorithm more  scalable to the dimension $D$ and to the number of observations $n$, we could use a variant based on parallel computing for Dirichlet process mixtures (see for instance \cite{meguelati2019dirichlet}).
It is quite possible that other nonparametric mixture models such as the Fisher-Gaussian kernels of \cite{mukhopadhyay2020estimating} would enjoy the same theoretical guarantees and we believe that our approximation result can be useful to study the theoretical properties of mixtures of  Fisher-Gaussian kernels  which are stronly related to Gaussian kernels. 

	\paragraph{Acknowledgments}
		We are very thankful to Marc Hoffmann and to Yann Chaubet for helpful discussions, as well as to Leo Zhang for spotting an initial mistake in the Gibbs sampling algorithm. \newclem{We would also like to express our gratitude to the referees and the associate editor for the great care and attention shown to the manuscript, which helped us improve the results and their exposition.}
	
	\paragraph{funding}
		This work has received   funding from the European Research Council
		(ERC) under the European Union’s Horizon 2020 research and innovation programme (grant agreement No 834175).

	\bibliographystyle{abbrv}
	\bibliography{aos}       

\begin{thebibliography}{10}

\bibitem{aamari2019nonasymptotic}
E.~Aamari and C.~Levrard.
\newblock Nonasymptotic rates for manifold, tangent space and curvature
  estimation.
\newblock {\em The Annals of Statistics}, 47(1):177--204, 2019.

\bibitem{alexander2006gauss}
S.~B. Alexander and R.~L. Bishop.
\newblock Gauss equation and injectivity radii for subspaces in spaces of
  curvature bounded above.
\newblock {\em Geometriae Dedicata}, 117(1):65--84, 2006.

\bibitem{GansMethods}
M.~Arjovsky and L.~Bottou.
\newblock Towards principled methods for training generative adversarial
  networks, 2017.

\bibitem{WGANs}
M.~Arjovsky, S.~Chintala, and L.~Bottou.
\newblock Wasserstein gan, 2017.

\bibitem{belkin2001laplacian}
M.~Belkin and P.~Niyogi.
\newblock Laplacian eigenmaps and spectral techniques for embedding and
  clustering.
\newblock {\em Advances in neural information processing systems}, 14, 2001.

\bibitem{LaplacianEigenmaps}
M.~Belkin and P.~Niyogi.
\newblock Laplacian eigenmaps and spectral techniques for embedding and
  clustering.
\newblock In T.~Dietterich, S.~Becker, and Z.~Ghahramani, editors, {\em
  Advances in Neural Information Processing Systems}, volume~14. MIT Press,
  2001.

\bibitem{berenfeld2021density}
C.~Berenfeld and M.~Hoffmann.
\newblock Density estimation on an unknown submanifold.
\newblock {\em Electronic Journal of Statistics}, 15(1):2179--2223, 2021.

\bibitem{supp:berenfeld:rosa:rousseau22}
C.~Berenfeld, P.~Rosa, and J.~Rousseau.
\newblock Supplementary material to ``estimating a density near an unknown
  manifold: a bayesian nonparametric approach'', 2022.

\bibitem{bingham2019pyro}
E.~Bingham, J.~P. Chen, M.~Jankowiak, F.~Obermeyer, N.~Pradhan, T.~Karaletsos,
  R.~Singh, P.~Szerlip, P.~Horsfall, and N.~D. Goodman.
\newblock Pyro: Deep universal probabilistic programming.
\newblock {\em The Journal of Machine Learning Research}, 20(1):973--978, 2019.

\bibitem{boissonnat2019reach}
J.-D. Boissonnat, A.~Lieutier, and M.~Wintraecken.
\newblock The reach, metric distortion, geodesic convexity and the variation of
  tangent spaces.
\newblock {\em Journal of applied and computational topology}, 3(1):29--58,
  2019.

\bibitem{camerlenghi2022wasserstein}
F.~Camerlenghi, E.~Dolera, S.~Favaro, and E.~Mainini.
\newblock Wasserstein posterior contraction rates in non-dominated bayesian
  nonparametric models.
\newblock {\em arXiv preprint arXiv:2201.12225}, 2022.

\bibitem{canale2017posterior}
A.~Canale and P.~De~Blasi.
\newblock Posterior asymptotics of nonparametric location-scale mixtures for
  multivariate density estimation.
\newblock {\em Bernoulli}, 23(1):379--404, 2017.

\bibitem{capitao2023deconvolution}
J.~Capitao-Miniconi and {\'E}.~Gassiat.
\newblock Deconvolution of spherical data corrupted with unknown noise.
\newblock {\em Electronic Journal of Statistics}, 17(1):607--649, 2023.

\bibitem{chae2021likelihood}
M.~Chae, D.~Kim, Y.~Kim, and L.~Lin.
\newblock A likelihood approach to nonparametric estimation of a singular
  distribution using deep generative models.
\newblock {\em arXiv preprint arXiv:2105.04046}, 2021.

\bibitem{chen2015asymptotic}
Y.-C. Chen, C.~R. Genovese, and L.~Wasserman.
\newblock Asymptotic theory for density ridges.
\newblock {\em The Annals of Statistics}, 43(5):1896--1928, 2015.

\bibitem{cleanthous2019minimax}
G.~Cleanthous, A.~G. Georgiadis, and E.~Porcu.
\newblock Minimax density estimation on sobolev spaces with dominating mixed
  smoothness.
\newblock {\em arXiv preprint arXiv:1906.06835}, 2019.

\bibitem{comte2013anisotropic}
F.~Comte and C.~Lacour.
\newblock Anisotropic adaptive kernel deconvolution.
\newblock In {\em Annales de l'IHP Probabilit{\'e}s et statistiques},
  volume~49, pages 569--609, 2013.

\bibitem{constantine1996multivariate}
G.~Constantine and T.~Savits.
\newblock A multivariate {Faa di Bruno} formula with applications.
\newblock {\em Transactions of the American Mathematical Society},
  348(2):503--520, 1996.

\bibitem{devroye2018total}
L.~Devroye, A.~Mehrabian, and T.~Reddad.
\newblock The total variation distance between high-dimensional gaussians.
\newblock {\em arXiv preprint arXiv:1810.08693}, 2018.

\bibitem{Divol_minimax}
V.~Divol.
\newblock Minimax adaptive estimation in manifold inference, 2020.

\bibitem{divol2021reconstructing}
V.~Divol.
\newblock Reconstructing measures on manifolds: an optimal transport approach.
\newblock {\em arXiv preprint arXiv:2102.07595}, 2021.

\bibitem{chen2010compressive}
D.~Dunson, M.~Chen, J.~Silva, J.~Paisley, C.~Wang, and L.~Carin.
\newblock Compressive sensing on manifolds using a nonparametric mixture of
  factor analyzers: Algorithm and performance bounds.
\newblock {\em IEEE Transactions on Signal Processing}, 58(12):6140--6155,
  2010.

\bibitem{dunson2021inferring}
D.~B. Dunson and N.~Wu.
\newblock Inferring manifolds from noisy data using gaussian processes.
\newblock {\em arXiv preprint arXiv:2110.07478}, 2021.

\bibitem{Fed59}
H.~Federer.
\newblock Curvature measures.
\newblock {\em Transactions of the American Mathematical Society},
  93(3):418--491, 1959.

\bibitem{fefferman2016testing}
C.~Fefferman, S.~Mitter, and H.~Narayanan.
\newblock Testing the manifold hypothesis.
\newblock {\em Journal of the American Mathematical Society}, 29(4):983--1049,
  2016.

\bibitem{Genovese_2012}
C.~R. Genovese, M.~Perone-Pacifico, I.~Verdinelli, and L.~Wasserman.
\newblock Manifold estimation and singular deconvolution under hausdorff loss.
\newblock {\em The Annals of Statistics}, 40(2), apr 2012.

\bibitem{genovese2014nonparametric}
C.~R. Genovese, M.~Perone-Pacifico, I.~Verdinelli, and L.~Wasserman.
\newblock Nonparametric ridge estimation.
\newblock {\em The Annals of Statistics}, 42(4):1511--1545, 2014.

\bibitem{ghahramani1996algorithm}
Z.~Ghahramani, G.~E. Hinton, et~al.
\newblock The em algorithm for mixtures of factor analyzers.
\newblock Technical report, Technical Report CRG-TR-96-1, University of
  Toronto, 1996.

\bibitem{ghosal2000convergence}
S.~Ghosal, J.~K. Ghosh, and A.~W. Van Der~Vaart.
\newblock Convergence rates of posterior distributions.
\newblock {\em Annals of Statistics}, pages 500--531, 2000.

\bibitem{ghosal2007posterior}
S.~Ghosal and A.~Van Der~Vaart.
\newblock Posterior convergence rates of dirichlet mixtures at smooth
  densities.
\newblock {\em The Annals of Statistics}, pages 697--723, 2007.

\bibitem{Ghosal_van_der_Vaart_2017}
S.~Ghosal and A.~van~der Vaart.
\newblock {\em Fundamentals of Nonparametric Bayesian Inference}.
\newblock Cambridge Series in Statistical and Probabilistic Mathematics.
  Cambridge University Press, 2017.

\bibitem{ghosal2001entropies}
S.~Ghosal and A.~W. Van Der~Vaart.
\newblock Entropies and rates of convergence for maximum likelihood and bayes
  estimation for mixtures of normal densities.
\newblock {\em The Annals of Statistics}, 29(5):1233--1263, 2001.

\bibitem{goldenshluger2011bandwidth}
A.~Goldenshluger and O.~Lepski.
\newblock Bandwidth selection in kernel density estimation: oracle inequalities
  and adaptive minimax optimality.
\newblock {\em The Annals of Statistics}, 39(3):1608--1632, 2011.

\bibitem{goldenshluger2014adaptive}
A.~Goldenshluger and O.~Lepski.
\newblock On adaptive minimax density estimation on rd.
\newblock {\em Probability Theory and Related Fields}, 159(3):479--543, 2014.

\bibitem{DeepLearning}
I.~Goodfellow, Y.~Bengio, and A.~Courville.
\newblock {\em Deep learning}.
\newblock Adapt. Comput. Mach. Learn. Cambridge, MA: MIT Press, 2016.

\bibitem{GANs}
I.~J. Goodfellow, J.~Pouget-Abadie, M.~Mirza, B.~Xu, D.~Warde-Farley, S.~Ozair,
  A.~Courville, and Y.~Bengio.
\newblock Generative adversarial networks, 2014.

\bibitem{hoff2009simulation}
P.~D. Hoff.
\newblock Simulation of the matrix bingham--von mises--fisher distribution,
  with applications to multivariate and relational data.
\newblock {\em Journal of Computational and Graphical Statistics},
  18(2):438--456, 2009.

\bibitem{hoffman2002random}
M.~Hoffman and O.~Lepski.
\newblock Random rates in anisotropic regression (with a discussion and a
  rejoinder by the authors).
\newblock {\em The Annals of Statistics}, 30(2):325--396, 2002.

\bibitem{horvat2021density}
C.~Horvat and J.-P. Pfister.
\newblock Density estimation on low-dimensional manifolds: an
  inflation-deflation approach.
\newblock {\em arXiv preprint arXiv:2105.12152}, 2021.

\bibitem{jauch2021monte}
M.~Jauch, P.~D. Hoff, and D.~B. Dunson.
\newblock Monte carlo simulation on the stiefel manifold via polar expansion.
\newblock {\em Journal of Computational and Graphical Statistics},
  30(3):622--631, 2021.

\bibitem{kerkyacharian2001nonlinear}
G.~Kerkyacharian, O.~Lepski, and D.~Picard.
\newblock Nonlinear estimation in anisotropic multi-index denoising.
\newblock {\em Probability theory and related fields}, 121(2):137--170, 2001.

\bibitem{kim2019Kernel}
J.~Kim, J.~Shin, A.~Rinaldo, and L.~Wasserman.
\newblock Uniform convergence rate of the kernel density estimator adaptive to
  intrinsic volume dimension.
\newblock pages 3398--3407, 2019.

\bibitem{VAEBayes}
D.~P. Kingma and M.~Welling.
\newblock Auto-encoding variational bayes, 2013.

\bibitem{klein1970vowel}
W.~Klein, R.~Plomp, and L.~C. Pols.
\newblock Vowel spectra, vowel spaces, and vowel identification.
\newblock {\em The Journal of the Acoustical Society of America},
  48(4B):999--1009, 1970.

\bibitem{kruijer2010adaptive}
W.~Kruijer, J.~Rousseau, and A.~Van Der~Vaart.
\newblock Adaptive bayesian density estimation with location-scale mixtures.
\newblock {\em Electronic Journal of Statistics}, 4:1225--1257, 2010.

\bibitem{lee2007nonlinear}
J.~A. Lee and M.~Verleysen.
\newblock {\em Nonlinear dimensionality reduction}, volume~1.
\newblock Springer, 2007.

\bibitem{lee2006riemannian}
J.~M. Lee.
\newblock {\em Riemannian manifolds: an introduction to curvature}, volume 176.
\newblock Springer Science \& Business Media, 2006.

\bibitem{ma2012manifold}
Y.~Ma and Y.~Fu.
\newblock {\em Manifold learning theory and applications}, volume 434.
\newblock CRC press Boca Raton, 2012.

\bibitem{dirichletprocess}
D.~Markwick and G.~J.~Ross.
\newblock dirichletprocess: Build dirichlet process objects for bayesian
  modelling, 2020.
\newblock https://cran.r-project.org/package=dirichletprocess.

\bibitem{maugis2013adaptive}
C.~Maugis-Rabusseau and B.~Michel.
\newblock Adaptive density estimation for clustering with gaussian mixtures.
\newblock {\em ESAIM: Probability and Statistics}, 17:698--724, 2013.

\bibitem{meguelati2019dirichlet}
K.~Meguelati, B.~Fontez, N.~Hilgert, and F.~Masseglia.
\newblock Dirichlet process mixture models made scalable and effective by means
  of massive distribution.
\newblock In {\em Proceedings of the 34th ACM/SIGAPP Symposium on Applied
  Computing}, pages 502--509, 2019.

\bibitem{MengersenCovid}
K.~Mengersen, E.~Santos-Fernandez, F.~Denti, A.~Mira, and A.~Varghese.
\newblock On the intrinsic dimensionality of covid-19 data: a global
  perspective, 2022.

\bibitem{mukhopadhyay2020estimating}
M.~Mukhopadhyay, D.~Li, D.~B. Dunson, et~al.
\newblock Estimating densities with non-linear support by using
  fisher--gaussian kernels.
\newblock {\em Journal of the Royal Statistical Society Series B},
  82(5):1249--1271, 2020.

\bibitem{naulet2017posterior}
Z.~Naulet and J.~Rousseau.
\newblock Posterior concentration rates for mixtures of normals in random
  design regression.
\newblock {\em Electronic Journal of Statistics}, 11(2):4065--4102, 2017.

\bibitem{neal2000markov}
R.~M. Neal.
\newblock Markov chain sampling methods for dirichlet process mixture models.
\newblock {\em Journal of computational and graphical statistics},
  9(2):249--265, 2000.

\bibitem{nikol2012approximation}
S.~M. Nikol'skii.
\newblock {\em Approximation of functions of several variables and imbedding
  theorems}, volume 205.
\newblock Springer Science \& Business Media, 2012.

\bibitem{niyogi2008finding}
P.~Niyogi, S.~Smale, and S.~Weinberger.
\newblock Finding the homology of submanifolds with high confidence from random
  samples.
\newblock {\em Discrete \& Computational Geometry}, 39(1):419--441, 2008.

\bibitem{ozakin2009submanifold}
A.~Ozakin and A.~Gray.
\newblock Submanifold density estimation.
\newblock {\em Advances in Neural Information Processing Systems}, 22, 2009.

\bibitem{Puchkin2022}
N.~Puchkin and V.~G. Spokoiny.
\newblock Structure-adaptive manifold estimation.
\newblock {\em J. Mach. Learn. Res.}, 23:1--62, 2022.

\bibitem{rousseau:scricciolo23}
J.~Rousseau and C.~Scricciolo.
\newblock Wasserstein convergence in bayesian and frequentist deconvolution
  models.
\newblock {\em arXiv preprint arXiv:2309.15300}, 2023.

\bibitem{roweis2000nonlinear}
S.~T. Roweis and L.~K. Saul.
\newblock Nonlinear dimensionality reduction by locally linear embedding.
\newblock {\em Science}, 290(5500):2323--2326, 2000.

\bibitem{rockova:rousseau:2023}
V.~Ročková and J.~Rousseau.
\newblock Ideal bayesian spatial adaptation.
\newblock {\em Journal of the American Statistical Association}, 0(0):1--14,
  2023.

\bibitem{scholkopf1998nonlinear}
B.~Sch{\"o}lkopf, A.~Smola, and K.-R. M{\"u}ller.
\newblock Nonlinear component analysis as a kernel eigenvalue problem.
\newblock {\em Neural computation}, 10(5):1299--1319, 1998.

\bibitem{shen2013adaptive}
W.~Shen, S.~T. Tokdar, and S.~Ghosal.
\newblock Adaptive bayesian multivariate density estimation with dirichlet
  mixtures.
\newblock {\em Biometrika}, 100(3):623--640, 2013.

\bibitem{tang2022minimax}
R.~Tang and Y.~Yang.
\newblock Minimax rate of distribution estimation on unknown submanifold under
  adversarial losses.
\newblock {\em arXiv preprint arXiv:2202.09030}, 2022.

\bibitem{tenenbaum2000global}
J.~B. Tenenbaum, V.~d. Silva, and J.~C. Langford.
\newblock A global geometric framework for nonlinear dimensionality reduction.
\newblock {\em science}, 290(5500):2319--2323, 2000.

\bibitem{vincent2002manifold}
P.~Vincent and Y.~Bengio.
\newblock Manifold parzen windows.
\newblock {\em Advances in neural information processing systems}, 15, 2002.

\bibitem{weinberger2006unsupervised}
K.~Q. Weinberger and L.~K. Saul.
\newblock Unsupervised learning of image manifolds by semidefinite programming.
\newblock {\em International journal of computer vision}, 70(1):77--90, 2006.

\end{thebibliography}

	\begin{appendix}
		\section{Some facts on submanifolds with bounded reach} \label{app:A}

		\subsection{The geometry of submanifolds with bounded reach} \label{app:A1}
		The reach $\tau_K$ of a closed subset $K \subset \bbR^D$, initially introduced by \cite[Def 4.1 p.432]{Fed59}, is the supremum of all the $r \geq 0$ such that the orthogonal projection from the $r$-offset $K^r = \{x \in \bbR^D ~|~d(x,K)\leq r\}$ to $K$ is well-defined, namely
		$$
		\tau_K := \sup\{r \geq 0~|~\forall x \in K^r, \forall y,z \in K,~ d(x,K) = \|x-y\| = \|x-z\| \implies y=z  \}.
		$$
		When the reach of a closed submanifold $M \subset \bbR^D$ is bounded away from zero, $M$ enjoys a number useful properties, that we list and prove below. In all the results stated hereafter, the reach of $M$ is bounded from below by some $\tau > 0$. Lemma \ref{lem:geolemma} provide already existing results from the literature, sometimes slightly rephrased to better suit our needs. We start off with a control of the exponential map over submanifold with bounded curvature together with a comparison between the intrinsic distance on $M$, denoted by $d_M(.,.)$, and the  ambient Euclidean distance. Recall that for any $x,y \in M$, $d_M(x,y)$ is the infimum of the length of all continuous path between $x$ and $y$ in $M$ and if $x$ and $y$ are in two separate path-connected components, then $d_M(x,y) = \infty$.
	
		\begin{lem}\label{lem:geolemma}
		The following facts hold true
		\bitem
		\item[i)]
		For any $x \in M$, the exponential map $\exp_{x}$ is a diffeomorphism from $B_{T_{x} M}(0,\pi\tau)$ to $\exp_{x}\{ B_{T_{x_0} M}(0,\pi\tau)\}$.
		\item[ii)] It is double Lipschitz from $\ball_{T_x M}(0,\tau/4)$ to its image with
			\beq \label{lem:doublelip}
			\forall v,w\in\ball_{T_x M}(0,\tau/4),~~~~\frac{11}{16} \|v-w\| \leq \|\exp_x(v) - \exp_x(w) \| \leq  \frac{21}{16} \|v-w\|.
			\eeq
		\item[iii)]
Letting $\kappa : \gamma \mapsto 2(1-\sqrt{1-\gamma})/\gamma$. If $\|x-y\|\leq \gamma\tau/2$ with $\gamma \leq 1$, then
			\beq \label{lem:useful}
			\|x-y\| \leq d_M(x,y) \leq \kappa(\gamma) \|x-y\|.
			\eeq			
		\item[iv)] Finally, if $\|x-y\| \leq \tau/2$, there holds
		\beq \label{lem:angles}
			\|\pr_{T_x M} - \pr_{T_y M}  \|_{\op}  \leq \frac{d_M(x,y)}{\tau} \leq \frac{2}{\tau}\|x-y\|.
			\eeq	
		\eitem
		\end{lem}	
\begin{proof}
	The first result on $\exp_{x}$ is an application of \cite[Thm 1.3]{alexander2006gauss} . For ii), denoting $R_x(v) = \exp_x(v)-x-v$, there holds that, 
			$$
			\left| \|\exp_x(v) - \exp_x(w) \| -\|v-w\| \right | \leq \|R_x(v) - R_x(w) \| \leq \frac{5}{16} \|v-w\|
			$$
			where we used \cite[Lem 1]{aamari2019nonasymptotic}.
Finally iii) comes from the monotonicity of $\kappa$ and \cite[Prp 6.3]{niyogi2008finding}, and 
using \cite[Lem 6]{boissonnat2019reach}, there holds $\|\pr_{T_x M} - \pr_{T_y M}  \|_{\op}  \leq d_M(x,y)/\tau$, which, together with iii) for $\gamma = 1$, leads to iv). 
\end{proof}

		We now wish to define the parametrization of the $\tau/2$-offset of $M$ that we introduced in \subsecref{manihold}. This requires to identify in a non-ambiguous way every normal fiber $N_x M$ to the base fiber $N_{x_0} M$ for $x$ in the vicinity of $x_0$. A natural way to do that would be to use parallel transport, and define
		$$
		N_{x_0}(v,\eta) := t_{\gamma}(\eta)
		$$ 
		where $t_{\gamma} : N_{x_0} M \to N_{\exp_{x_0}(v)} M$ is the parallel transport along the path $\gamma(s) := \exp_{x_0}(sv)$. We refer to \cite[Sec 4]{lee2006riemannian} for a formal introduction to parallel transport. 
		\newcommand{\Norm}{\mathrm{Norm}}
		
		In order to make things more comprehensible for the reader who is unfamiliar with parallel transports, and in order to have clear and quantitative controls and the quantity at stake, we suggest another, more elementary approach. We assert that the two approaches yields similar regularity classes as introduced in \subsecref{manihold}.  We start off with a few notations. For a matrix $A \in \bbR^{D\times D}$ and $1 \leq k \leq D$, we let $V_k(A)$ be the vector space spanned by the first $k$ columns of $A$. We denote by $\Norm : x \mapsto x/\|x\|$. We let
		$$
		\cG : GL(D,\bbR) \to O(D,\bbR)
		$$
		be the Gram-Schmidt process, defined recursively on the columns of any invertible matrix $A = (A_1,\dots,A_D)$ as $\cG(A) = (\cG_1(A),\dots,\cG_D(A))$ with 
		\begin{align*} 
			\cG_1(A) := \Norm(A_1)~~~\text{and}~~~\forall {1\leq j \leq D-1},~~&\cG_{j+1}(A) := \Norm\left(\bar\cG_{j+1}(A)\right)~~
			\\\text{where}~~&\bar\cG_{j+1}(A) := A_{j+1} - \sum_{1\leq i \leq j} \inner{A_{j+1}}{\cG_i(A)}\cG_i(A).
		\end{align*} 
		Because $\cG$ is such that $V_k(A) = V_k(\cG(A))$ for every $1\leq k \leq D$, there holds that $\bar\cG_{k+1}(A) = \pr_{V_k(A)^\perp}(A_{k+1})$, and that $\bar \cG_j(A)$ is thus non zero everywhere, so that $\cG$ is a  well-defined, smooth application. In order to bound its derivatives, we need to control how $\bar \cG_k$ is far away from zero. We introduce
		$$
		GL_\ve(D,\bbR) := \{A \in GL(D,\bbR)~|~d(A_{k+1},V_k(A)) \geq \ve ~~\forall 1\leq k \leq D\},
		$$ 
		so that $\|\bar\cG_k(A)\| \geq \ve$ for every $k$ and any $A \in GL_\ve(D,\bbR)$, and thus straightforwardly all the derivatives, up to any order, of $\cG$ are bounded on $GL_\ve(D,\bbR)$.
		We let $B_0$ be an arbitrary basis of $T_{x_0} M$, $B_\perp$ be an arbitrary basis of $N_{x_0} M$ and let $B = (B_0,B_\perp)$. We define, for $v \in \ball_{T_{x_0}M}(0,\tau/4)$, $A_{x_0}(B,v) := (\diff \Psi_{x_0}(v)[B_0], B_\perp)$. 	Note that since $\Psi_{x_0}$ is a diffeomorphism, there holds
		\beq \label{tanspan}
		V_d(A_{x_0}(B,v)) = \Vect(\diff \Psi_{x_0}(v)[B_0]) = T_{\Psi_{x_0}(v)} M.
		\eeq
		Set 
		$$
		N_{x_0}^B(v,\cdot) := \cG(A_{x_0}(B,v))[\cdot], \quad 	\bar \Psi^B_{x_0}(v,\eta) := \Psi_{x_0}(v) + N^B_{x_0}(v,\eta). 
		$$
		We show in Lemma \ref{lem:defNB} that $N_{x_0}^B$ and $\bar\Psi^B_{x_0}$  are well defined and smooth, which combined with \eqref{tanspan} yields in particular  that $N_{x_0}^B(v,\cdot)$ is an isometry between $N_{x_0} M$ and $N_{\Psi_{x_0}(v)} M$.
	
		\begin{lem} \label{lem:defNB} For any $v \in \ball_{T_{x_0}M}(0,\tau/4)$, there holds that $A_{x_0}(B,v) \in GL_{\ve}(D,\bbR)$ with $\ve = 1/2^d$. Consequently, the map $v \mapsto N_{x_0}^B(v,\cdot)$ is in $\cH_{\iso}^{\beta_M-1}(\ball_{T_{x_0}M}(0,\tau), C)$ for some constant $C_M$, $\tau$, $\beta_M$ and $D$ (and not on $B$).
		
		Moreover for any basis $B$, it holds that 
			$|\det \diff\bar \Psi_{x_0}^B(v,\eta) | \geq (3/16)^d$ for any $v \in \ball_{T_{x_0} M}(0,\tau/4)$ and $\eta \in  \ball_{N_{x_0} M}(0,\tau/2)$.
		\end{lem}
		
		The proof of \lemref{defNB} can be found in Section \ref{pr:lem:defNB}. An important feature of the parametrizations $\bar\Psi^B_{x_0}$ is that the subsequent H\"older classes as defined in Definition \ref{def:defmanihold} do not depend, up to a universal constant, on the choice of the  collection of basis $(B_{x_0})_{x_0 \in M}$, nor on the choice of the parametrisation.  This is shown in Section \ref{sec:stability-basis} 

		\subsection{Taylor expansion of  M-anisotropic H\"older functions} \label{app:manihold}

		In this section, we  derive a Taylor expansion for manifold-anisotropic H\"older functions. Recall from \secref{theoretical} that for any $\sigma,\delta > 0$,
		$$
		\Delta_{\sigma,\delta} = \begin{pmatrix} \sigma^{\alpha_0} \Id_{d} & 0 \\ 0 & \delta \sigma^{\alpha_\perp} \Id_{D-d} \end{pmatrix}.
		$$
		
		\begin{cor} \label{cor:taylor} Let $f \in \hbm$. Then, for any $x_0 \in M$, any $w \in \cW_{x_0,\delta}$, and any $z \in T_{x_0} M \times N_{x_0} M$ such that $w+\Delta_{\sigma,1} z \in \cW_{x_0,\delta}$, there holds  
			$$
			\bar f_{x_0,\delta}(w+\Delta_{\sigma,1} z ) = \bar f_{x_0,\delta}(w) + \sum_{0 < \inner{k}{\bm \alpha} < \beta} \sigma^{\inner{k}{\alpha}} \frac{z^k}{k!} \Diff^k \bar f_{\delta,x_0}(w) + R(w,z), 
			$$
			where the remainder $R$ satisfies the following bound
			\begin{align*} 
				|R(w,z)| \leq D \sigma^\beta \|1 \vee z\|_1^{\beta_{\max}}  L_{x_0,\delta}(w).
			\end{align*} 
		\end{cor}
		
		\begin{proof}[Proof of \corref{taylor}] Simply applying \prpref{taylor} to $\bar f_{x_0,\delta}$ yields
			$$
			\bar f_{x_0,\delta}(w+\Delta_{\sigma,1} z ) = \bar f_{x_0,\delta}(w) + \sum_{0 < \inner{k}{\bm \alpha} < \beta} \frac{(\Delta_{\sigma,1} z)^k}{k!} \Diff^k \bar f_{\delta,x_0}(w) + R_{x_0,\delta}(w,w+\Delta_{\sigma,1}z)
			$$
			where $R_{x_0,\delta}(w,\Delta_{\sigma,1}z)$ satisfies
			\begin{align*} 
				|R_{x_0,\delta}(w,w+\Delta_{\sigma,1}z)| &\leq L_{x_0,\delta}(w)  \sum_{\beta-\alpha_{\max} \leq \inner{k}{\alpha} <\beta} \frac{|\Delta_{\sigma,1}z|^k}{k!}\sum_{j=1}^D  |(\Delta_{\sigma,1}z)_j|^{\frac{\beta- \inner{k}{\bm\alpha}}{\alpha_j} }\\
				&= L_{x_0,\delta}(w) \sum_{\beta-\alpha_{\max} \leq \inner{k}{\bm\alpha} <\beta} \sigma^{\inner{k}{\bm\alpha}}\frac{|z|^k}{k!}\sum_{j=1}^D  \sigma^{\beta- \inner{k}{\bm\alpha}} |z_j|^{\frac{\beta- \inner{k}{\bm\alpha}}{\alpha_j}} \\
				&\leq D \sigma^\beta L_{x_0,\delta}(w) \|1 \vee z \|_1^{\beta_{\max}}
			\end{align*} 
			where we used the fact that $k_j + \frac{\beta-\inner{k}{\bm\alpha}}{\alpha_j} \leq \beta_j \leq \beta_{\max}$ for any $1 \leq j \leq D$. 
		\end{proof}

		\section{Additional results and proofs on the geometry of submanifolds} \label{appS:A}

%

		\subsection{Proof of \lemref{defNB}} \label{pr:lem:defNB}
		
		\begin{proof}[Proof of \lemref{defNB}]
		First notice that for $d \leq k \leq D-1$ and $v \in \ball_{T_{x_0}M}(0,\tau/4)$, there holds, letting $A = A_{x_0}(B,v)$, and because of \eqref{tanspan} and the fact that $B_\perp$ is an orthonormal frame of $N_{x_0} M$,
			$$
			d(A_{k+1},V_k(A)) = d(A_{k+1}, T_{\Psi_{x_0}(v)} M) \geq 1 - \| \pr_{T_{x_0} M}-\pr_{T_{\Psi_{x_0}(v)} M}\|_{\op} \geq  3/4
			$$ 
			where we used \eqref{lem:angles}. Now we let $1 \leq k \leq d-1$ and let $V_0(A) = \{0\}$. Letting $Q = (A_1,\dots,A_d)$, here holds that 
			$$\prod_{k=0}^{d-1} d(A_{k+1}, V_k(A))^2 = \det Q^\top Q = \det \diff \Psi_{x_0}(v)^\top \diff \Psi_{x_0}(v).$$
			Using \cite[Lem 1]{aamari2019nonasymptotic}, it holds $\|\diff \Psi_{x_0}(v)-\iota \|_{\op} \leq 5/16$ for $v \in \ball_{T_{x_0}M}(0,\tau/4)$ where $\iota : T_{x_0} M \to \bbR^D$ is the inclusion so that for any $h \in T_{x_0} M$, $\|\diff \Psi_{x_0}(v)[h]\| \geq (1-5/16)\|h\| = 11/16\|h\|$. In particular $\det  \Psi_{x_0}(v)^\top \diff \Psi_{x_0}(v) \geq (11/16)^{2d}$. Since now
			$$
			d(A_{k+1}, V_k(A))^2 \leq \|A_{k+1}\|^2 \leq \|\diff \Psi_{x_0}(v)\|_{\op}^2 \leq (21/16)^2 
			$$
			there holds that for any $1 \leq k \leq d$,
			$$
			d(A_{k+1}, V_k(A))^2 \geq (16/21)^{2d} (11/16)^{2d} \geq (11/21)^{2d} \geq 1/2^{2d}. 
			$$
Using that the Gram-Schmidt transform is smooth on $GL_\ve(D,\bbR)$ with $\ve = 1/2^d$ and \prpref{comp}, we obtain that the map $v \mapsto N_{x_0}^B(v,\cdot)$ is in $\cH_{\iso}^{\beta_M-1}(\ball_{T_{x_0}M}(0,\tau), C)$.

Also, in $B$ and in orthonormal bases of $T_{\Psi_{x_0}(v)} M$ and $N_{\Psi_{x_0}(v)}$, the Jacobian of $\bar \Psi_{x_0}^B(v,\eta)$ write
			$$
			\begin{pmatrix}
				\diff \Psi_{x_0}(v) + \pr_{T_{\Psi_{x_0}(v)} M} \circ \diff_v N_{x_0}^B(v,\eta)  & 0 \\
				\pr_{N_{\Psi_{x_0}(v)} M} \circ \diff_v N_{x_0}^B(v,\eta) & N_{x_0}^B(v,\cdot)
			\end{pmatrix}
			$$ 
			so that $|\det \diff\bar \Psi_{x_0}^B(v,\eta) | = \left|\det\{\diff \Psi_{x_0}(v) + \pr_{T_{\Psi_{x_0}(v)} M} \circ \diff_v N(v,\eta)\}\right|$. We saw earlier in the proof that $\|\diff \Psi_{x_0}(v)\|_{\op} \geq 11/16$. Furthermore, using \eqref{lem:angles}, we find that for any small $w \in T_{x_0}M$,
			\begin{align*} 
				\|\pr_{T_{\Psi_{x_0}(v)}}\{N_{x_0}^B(v+w,\eta)-N_{x_0}^B(v,\eta)\}\| &= \|(\pr_{T_{\Psi_{x_0}(v+w)}}- \pr_{T_{\Psi_{x_0}(v)}}) N_{x_0}^B(v+w,\eta) \| \\
				&\leq \frac{\|w\|}{\tau} \|\eta\|,
			\end{align*} 
			and consequently $\|\pr_{T_{\Psi_{x_0}(v)}} \circ \diff_v N_{x_0}^B(v,\eta)\|_{\op} \leq \|\eta\|/\tau \leq 1/2$. Thus, for any $h \in T_{x_0}M$, $\| \diff \Psi_{x_0}(v)[h] + \pr_{T_{\Psi_{x_0}(v)} M} \circ \diff_v N(v,\eta)[h]\|_{\op} \geq (11/16-1/2)\|h\| \geq 3/16 \|h\|$ and thus $$|\det \diff\bar \Psi_{x_0}^B(v,\eta) | \geq (3/16)^d.$$
			\end{proof}

\subsection{Stability by a change of basis} \label{sec:stability-basis}		
		We now show  that a change of basis does not interfere with the anisotropic regularity of a map seen through $\bar\Psi^B_{x_0}$.  
		\begin{lem} \label{lem:basis} For any orthonormal basis $B' = (B_0',B_\perp')$ subordinated to $T_{x_0} M$ and $N_{x_0} M$, and for any $\delta > 0$, it holds that
			\beq\label{basis}
			\left(\bar\Psi_{x_0,\delta}^B\right)^{-1} \circ \bar\Psi_{x_0,\delta}^{B'}(v,\eta) = \left(v, C_{B,B'}(v)\eta\right) 
			\eeq
			where $C_{B,B'}$ is independent of $\delta$ and is in $\cH_{\iso}^{\beta_M-1}(\ball_{T_{x_0}M}(0,\tau), C)$ for some constant $C$ depending on $C_M$, $\tau$, $\beta_M$ and $D$
			(and not on $B$ and $B'$).
		\end{lem}
		\begin{proof} Short and simple computations shows that $C_{B,B'}(v) := N_{x_0}^B(v,\cdot)^{\top} N_{x_0}^{B'}(v,\cdot)$ so that an application of Lemma \ref{lem:defNB} with \prpref{mult} immediately yields the result.
		\end{proof}
		\begin{cor}  \label{cor:basis} In the context of \subsecref{manihold},  assume that $\beta_0 \leq \beta_\perp \leq \beta_M-1$. 
			Then, there exists a constant $C$ depending on $C_M$, $\tau$, $\beta_M$ and $D$ such that, if there exists a basis $B$ such that $f \circ \bar \Psi_{x_0,\delta}^B \in \cH^{\bm\beta}_{\an}(\cW_{x_0,\delta},L^B_{x_0,\delta})$, then, for any other orthonormal basis $B'$, $f \circ  \bar\Psi_{x_0,\delta}^{B'} \in \cH^{\bm\beta}_{\an}(\cW_{x_0,\delta},CL_{x_0,\delta}^{B'})$.
		\end{cor}
		
		\begin{proof} Using the lemma above, there holds
			$$
			f \circ  \bar\Psi_{x_0,\delta}^{B'}  = f \circ \bar \Psi_{x_0,\delta}^B(v, C_{B,B'}(v)\eta) =  f \circ \bar \Psi_{x_0,\delta}^B \circ J (v,\eta)
			$$
			 where $J$ is defined through \eqref{basis}. We denote for short $f^{B'} = f \circ  \bar\Psi_{x_0,\delta}^{B'}$, $f^{B} = f \circ  \bar\Psi_{x_0,\delta}^{B'}$ so that $f^{B'} = f^B \circ J$. Taking $\inner{k}{\bm\alpha} < \beta$ and using the multivariate Faa di Bruno formula \cite{constantine1996multivariate}, we find that $\Diff^k f^{B'}$ is a sum of products of the form
			$$
			\Diff^{\ell} (f^B) \circ J \times \prod (\Diff^{k^{(j)}} J)^{\ell^{(j)}}
			$$
			subject to $|\ell|\leq |k|$, $\sum_{j} \ell^{(j)} = \ell$ and $\sum_j |\ell^{(j)} | k^{(j)} = k$ with $k^{(j)} \neq 0$. Now notice that $(\Diff^{k^{(j)}} J)_i = 0$ for $1 \leq i \leq d$ as soon as $k^{(j)}_\perp \neq 0$. For a configuration of $\ell, \ell^{(j)}$ and $k^{(j)}$ such that the above product is not zero, there thus holds
			\begin{align*} 
				|\ell_0| = \sum_j |\ell_0^{(j)}| = \sum_{k^{(j)}_\perp = 0} |\ell_0^{(j)}| \leq  \sum_{k^{(j)}_\perp = 0} |\ell_0^{(j)}| |k^{(j)}_0 | \leq |k_0|
			\end{align*} 
			which, together with $|\ell | \leq |k$ and $\alpha_0 \geq \alpha_\perp$, yields that $\inner{\ell}{\bm\alpha} \leq \inner{k}{\bm\alpha}$ whenever the above product is non zero. We conclude with a telescopic argument with \lemref{basis}. 
		\end{proof}
	
More generally, as noted in Remark \ref{stabi:def:main}	  the definition of $\cH_{\delta}^{\beta_0,\beta_\perp}(M,L)$ through the construction given in Appendix \ref{app:A} does not depend on the choice of $\Psi_{x_0}$ as long as the parametrization $\Psi_{x_0}$ is $\beta_M$ regular.

		\newpaul{
  \begin{prp} \label{prop:stabilityPsi}
      Let $(\Psi_{x_0}')_{x_0 \in M}$ be a set of parametrizations of $M$ satisfying $\Psi'_{x_0} \in \cH_{\iso}^{\beta_M}\left( \mathcal{V}_{x_0},C_M\right)$ for all $x_0 \in M$. Let $\beta_0 \leq \beta_\perp \leq \beta_M - 3$. We let $\bar\Psi$ and $\bar\Psi'$ be the  extensions of $\Psi$ and $\Psi'$ constructed in  Appendix \ref{app:A}. Then $\cH_\delta^{\beta_0,\beta_\perp}(M,L,\bar\Psi) \subset \cH_\delta^{\beta_0,\beta_\perp}(M,CL,\bar\Psi')$ for some $C>0$ depending on $\beta_0,\beta_\perp,D$ and $C_M$.
  \end{prp}
  \begin{proof}
      Let  $f \in \cH_\delta^{\beta_0,\beta_\perp}(M,L,\bar\Psi) $ and $x_0 \in M$. We need to prove that $\delta^{D-d} f \circ \bar\Psi_{x_0}' \in \cH_{\an} (\mathcal{W}_{x_0,\delta},\delta^{D-d} L \circ \bar\Psi_{x_0,\delta}')$. We have $\delta^{D-d} f \circ \bar\Psi_{x_0,\delta}' = \bar f_{x_0,\delta} \circ \left[ \bar\Psi_{x_0,\delta}^{-1} \circ \bar\Psi_{x_0,\delta}' \right] = \bar f_{x_0,\delta} \circ \tau_{x_0,\delta}$ where we have defined $\tau_{x_0,\delta} = \bar\Psi_{x_0,\delta}^{-1} \circ \bar\Psi_{x_0,\delta}'$. By proposition \ref{prp:comp} we find that $\tau_{x_0,\delta} \in \cH_{\iso}^{\beta_M - 1}(\mathcal{W}_{x_0,\delta},C_M^*)$ for some $C_M^*>0$. Moreover by assumption we have $\bar f_{x_0,\delta} \in \cH_{\an}^{\bm\beta_{0,\perp}}(W_{x_0,\delta},L_{x_0,\delta})$, and $\tau_{x_0,\delta}$ can be expressed as
      \begin{align*}
      \tau_{x_0,\delta}(v,\eta) = & \left( \Psi_{x_0}^{-1} \circ \Psi_{x_0}'(v), \delta N_{x_0}(\Psi_{x_0}^{-1} \circ \Psi_{x_0}'(v), \cdot)^{-1} (\delta N_{x_0}'(\Psi_{x_0}^{-1} \circ \Psi_{x_0'}(v),\eta)) \right) \\
      = & \left( \tau^{(1)}(v), \tau^{(2)}(v,\eta) \right)
      \end{align*}
      With $\tau^{(1)} \in \cH_{\iso}^{\beta_M} \left( \mathcal{V}_{x_0},C \right)$ and $\tau^{(2)} \in \cH_{\iso}^{\beta_M-1} \left(\ball_{N_{x_0}M} \left(0, \tau/2\delta \right), C \right)$ for some $C>0$. What is crucial here is that $\tau^{(1)}$ takes only $v$ as argument, which implies that for every multi-index $k$ such that $\inner{k}{\bm\alpha} < \beta$, the partial derivative $\Diff^k \left( \bar f_{x_0,\delta} \circ \tau_{x_0,\delta} \right)$ exists and is a linear combination of terms of the form
      \[
     \Diff^{\ell}\bar f_{x_0,\delta} (\tau_{x_0,\delta}) \prod_{i=1}^{|\ell|} \left( \Diff^{r_i} \tau_{x_0,\delta} \right)^{q_i}
      \]
      where $|\ell| \leq |k|$, $\inner{\ell}{\bm\alpha} < \beta$, $r_1,\dots,r_{|\ell|} \in \mathbb{N}^D$, $q_1,\ldots,q_{|\ell|} \in \mathbb{N}$, and $\sum_{i=1}^D |r_i|q_i = |\ell|$. Since the derivatives of $\tau_{x_0,\delta}$ are bounded this implies that for some $C>0$ we have
      \[
      \left| \Diff^k \left( \bar f_{x_0,\delta} \circ \tau_{x_0,\delta} \right) \right| \leq CL_{x_0,\delta} \circ \tau_{x_0,\delta} = C\delta^{D-d} L \circ \bar\Psi_{x_0,\delta}'
      \]
      Moreover, for every $\beta - \alpha_{\max} \leq \< k,\bm\alpha \> < \beta$ and every $x,y \in \mathcal{W}_{x_0,\delta}$, as above we can upper bound $\left| \Diff^k \left( \bar f_{x_0,\delta} \circ \tau_{x_0,\delta} \right)(x) - \Diff^k \left( \bar f_{x_0,\delta} \circ \tau_{x_0,\delta} \right)(y) \right|$ by a linear combination of terms of the forms
      \beq \label{eq:ft}
      \left| \Diff^\ell \bar f_{x_0,\delta}(\tau_{x_0,\delta}(x)) - \Diff^\ell \bar f_{x_0,\delta}(\tau_{x_0,\delta}(y)) \right| \times \prod_{i=1}^{|\ell|} \left| (\Diff^{r_i} \tau_{x_0,\delta})^{q_i}(y) \right|
      \eeq
      and
      \beq \label{eq:st}
       \left|\prod_{i=1}^{|\ell|} (\Diff^{r_i} \tau_{x_0,\delta})^{q_i}(x) - \prod_{i=1}^{|\ell|} (\Diff^{r_i} \tau_{x_0,\delta})^{q_i}(y) \right| \times \left| \Diff^{\ell} \bar f_{x_0,\delta}(\tau_{x_0,\delta})(x) \right|
      \eeq
      The functions $(\Diff^{r_i} \tau_{x_0,\delta})^{q_i}$ are in $\cH_{\iso}^{\beta_M-1} \left(\ball_{N_{x_0}M} \left(0, \tau/2\delta \right), C \right)$ by Proposition \ref{prp:mult} for some $C>0$. Hence, as $\beta_M - 1 \geq \beta_\perp + 2 \geq 1$ and 
      $
      |\Diff^\ell \bar f_{x_0,\delta}(\tau_{x_0,\delta}(x)) | \leq L_{x_0,\delta}(\tau_{x_0,\delta}(x))
      $
      the term in \eqref{eq:st} is bounded by a multiple of 
      $$
     \delta^{D-d} L \circ \bar\Psi_{x_0,\delta}'(x) \times \sum_{i=1}^D \left| x_i-y_i \right| \leq \delta^{D-d} L \circ \bar\Psi_{x_0,\delta}'(x) \times \sum_{i=1}^D \left| x_i-y_i \right|^{1 \wedge \frac{\beta - \<k, \bm \alpha\>}{\alpha_i}}
      $$
      Whereas since $(\Diff^{r_i} \tau_{x_0,\delta})^{q_i}$ are bounded and $\tau_{x_0,\delta}$ is Lipschitz, the term in \eqref{eq:ft} is upper bounded by a multiple of
      $$
      \left| \Diff^\ell \bar f_{x_0,\delta}(\tau_{x_0,\delta}(x)) - \Diff^\ell \bar f_{x_0,\delta}(\tau_{x_0,\delta}(y)) \right| \leq L_{x_0,\delta}(\tau_{x_0,\delta}(x)) \sum_{i=1}^D \left|\tau_i(x) - \tau_i(y) \right|^{1 \wedge \frac{\beta - \inner{k}{\bm\alpha}}{\alpha_i}}
      $$
      As noticed above, only $(x_i)_{i \leq d}$ are involved in the computation of $(\tau_i(x))_{i \leq d}$. Hence, using in addition that $\tau_{x_0,\delta}$ is Lipschitz and that $\alpha_i = \alpha_0$ for every $i \leq d$, we have
      \[
      \sum_{i=1}^d \left|\tau_i(x) - \tau_i(y) \right|^{1 \wedge \frac{\beta - \<k,\bm\alpha \>}{\alpha_i}} \lesssim \sum_{i=1}^d \left|x_i - y_i \right|^{1 \wedge \frac{\beta - \<k,\bm\alpha \>}{\alpha_i}}
      \]
      Moreover since for every $i > d$ we have $\alpha_i = \alpha_\perp \leq \alpha_0$, there holds
      \[
      \sum_{i>d} \left|g_i(x) - g_i(y) \right|^{1 \wedge \frac{\beta - \<k,\bm\alpha \>}{\alpha_i}} \lesssim \sum_{i=1}^D \left|x_i - y_i \right|^{1 \wedge \frac{\beta - \<k,\bm\alpha \>}{\alpha_i}}
      \]
   which ends the proof.
  \end{proof}
  }

	\subsection{Partitions of unity and packings} \label{sec:partition}	
The approximation result uses a particular covering of an offset of the manifold $M$, which we describe here.	 Take $x_0 \in M$ and define
		$$
		\cW_{x_0}^j := \ball_{T_{x_0}}\(0,\frac{2+j}{16}\tau\) \times \ball_{N_{x_0}}\(0,\frac{6+j}{8}\tau\)~~~\text{for}~~j \in\{0,1,2\},
		$$
		and $\cO_{x_0}^j = \bar\Psi_{x_0}(\cO_{x_0}^j)$ for $j \in \{0,1,2\}$. We have
		$$
		\cW_{x_0}^0 \subset \cW_{x_0}^1 \subset \cW_{x_0}^2~~~\text{and} ~~~\cO_{x_0}^0 \subset \cO_{x_0}^1 \subset \cO^2_{x_0}.
		$$ 
		Furthermore, the sets $\cO_{x_0}^0$ for $x_0 \in M$ forms a covering of $M^{3\tau/4}$. See \figref{wox0} for an illustration of these open sets. 
		
		\begin{figure}[h] 
			\centering
			\includegraphics[width=6cm]{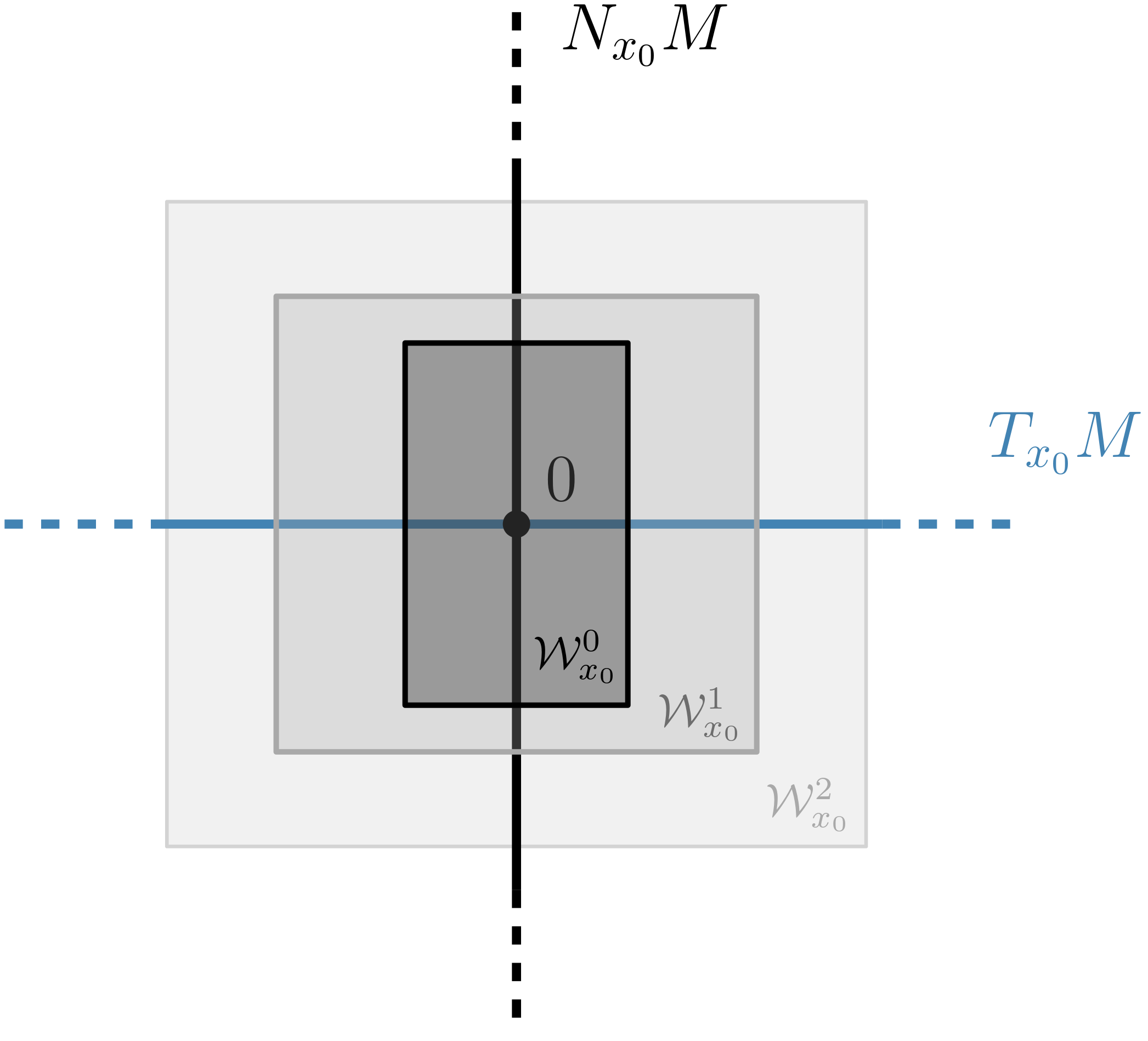}  \hspace{20pt}
			\includegraphics[width=6cm]{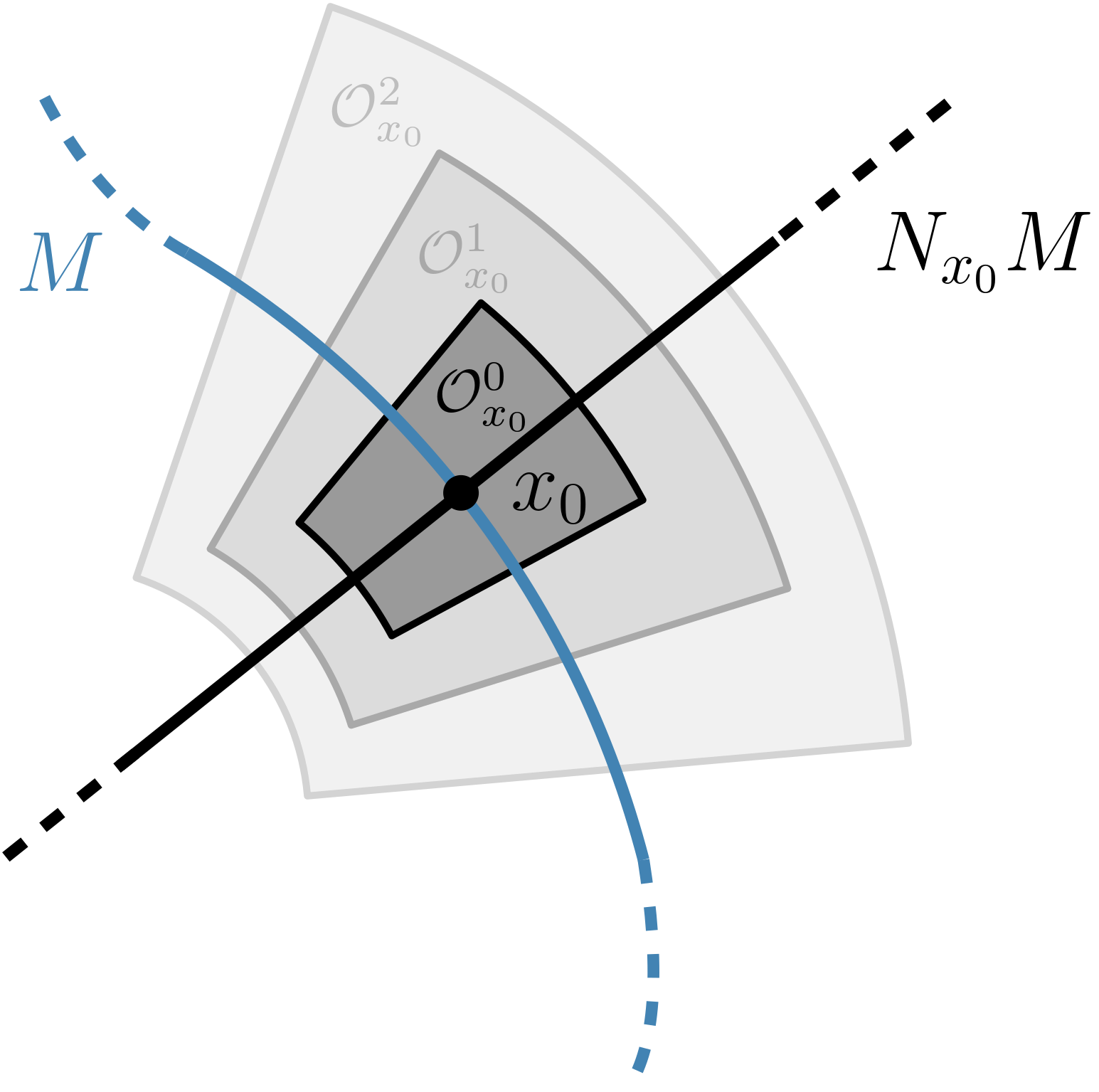} 
			\caption{A visual representation of the sets used in the proof of \secref{theoretical}.}
			\label{fig:wox0}
		\end{figure}
		
		In what follows, we will need the notion of packing. An $\ve$-packing of a subset $\cA \subset \bbR^D$ is a set $\{y_1,\dots,y_J\}$ of points of $\cA$ such that $\|x_j - x_k\| > \ve$ for any $1 \leq i \neq j \leq J$ and such that no set of $J+1$ points has this property. We denote by $\pk(\cA,\ve) := J$ the $\ve$-packing number of $\cA$. By maximality of $J$, it is straightforward to see that a $\cA$ is covered by the union of the balls $\ball(x_j,\ve)$. Furthermore, the balls $\ball(x_j,\ve/2)$ must be disjoint by definition of a packing so that
		$$
		\pk(\cA,\ve) \times \min_{x \in \cA} \vol(\cA \cap \ball(x,\ve/2)) \leq \vol \{\cA \cap \bigcup_{1 \leq j \leq J} \ball(x_j,\ve/2)\} \leq \vol \cA.  
		$$
		In the case where $\cA$ is a ball of radius $R$ with $R$ large before $\ve$, it is straightforward to see that $\vol(\cA \cap \ball(x,\ve/2)) \gtrsim \ve^D$ for any $x \in \cA$ so that $\pk(\cA,\ve) \lesssim (R/\ve)^D$. We now let 
		$$\rho(x) := \exp\{-\frac{1}{(1-\|x\|^2)_+}\}
		$$
		which is an infinitely differentiable, radially symmetric function from $\bbR^D$ to $[0,1]$ supported on $\ball(0,1)$. For any $x_0 \in M$, we define
		$$
		\rho_{x_0}(x) := \rho\left(32\frac{x-x_0}{\tau}\right)
		$$
		For any large $R > 0$, one can take a $\tau/64$-packing of $M \cap \ball(0,R)$, say $\{x_1,\dots,x_J\}$ with $J$ of order less than $\pk(\ball(0,R), \tau/64) \lesssim R^D$. We can define
		$$
		\chi_{j}(x) := \frac{\rho_{x_j}(x)
		}{\sum_{i=1}^J \rho_{x_i}(x)}.
		$$
		In a similar fashion as what is done in \cite{divol2021reconstructing}, we first review a few properties satisfied by the maps $\{\chi_{j}\}_{1\leq j \leq J}$, which forms a partition of unity associated with the covering $\{\ball(x_j,\tau/32)\}_{1 \leq j \leq J}$ of $M \cap \ball(0,R)$. See \figref{chix0} for a geometric interpretation of the situation.
		
		\begin{lem} \label{lem:chixj} The following assertions hold true: 
			\bitem
			\item[i)] For any $1 \leq j\leq J$, $\supp \chi_{j} \subset \cO_{x_j}^0$;
			\item[ii)] There exists a numeric constant $\gamma > 0$ such that $M^{\gamma\tau}\cap \ball(0,R) \subset \supp \sum_{j} \chi_{j}$;
			\item[iii)] There exists a numeric constant $\nu > 0$, such that for any $x \in M^{\gamma \tau} \cap \ball(0,R)$, there holds $\sum_{j=1}^J \rho_{x_j}(x) \geq \nu$;
			\item[iv)] For any $|k| \leq K$, there holds that $\|\Diff^{k} \chi_{j} \|_\infty \leq C < \infty$ with $C$ depending on $K$ and $\tau$;
			\item[v)] For any $|k| \leq K$, there exists a non-negative function $I_K $ such that, for any $1 \leq j \leq J$, 
			$$|\Diff^k \chi_{j}(x)| \leq I_K(x-x_j) \chi_{j}(x),~~~~\forall x \in M^{\gamma \tau} \cap \ball(0,R)$$ 
			with $I_K$ being such that for any $\omega > 0$
			$$\sup_{ x\in M^{\gamma \tau}} \left(I_K(x-x_j)\right)^\omega \chi_{j}(x) \leq C < \infty$$
			where $C$ depends on $K$, $\tau$ and $\omega$.
			\eitem
		\end{lem}
		We stress out that the constants appearing in \lemref{chixj} do not depend on $R$ or $J$ at all, so that it can be applied for a family of covering $\{x_1,\dots,x_J\}$ indexed by $R \to \infty$, as it will be the case in the proofs below.

		\begin{figure}[h]
			\centering
			\includegraphics[scale=.10]{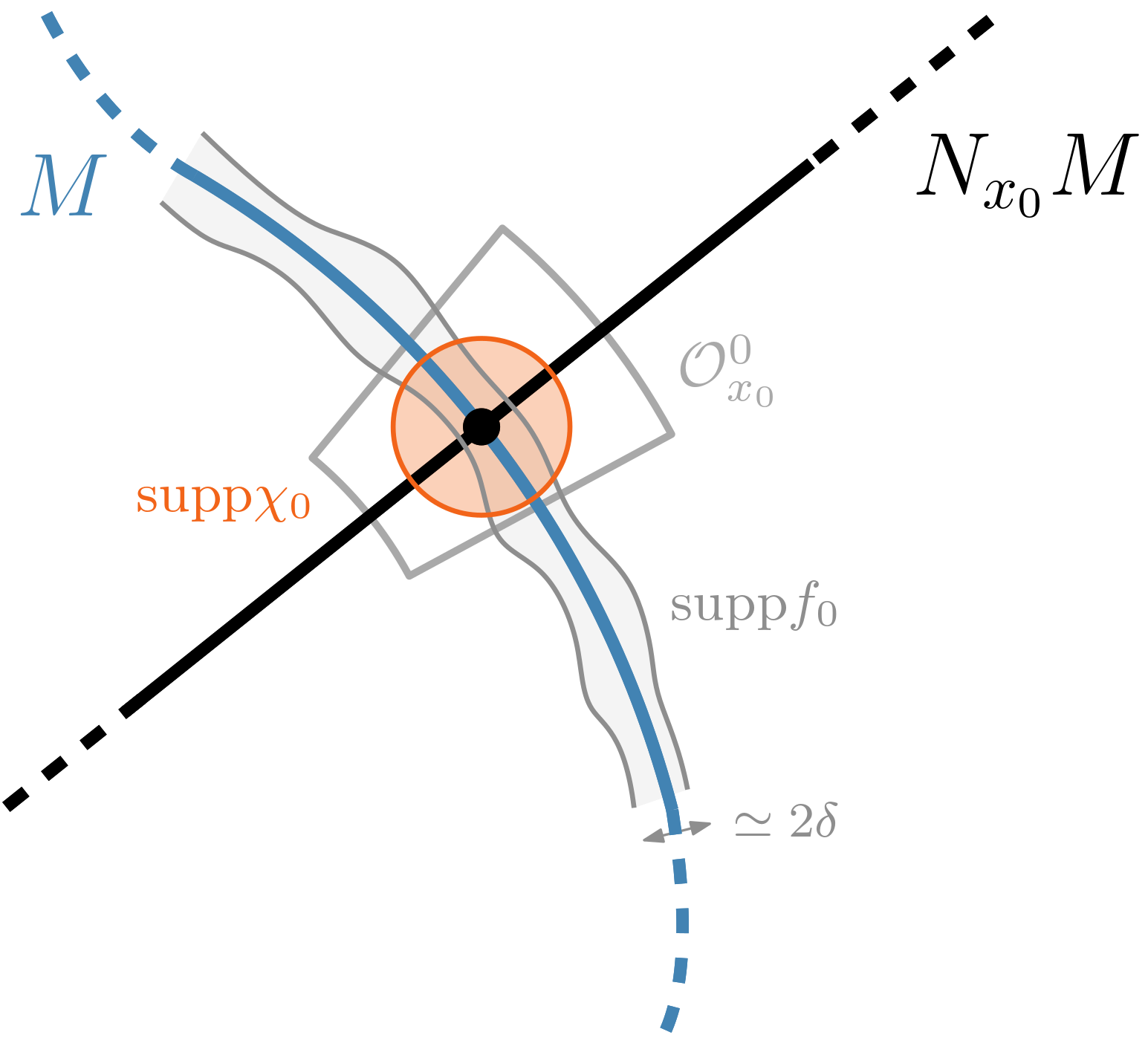}  \hspace{.5in}
			\includegraphics[scale=.10]{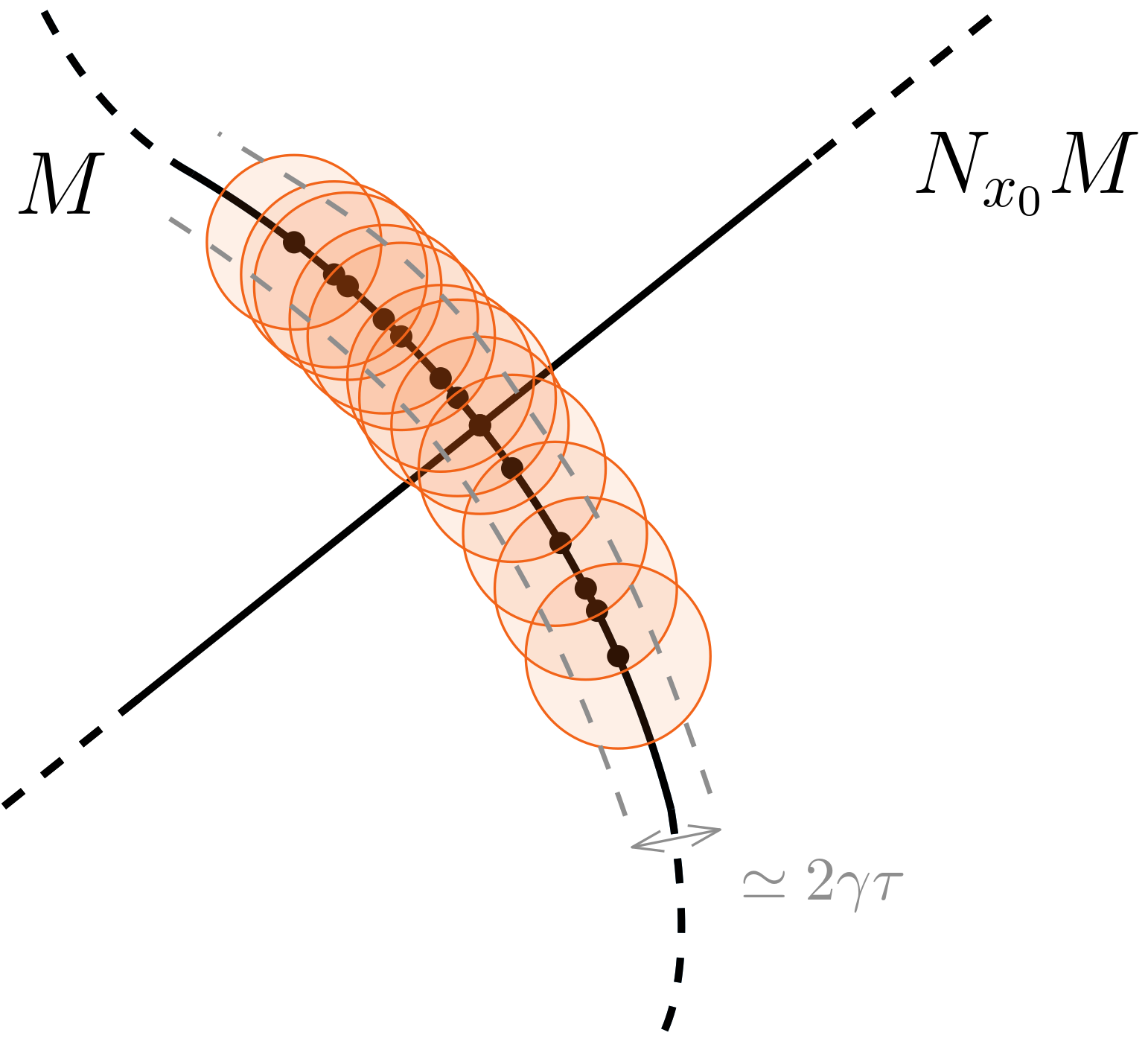} 
			\caption{A visual interpretation of the framework and results of \lemref{chixj}.}
			\label{fig:chix0}
		\end{figure}
		
		\begin{proof}[Proof of \lemref{chixj}] We start with proving i). Let $x \in \ball(x_j,\tau/32)$ and let $z = \pr_M x$. There holds that
			$$
			\|z - x_j\|^2 = \|z - x\|^2  +\|x - x_j\|^2  + 2\inner{z - x}{x-x_j} =  \|x - x_j\|^2 - \|z - x\|^2 + 2\inner{z - x}{z-x_j}.
			$$
			But now using \cite[Thm 4.7 (8)]{Fed59}, there holds that $2\inner{z - x}{z-x_j}\leq \|z - x_j\|^2 \|z -x\|/\tau$ so that in the end
			$$
			\|z - x_j\|^2 \leq \frac{\|z-x\|^2}{1 - \|z-x\|/\tau} \leq \{\sqrt{\frac{32}{31}} \frac{\tau}{32}\}^2,
			$$
			where we used the fact that $\|z-x\| \leq \|x_j-x\| \leq \tau/32$. In particular, using \eqref{lem:useful} in Lemma \ref{lem:geolemma}, we obtain 
			$$
			d_M(z,x_j) \leq \kappa\left(\sqrt{\frac{32}{31}}\frac{1}{16}\right) \|z - x_j\| \leq \kappa\left(\sqrt{\frac{32}{31}}\frac{1}{16}\right) \sqrt{\frac{32}{31}} \frac{\tau}{32} < \tau/16
			$$
			where the last inequality was checked with a calculator, so that $z = \Psi_{x_j}(v)$ with $v \in \ball_{T_{x_j} M}(0,\tau/16)$. Finally $\|x-z\| \leq \tau/32$ so that $x-z = N_{x_j}(v,\eta)$ for some $\eta \in \ball_{N_{x_j} M}(0,\tau/32)$. In the, end $x =\bar\Psi_{x_j}(v,\eta)$ with $(v,\eta)\in\cW_{x_j}^0$ so $x \in \cO_{x_j}^0$, ending the proof of the assertion i).
			
			We continue with proving ii) and iii) with $\gamma = 1/128$ and $\nu = \exp\{-16/7\}$. Take $x \in M^{\gamma \tau} \cap \ball(0,R)$ and $z = \pr_M(x)$. There exists $x_j$ such that $\|z - x_j\| \leq \tau /64$, leading to $\|x - x_j \| \leq 3 \tau / 128$ and thus $x \in \supp \chi_{x_j}$. Furthermore, 
			$$\rho_{x_j}(x)  =   \exp\{-\frac{1}{(1- (32\|x - x_j\|/\tau)^2)_+}\} \geq \exp\{-\frac{1}{(1- (3/4)^2)_+}\} = \nu,
			$$
			so that points ii) and iii) are proven.
			
			We finish with the proof of iv) and v). We let $R(x) = (\sum_{i=1}^J \rho_{x_i}(x))^{-1}$, which is well defined on $M^{\gamma\tau} \cap \ball(0,R)$ and bounded from above by $\nu^{-1}$ so that $\chi_{j}(x) = R(x) \rho_{x_j}(x)$ and 
			$$
			\Diff^k \chi_{j}(x) = \sum_{|\ell|\leq k} {k \choose \ell} \Diff^{k-\ell} R(x)  \Diff^\ell \rho_{x_j}(x).
			$$
			Outside of $\ball(x_j,\tau/32)$ one can take $I_K = 0$. Now if $x \in \ball(x_j,\tau/32)$, one get that for any $|\ell| < k$, using the Faa Di Bruno formula yields that $\Diff^\ell \rho_{x_j}(x)$ is a sum of term of the form
			$$
			(-1)^{s} \prod_{i=1}^s \Diff^{r_i} \Upsilon(x) \times \rho_{x_j}(x)~~~\text{with}~~~\Upsilon(x) := \frac{1}{1 - \frac{32^2}{\tau^2} \|x-x_j\|^2}~~\text{and}~~\sum_{i=1}^s r_i = \ell.
			$$
			In particular, we get that $\Diff^{\ell} \rho_{x_j}(x)$ is bounded from above by some constant depending on $K$ and $\tau$ on $\ball(x_j,\tau/32)$. Furthermore, for any $|\ell| < k$, the derivative $\Diff^\ell R(x)$ is a sum of terms of the form
			$$
			R(x)^{r} \prod_{i=1}^s \Diff^{r_i} R^{-1}(x)~~~\text{with}~~~\sum_{i=1}^s r_i \leq \ell~~\text{and}~~ r + \sum_{i=1}^s |r_i| = |\ell|+2,
			$$ 
			and $ \Diff^{r_i} R^{-1}(x) = \sum_q \Diff^{r_i} \rho_{x_q}(x)$ which is bounded from above by something depending on $\tau$ and $K$. All in all, we get the uniform boundedness of $\Diff^{k} \chi_{j}$ and that
			\begin{align*} 
			|\Diff^k \chi_{j}(x) | &\leq I_K(x-x_j) \chi_{j}(x) \\
			~~~\text{with}~~~ I_K(x) &:= \begin{cases} \displaystyle C_{K,\tau}\{1 - \frac{32^2}{\tau^2} \|x\|^2\}^{-(K+1)} ~&\text{if}~~x \in \ball(0,\tau/32), \\ 0 ~~ &\text{otherwise,} \end{cases}
			\end{align*}
			where $C_{K,\tau}$ depends on $K$ and $\tau$, ending the proof.
		\end{proof}

		\section{Appendix to \secref{model}} \label{app:anis}
		
		\subsection{Auxiliary results on general anisotropic H\"older functions}\label{app:auxhold}
		
		We first extend \dfnref{anihold} to functions defined on general open sets of $\bbR^D$. For any open set $\cU \subset \bbR^D$, any function $L : \cU \to \bbR_+$ and any positive real number $\zeta > 0$, we define the anisotropic H\"older spaces $\hbu$ as the set of functions $f : \cU \to \bbR^D$ satisfying that 
		\bitem
		\item[i)] For any multi-index $k \in \bbN^D$ such that $\inner{k}{\bm\alpha} < \beta$ the partial derivative $\Diff^k f$ is well defined on $\cU$ and
		$|\Diff^k f(x)| \leq L(x)$ for all $x \in \cU$; 
		\item[ii)] For any multi-index $k \in \bbN^D$ such that $\beta - \alpha_{\max} \leq \inner{k}{\bm\alpha} < \beta$, there holds
		\beq
		\label{eq:hold} 
		|\Diff^k f(y) - \Diff^k f(x)| \leq L(x) \sum_{i=1}^D |y_i - x_i|^{\frac{\beta- \inner{k}{\bm\alpha}}{\alpha_i} \wedge 1}~~~ \forall x,y \in \cU,~ \|x-y\| \leq \zeta.
		\eeq
		\eitem
		Like in \subsecref{anis}, we define in a similar fashion $\cH_{\iso}^{\beta}(\cU,L,\zeta)$, $\cH_{\an}^{\bm\beta}(\cU,C,\zeta)$ and $\cH_{\iso}^{\beta}(\cU,C,\zeta)$ for some constant $C > 0$. We now list a number of useful results which hold for  our definition of anisotropy. The proofs of these results are provided below.

		\begin{prp} \label{prp:der} Let $f\in \hbu$. Then for any $k \in \bbN^{D}$ such that $\inner{k}{\bm\alpha} < \beta$, the partially differentiated function $\Diff^k f$ is in $\cH^{\bm\beta^{(k)}}_{\an}(\cU,L,\zeta)$ with 
			$$
			\bm \beta^{(k)} = \{1 - \frac{\inner{k}{\bm\alpha}}{\beta}\} \bm\beta.
			$$
		\end{prp}

		Anisotropic H\"older functions enjoy the same convenient Taylor expansion as usual H\"older function. 
		\begin{prp} \label{prp:taylor} Let $f\in \hbu$. Then, for any $x,y \in \cU$ with $\|x-y\| \leq \zeta$, 
			\begin{align*} 
				f(y) = f(x) + \sum_{0 < \inner{k}{\bm \alpha} < \beta} \frac{(y-x)^k}{k!} \Diff^k f(x) + R(x,y), 
			\end{align*} 
			where the remainder $R$ satisfies the following bound
			\begin{align*} 
				|R(x,y)| \leq L(x)  \sum_{\beta-\alpha_{\max} \leq \inner{k}{\alpha} <\beta} \frac{|y-x|^k}{k!}\sum_{i=1}^D  |y_i - x_i|^{\frac{\beta- \inner{k}{\bm\alpha}}{\alpha_i} }.
			\end{align*} 
		\end{prp}

		\begin{prp} \label{prp:der2} Let $f \in \hbu$. Then, for any $k \in \bbN^D$ such that $\inner{k}{\bm\alpha} < \beta$, and any $x,y \in \cU$ with $\|x-y\|\leq\zeta$,
			\begin{align*} 
				|\Diff^k f(y) - \Diff^k f(x) | \leq C L(x) \{\sum_{i=1}^D  |y_i - x_i|^{\frac{\beta- \inner{k}{\bm\alpha}}{\alpha_i} \wedge 1} \}.
			\end{align*} 
			with $C$ depending on $k$,$\zeta$ and $D$.
		\end{prp}
		
			\begin{prp} \label{prp:less} Let $f \in \hbu$. Then, for any $\bm\beta' \leq \bm\beta$, $f \in \cH^{\bm\beta'}_{\an}(\cU, C L,\zeta)$ with $C$ depending on $\bm\beta$, $\zeta$ and $D$.
		\end{prp}

		\begin{prp}\label{prp:tens} Let $f \in \cH^{\bm \beta_1}_{\an}(\cU_1,L_1,\zeta_1)$ and $g\in \cH^{\bm\beta_2}_{\an}(\cU_2,L_2,\zeta_2)$ with $\cU_1 \subset \bbR^{D_1}$ and $\cU_2 \subset \bbR^{D_2}$. Then
			$$
			f\otimes g \in \cH^{(\bm\beta_1,\bm\beta_2)}_{\an}(\cU_1 \times\cU_2, C L_1\otimes L_2,\zeta)
			$$
			where $f\otimes g(x,y) = f(x)g(y)$ and $\zeta = \zeta_1\wedge\zeta_2$, and where $C$ depends on $D_1,D_2,\zeta,\bm\beta_1$ and $\bm\beta_2$.
		\end{prp}

		\begin{prp} \label{prp:mult} Let $f \in \cH^{\bm\beta_1}_{\an}(\cU,L_1,\zeta)$ and $g \in \cH^{\bm \beta_2}_{\an}(\cU,L_2,\zeta)$. Then $f\times g \in  \cH^{\bm\beta}_{\an}(\cU,L,\zeta)$ with $\bm\beta = \bm\beta_1\wedge\bm\beta_2$, and where, for some constant $C$ depending on $\zeta$, $\bm\beta$ and $D$, 
			$$
			L(x) = C L_1(x)  L_2(x).
			$$
		\end{prp}


	\begin{prp} \label{prp:comp} Let $f \in \cH_{\iso}^{\beta_1}(\cU_1,L_1)$ and $g  \in \cH_{\iso}^{\beta_0}(\cU_0,C_0)$  where  $g$ takes its value in $\cU_1$, and where $\cU_0$ is bounded. Assume furthermore that $\beta_0 \geq 1$. Then $f \circ g \in \cH^\beta_{\iso}(\cU_0,L_1 \circ g)$ where $\beta = \beta_0 \wedge \beta_1$ and where $C$ depends on $C_0, C_1$ and $\diam \cU_0$.
		\end{prp}

		\begin{proof}[Proof of Proposition \ref{prp:der}] 
		  Let  $k \in \bbN^D$ such that $\inner{k}{\alpha} < \beta$. Note that
			$$
			\bm \alpha = \beta \frac{1}{\bm \beta} =  \beta^{(k)} \frac{1}{\bm \beta^{(k)}}.
			$$ 
			Let $\ell \in \bbN^D$ such that $\inner{\ell}{\bm\alpha} < \beta^{(k)}$. Then $\inner{k+\ell}{\alpha} < \beta$ so that by definition, $\Diff^{\ell} \Diff^{k} f = \Diff^{k+\ell} f$ exists and is bounded from above by $L$. If now $\ell$ is such that $\beta^{(k)} - \alpha_{\max} \leq \inner{\ell}{\bm\alpha} < \beta^{(k)}$, then $\beta - \alpha_{\max} \leq \inner{k+\ell}{\bm\alpha} < \beta$ and for any $x,y \in \cU$ that are at most $\zeta$-apart, 
			$$
			|\Diff^\ell \Diff^{k} f(y) - \Diff^\ell \Diff^k f(x)| \leq L(x) \sum_{i=1}^D |y_i - x_i|^{\frac{\beta- \inner{k+\ell}{\bm\alpha}}{\alpha_i}\wedge1} = L(x) \sum_{i=1}^D |y_i - x_i|^{\frac{\beta^{(k)}- \inner{\ell}{\bm\alpha}}{\alpha_i}\wedge 1},
			$$
			ending the proof.
		\end{proof}

		\begin{proof}[ Proof of Proposition \ref{prp:taylor}]
		
			We let for any $0\leq i \leq D$,
			$$
			z^{(i)} = (y_1,\dots,y_i,x_{i+1},\dots,x_D) \in \bbR^D
			$$
			so that $z^{(0)} = x$ and $z^{(D)} = y$. We prove the results recursively on the integer $N = | \ceil{\bm\beta-1} |$. If $N = 0$, then every coefficient $\beta_i$ are strictly less than $1$, and the results follow immediately from the definition of being $\bm\beta$-H\"older (there are no $k$ that satisfies $\inner{k}{\bm\alpha} < \beta$ except $k = 0$). If $N \geq 1$, we can order without loss of generality and for ease of notations 
			$$\beta_1 \leq \cdots \leq \beta_k < 1 \leq \beta_{k+1} \leq ..  \leq \beta_D$$
			with $k \leq D-1$ because $N \geq 1$. We write 
			$$
			f(y) - f(x) = f(z^{(k)}) - f(x) +  \sum_{i=k}^{D-1} f(z^{(i+1)})-f(z^{(i)}).
			$$
			The first term is simply bounded from above by
			$$
			|f(z^{(k)}) - f(x)|  \leq L(x) \sum_{i=1}^k  |y_i - x_i|^{\beta/\alpha_i}.
			$$
			Recall we write $e_i = (0, \dots,0,1,0,\dots, 0)$ with $1$ at the $i$-th position.  For the other terms, we do the Taylor expansion with integral remaining term: 
			\begin{align}
			\begin{split}
				&f(z^{(i)})-f(z^{(i-1)})  = \sum_{\ell=1}^{L_i-1} \frac{ (y_i-x_i)^\ell}{ \ell!}
				\Diff^{\ell e_i} f(z^{(i-1)}) + R_i(y,x), \quad L_i = \ceil{\beta_i-1} \geq 1 \label{devi} \\
				&R_i(y,x)  = \frac{ (y_i-x_i)^{L_i} }{ (L_i-1)!} \int_0^1(1-t)^{L_i-1} \Diff^{L_ie_i}f(z^{(i)}(t))\diff t,\\
				& \quad \quad \quad \quad \quad \quad\quad\quad\quad\quad \text{with}\quad z^{(i)}(t) = t z^{(i)} + (1-t) z^{(i-1)}. 
				\end{split}
			\end{align}
			
			Note that $L_i\alpha_i = \frac{ \ceil{\beta_i - 1} }{ \beta_i } \beta $ and for all $ k = (k_{1:i},0) \neq 0$ , $L_i\alpha_i + \inner{k}{\bm\alpha} \geq \beta$.
			Indeed, because the coefficient of $\bm\beta$ are ordered, there holds
			$$ L_i\alpha_i + \inner{k}{\bm\alpha} = \frac{ \beta  }{\beta_i}  \left[ \ceil{\beta_i-1}  + \sum_{j\leq i} k_j \beta_i/\beta_j  \right]\geq  \frac{ \beta  }{\beta_i}  ( \ceil{\beta_i-1}+1 )\geq \beta.$$
			Therefore 
			\begin{align}\label{Ri}
				\left| \Diff^{L_ie_i} f(z^{(i)}(t) ) - \Diff^{L_ie_i} f(x )\right| &\leq L(x)\sum_{j=1}^i|y_j-x_j|^{ \frac{\beta - L_i\alpha_i}{\alpha_j}} 
			\end{align}
			and in particular
			\begin{align*} 
				R_i(x,y) &= \frac{ (y_i-x_i)^{L_i} }{ L_i!} \Diff^{L_ie_i} f(x) + \wt R_i(x,y) \\
				\text{where}~~~~|\wt R_i(x,y)| &\leq L(x) \frac{|y_i-x_i|^{L_i}}{L_i !} \sum_{j=1}^i|y_j-x_j|^{ \frac{\beta - L_i\alpha_i}{\alpha_j}}. 
			\end{align*} 
			
			Now we use the induction hypothesis on $f_{i,\ell} =\Diff^{\ell e_i} f $ which belongs to $ \mathcal H^{\bm\beta^{(\ell e_i)}}(\mathcal U, L , \zeta)$ (according to \prpref{der}): 
			$$
			f_{i,\ell}(z^{(i-1)}) - f_{i,\ell}(x) = \sum_{\substack{0 < \inner{k}{\bm \alpha} < \beta - \ell \alpha_i \\ k_{i:D} = 0}} \frac{(y-x)^k}{k!} \Diff^{\ell e_i + k} f(x) + R_{i,\ell}(x,y), 
			$$
			where the remainder $R$ satisfies the following bound
			\begin{align*} 
				|R_{i,\ell}(x,y)| \leq L(x)  \sum_{\substack{\beta-\alpha_{\max} \leq \inner{k+\ell e_i}{\alpha} <\beta \\ k_{i:D} = 0}} \frac{|y-x|^k}{k!}\sum_{i=1}^D  |y_j - x_j|^{\frac{\beta- \inner{k + \ell e_i}{\bm\alpha}}{\alpha_j} }.
			\end{align*} 
			All in all, gathering all the developments yields
			\begin{align*} 
				&f(z^{(i)}) - f(z^{(i-1)}) \\
				&= \sum_{\ell = 1}^{L_i-1} \frac{(y_i-x_i)^\ell}{\ell!} \{\Diff^{\ell e_i} f(x) + \sum_{\substack{0 < \inner{k}{\bm \alpha} < \beta - \ell \alpha_i \\ k_{i:D} = 0}} \frac{(y-x)^k}{k!} \Diff^{\ell e_i + k} f(x) + R_{i,\ell}(x,y)\} + R_i(y,x) \\
				&=  \sum_{\ell = 1}^{L_i-1} \{\sum_{\substack{0 < \inner{k+\ell e_i}{\bm \alpha} < \beta  \\  k_{i:D} = 0}} \frac{(y-x)^{k+\ell e_i}}{(k+\ell e_i)!} \Diff^{\ell e_i + k} f(x) + \frac{(y_i-x_i)^\ell}{\ell!} R_{i,\ell}(y,x)\}+ R_i(y,x) \\
				&=   \sum_{\substack{0 < \inner{k}{\bm \alpha} < \beta  \\  k_i \neq 0,~k_{(i+1):D} = 0}} \frac{(y-x)^{k}}{k!} \Diff^{k} f(x) + \{\wt R_i(y,x) + \sum_{\ell = 1}^{L_i-1} \frac{(y_i-x_i)^\ell}{\ell!} R_{i,\ell}(y,x)\}
			\end{align*} 
			so that
			\begin{align*} 
				f(y) - f(x) &=   \sum_{0 < \inner{k}{\bm \alpha} < \beta} \frac{(y-x)^{k}}{k!} \Diff^{k} f(x) \\
				&+ \underbrace{\sum_{i \geq k+1} \{\wt R_i(y,x) + \sum_{\ell = 1}^{L_i-1} \frac{(y_i-x_i)^\ell}{\ell!} R_{i,\ell}(y,x)\} + R_0(y,x)}_{:= R(y,x)}
			\end{align*} 
			with $R_0(y,x) = f(z^{(k)})-f(x)$, and with $R(y,x)$ being exactly bounded from above by
			$$
			|R(y,x)| \leq L(x)  \sum_{\beta-\alpha_{\max} \leq \inner{k}{\alpha} <\beta} \frac{|y-x|^k}{k!}\sum_{j=1}^D  |y_j - x_j|^{\frac{\beta- \inner{k}{\bm\alpha}}{\alpha_j} }
			$$
			and \prpref{taylor} is proved. 
		\end{proof}

		\begin{proof}[Proof of Proposition \ref{prp:der2}] 
		 Either $\inner{k}{\bm\alpha} \geq \beta-\alpha_{\max}$ and then there is nothing to show, or $\inner{k}{\bm\alpha} < \beta-\alpha_{\max}$, in which case $\inner{k+e_i}{\bm\alpha} <\beta$ for any $i$ and thus $\diff \Diff^k f$ is well defined. There thus exist $z \in [x,y]$ such that $\Diff^k f(y) - \Diff^k f(x) = \diff \Diff^k f(z)[y-x]$ leading to
			$$
			|\Diff^k f(y) - \Diff^k f(x)| \leq \sum_{i=1}^D |\Diff^{k+e_i} f(z)|\times |y_i-x_i|. 
			$$
			Using an induction argument, there exists $C_{i} > 0$ such that
			$$
			|\Diff^{k+e_i} f(z) - \Diff^{k+e_i} f(x) | \leq C_{i} L(x) \{\sum_{i=1}^D  |y_i - x_i|^{\frac{\beta- \inner{k+e_i}{\bm\alpha}}{\alpha_i} \wedge 1} \} \leq C_{i} D (1\vee \zeta) L(x)
			$$
			leading the the right results with $C = 1 + D (1\vee \zeta) \max_{1 \leq i \leq D} C_{i}$.
		\end{proof}

				\begin{proof}[Proof of Proposition \ref{prp:less}]

		 Let $\beta'$ be the harmonic mean of $\bm\beta'$ and $\bm\alpha' = \beta'/\bm\beta'$. For any $k \in \bbN^D$ such that $\inner{k}{\bm\alpha'} < \beta'$, we have $\inner{k}{\bm\alpha} <\beta$ so that $\Diff^k f$ is well defined and bounded from above by $L$. What's more, using \prpref{der2}, we have
			\begin{align*} 
				|\Diff^k f(y) - \Diff^k f(x) | \leq C L(x) \{\sum_{i=1}^D  |y_i - x_i|^{\frac{\beta- \inner{k}{\bm\alpha}}{\alpha_i} \wedge 1} \}.
			\end{align*} 
			Now notice that
			$$
			\frac{\beta- \inner{k}{\bm\alpha}}{\alpha_i} = \beta_i \{1 - \inner{k}{1/\bm\beta}\} \geq \beta_i' \{1 - \inner{k}{1/\bm\beta'}\} = \frac{\beta'- \inner{k}{\bm\alpha'}}{\alpha_i'}
			$$
			so that
			$$
			|y_i - x_i|^{\frac{\beta- \inner{k}{\bm\alpha}}{\alpha_i} \wedge 1}  \leq (1\vee\zeta) \times |y_i - x_i|^{\frac{\beta'- \inner{k}{\bm\alpha'}}{\alpha_i'} \wedge 1}
			$$
			yielding the result.
		\end{proof}

		\begin{proof}[Proof of Proposition \ref{prp:tens}] Let $k = (k_1,k_2)\in \bbN^{D_1+D_2}$. Let $\bm\beta = (\bm\beta_1,\bm\beta_2)$ and $\beta, \beta_1, \beta_2$ be the harmonic means of $\bm\beta,\bm\beta_1$ and $\bm\beta_2$. Let also $\bm\alpha = \beta/\bm\beta$ and $\bm\alpha_i = \beta_i/\bm\beta_i$ for $i \in \{1,2\}$. If $\inner{k}{\bm \alpha} < \beta$, then
			$$
			\inner{k_1}{1/\bm\beta_1} +\inner{k_2}{1/\bm\beta_2} < 1 
			$$
			so that both term is stricly less that one $\Diff^{k_1}f$ and $\Diff^{k_2} g$ are well defined and $\Diff^k (f\otimes g) = \Diff^{k_1}f \otimes \Diff^{k_2}g$ is well defined as well. Furthermore, if $x = (x_1,x_2)$ and $y = (y_1,y_2)$, is such that $\|x-y\|\leq\zeta$, then $\|x_1-y_1\| \leq \zeta \leq \zeta_1$ and $\|x_2-y_2\| \leq \zeta \leq\zeta_2$. Using \prpref{der2}, one find that
			\begin{align*} 
				&|\Diff^k (f\otimes g)(x)-\Diff^k (f\otimes g)(y)| \\
				&\leq  |\Diff^{k_2} g(x_2)|\times |\Diff^{k_1} f(x_1)-\Diff^{k_1} f(y_1)| +  |\Diff^{k_1} f(y_1)|\times |\Diff^{k_2} g(x_2)-\Diff^{k_2} f(y_2)| \\
				&\leq C L_2(x_2)L_1(x_1) \{\sum_{i=1}^{D_1}  |y_{1,i} - x_{1,i}|^{\frac{\beta_1- \inner{k_1}{\bm\alpha_1}}{\alpha_{1,i}} \wedge 1} \} 
				\\&+ C|\Diff^{k_1} f(y_1)| L_2(x_2)  \{\sum_{i=1}^{D_2}  |y_{2,i} - x_{2,i}|^{\frac{\beta_2- \inner{k_2}{\bm\alpha_2}}{\alpha_{2,i}} \wedge 1} \}.
			\end{align*} 
			for some constant $C > 0$ depending on $D_1,D_2,\zeta$ and $\bm\beta$.
			Now notice first that 
			$$|\Diff^{k_1} f(y_1)| \leq |\Diff^{k_1} f(y_1) - \Diff^{k_1} f(x_1)|+|\Diff^{k_1} f(x_1)| \leq (C D (1\vee \zeta) + 1)L_1(x_1),$$
			and notice furthermore that of either $j\in\{1,2\}$,
			$$
			\frac{\beta_j- \inner{k_j}{\bm\alpha_j}}{\alpha_{j,i}} = \bm\beta_{j,i}\{1 - \inner{k_j}{1/\bm\beta_j}\} \geq   \bm\beta_{j,i}\{1 - \inner{k}{1/\bm\beta}\} = \frac{\beta- \inner{k}{\bm\alpha}}{\alpha_{i}} ~~\left(= \frac{\beta- \inner{k}{\bm\alpha}}{\alpha_{i+D_1}}~\text{if}~j=2\right)
			$$
			so that
			\begin{align*} 
				&|\Diff^k (f\otimes g)(x)-\Diff^k (f\otimes g)(y)| \\
				&\leq (1\vee\zeta) C L_2(x_2)L_1(x_1) \{\sum_{i=1}^{D_1}  |y_{1,i} - x_{1,i}|^{\frac{\beta- \inner{k}{\bm\alpha}}{\alpha_{i}} \wedge 1} \} \\
				&+ (1\vee\zeta) C (C D (1\vee \zeta) + 1) L_1(x_1)L_2(x_2)  \{\sum_{i=1}^{D_2}  |y_{2,i} - x_{2,i}|^{\frac{\beta- \inner{k}{\bm\alpha}}{\alpha_{i+D_1}} \wedge 1} \} \\
				&\leq (1\vee\zeta) C (C D (1\vee \zeta) + 1) (L_1\otimes L_2)(x)  \{\sum_{i=1}^{D_1+D_2}  |y_{i} - x_{i}|^{\frac{\beta- \inner{k}{\bm\alpha}}{\alpha_{i}} \wedge 1} \}
			\end{align*} 
			ending the proof.
		\end{proof}

		\begin{proof}[ Proof of Proposition \ref{prp:mult}]
		
		 We let $\bm\alpha = \beta/\bm\beta$ with $\beta$ the harmonic mean of $\bm\beta$. Now for any multi-index $k$ such that $\inner{k}{\bm\alpha} < \beta$, we have $\inner{\ell}{\bm\alpha} < \beta$ for any $\ell \leq k$ so that $\Diff^{k} fg$ is indeed well-defined and
			$$
			\Diff^{k}fg(x) = \sum_{\ell \leq k} {k\choose \ell} \Diff^{\ell} f(x) \Diff^{k-\ell}g(x) 
			$$
			so that $|\Diff^{k}fg(x)| \leq 2^{|k|} L_1(x) L_2(x)$ and 
			$$
			|\Diff^{k}fg(y)-\Diff^{k}fg(x)| \leq L_2(x) \sum_{\ell \leq k} {k\choose \ell} | \Diff^{\ell} f(y) - \Diff^{\ell} f(x)|+ L_1(x) \sum_{\ell \leq k} {k\choose \ell} | \Diff^{\ell} g(y) - \Diff^{\ell} g(x)|.
			$$
			Using now \prpref{der2}, we find that there exists $C > 0$ depending on $\bm\beta$, $D$ and $\zeta$ such that
			\begin{align*} 
				|\Diff^{k}fg(y)-\Diff^{k}fg(x)| &\leq  2 C   L_1(x)L_2(x) \sum_{\ell \leq k} {k\choose \ell} \sum_{i=1}^D |y_i - x_i|^{\frac{\beta- \inner{\ell}{\bm\alpha}}{\alpha_i} \wedge 1} \\
				&\leq 2^{|k|+1} (1\vee\zeta) C L_1(x)L_2(x) \sum_{i=1}^D |y_i - x_i|^{\frac{\beta- \inner{k}{\bm\alpha}}{\alpha_i} \wedge 1}
			\end{align*}
			where we used again that 
			$$
			|y_i - x_i|^{\frac{\beta- \inner{\ell}{\bm\alpha}}{\alpha_i} \wedge 1}  \leq (1\vee\zeta)\times |y_i - x_i|^{\frac{\beta- \inner{k}{\bm\alpha}}{\alpha_i} \wedge 1},   
			$$ 
			for $\ell \leq k$, ending the proof.
		\end{proof}

%
%
	
		\begin{proof}[ Proof of Proposition \ref{prp:comp}]
		
		A simple use of the multivariate Faa Di Bruno formula \cite{constantine1996multivariate} yields that, for any $k \in \bbN^D$ with $|k| \leq \beta$, $\Diff^{k} (f \circ g)(x)$ for $x \in \cU_0$ is a sum a term of the form
			$$
			\Diff^{\ell} f(g(x)) \prod_{j=1}^s   (\Diff^{k^{(j)}} g(x))^{\ell^{(j)}}
			$$
			with $|\ell| \leq |k|$, $s \leq |k|$, $\sum_j \ell^{(j)} = \ell$, and $\sum_{j} |\ell^{(j)}| k^{(j)} = k$. In particular, it is bounded by $C L_1(g(x))$ where $C$ depends on $C_0$ and $\beta$. Likewise, a telescopic argument would yield that 
			$$
			|\Diff^{k} (f\circ g)(y)-\Diff^{k} (f\circ g)(x) | \leq C L_1(g(x)) \prod_{i=1}^D |x-y|^{(\beta - |k|) \wedge 1}
			$$
			where $C$ depends on $C_0$, $\beta$ and $\diam \cU_0$, ending the proof.
		\end{proof}

		\subsection{Proofs associated to the examples of \prpref{model}} \label{app:prpmodel}
		
	\begin{proof}[Proof of \prpref{model}]
	
	 Take $x_0 \in M$. In the orthonormal noise model, the density $f$ takes the very simple form
			\begin{align*} 
				\bar f_{x_0,\delta}(v,\eta) &= \delta^{D-d} f \circ \bar\Psi_{x_0,\delta}(v,\eta) =  \delta^{D-d} \times f_0(\Psi_{x_0}(v)) \times \delta^{-(D-d)} c_{\bot} K\left(\frac1\delta N_{x_0}(v,\delta \eta)\right)\\
				&=  f_{x_0}(v) \times c_{\bot} K\left(\eta\right)
			\end{align*} 
			where we used the isotropy of $K$. Thus, $\bar f_{x_0,\delta}$ lies in $\cH^{\bm\beta}_{\an}(\cW_{x_0,\delta}, L_{x_0,\delta})$ where $L$ was defined in the statement of \prpref{model}, through \prpref{tens} and the isotropy of $L_\perp$. 
			
			In the isotropic noise model, one can write
			$$
			\bar f_{x_0,\delta}(v,\eta) = \delta^{D-d}\int_{M} \delta^{-D}K\left(\frac{\bar\Psi_{x_0,\delta}(v,\eta)-x}{\delta}\right)f_0(x) \diff \mu(x). 
			$$
			If $\| \eta \| \geq 1$, the integrand is trivially $0$, so one may focus on $\|\eta\| \leq 1$. Now if $x \in M$ is at least $2\delta$ apart from $\Psi_{x_0}(v)$, there holds
			$$
			\|x-\bar\Psi_{x_0,\delta}(v,\eta) \| \geq \|x - \Psi_{x_0}(v)\| - \delta  \geq \delta
			$$
			so that the integrand in the integral above is zero for $x$ outside of $B(\Psi_{x_0}(v),2\delta)$. Doing the variable change $x = \Psi_{x_0}(w)$, we can write
			\begin{align*} 
				\bar f_{x_0,\delta}(v,\eta) &= \frac1{\delta^{d}}\int_{\exp_{x_0}^{-1}\ball(\Psi_{x_0}(v),2\delta)} K\left(\frac{\bar\Psi_{x_0,\delta}(v,\eta)-\Psi_{x_0}(w)}{\delta}\right)f_{x_0}(w) |\det \diff \Psi_{x_0}(w)| \diff w \\
				&= \int_{\cZ_\delta} \underbrace{K\left(\frac{\Psi_{x_0}(v)-\Psi_{x_0}(v+\delta s)}{\delta} + N_{x_0}(v,\eta)\right)}_{=: K\circ g_{\delta,s}(v,\eta)} \underbrace{f_{x_0}(v + \delta s) |\det \diff \Psi_{x_0}(v + \delta s)|}_{=: h_{\delta,s}(v)} \diff s \\
			\end{align*} 
			where 
			$$\cZ_\delta = \frac{1}{\delta}\{\exp_{x_0}^{-1}\ball(\Psi_{x_0}(v),2\delta) - v\}.$$
			Now notice that, using the first inequality of \eqref{lem:useful}, we have $\ball(\Psi_{x_0}(v),2\delta) \subset \ball (x_0,\tau/8+2\delta) \subset \ball(x_0,3\tau/16)$, so that, using the second inequality of \eqref{lem:useful} with $\gamma = 3/8$, we find
			$$
			\exp_{x_0}^{-1}\ball(\Psi_{x_0}(v),2\delta) \subset \ball_{T_{x_0}M}(0, \kappa(3/8)3\tau/16) \subset  \ball_{T_{x_0}M}(0, \tau/4) 
			$$
			so that in particular, according to \eqref{lem:doublelip}, $\exp_{x_0}^{-1}$ is $16/11$-Lipschitz on $\ball(\Psi_{x_0}(v),2\delta)$ and consequently, $\cZ_\delta \subset \ball_{T_{x_0} M}(0,32/11)$. Furthermore, notice that
			$$v + \delta\ball_{T_{x_0} M}(0,32/11) \subset \ball_{T_{x_0} M}(0,19\tau/88)
			$$
			with $19\tau/88$ being smaller than the injectivity radius, so that finally
			\beq \label{kgh}
			\bar f_{x_0,\delta}(v,\eta) =  \int_{ \ball_{T_{x_0} M}(0,32/11)}  K\circ g_{\delta,s}(v,\eta) \times h_{\delta,s}(v) \diff s
			\eeq
			It is straighforward to see that 
			$$ v\mapsto \frac{\Psi_{x_0}(v)-\Psi_{x_0}(v+\delta s)}{\delta}
			$$
			has derivatives bounded from above, in virtue of \prpref{der2}, by something depending on $C_M$ and $\beta_M$ only (and not on $\delta$) and up to order $\ceil{\beta_M-1}\geq \beta_\perp$. Furthermore, $(v,\eta) \mapsto N_{x_0}(v,\eta)$ is $(\beta_M - 1)$-H\"older by construction. It so happens that $g_{\delta,s} \in \cH_{\iso}^{\beta_\perp}(\cW_{x_0,\delta}, C_g)$ for some $g$ depending on $C_M$ and $\beta_M$. Using \prpref{comp}, we get that $K \circ g_{\delta,s} \in \cH^{\beta_\perp}_{\iso}(\cW_{x_0,\delta},C_1 L_\perp \circ g_{\delta,s})$ for some constant $C_1$ depending on $C_g$, $C_K$ and $\beta_\perp$. Similar reasoning and the use and \prpref{mult} would lead to $h_{\delta,s} \in  \cH^{\beta_0}_{\iso}(\cV_{x_0}, C_2 L_0\circ \Psi_{x_0} \circ \tau_{\delta s})$ where $\tau_{\delta s}(v) = v +\delta s$ and for some constant $C_2$ depending on $C_M$, $\tau$ and $\beta_0$. Now seeing $h_{\delta,s}$ as a function of $(v,\eta)$ would trivially lead to $h_{\delta,s} \in  \cH^{\bm\beta}_{\an}(\cW_{x_0,\delta}, C_2 L_0 \circ \Psi_{x_0} \circ \tau_{\delta s} \circ \pi_1)$ with $\bm\beta = (\beta_0,\dots,\beta_0,\beta_\perp,\dots,\beta_\perp)$ and $\pi_1(v,\eta) = v$, and an application of \prpref{mult} together with the observation that $\beta_0 \leq \beta_\perp$ yields 
			$$K\circ g_{\delta,s} \times h_{\delta,s} \in  \cH^{\bm\beta}_{\an}(\cW_{x_0,\delta}, C L_\perp \circ g_{\delta,s} \times L_0 \circ \Psi_{x_0} \circ \tau_{\delta s} \circ \pi_1)$$
			for some constant $C$ depending on $C_2$, $C_1$ and $\beta_0,\beta_\perp$. We conclude by using the linearity of the integral, so that integrating \eqref{kgh} gives $\bar f_{x_0,\delta} \in \cH^{\bm\beta}_{\an}\left(\cW_{x_0,\delta}, \wt L\right)$, where
				\begin{align*}
				\wt L(v,\eta) &= C  \int_{ \ball_{T_{x_0} M}(0,32/11)}  L_\perp \circ g_{\delta,s}(v,s) \times L_0 \circ \Psi_{x_0}(v+\delta s)\diff s \\
				&\lesssim  \int_{ \ball_{T_{x_0} M}(0,32/11)}  L_\perp \circ g_{\delta,s}(v,s) \times L_0 \circ \Psi_{x_0}(v+\delta s) \times |\det \diff \Psi_{x_0}(v+\delta s) | \diff s  \\
				&=  \delta^{-d} \int_{M}  L_\perp\left(\frac{\bar \Psi_{x_0,\delta}(v,\eta) - x}{\delta}\right) \times L_0(x) \diff \mu_M(x) 
				\end{align*}
	so that 
	\begin{align*}
				 L(y) &= C\delta^{-D}\int_{M}  L_\perp\left(\frac{y - x}{\delta}\right) \times L_0(x) \diff \mu_M(x) 
				\end{align*}
			where $L$ was defined in the statement of \prpref{model}, and where we used \lemref{defNB}  to justify the introduction of $|\det \diff \Psi_{x_0}(v+\delta s)|$. This ends the proof.
		\end{proof}

			\begin{proof}[Proof of \lemref{moments}]

		 In the orthogonal noise model, recall (see proof of \prpref{model}) that 
			$$
			\begin{cases}
				f_0(x) = f_*(\pr_M(x)) \times c_\perp \delta^{-(D-d)} K((x-\pr_M x)/\delta) \\
				L(x) = C  \delta^{-(D-d)} L_0(\pr_M(x)) \times  L_\perp((x-\pr_M x)/\delta),
			\end{cases}
			$$
			for some $C > 0$. There thus holds, letting $\cO_{x_0} = \bar\Psi_{x_0,1}(\cW_{x_0,1})$, and using Cauchy-Schwartz inequality
			\begin{align*} 
				&\{\int_{\cO_{x_0}} \{\frac{L(x)}{f_0(x)}\}^{\omega^*/2} f_0(x) \diff x\}^2 \\
				&\lesssim \int_{\cO_{x_0}} \{\frac{L_0(\pr_M x)}{f_*(\pr_M x)}\}^{\omega^*} f_0(x) \diff x  \int_{\cO_{x_0}} \{\frac{L_\perp ((x- \pr_M x)/\delta)}{K((x-\pr_M x)/\delta)}\}^{\omega^*} f_0(x) \diff x.
			\end{align*} 
			The first integral in the RHS is simply, up to the constant $c_\perp$, 
			$$
			\int_{\Psi_{x_0}(\cV_{x_0})} \{\frac{L_0(z)}{f_*(z)}\}^{\omega^*} f_*(z) \diff\mu_M(z),
			$$
			which is bounded by assumption. The second integral is exactly
			$$
			\int_{\Psi_{x_0}(\cV_{x_0})} f_*(z) \diff\mu_M(z) \times \int_{\ball_{\bbR^{D-d}}(0,1)} \{\frac{L_\perp (\eta)}{K(\eta)}\}^{\omega^*} K(\eta) \diff x.
			$$
			which is also bounded by assumption. Hence \eqref{assint} holds true in this case.
			
			In the isotropic noise model, there holds this time
			$$
			\begin{cases} \displaystyle
				f_0(x) = \delta^{-D} \int_{M \cap \ball(x,\delta)}  K((y-x)/\delta) f_*(y) \diff \mu_M(y) \\ \displaystyle
				L(x) = C \delta^{-D} \int_{M \cap \ball(x,\delta)}  L_\perp ((y-x)/\delta) L_0(y) \diff \mu_M(y),
			\end{cases}
			$$
			for some $C > 0$. Using the fact that for any integrable function $\phi,\psi \geq 0$ and any $\omega \geq 1$, there holds
			$$
			\{\int \phi\}^\omega \leq \{\int \frac{\phi^\omega}{\psi^{\omega-1}}\} \times \{\int \psi\}^{\omega-1},
			$$
			(which is simple consequence of the H\"older inequality), one find that
			\begin{align*} 
				&\int_{\cO_{x_0}} \{\frac{L(x)}{f_0(x)}\}^{\omega^*/2} f_0(x) \diff x \\
				&\lesssim \int_{\cO_{x_0}}\int_{M} \{\frac{ L_\perp ((y-x)/\delta) L_0(y)}{K((y-x)/\delta) f_*(y)}\}^{\omega^*/2} \delta^{-D}  K((y-x)/\delta) f_*(y) \diff \mu_M(y) \diff x \\
				&\leq \delta^{-D} I_1^{1/2} I_2^{1/2},
			\end{align*} 
			where
			\begin{align*} 
				I_1 &:= \int_{\cO_{x_0}} \int_{M} \{\frac{L_0(y)}{f_*(y)}\}^{\omega^*} K((y-x)/\delta) f_*(y) \diff \mu_M(y) \diff x \\
				&\lesssim \delta^D \int_{M } \{\frac{L_0(y)}{f_*(y)}\}^{\omega^*} f_*(y) \diff \mu_M(y) \lesssim \delta^D	,\end{align*} 
			by assumption, and where
			\begin{align*} 
				I_2 &:=  \int_{\cO_{x_0}}\int_{M} \{\frac{ L_\perp ((y-x)/\delta)}{K((y-x)/\delta)}\}^{\omega^*}  K((y-x)/\delta) f_*(y) \diff \mu_M(y) \diff x \\
				&\lesssim \delta^D \int_M f \diff \mu_{M} \int_{\ball(0,1)}\{\frac{ L_\perp (\eta)}{K(\eta)}\}^{\omega^*}  K(\eta)  \diff \eta,
			\end{align*} 
			so that indeed $ \delta^{-D} I_1^{1/2} I_2^{1/2} \lesssim 1$, which ends the proof.
			\end{proof}

		\subsection{Proof of \prpref{hybridDP}} \label{app:prior}
		
			If $Q_2 =Q_0^{\otimes D} $ with  $Q_0\sim \DP( B H_\lambda):= \Pi_{\DP}$ then, for $0<x_1\leq x_2/2$ and $x_1 < 1$, letting $A_1 =   [x_1,x_1(1+x_1^b)] $, $A_2  =  [x_2,x_2(1+x_1^b)] $ and $A_3 = (A_1 \cup A_2)^c$, we have
			$$(Q_0(A_1), Q_0(A_2), Q_0(A_3) ) \sim \mathcal D( B H_\lambda(A_1),  B H_\lambda(A_2), B H_\lambda(A_3)),$$  
			and 
			$$Q_2(A_1^d \times A_2^{D-d}) = Q_0(A_1)^dQ_0(A_2)^{D-d}.$$
			We let $\gamma =  \Gamma (B) / \(\Gamma( BH_\lambda(A_1))\Gamma( BH_\lambda(A_2))\Gamma( BH_\lambda(A_3)\)$. Noticing that when $x_1, x_2$ are small, $ H_\lambda(A_3) =1-o(1)$, there holds, \newclem{letting $i=1$ if $x_1 \leq x_2$ and $i=2$ otherwise,
			\begin{align*}
				\tilde\Pi_\Lambda &\big[Q_2\left(A_1^d A_2^{D-d} \right) \geq x_1^{B_0} \big] \geq 
				\Pi_{DP}\( Q_0(A_1)> x_1^{B_0/D}, Q_0(A_2)> x_1^{B_0/D} \)\\
				&\geq \frac{ \gamma }{2^{(B-1)_+}} \int_{ x_1^{B_0/D}}^{1/4}x^{BH_\lambda(A_1)-1}dx\int_{ x_1^{B_0/D}}^{1/4}x^{BH_\lambda(A_2)-1}dx\\
				& \geq \frac{ \Gamma (B) }{ (1 + o(1))2^{(B-1)_+}}\left[ 4^{-BH_\lambda(A_1)} - x_1^{\frac{BH_\lambda(A_1)B_0}{D}} \right]\left[ 4^{-BH_\lambda(A_2)} - x_1^{\frac{BH_\lambda(A_2)B_0}{D}} \right]\\
				&\geq  C_1 \left[1 - C_2 e^{- {\frac{BH_\lambda(A_i)B_0}{D}}}\right]^2,
			\end{align*}
			for some constant $C_1,C_2 > 0$, where we used the fact that for $x_1$ and $x_2$ small, $H_\lambda(A_i) = \min\{H_\lambda(A_1), H_\lambda(A_2)\}$ and $H_\lambda(A_1)$ and $H_\lambda(A_2)$ are both small. 
			Under the  assumption on $H_\lambda$ near $0$, 
			$$H_\lambda(A_i) \gtrsim e^{- c_2 x_i^{-1/4}} x_i^{1+b} \gtrsim e^{- c'_2 x_i^{-1/4}} $$ 
			for some $c'_2 > 0$, which in turn implies 
			$$ 
			\tilde\Pi_\Lambda \big[Q_2\left(A_1^d A_2^{D-d} \right) \big]  \gtrsim e^{- c'_2 x_i^{-1/4}} \geq e^{- c'_2 x_1^{-d/2}}
			$$
			where we used that $x_i^{-1/4} \leq x_1^{-1/2} \leq x_i^{-d/2}$, so that \eqref{cond:piH:LB} holds.} Condition \eqref{cond:piH:UB} is verified similarly: first note that 
			$ Q_2\left(\min_{i\leq D} \lambda_i \leq x \right) \leq D Q_0((0,x]) $ and using the fact that $Q_0((0,x])$ follows a Beta variable with parameters $(BH_\lambda((0,x]), B(1 +o(1)))$ we obtain that 
			$$ \bbE_{\tilde \Pi_\Lambda}\left[Q_2\left(\min_{i\leq D} \lambda_i \leq x \right) \right] \lesssim H_\lambda((0,x]) \lesssim e^{- c_3 x^{-b_3}}, 
			$$ 
			for small $x$. Similar computation terminates the proof of \eqref{cond:piH:UB}. 
			

		\section{Appendix to Section \ref{app:thmapprox1}: proof and framework of \thmref{approx}} \label{app:theoretical}

		\subsection{Technical Lemmata} \label{app:thmapprox2}
		We let $\chi_1,\dots,\chi_J$ be the partition of unity defined in Section \ref{sec:partition} associated with a $\tau/64$-packing of $\{x_1,\dots,x_J\}$ of $M \cap \ball(0,R)$. We recall that $J \lesssim R^D$ and we take $R = (H \log(1/\sigma))^{1/\kappa}$. We define
		$$
		f_j := c_j^{-1} \chi_j \times f_0~~~~\text{where}~~~c_j = \int_{\bbR^D} \chi_j(x) f_0(x) \diff x.
		$$
		Before we give the proof of \corref{approx} we need a few technical Lemmata. 
		
		\begin{lem}
			For all $ 1\leq j \leq J$, $f_j \in \cH^{\beta_0,\beta_\bot}_\delta(M,L_j)$ with $L_j(x) := C c_j^{-1} I_K(x-x_j) \chi_j(x) L(x)$ where $C > 0$ is a constant depending on $\tau$ and where we took $K = \ceil{\beta_\perp}$.
		\end{lem}
		\begin{proof} The proof is a direct consequence of  \prpref{mult} and \prpref{comp}  applied to $\chi_j(x)$, along with the assertion iv) from \lemref{chixj}.
		\end{proof}
		We let $L_{j,\delta} := \delta^{D-d} L_j \circ \bar \Psi_{x_j,\delta}$ and
		$$
		\bar f_{j,\delta} := f_j \circ \bar \Psi_{x_j,\delta} : \cW_{x_j,\delta} \to \bbR^D.
		$$
		and also introduce the functions $g_j : \bbR^D \to \bbR$ from the proof of \thmref{approx}. Recall that these functions are of the form
		$$
		g_j(x) := f_j(x) + \frac{1}{\delta^{D-d}}\sum_{0 < \inner{k}{\bm\alpha} < \beta} \sigma^{\inner{k}{\bm\alpha}} d_{j,k}(x,\sigma,\delta) \Diff^{k} \bar f_{j,\delta}(z_x)~~~~\text{where}~~z_x :=\Delta_{1,\delta}^{-1} \bar\Psi_{x_j}^{-1}(x) 
		$$
		and satisfy
		\benum
		\item They are supported on $\cO_{x_j}^0$;
		\item The functions $d_{j,k}$ are uniformly bounded by a constant $C$ depending on $C_M$;
		\item $|K_\Sigma g_j(x) - f_j(x)| \lesssim \sigma^{\beta} L_j(x)$ on $\cO_{x_j}^1$;
		\item For all $H>0$, $|K_\Sigma g_j(x) - f_j(x)| \lesssim \sigma^{H} \|L_j\|_\infty$ outside of $\cO_{x_j}^1$.
		\eenum 
		
		\begin{lem} \label{lem:gkgb}
			Under \eqref{assint}, there holds, for any $\ve \in (0,\omega-2\beta)$, any $\inner{k}{\bm\alpha} < \beta$ and any $1 \leq j \leq J$,
			\begin{align*} 
				\int \left( \frac{|\Diff^k \bar f_{j,\delta} (z)|}{\bar f_{j,\delta} (z)} \right)^{\frac{2\beta+\ve}{\inner{k}{\bm\alpha}}} \bar f_{j,\delta} (z)\diff z \lesssim  c_j^{-1}~~~\text{and}~~~\int \left( \frac{L_{j,\delta}(z)}{\bar f_{j,\delta} (z)} \right)^{\frac{2\beta+\ve}{\beta}} \bar f_{j,\delta} (z)\diff z \lesssim c_j^{-1},
			\end{align*} 
			up to a constant depending on the parameters.
		\end{lem}
		\begin{proof} We denote by $\bar \chi_{j,\delta} = \chi_j \circ \bar\Psi_{x_j,\delta}$ and $\bar I_{j,\delta} := I_K(\bar\Psi_{x_j,\delta}(\cdot) - x_j)$. Writing that
			$$
			\left | \frac{\Diff^k \bar f_{j,\delta}}{\bar f_{j,\delta}} \right | \leq \sum_{\ell\leq k} {k \choose \ell} \left| \frac{\Diff^{k-\ell}\bar \chi_{j,\delta}}{\bar \chi_{j,\delta}} \right|\times \left| \frac{\Diff^{\ell}\bar f_{x_j,\delta}}{\bar f_{x_j,\delta}}\right| \lesssim \bar I_{j,\delta} \sum_{\ell  \leq k} \left| \frac{\Diff^{\ell}\bar f_{x_j,\delta}}{\bar f_{x_j,\delta}}\right| 
			$$
			we easily see that
			\begin{align*} 
				\int \left| \frac{\Diff^k \bar f_{j,\delta}}{\bar f_{j,\delta}} \right|^{\frac{2\beta+\ve}{\inner{k}{\bm\alpha}}} \bar f_{j,\delta} &\lesssim c_j^{-1} \sum_{\ell  \leq k} \int \{\bar I_{j,\delta} \left| \frac{\Diff^{\ell}\bar f_{x_j,\delta}}{\bar f_{x_j,\delta}}\right| \}^{\frac{2\beta+\ve}{\inner{k}{\bm\alpha}}} \bar \chi_{j,\delta} \bar f_{x_j,\delta} \\
				&\lesssim c_j^{-1}  \sum_{\ell \leq k}  \int \left| \frac{\Diff^{\ell}\bar f_{x_j,\delta}}{\bar f_{x_j,\delta}}\right| ^{\frac{2\beta+\ve}{\inner{k}{\bm\alpha}}} \bar f_{x_j,\delta} \lesssim c_j^{-1}
			\end{align*} 
			where we used both \lemref{chixj} iv) and \eqref{assint} with the fact that $\inner{\ell}{\bm\alpha} \leq \inner{k}{\bm\alpha}$ for $\ell \leq k$. The bound on the second integral follows from the same line of reasoning.
		\end{proof}
		
		\begin{lem} \label{lem:ajs} We let $\cA_{j,\sigma}$ be the event of all $x \in \cO_{x_j}^2$ such that 
			\begin{align*} 
				\forall 0 < \inner{k}{\bm\alpha} < \beta, \frac{|\Diff^k \bar f_{j,\delta}(z_x)|}{\bar f_{j,\delta}(z_x)} \leq \sigma^{-\< k,\alpha \>} \log(1/\sigma)^{-\< k,\alpha \>/2} ~~\text{and}~~
				\frac{L_{j,\delta} (z_x)}{\bar f_{j,\delta} (z_x)} \leq \sigma^{-\beta} \log(1/\sigma)^{-\beta/2}.
			\end{align*} 
			Then, the following assertions hold true for $\sigma$ small enough 
			\bitem
			\item[i)] $g_j(x) \geq f_j(x)/2$ for all $x \in \cA_{j,\sigma}$;
			\item[ii)] $P_{f_j}(\cA_{j,\sigma}^c) \lesssim   c_j^{-1} \{\sigma \log(1/\sigma)\}^{2\beta+\ve}$;
			\item[iii)] For any $\inner{k}{\bm\alpha} <\beta$,
			$
			\int_{\cA_{j,\sigma}^c} |\Diff^{k} \bar f_{j,\delta}(z_x)| \diff x \lesssim  c_j^{-1}  \delta^{D-d} \{\sigma \log(1/\sigma)\}^{2\beta+\ve - \inner{k}{\bm\alpha}}.
			$
			\eitem
		\end{lem}

		\begin{proof}
			We start by proving i). For any $x \in \cA_{j,\sigma}$, there holds, using the fact that the functions $d_{j,k}$ are uniformly bounded,
			\begin{align*}
				\left| g_j(x) - f_j(x) \right| &=  \left| \frac{1}{\delta^{D-d}} \sum_{0 < \<k,\alpha \> < \beta} \sigma^{\<k,\alpha \>}d_k(y,\sigma,\delta) \Diff^k \bar f_{j,\delta} (z_y) \right|  
				\\
				&\lesssim \frac{ \bar f_{j,\delta}(z_x)}{\delta^{D-d}} \sum_{0 < \<k,\alpha \> < \beta} \sigma^{\<k,\alpha \>}  \left|\frac{\Diff^k \bar f_{j,\delta} (z_x)}{\bar f_{j,\delta} (z_x)} \right| \\
				&\lesssim  f_j(x) \sum_{0 < \<k,\alpha \> < \beta}  \log(1/\sigma)^{-\< k,\alpha \>/2} \lesssim \log^{-\alpha_\perp/2}(1/\sigma) f_j(x) < \frac12 f_j(x)  
			\end{align*}
			provided that $\sigma$ is chosen small enough. For ii), notice that	
			\begin{align*} 
				P_{f_j}(\cA_{j,\sigma}^c) &\leq \sum_{0 < \inner{k}{\bm\alpha} < \beta} P_{f_j} \left(\frac{|\Diff^k \bar f_{j,\delta}(z_x)|}{\bar f_{j,\delta}(z_x)} > \sigma^{-\< k,\alpha \>} \log(1/\sigma)^{-\< k,\alpha \>/2} \right)  \\
				&\leq \sum_{0 < \inner{k}{\bm\alpha} < \beta}  \sigma^{2\beta+\ve} \log(1/\sigma)^{2\beta+\ve}  \int_{\cO_{x_j}^0} \left(\frac{|\Diff^k \bar f_{j,\delta}(z_x)|}{\bar f_{j,\delta}(z_x)}\right)^{\frac{2\beta+\ve}{\inner{k}{\bm\alpha}}} f_j(x) \diff x \\
				&\lesssim  c_j^{-1} \sigma^{2\beta+\ve} \log(1/\sigma)^{2\beta+\ve} 
			\end{align*}
			where we used \lemref{gkgb} after a variable change $z = z_x$. Finally, for iii), there holds
			\begin{align*} 
				\int_{\cA_{j,\sigma}^c} &|\Diff^{k} \bar f_{j,\delta}(z_x)| \diff x = \int \frac{|\Diff^{k} \bar f_{j,\delta}(z_x)|}{\bar f_{j,\delta}(z_x)} \bar f_{j,\delta}(z_x) \ind_{\cA_{j,\sigma}^c}(x) \diff x \\
				&\leq \{\int \{\frac{|\Diff^{k} \bar f_{j,\delta}(z_x)|}{\bar f_{j,\delta}(z_x)}\}^{\frac{2\beta+\ve}{\inner{k}{\bm\alpha}}} \bar f_{j,\delta}(z_x) \ind_{\cA_{j,\sigma}^c}(x) \diff x\}^{\frac{\inner{k}{\bm\alpha}}{2\beta+\ve}} \times \{ \delta^{D-d} P_{f_j}(\cA_{j,\sigma}^c)\}^{1 - \frac{\inner{k}{\bm\alpha}}{2\beta+\ve} } \\
				&\lesssim   \{c_j^{-1}\delta^{D-d}\}^{\frac{\inner{k}{\bm\alpha}}{2\beta+\ve}} \times 
				\{c_j^{-1} \delta^{D-d} \{\sigma \log(1/\sigma)\}^{2\beta+\ve} \}^{1 - \frac{\inner{k}{\bm\alpha}}{2\beta+\ve}} \\
				&\lesssim  c_j^{-1}\delta^{D-d} \{\sigma \log(1/\sigma)\}^{2\beta+\ve-\inner{k}{\bm\alpha}}.
			\end{align*} 
			ending the proof.
		\end{proof}

	\subsection{Proof of \corref{approx}}	
		
			Let define
			$$
			\wt h_j := g_j \ind_{g_j > f_j/2} + \frac12 f_j \ind_{g_j \leq f_j/2}
			$$
			and $h_j = \wt h_j / I_j$ where $I_j = \int \wt h_j$. The function $h_j$ is a probability measure on $\cO_{x_j}^0$ and furthermore, using the convexity of $f \rightarrow \dhed(f,f_0)$, there holds
			$$
			\dhed(K_\Sigma(h),f_0) \leq \sum_{j=1}^J c_j \dhed(K_\Sigma(h_j),f_j)~~~\text{where}~~h = \sum_{j=1}^J c_j h_j.
			$$
			We will control each term separately. First notice that
			$$
			\sum_{j=1}^J c_j \dhed(K_\Sigma(h_j),f_j) \leq J \sigma^{2\beta} + \sum_{c_j \geq \sigma^{2\beta}} c_j \dhed(K_\Sigma(h_j),f_j).
			$$ 
			Now take $1 \leq j \leq J$ such that $c_j \geq \sigma^{2\beta}$ and define $\cU_{j,\sigma} = \{x \in \bbR^D~|~ f_j (x) \geq \sigma^{H_1}\}$ for some $H_1 > 0$ to be specified later.
			Then there holds
			\begin{align*}
				\dhed (K_\Sigma h_j, f_j) \leq & \int_{\cU_{j,\sigma} \cap \cO_{x_j}^1} \left(\sqrt{K_\Sigma(h_j)} - \sqrt{f_j}\right)^2 + \int_{(\cO_{x_j}^1)^c} K_\Sigma(h_j) + \int_{\cU_{j,\sigma}^c}[K_\Sigma(h_j) + f_j]
			\end{align*}
			and 
			\begin{align*}
				\int_{\cU_{j,\sigma} \cap \cO_{x_j}^1} &\left(\sqrt{K_\Sigma(h_j)} - \sqrt{f_j}\right)^2 \leq  \int_{\cU_{j,\sigma} \cap \cO_{x_j}^1} \frac{(K_\Sigma(h_j) - f_j)^2}{K_\Sigma(h_j) + f_j} \\
				&\leq \int_{\cU_{j,\sigma} \cap \cO_{x_j}^1} \frac{K_\Sigma(h_j - \tilde{h}_j)^2}{K_\Sigma(h_j) + f_j}  + \int_{\cU_{j,\sigma} \cap \cO_{x_j}^1} \frac{K_\Sigma(\tilde{h}_j - g_j)^2}{K_\Sigma(h_j) + f_j} + \int_{\cU_{j,\sigma} \cap \cO_{x_j}^1} \frac{(K_\Sigma(g_j)-f_j)^2}{K_\Sigma(h_j) + f_j}
			\end{align*}
			and therefore
			\begin{align}
				\dhed (f_j,K_\Sigma(h_j)) \leq&  \int_{(\cO_{x_j}^1)^c} K_\Sigma(h_j)\label{int:1} \\
				&+ \int_{\cU_{j,\sigma}^c}[K_\Sigma(h_j) + f_j]\label{int:2} \\
				&+ (1-I_j)^2\label{int:3} \\
				&+ \int_{\cU_{j,\sigma} \cap \cO_{x_j}^1} \frac{K_\Sigma\left((f_j/2 - g_j)\ind_{g_j < f_j/2}\right)^2}{K_\Sigma(h_j) + f_j}\label{int:4} \\
				&+ \int_{\cU_{j,\sigma} \cap \cO_{x_j}^1} \left(\frac{K_\Sigma(g_j)-f_j}{f_j}\right)^2 f_j\label{int:5}
			\end{align}
			where each term will be bounded independently.
			\begin{enumerate}
				\item For (\ref{int:1}), notice that for any $y  \in \cO_{x_j}^0$ and any $x \notin \cO_{x_j}^1$,  there holds that $ \|x-y\|_{\Sigma_y^{-1}}^2 \geq C/\rho^2(\delta, \sigma($ where $\rho(\delta, \sigma) = \max( \delta \sigma^{\alpha_\perp}, \sigma^{\alpha_0})$ for some $C$ depending on $\tau$. We can thus write
				\begin{align*}
					\int_{(\cO_{x_j}^1)^c} K_\Sigma(h_j) &=  \int_{(O_{x_j}^1)^c} \int h_j(y) \varphi_{\Sigma(y)}(x-y)\diff y \diff x \\
					&=  \int_{(\cO_{x_j}^1)^c} \int_{\cO_{x_j}^0} h_j(y) \frac{\exp\left(- \frac{1}{2} \|x-y\|_{\Sigma^{-1}(y)}^2\right)}{(2\pi)^{D/2}\sigma^D \delta^{D-d}}\diff y \diff x \\
					&\leq  \exp\{- \frac{C^2}{4\delta^2 \sigma^{2\alpha_\perp}}\} \int_{(\cO_{x_j}^1)^c} \int_{\cO_{x_j}^0} h_j(y) \frac{\exp\left(- \frac{1}{4} \|x-y\|_{\Sigma^{-1}(y)}^2\right)}{(2\pi)^{D/2}\sigma^D \delta^{D-d}}\diff y \diff x \\
					&\lesssim  \exp\{- \frac{C^2}{4\rho(\delta, \sigma)^2 }\} \int_{(O_{x_j}^1)^c} K_{2\Sigma}(h_j) \lesssim  \exp\{- \frac{C^2}{4\rho(\delta, \sigma)^2}\} \lesssim \sigma^{2\beta},
				\end{align*}
				because $\rho(\delta,\sigma) = o(|\log(\sigma)|^{1/2})$ by assumption.
				\item For (\ref{int:5}), there holds
				\begin{align*}
					\int_{\cU_{j,\sigma} \cap \cO_{x_j}^1} \left(\frac{K_\Sigma(g_j)-f_j}{f_j}\right)^2 f_j &\leq  \int_{\cO_{x_j}^1} \left(\frac{L_j \sigma^\beta}{f_j}\right)^2 f_j \leq  \sigma^{2\beta} \{\int_{\cO_{x_j}^1} \left(\frac{L_j}{f_j}\right)^{\frac{2\beta+\ve}{\beta}} f_j \}^{\frac{2\beta}{2\beta+\ve}} \\
					&\leq  \sigma^{2\beta}\{\int_{\cW_{x_j}^1} \left(\frac{L_{j,\delta}}{\bar f_{j,\delta}}\right)^{\frac{2\beta+\ve}{\beta}} \bar f_{j,\delta} \frac{|\det \diff \bar{\Psi}_{x_j,\delta}|}{\delta^{D-d}}\}^{\frac{2\beta}{2\beta+\ve}} \\
					&\lesssim  \sigma^{2\beta}\{\int_{\cW_{x_j}^1} \left(\frac{L_{j,\delta}}{\bar f_{j,\delta}}\right)^{\frac{2\beta+\ve}{\beta}} \bar f_{j,\delta} \}^{\frac{2\beta}{2\beta+\ve}}  \lesssim c_j^{-\frac{2\beta}{2\beta+\ve}} \sigma^{2\beta} \lesssim c_j^{-1}\sigma^{2\beta},
				\end{align*}
				where we used that $|\det \diff \bar\Psi_{x_j,\delta}| \lesssim \delta^{D-d}$ and \lemref{gkgb}.
				\item For (\ref{int:3}), notice that
				$$
				I_j = \int g_j \1_{\{f_j < 2g_j \}} + \int \frac{f_j}{2} \1_{\{f_j \geq 2g_j \}} = \int g_j + \int \frac{(f_j-2g_j)}{2} \1_{\{f_j \geq 2g_j \}} 
				$$
				and that $\int g_j = \int K_\Sigma(g_j) $. Moreover
				\[
				\int_{\cO_{x_j}^1} K_\Sigma(g_j) = \int_{\cO_{x_j}^1} f_j + \int_{\cO_{x_j}^1} (K_\Sigma(g_j) - f_j) = 1 + \int_{\cO_{x_j}^1} (K_\Sigma(g_j) - f_j)
				\]
				and since 
				\begin{align*}
					\left|\int_{\cO_{x_j}^1} K_\Sigma(g_j) - f_j \right| &\leq  \int_{\cO_{x_j}^1} L_j \sigma^\beta = \sigma^{\beta} \int_{\cW_{x_j}^1} L_{j,\delta} \times \frac{|\det \diff \bar{\Psi}_{x_j,\delta}|}{\delta^{D-d}} \\
					&\lesssim  \sigma^\beta  \{\int_{\cW_{x_j}^1} \left(\frac{L_{j,\delta}}{\bar f_{j,\delta}}\right)^{\frac{2\beta+\ve}{\beta}} \bar f_{j,\delta}\}^{\frac{\beta}{2\beta+\ve}}  \lesssim c_j^{-\frac{\beta}{2\beta+\ve}}\sigma^{\beta} \lesssim c_j^{-1/2} \sigma^{\beta}.
				\end{align*}
				where we again used that $|\det \diff \bar\Psi_{x_j,\delta}| \lesssim \delta^{D-d}$ and \lemref{gkgb},  we have that 
				\[ \int_{\cO_{x_j}^1} K_\Sigma(g_j) = 1 + O( c_j^{-1/2} \sigma^{\beta}).\]
				We also have, as in (\ref{int:1}) that 
				\begin{align*}
					\int_{(\cO_{x_j}^1)^c} K_\Sigma(|g_j|) &\lesssim e^{-\frac{ C^2  }{ 4 \delta^2 \sigma^{2 \alpha_\perp }}}  	\int_{\cO_{x_j}^1} |g_j|(y)dy \\
					& \lesssim  e^{-\frac{ C^2  }{ 4 \delta^2 \sigma^{2 \alpha_\perp }}}  \|L_j\|_\infty = o( c_j^{-1}  \sigma^{2\beta}) = o( c_j^{-1/2}  \sigma^{\beta}),
				\end{align*}
				where the two last inequality comes from $\rho^2(\delta, \sigma)= o(1/\log (1/\sigma))$ and $c_j \geq \sigma^{2\beta}$.  Therefore $|1 - \int g_j |  \lesssim c_j^{-1/2}  \sigma^\beta$. Moreover, using this time \lemref{ajs}
				\begin{align*}
					\int_{f_j > 2g_j}( f_j-2g_j) & \lesssim  P_{f_j}(f_j > 2g_j)  + \sum_{0<\inner{k}{\bm\alpha}<\beta} \frac{ \sigma^{\inner{k}{\bm\alpha}} }{\delta^{D-d}} \int_{f_j > 2g_j} |\Diff^k \bar f_{j,\delta} (z_x)|dx \\
					&\lesssim  c_j^{-1}\sigma^{2 \beta + \varepsilon}.
				\end{align*} 
				All in all, we find that
				\begin{align*} 
					(1-I_j)^2 &\lesssim \exp\{-C^2/(4\delta^2 \sigma^{2\alpha_\perp})\} +  \{c_{j}^{-1} \sigma^{2\beta} \}\vee \{c_{j}^{-2} \{\sigma \log(1/\sigma)\}^{4\beta+2\ve}\} \\
					&\lesssim \exp\{-C^2/(4\delta^2 \sigma^{2\alpha_\perp})\} +  c_{j}^{-1} \sigma^{2\beta} 
				\end{align*} 
				where the last inequality again holds because $c_j \geq \sigma^{2\beta}$.
				\item For (\ref{int:2}), we start with taking $\inner{k}{\bm\alpha} < \beta$. Notice that, thanks to \prpref{der2}, there exists a constant $C > 0$ such that for any $x,y \in \cO_{x_j}^0$, $|\Diff^k \bar f_{j,\delta}(z_x)-\Diff^k \bar f_{j,\delta}(z_y)|\leq C \delta^{D-d} L_j(x)$. 
				Then, using \eqref{Ksigma:decomp1}, together with $|\det \diff\bar \Psi(w)|\leq C$, 
				\begin{align*} 
					| K_\Sigma(\Diff^k \bar f_{j,\delta}(z_{\square}))(x)| &\lesssim 
					\int_{\Delta^{-1}_{\sigma,1}(\cW^0-z_x)} e^{-B_\sigma(x,z) } |\Diff^k \bar f_{j,\delta}(\Delta_{\sigma,1}z + z_x)| dz \\
					&\lesssim \delta^{D-d} L_j(x) \int_{\Delta^{-1}_{\sigma,1}(\cW^0-z_x)} e^{-B_\sigma(x,z) } dz \lesssim \delta^{D-d} L_j(x) .
				\end{align*}
				
				Now notice that
				\begin{align*} 
					&\int_{\cU_{j,\sigma}^c} |\Diff^k \bar f_{j,\delta}(z_x)| \diff x \\
					&\leq \{\int_{\cU_{j,\sigma}^c} \{\frac{|\Diff^k \bar f_{j,\delta}(z_x)|}{\bar f_{j,\delta}(z_x)}\}^{\frac{2\beta+\ve}{\inner{k}{\bm\alpha}}} \bar f_{j,\delta}(z_x) \diff x\}^{\frac{\inner{k}{\bm\alpha}}{2\beta+\ve}}  \times \{ \int_{\cU_{j,\sigma}^c} \bar f_{j,\delta}(z_x) \diff x \}^{1-\frac{\inner{k}{\bm\alpha}}{2\beta+\ve}} \\
					&\lesssim  c_j^{-1}\delta^{D-d} \sigma^{H_1 \times \{1 - \inner{k}{\bm\alpha}/(2\beta+\ve)\}} \lesssim c_j^{-1}\delta^{D-d}\sigma^{H_1/2}
				\end{align*} 
				and likewise, $\int_{\cU_{j,\sigma}^c} L_j(x) \diff x \lesssim  c_j^{-1} \delta^{D-d} \sigma^{H_1 \times \{1 - \beta/(2\beta+\ve)\}}  \lesssim \delta^{D-d}\sigma^{H_1/2}$. We thus have shown that
				\beq \label{intks}
				\int_{\cU_{j,\sigma}^c} |K_\Sigma(\Diff^k \bar f_{j,\delta}(z_{\square}))(x)| \lesssim c_j^{-1} \delta^{D-d}\sigma^{H_1/2}.
				\eeq
				Coming back to (\ref{int:2}), there immediately holds that $\int_{\cU_{j,\sigma}^c} f_j  \lesssim \sigma^{H_1}$, and furthermore, noticing that
				$$
				h_j \lesssim \wt h_j \lesssim  f_j + 
				\frac{1}{\delta^{D-d}}\sum_{0 \leq \inner{k}{\bm\alpha} < \beta} \sigma^{\inner{k}{\bm\alpha}} |\Diff^k \bar f_{j,\delta}(z_{\square})|
				$$
				and using \eqref{intks} and the monotonicity of $K_\Sigma$, we find that $\int_{\cU_{j,\sigma}^c} K_\Sigma(h_j) \lesssim c_j^{-1}\sigma^{H_1/2}.$
				
				\item Finally, for (\ref{int:4}), notice that $f_j/2 - g_j$ is a sum of terms of the form $\sigma^{\<k,\alpha \>}\delta^{-(D-d)}\Diff^k \bar f_{j,\delta} (z_\square)$ for  $0 \leq \langle k,\alpha \rangle < \beta$. The term $K_\Sigma((f_j/2 - g_j) \ind_{g_j \leq f_j/2})^2$ is thus upper bounded by a sum of terms of the form $\sigma^{2\inner{k}{\bm\alpha}} \delta^{-2(D-d)} K_{\Sigma}(  \Diff^k \bar f_{j,\delta} (z_\square) \ind_{g_j \leq f_j/2})^2$. 
				Now there holds
				\begin{align*}
					\int_{\cU_{j,\sigma}\cap \cO_{x_j}^1} &\frac{K_{\Sigma}(  \Diff^k \bar f_{j,\delta} (z_\square) \ind_{g_j \leq f_j/2})^2}{K_\Sigma(h_j) + f_j} \\
					&\lesssim \sigma^{-H_1}  \| K_{\Sigma}(  \Diff^k \bar f_{j,\delta} (z_\square) \ind_{g_j\leq f_j/2})\|_\infty \times \int K_{\Sigma}( | \Diff^k \bar f_{j,\delta} (z_\square) |\ind_{g_j\leq f_j/2}) \\
					&\lesssim\frac{ \sigma^{-H_1}}{c_j} \int_{\cO_{x_j}^1}| \Diff^k \bar f_{j,\delta} (z_y) |\ind_{g_j\leq f_j/2}) dy \lesssim \frac{ \sigma^{-H_1}}{c_j^2}    \delta^{D-d} \{\sigma \log(1/\sigma)\}^{2\beta+\ve - \inner{k}{\bm\alpha}}
				\end{align*} 
				where we used Lemma \ref{lem:ajs}. 
				Summing all the bounds in $k$, we obtain that (\ref{int:4}) is bounded from above, up to constant, by 
				$$\max_k c_j^{-2} \sigma^{2\beta+\ve-H_1+2\inner{k}{\bm\alpha}} \delta^{-(D-d)}   {\log(1/\sigma)}^{2\beta+\ve} \lesssim c_j^{-1} \delta^{-(D-d)}  \sigma^{\ve-H_1} {\log(1/\sigma)}^{2\beta+\ve} $$
				where we used that $c_j \geq \sigma^{2\beta}$.
			\end{enumerate}
			Collecting all the bounds on (\ref{int:1}-\ref{int:5}) together with $ \delta^{2} \sigma^{2\alpha_\perp} = o(1/\log (1/\sigma))$, we get that
			$$
			c_j \dhed(K_\Sigma(h_j),f_j) \lesssim c_j \left[  \sigma^{2\beta} \log(1/\sigma)^{4\beta+2\ve}+ \sigma^{H_1/2} + \sigma^{\ve-H_1} \delta^{-(D-d)}{\log(1/\sigma)}^{2\beta+\ve}\right].
			$$
			Choosing $H_1 = 4 \beta$ and $\ve \geq 6 \beta + \ve_1$ where $\sigma^{\ve_1} \leq \delta^{D-d}$, we obtain the result. 
			

		\section{Appendix to Section \ref{app:pr:thmain}: proof of \thmref{main}} \label{app:secproofs}

		For any probability distribution $P$ on $\bbR^{D} \times \cS^{++}(D,\bbR)$, one defines the probability density function on $\bbR^D$
		$$
		f_P(x) := \int \varphi_{\Sigma}(x-y) \diff P(y,\Sigma).
		$$
		Note that when $g$ is a probability distribution on $\bbR^D$, then
		$$
		K_\Sigma g(x) = f_P(x)~~~\text{with}~~~\diff P(y,\Sigma) = \delta_{\Sigma(y)}(\Sigma) g(y) \diff y.
		$$
		
		\begin{lem} \label{lem:vj} Let $V_0,\dots,V_N$ be a partition of $\bbR^D$ and let $P = \sum_{j=1}^N \pi_j \delta_{z_j,\Sigma_j}$ with $z_j \in V_j$. Then, for any probability measure $Q$ on $\bbR^D \times \cS^{++}(D,\bbR)$,
			\begin{align*}
				\|f_Q - f_P\|_1 &\leq 2\sum_{j=1}^N |Q_1(V_j) - \pi_j | + \frac12\sup_{1\leq j\leq N} \|\Sigma_j^{-1/2}\|_{\op} \diam V_j +  \frac32 \sup_{1\leq j\leq N} \sup_{\Sigma \in \cS_j}\sqrt{\tr(\Sigma_j^{-1}\Sigma-\Id)^2},
			\end{align*} 
			where $Q_1$ is the first marginal of $Q$ and $\cS_j$ is the support of the second marginal of $\ind_{V_j}(y) \diff Q(y,\Sigma)$.
		\end{lem}
		\begin{proof} We write that $f_Q(x) - f_P(x)$ is
			\begin{align*} 
			\int_{V_0} \varphi_\Sigma(x-y) &\diff Q(y,\Sigma) + \sum_{j=1}^N \int_{V_j} \{\varphi_\Sigma(x-y) - \varphi_{\Sigma_j}(x-z_j)\} \diff Q(y,\Sigma) \\
			&+ \sum_{j=1}^N (Q_1(V_j)-\pi_j)\varphi_{\Sigma_j}(x-z_j), 
			\end{align*} 
			The last term can be readily bounded in $L^1$-norm by
			$$
			\sum_{j=1}^N |Q_1(V_j)-\pi_j | \int \varphi_{\Sigma_j}(x-z_j) \diff x = \sum_{j=1}^N |Q_1(V_j)-\pi_j |,
			$$
			and the first one by $Q_1(V_0) = 1 - \sum_{j = 1}^N Q_1(V_j) \leq \sum_{j=1}^N |Q_1(V_j) - \pi_j|$. For the second term, notice that each term of the sum is upper-bounded in $L^1$-norm by
			\begin{align*} 
				&\int_{V_j} \|\varphi_\Sigma(\cdot-y) - \varphi_{\Sigma_j}(\cdot-z_j) \|_1 \diff Q(y,\Sigma)  \\
				&\leq   \int_{V_j} \|\varphi_{\Sigma_j}(\cdot-y) - \varphi_{\Sigma_j}(\cdot-z_j) \|_1 \diff Q(y,\Sigma) + \int_{V_j} \|\varphi_\Sigma(\cdot-y) - \varphi_{\Sigma_j}(\cdot-y) \|_1 \diff Q(y,\Sigma) \\
				&\leq \frac12  \int_{V_j} \|\Sigma^{-1/2}_j(y-z_j)\| \diff Q_1(y) + \frac32\int_{V_j} \sqrt{\tr(\Sigma_j^{-1}\Sigma-\Id)^2} \diff Q(y,\Sigma) \\
				&\leq \frac12 Q_1(V_j) \|\Sigma_{j}^{-1/2}\|_{\op} \diam V_j + \frac32 Q_1(V_j) \sup_{\Sigma \in \cS_j} \sqrt{\tr(\Sigma_j^{-1}\Sigma-\Id)^2},
			\end{align*} 
			where we used Prp 2.1 and Thm 1.1 of \cite{devroye2018total}.
		\end{proof}
		
				\subsection{Proof of \lemref{sieve}} \label{app:sieve}



						Throughout the proof, $C$ denotes a generic constant whose value depends only on $D$.  We start with bounding the probability measure of $\cF_n^c$. There holds
						\begin{align*}
							\Pi(\mathcal F_n^c) &\leq \Pi\{\exists h \leq H_n, \mu_h \notin B(0,R_n)\} +  \Pi\{\sum_{h > H_n} \pi_h \geq \ve_n\} + \Pi(\exists h \leq H_n, \Lambda_h \notin [\underline{\sigma}_n^2,\bar{\sigma}_n^2]^{D}).
						\end{align*} 
						The first mass is bounded using \eqref{condH:partial}:
						$$
						\Pi\{\exists h \leq H, \mu_h \notin B(0,R_n)\} \lesssim H_n \int \|\mu\|^{-b_2}\ind_{\|\mu\| \geq R_n} \diff \mu \lesssim H_n R_n^{-b_2} \lesssim e^{-c_1 n \ve_n^2},
						$$
						as soon as $b_2 R_0\geq 2c_1$. The second term is bounded in \cite[p. 15]{shen2013adaptive} by
						$$
						\Pi\{\sum_{h > H_n} \pi_h \geq \ve_n\} \leq \{\frac{eB}{H_n}\log(1/\ve)\}^{H_n}\lesssim e^{-c_1 n \ve_n^2},$$
						as soon as $H_0$ is large enough. Finally, to bound the last probability, we consider separately the partial and hybrid location-scale priors. 
						\begin{itemize}
							\item Partial location-scale: $\Lambda_h = \Lambda_1$ for all $h$ and using \eqref{cond:piL}, 
							\begin{align*}
								\Pi( \Lambda_1 \notin \cQ_n) \leq  \Pi_\Lambda\{ \min_i \Lambda_{i} \leq \underline{\sigma}_n^2\} +  \Pi_\Lambda\{ \max_i \Lambda_{i} \geq \bar{\sigma}_n^2\} \lesssim e^{-c_3 \sigma_0^{-2b_3} n\ve_n^2 } + \bar \sigma_n^{-2b_4} \leq e^{-c_1 n\ve_n^2 }, 
							\end{align*}
							as soon as $\sigma_1^2$ is large  enough and $\sigma_0^2$ is small enough. 
							\item Hybrid location-scale prior 
							\begin{align*}
								\Pi(\exists h \leq H_n, \,  \Lambda_h \notin \cQ_n) &\leq  H_n \mathbb E_{\tilde\Pi_\Lambda}\left[Q_2( \min_i \Lambda_{i} \leq \underline{\sigma}_n^2)\right] + H_n \mathbb E_{\tilde\Pi_\Lambda}\left[Q_2( \max_i \Lambda_{i} > \bar{\sigma}_n^2)\right] \\
								&\leq H_n e^{-c_3 \sigma_0^{-2b_3} n\ve_n^2 }  + H_n e^{-c_4 \sigma_1^{2 b_4} n\ve_n^2 } \lesssim e^{-c_1 n\ve_n^2}. 
							\end{align*}
						\end{itemize}		
						We then turn on bounding the entropy of the partions of $\cF_n$. Note that on $\mathcal F_n$ $\min_i \lambda_i \geq \underline \sigma_n^2 $ so that 
						$$\frac{ \max_i \lambda_i }{ \min_i \lambda_i } \leq \max_i \lambda_i  \underline \sigma_n^{-2} \leq \sigma_0^{-2} n^{2{t_0}/b_3} \times \max_i \lambda_i . $$  
						From \cite{canale2017posterior}, the covering number of $\mathcal F_{n, \mathbf j,\ell}$ is bounded by 
						\begin{align*} &N( \ve_n, \mathcal F_{n, \mathbf j,  \ell}, \| \cdot \|_1 ) \\
						&\leq \exp \left( 
						C H_n[\log n+\log (1/\ve_n)] +(D-1) \sum_{h\leq H_n} \log(j_h+1) +D(D-1) 2^{\ell-2}H_n\log n   \right), \end{align*} 
						and the covering number of 
						$\mathcal F_{n, \mathbf j}$ is bounded by 
						\begin{align*} 
							N( \ve_n, \mathcal F_{n, \mathbf j}, \| \cdot \|_1 ) &\leq \exp \left( 
							C H_n[\log n+\log (1/\ve_n)] + \sum_{h\leq H_n} \log( j_h+1) +  DH_n\log \bar \sigma_n^2 \right)
							\\
							&\lesssim  \exp \left( 
							C H_n[\log n+\log (1/\ve_n)] + \sum_{h\leq H_n} \log (j_h+1)\right). 
						\end{align*} 
						We bound $\Pi(\mathcal F_{n, \mathbf j,  \ell} )$ in the case of  the partial location-scale prior, for $\ell \geq 1$: 
						\begin{equation*} 
							\begin{split}
								&\Pi(\mathcal F_{n, \mathbf j, \ell} ) \lesssim 
								\prod_{h\leq H_n} (j_h\sqrt{n})^{-(b_2-D)\1_{j_h \geq 1} } \Pi_\Lambda\( \max_{i} \lambda_{i} > \sigma^{2}_0 n^{2^{\ell-1}-2{t_0}/b_3}\)  \\ 
								&\lesssim 
								\prod_{h\leq H_n} (j_h\sqrt{n})^{-(b_2-D)\1_{j_h \geq 1} } n^{ -b_4(2^{\ell-1}-2{t_0}/b_3)} \\
								&= \exp\{-(b_2-D) \sum_{h \leq H_n} \ind_{j_h \geq 1} \log(j_h)- \frac12(b_2-D) n H_n -  b_4(2^{\ell-1}-2{t_0}/b_3) H_n \log n\}.
							\end{split}
						\end{equation*} 
						We bound $\Pi(\mathcal F_{n, \mathbf j} )$ in the case of  the hybrid location-scale prior: 
						\begin{equation*} 
							\begin{split}
								\Pi(\mathcal F_{n, \mathbf j} ) &\lesssim 
								\prod_{h\leq H_n} (j_h\sqrt{n})^{-(b_2-D)\1_{j_h \geq 1} }  \\
								&= \exp\{-(b_2-D) \sum_{h \leq H_n} \ind_{j_h \geq 1} \log(j_h)- \frac12(b_2-D)  H_n\log n\}.
							\end{split}
						\end{equation*} 
						This implies in particular that for the partial location-scale prior
						\begin{align*}
							&\sum_{\mathbf j, \ell} \sqrt{\Pi(\mathcal F_{n, \mathbf j, \ell} )N( \ve_n, \mathcal F_{n, \mathbf j,  \ell}, \| \cdot \|_1 ) } \lesssim 
							\exp \left( 
							\frac12 C H_n[\log n+\log (1/\ve_n)] \right) \\
							&\times \sum_{\mathbf j, \ell}\exp\big( \frac{1}{2}\sum_h [ (D-1)\log(j_h+1)-(b_2-D)(\log(j_h) \ind_{j_h \geq 1}-1)] \\
							&~~~~~~~~~~~~~~+ \1_{\ell\geq 2} 2^{\ell-1 }\log n[ D(D-1)/2 - b_4 ] \big) \\
							& \lesssim 
							\exp \left( C H_n[\log n+\log (1/\ve_n)]/2 \right), 
						\end{align*}
						since $b_4 > D(D-1)/2$ and $b_2 > 2D-1$.
						Therefore, by choosing $M_0>0$ large enough
						$$\sum_{\mathbf j, \ell} \sqrt{\Pi(\mathcal F_{n, \mathbf j, \ell} )N( \ve_n, \mathcal F_{n, \mathbf j,  \ell}, \| \cdot \|_1 ) } e^{-M_0 n \ve_n^2 } = o(1) .$$ 
						We obtain a similar result for the hybrid location-scale prior.

	\subsection{Proof of \lemref{thickness}} \label{app:thickness}
	We let again $R = (H \log(1/\ve)/C_2)^{1/\kappa}$ and define $\sigma := \ve^{1/\beta}$.
			Thanks to \corref{approx}, we know there exists a density function $g$ supported on $M^\delta$ such that $\dhed(K_\Sigma g, f_0) \lesssim \sigma^{2\beta}\log^q(1/\sigma)$ for some $a>0$. We can in turn, thanks to \lemref{disc}, find a discrete probability measure $G$ on $M^\delta \cap \ball(0,R)$ with $N$ atoms at least $\sigma^{\alpha_1} \ve^2$-apart such that 
			$$
			\|K_\Sigma g - K_\Sigma G\|_1 \lesssim \ve^2 \log^{D/2}(1/\ve)~~~\text{and}~~~N \lesssim \sigma^{-D} \log^D(1/\ve).
			$$
			We thus have $\dhed(K_\Sigma G, f_0) \lesssim \sigma^{2\beta} \log^q(1/\sigma)  +  \ve^2 \log^{D/2}(1/\ve) \lesssim \ve^2 \log^{2r}(1/\ve)$ with $r := q/2 \vee D/4$. We let $z_1,\dots,z_N$ be the atoms of $G$ and denote by $p_j = G(z_j)$. We let $V_j$ be the ball centered around $z_j$ with radius $\sigma^{2\alpha_0}\ve^2/2$. We complete $V_1,\dots,V_N$ with sets $V_{N+1},\dots,V_J$ that forms a partition of $M^\tau \cap \ball(0,R)$ with $V_j$ included in balls of the form
			$
			\{x \in \bbR^D~|~\|x-z_j\|_{\Sigma^{-1}(z_j)} \leq 1\}  
			$
			for some $z_j \in M^\tau \cap \ball(0,R)$, so that we can take $J \lesssim N + (R/\sigma)^D \lesssim \sigma^{-D} \log^{D/\kappa}(1/\ve)$. We then set $V_0$ to be the complementary set of the reunion of the $V_j$ and set further $p_j = 0$ for $j$ greater than $N+1$.
			
We write  under the prior $\Pi$, $P = \sum_{h = 1}^\infty \pi_h \delta_{\mu_h,U_h,\Lambda_h}$ and $\Sigma_h = U_h^\top \Lambda_h U_h$. We use the convention that $\Lambda_h = \Lambda$ for all $h$ in the case of the Partial location scale prior and $\pi_h = 0$ for $h \geq K$ for the mixture of finite mixtures prior. 
Set			$\wt N = \log(1/\ve) \times N$, we consider the following events
			\begin{align*} 
				\cP_{J} &= \{\sum_{j=1}^{J} |p_j - P_\mu (V_j)| \leq \ve^2 ~~\text{and} ~~ \min_{1 \leq j \leq J} P_\mu (V_j) \geq \ve^4\},~~~\\
				\cF_{\wt N} &= \{\sum_{h \leq {\wt N}} \pi_h \geq 1 - \ve^8\}, \\
				\cO_{{\wt N}} &= \{\forall 1\leq h \leq {\wt N},~ \|U_h O_{\mu_h}^\top - \Id\|_{\op} \leq \sigma^{2\alpha_0}\ve^2\}, \\
				\text{and}~~\cL_{\wt N} &= \{\forall 1\leq h \leq {\wt N},~ \Lambda_h \in \cS(\sigma^{\alpha_0})^d \times \cS(\delta\sigma^{\alpha_\perp})^{D-d}\}\\
				~~\text{where}~~
				\cS(t) &= \{s~|~t^2 \leq s \leq t^2(1+\sigma^{2\beta})\}.
			\end{align*} 
			We first show that if $P \in \cP_{J} \cap \cF_{\wt N} \cap \cO_{{\wt N}} \cap \cL_{\wt N}$, then
			$$
			\dhe(f_P,f_0) \lesssim  \ve \log^{r}(1/\ve). 
			$$
			Indeed, we have
			$$
			\dhe(f_P,f_0) \leq \dhe(f_P,K_\Sigma G) + \dhe(K_\Sigma G,f_0) \leq \dhe(f_{\wh P},f_{\wh G}) +  \dhe(f_{\wh P},f_{P}) +  \ve \log^{r}(1/\ve),		$$
			where $\diff \wh G(y,\Sigma) = \delta_{\Sigma(y)} \diff G(y)$ and $\wh P = \sum_{h \geq 1} \pi_h \delta_{\mu_h,\wh \Sigma_h}$  with
			$$
			\wh \Sigma_h = \begin{cases} \Sigma_h~~~\text{if $h \leq \wt N$} \\
				\Sigma(z_h)~~~\text{if $h > \wt N$ and $\mu_h \in V_j$ with $j \leq N$;} \\ 0~~~~\text{otherwise},
			\end{cases}
			$$
			where, conventionally, $\varphi_\Sigma(\cdot -z)\diff z = \delta_{z}$ when $\Sigma = 0$.
			Since $P \in \cF_{\wt N}$, there easily holds $\|f_P - f_{\wh P}\|_1 \leq 2\ve^2$ and, using \lemref{vj}, we find that
			\begin{align*} 
				&\|f_{\wh P} - f_{\wh G}\|_1 \\
				&\leq 2\sum_{j=1}^N |P_\mu(V_j) - p_j | + \frac12 \sigma^{-\alpha_0} \sup_{1\leq j\leq N}\diam V_j + \frac32\sup_{1 \leq j \leq N}\sup_{\mu_h \in V_j} \sqrt{\tr(\Sigma(z_j)^{-1}\wh \Sigma_h-\Id)^2},
				\\
				&\lesssim \ve^2 + \sup_{1 \leq j \leq N}\sup_{\mu_h \in V_j} \|\Sigma(z_j)^{-1}\wh \Sigma_h-\Id\|_{\op}.
			\end{align*} 
			where we used that $P \in \cP_J$ and that $\diam V_j \lesssim \sigma^{\alpha_0}\ve^2$ for $j \leq N$. In the last supremum, if $h > \wt N$ and $\mu_h \in V_j$, then $\Sigma_h = \Sigma(z_j)$ so we only need to handle the case when $h \leq \wt N$. Moreover, 
			\begin{align*}
				\Sigma(z_j)^{-1}\wh \Sigma_h-\Id &= O_{z_j}^\top \Delta^{-2}_{\sigma,\delta} O_{z_j} U_{h}^\top \Lambda_h U_{h} - \Id \\
				&= O_{z_j}^\top \Delta^{-2}_{\sigma,\delta} (O_{z_j} U_{h}^\top - \Id)\Lambda_h U_{h} + O_{z_j}^\top \Delta^{-2}_{\sigma,\delta} \Lambda_h U_{h}  - \Id \\
				&=   O_{z_j}^\top \Delta^{-2}_{\sigma,\delta} (O_{z_j} U_{h}^\top - \Id)\Lambda_h U_{h} + O_{z_j}^\top( \Delta^{-2}_{\sigma,\delta} \Lambda_h - \Id) U_{h}  + O_{z_j}^\top U_{h} - \Id,
			\end{align*} 
			so that
			\begin{align*}
				\|\Sigma(z_j)^{-1}\wh \Sigma_h-\Id\|_{\op} &\lesssim \| \Delta^{-2}_{\sigma,\delta}\|_{\op} \|\Lambda_h\|_{\op} \| \|O_{z_j} U_{h}^\top - \Id\|_{\op} + \|\Delta^{-2}_{\sigma,\delta} \Lambda_h - \Id\|_{\op} +  \|O_{z_j}^\top U_{h} - \Id\|_{\op}  \\
				&\lesssim \ve^2+ \sigma^{2\beta} \simeq \ve^2,
			\end{align*} 
			where we used both that $h \leq \wt N$ and $P \in \cO_{\wt N} \cap \cL_{\wt N}$. 
			
			Using \cite[Lem B2]{shen2013adaptive}, we find that for $\lambda > 0$ small enough
			\begin{align*} 
				P_0 \log \frac{f_0}{f_P} &\lesssim \dhed(f_0,f_P)(1+\log(1/\lambda))+P_0\{\log\frac{f_0}{f_P} \ind_{f_P/f_0 < \lambda}\} \\
				P_0 \left(\log \frac{f_0}{f_P}\right)^2 &\lesssim \dhed(f_0,f_P)(1+\log^2(1/\lambda))+P_0\{\left(\log\frac{f_0}{f_P}\right)^2 \ind_{f_P/f_0 < \lambda}\},
			\end{align*} 
			up to numeric constant. Notice that by assumption, $f_0$ is bounded from above by $\|L\|_\infty$. Let $x \in M^\tau \cap \ball(0,R)$, and let $1 \leq j\leq J$ be such that $x \in V_j$. Notice that, since $P \in \cF_{\wt N} \cap \cP_J$, there holds
			$$
			\ve^4 \leq P_\mu(V_j) = \sum_{\mu_h \in V_j} \pi_h \leq \sum_{\substack{\mu_h \in V_j \\ h \leq \wt N}} \pi_h + \ve^8~~~~\text{so that}~~~\sum_{\substack{\mu_h \in V_j \\ h \leq \wt N}}  \pi_h \geq \ve^4/2.
			$$
			Now we can write
			$$
			f_P(x) \geq \sum_{\substack{\mu_h \in V_j \\ h \leq \wt N}} \pi_h \varphi_{\Sigma_h}(x-\mu_h) \gtrsim  \frac{\ve^4}{\delta^{D-d} \sigma^D},  
			$$
			where we used that $|\det \Sigma_h|^{1/2} \leq  (1 + \sigma^{2\beta})^{D/2} \delta^{D-d} \sigma^{D} \lesssim \delta^{D-d} \sigma^{D}$ and that $\|x - y\|_{\Sigma^{-1}(y)} \lesssim 1$ for any $x,y \in V_j$ with a similar line of reasoning as in the proof of \lemref{disc}, relying on $P \in \cO_{\wt N} \cap \cL_{\wt N}$. Furthermore, since $f_P(x) \leq \ve^8 + \sum_{h \leq \wt N} \pi_h \varphi_{\Sigma_h}(x-\mu_h)$, there must be some $h \leq \wt N$ such that both $\pi_{h} \geq \ve^2$ and $\mu_{h} \in \ball(0,R)$ holds, otherwise one would get a contradiction looking at the mass of $f_0$ since one would get
			\begin{align*} 
			\|f_0\|_1 &= \|f_0 \ind_{\ball(0,R/2)}\|_1+ \|f_0 \ind_{\ball(0,R/2)^c}\|_1 \\
			&\leq 
			\underbrace{\|f_P \ind_{\ball(0,R/2)}\|_1}_{\lesssim \ve^2 + \ve^H}+
			\underbrace{\|(f_P - f_0)\ind_{\ball(0,R/2)}\|_1}_{\lesssim \ve + \sigma^\beta}+ 
			\underbrace{\|f_0 \ind_{\ball(0,R/2)^c}\|_1}_{\lesssim \ve^H},
			\end{align*} 
			where the inequalities occurs up to log-term.
			For $x \notin \ball(0,R)$, notice that, for this particular $h \leq \wt N$,
			$
			f_P(x) \gtrsim \ve^2 \varphi_{\Sigma_h}\left(x - \mu_h\right).
			$
			Taking $\lambda = C \ve^4/(\delta^{D-d} \sigma^D)$ for some small constant $C$, one get for any $\ell \geq 1$, 
			\begin{align*}
				&P_0\{\left(\log\frac{f_0}{f_P}\right)^\ell \ind_{f_P/f_0 < \lambda}\} \leq P_0\{\left(\log\frac{f_0}{f_P}\right)^\ell \ind_{\ball(0,R)^c}\} \\
				&\lesssim \log^\ell \frac{\ve^2}{\sigma^D \delta^{D-d}} \times P_0(\ball(0,R)^c)+ \int_{\ball(0,R)^c} \|x-\mu_h\|_{\Sigma_h^{-1}}^{2\ell} f_0(x) \diff x \\
				&\lesssim \ve^{H} \log^\ell \frac{\ve^2}{\sigma^D \delta^{D-d}} + \sigma^{-2 \ell \alpha_0} \{\int \|x\|^{4\ell} f_0(x) \diff x\}^{1/2} P_0(\ball(0,R)^c)^{1/2} \\
				&\lesssim \ve^2 \log^2(1/\ve),
			\end{align*} 
			where the last inequality holds for $\ell \in \{1,2\}$, provided that we chose $H \geq 8\alpha_0+4\beta$.
			This shows that $f_P \in \ball(f_0,\tilde \ve)$ for $\tilde \ve \approx \ve \log^{s}(1/\ve)$ with $s  = r \vee 1$. It only remains to lower bound the prior mass of the event $\cP_{J} \cap \cF_{\wt N} \cap \cO_{{\wt N}} \cap \cL_{\wt N}$. 
			
			For this, one can use \eqref{condH:partial} and the fact that the scales are drawn independently to rest to find that
			$$
			\Pi(\cP_{J} \cap \cF_{\wt N} \cap \cO_{{\wt N}} \cap \cL_{\wt N}) \geq \Pi(\cP_{J} \cap \cF_{\wt N}) \times c_o^{\wt N} \prod_{h \leq \wt N} O\{\|UO_{\mu_h}-\Id\|_{\op} \leq \sigma^{2\alpha_0} \ve^2\} \times \Pi(\cL_{\wt N}).
			$$ 
			We easily get that $O\{\|UO_{\mu_h}-\Id\|_{\op} \leq \sigma^{2\alpha_0} \ve^2\} \gtrsim (\sigma^{2\alpha_0} \ve^2)^{D(D-1)/2}$. For $\Pi(\cL_{\wt N})$, we use  \eqref{cond:piL} or  \eqref{cond:piH:LB} along with a simple Markov inequality to find that
			\begin{align*} 
				\Pi(\cL_{\wt N}) &= \bbE_{\Pi}\{Q_2\left([\sigma^{2\alpha_0},(1+\sigma^{2\beta})\sigma^{2\alpha_0}]^d \times [\delta^2 \sigma^{2\alpha_\perp},(1+\sigma^{2\beta})\delta^2\sigma^{2\alpha_\perp}]^{D-d}\right)^{\wt N}\} \\
				&\gtrsim \exp(-2 \alpha_0 B_0 \log(1/\sigma) \wt N) \exp\{-c_2 D \sigma^{-D}\}.
			\end{align*} 
			For $\Pi(\cP_{J} \cap \cF_{\wt N})$, we write
			$$
			\Pi(\cP_{J} \cap \cF_{\wt N}) \geq \Pi(\cP_{J}) -  \Pi(\cF_{\wt N}^c) \gtrsim e^{-C J \log(1/\ve)} - e^{-\wt N\log(\wt N)},
			$$
			where we used \cite[Lem 10]{ghosal2007posterior} and the bound  \cite[p. 15]{shen2013adaptive}. We conclude by noticing that $e^{-\wt N\log(\wt N)} \ll e^{- N \log(1/\ve)}$ so that in the end 
			$$
			\Pi(\cP_{J} \cap \cF_{\wt N} \cap \cO_{{\wt N}} \cap \cL_{\wt N})  \gtrsim \exp(-C N \log^t(1/\ve))),
			$$
			for some $C$ depending on the parameters and $t = D/\kappa + 2$, ending the proof.

\subsection{Proof of \lemref{disc}} \label{app:disc}
		\begin{proof}[Proof of \lemref{disc}]
			Let $g : \bbR^D \to \bbR$ be a density supported on $M^\delta$ and satisfying \eqref{expdec} and \eqref{assint}. Let
			$R = (H \log(1/\ve)/C_2)^{1/\kappa}
			$,
			so that $g$ is less than $C_1 \ve^H$ outside of $\ball(0,R)$.
			We consider the functions $ \{\chi_i\}_{i \in \cI}$  introduced in \lemref{chixj} associated with a $\tau/64$-packing $\{x_i\}_{i \in \cI}$ of $M \cap \ball(0,R)$. Recall that the number $|\cI|$ of functions is less than of order $R^D$. We can write
			$$
			K_\Sigma g(x) = K_\Sigma (g \ind_{\ball^c(0,R)})(x) + \sum_{i \in \cI} K_\Sigma (\chi_i g \ind_{\ball(0,R)})(x) =K_\Sigma (g \ind_{\ball^c(0,R)})(x) + \sum_{i \in \cI} c_i K_\Sigma g_i(x) 
			$$
			where $c_i = \|\chi_i g \ind_{\ball(0,R)}\|_1$ and $g_i = \chi_i g \ind_{\ball(0,R)} /c_i$ is a density supported on $\cO^0_{x_i}$. Notice that for any $x \in \bbR^D$
			$$
			K_\Sigma (g \ind_{\ball^c(0,R)})(x)  \leq C_1 \ve^H K_\Sigma (\ind_{\ball^c(0,R)})(x) \leq C_1 \ve^H
			$$
			and that
			\begin{align*} 
				\|K_\Sigma(g \ind_{\ball^c(0,R)})\|_1 &= \|g \ind_{\ball^c(0,R)}\|_1 \leq \int_{\|x\| \geq R} C_1 e^{-C_2\|x\|^\kappa} \lesssim  \int_{r \geq R} e^{-C_2 r^\kappa} r^{D-1} \diff r \\
				&\lesssim \log^{(D-2)/\kappa+1} (1/\ve) \times \ve^H,
			\end{align*} 
			where we used a NP-bound on the incomplete Gamma function in the last inequality. Consequently, the term $K_\Sigma(g \ind_{\ball^c(0,R)})$ need not be discretized. Take now $i \in \cI$. As in the proof of \thmref{approx}, there holds that
			$$
			\varphi_{\Sigma(y)}(x - y) \lesssim \frac{\ve^H}{\sigma^D\delta^{D-d}} ~~~~\forall x \notin \cO_{x_i}^1, y \in \cO^0_{x_i},
			$$
			provided that $\tau \gtrsim  R \delta\sigma^{\alpha_\perp}$. This means that $|K_\Sigma g_i(x)| \lesssim  \frac{\ve^H}{\sigma^D\delta^{D-d}}$ for all $x \notin \cO_{x_i}^1$. If now $x \in \cO_{x_i}^1$, there holds   
			\begin{align*} 
				K_\Sigma g_i(x) &= \int_{\cW_{x_i}^0} \varphi_{\Sigma(\bar\Psi_{x_i}(w))}(x-\bar\Psi_{x_i}(w)) \bar g_i(w) \diff w \\
				&=  \sum_{j \in \cJ_i}  \bar c_{i,j} \int_{\cW_{i,j}} \varphi_{\Sigma(\bar\Psi_{x_i}(w))}(x-\bar\Psi_{x_i}(w)) \bar g_{i,j}(w) \diff w
			\end{align*} 
			where $\bar g_i(w) = g_i(\bar \Psi_{x_i}(w)) \times | \det \diff \bar\Psi_{x_i}(w)|$, $\bar c_{i,j} = \int_{\cW_{i,j}} \bar g_i$ and $ \bar g_{i,j} = \bar g_i / \bar c_{i,j}$. The sets $\cW_{i,j}$ form a partition of $\cW_{x_i}^0$ that are included in
			$$
			\prod_{\ell=1}^d [w^0_{t_\ell}, w^0_{t_\ell +1} ) \times  \prod_{\ell=d+1}^D [w^\perp_{t_\ell}, w^\perp_{t_\ell +1} ) \cap \cW_{x_i}^0,
			$$
			for $(w^0_{t_1},\dots,w^0_{t_d}, w^\perp_{t_{d+1}},\dots,w^\perp_{t_D})$ vary on a grid of size $(\sigma^{\alpha_0},\dots,\sigma^{\alpha_0},\delta\sigma^{\alpha_\perp},\dots,\delta\sigma^{\alpha_\perp})$. Notice that
			$$
			\Card \cJ_i  \simeq \frac{1}{\sigma^{\alpha_0 d}} \times \left(\frac{\delta}{\delta \sigma^{\alpha_\perp}} \right)^{D-d} = \sigma^{-D}. 
			$$
			Let  $i \in \cI$ and  $j \in \cJ_i$ be fixed  and denote for short $\bar\Psi = \bar\Psi_{x_i}$ and $\bar \Sigma = \Sigma \circ \bar \Psi$. We let $\Gamma > 0$ and we distinguish two cases: $ \inf_{w \in \cW_{i,j}} \|x - \bar \Psi(w) \|^2_{\bar\Sigma^{-1}(w)} \geq \Gamma \log(1/\ve)$ and  $ \inf_{w \in \cW_{i,j}} \|x - \bar \Psi(w) \|^2_{\bar\Sigma^{-1}(w)} < \Gamma \log(1/\ve)$. In the former, 
			$$
			\int_{\cW_{i,j}} \varphi_{\bar\Sigma(w)}(x-\bar\Psi(w)) \bar g_{i,j}(w) \diff w \lesssim \frac{\ve^{\Gamma/2}}{\sigma^D \delta^{D-d}}.
			$$
			While if $\inf_{w \in \cW_{i,j}} \|x - \bar \Psi(w) \|^2_{\bar\Sigma^{-1}(w)} \leq \Gamma \log(1/\ve)$, we first show that $\sup_{w \in \cW_{i,j}} \|x - \bar \Psi(w) \|^2_{\bar\Sigma^{-1}(w)} \leq \Gamma' \log(1/\ve)$ for some $\Gamma'>\Gamma$. We let $w_0 \in \cW_{i,j}$ such that $\|x - \bar \Psi(w_0) \|^2_{\bar\Sigma^{-1}(w_0 )} \leq 2\Gamma \log(1/\ve)$ and take $w \in \cW_{i,j}$. Using \eqref{lem:angles}, we get that there $\|\pr_{T_{\bar\Psi(w)}} - \pr_{T_{\bar\Psi(w_0)}}\|_{\op} \lesssim \sigma^{\alpha_0}$ and the same holds for $\pr_{N_{\bar\Psi(w)}} - \pr_{N_{\bar\Psi(w_0)}}$. 
			\newclem{Let us denote $y = \bar \Psi(w_0)$ and $z = \bar\Psi(w)$. First notice that if $x \notin \cW_{x_i}^1$, then $\|x-y\|_{\Sigma^{-1}(y)} \gtrsim \sigma^{-\alpha_0} \wedge \delta^{-1} \sigma^{-\alpha_\perp} > \Gamma \log(1/\ve)$ so that necessarily $x \in \cW_{x_i}^1$. In particular, $\|x-y\| \leq B$ for some constant $B > 0$ not depending on $\Gamma$. Now recall that
			$$ \|x-z\|_{\Sigma^{-1}(z)} = \left\|\frac{1}{\sigma^{\alpha_0}}\pr_{T_{z}}(x-z) +  \frac{1}{\sigma^{\alpha_\perp} \delta}\pr_{N_{z}}(x-z)\right\|.$$
			Furthermore, we can write
			\begin{align} 
			x - y &= \Psi(v_x) - \Psi(v_y) + \delta \( N(v_x) \eta_x - N(v_y) \eta_x\)\label{eq:xy1} \\
			&= \diff \Psi(v_y)[v_x - v_y] + \delta \diff N(v_y)[v_x-v_y]\eta_x + \delta N(v_y)[\eta_x-\eta_y] + O(\sigma^{2\alpha_0}).\label{eq:xy2}
			\end{align} 
			Using \eqref{eq:xy1}, there holds
			\begin{align*} 
			\pr_{T_y}(x-y) &= \pr_{T_y}(\Psi(v_x) - \Psi(v_y)+\delta N(v_x)\eta_x)  \\
			&= \pr_{T_y}(\Psi(v_x) - \Psi(v_{z})+\delta N(v_x)\eta_x) + O(\sigma^{\alpha_0}) \\
			&= \pr_{T_{z}}(\Psi(v_x) - \Psi(v_{x_{i,j}})+\delta N(v_x)\eta_x) + O(\sigma^{\alpha_0}) \\
			&= \pr_{T_{z}}(x-z) + O(\sigma^{\alpha_0})
			\end{align*} 
			where $O(\sigma^{\alpha_0})$ is independant of $\Gamma$. Now using \eqref{eq:xy2}, we find that
			\begin{align*} 
			\pr_{N_y}(x-y) &= \delta N(v_y)[\eta_x-\eta_y]+ \delta \pr_{N_y}(\diff N(v_y)[v_x-v_y]\eta_x) + O(\sigma^{2\alpha_0}) \\
			&=  \delta N(v_{z})[\eta_x-\eta_{z}] + \delta  \pr_{N_{z}}(\diff N(v_y)[v_x-v_{x_{i,j}}]\eta_{z}) + O(\delta \sigma^{\alpha_\perp}) + O(\sigma^{2\alpha_0}) \\
			&= \pr_{N_{z}}(x-z) + O(\delta \sigma^{\alpha_\perp} + \sigma^{2\alpha_0}) 
			\end{align*} 
			where we used that $\alpha_0 \geq \alpha_\perp$, and where again $O(\delta \sigma^{\alpha_\perp} + \sigma^{2\alpha_0})$ does not depend on $A$. All in all, we find that
			\begin{align} 
			\|x-y\|_{\Sigma^{-1}(y)} &= \left\|\frac{1}{\sigma^{\alpha_0}}\pr_{T_{x}}(x-y) +  \frac{1}{\sigma^{\alpha_\perp} \delta}\pr_{N_{y}}(x-y)\right\| \nonumber\\
			&= \left\|\frac{1}{\sigma^{\alpha_0}}\pr_{T_{z}}(x-z) +  \frac{1}{\sigma^{\alpha_\perp} \delta}\pr_{N_{z}}(x-z)\right\| + O\(1 + \sigma^{2\alpha_0}/(\delta \sigma^{\alpha_\perp})\) \nonumber\\
			&= \|x-z\|_{\Sigma^{-1}(z)} + O\(1 + \sigma^{2\alpha_0}/(\delta \sigma^{\alpha_\perp})\). \label{eq:xyz}
			\end{align} 
			Now using that $\sigma^{2\alpha_0} = o(\delta \sigma^{\alpha_\perp})$, there thus holds that $\|x - \bar\Psi(w)\|_{\bar\Sigma^{-1}(w)} \leq 2 \Gamma \log(1/\ve)$ for $\Gamma$ large enough.}
			

			Denote $R_T(u) = \exp(u)-\sum_{t=0}^{T-1} u^t/t!$, then $|R_T(u)| \leq e^{u} |u|^T/T!$ and
			\begin{align*} 
				\exp&\{-\frac12\|x - \bar\Psi(w)\|_{\bar\Sigma^{-1}(w)}^{2}\}  = \sum_{t=0}^{T-1} \frac{(-1)^t}{2^t t!} \|x - \bar\Psi(w)\|_{\bar\Sigma^{-1}(w)}^{2t} + \underbrace{R_T\left(-\|x - \bar\Psi(w)\|_{\bar\Sigma^{-1}(w)}^{2}/2\right).}_{:=~ R_T(x,w)}
			\end{align*}  
		Note that $R_T(x,w)$ is uniformly bounded by
			$$
			|R_T(x,w)| \leq e^{5 \Gamma \log(1/\ve)} \frac{\left(5 \Gamma \log(1/\ve)\right)^T}{T!} \approx \frac{\ve^{-5\Gamma}\log^T(1/\ve)}{T!}.
			$$
			Set,
			$$
			A_{i,j}(w) := \inner{e_i}{\bar\Sigma^{-1}(w) e_j},~~~ B_i(w) :=  \inner{e_i}{\bar\Sigma^{-1}(w) \bar\Psi(w)},~~~\text{and}~~~C(w) := \|\bar\Psi(w)\|^2_{\bar\Sigma^{-1}(w)},
			$$
			so that all functions $A_{i,j}$, $B_i$ and $C$ are continuous functions of $w$ and if $x = (x_1,\dots,x_D)$, 
			\begin{align*} 
				&\|x - \bar\Psi(w)\|_{\bar\Sigma^{-1}(w)}^{2t} = \{\|x\|^2_{\bar\Sigma^{-1}(w)} - 2 \inner{x}{\bar\Sigma^{-1}(w)\bar\Psi(w)} + \|\bar\Psi(w)\|_{\bar\Sigma^{-1}(w)}^2\}^t \\
				&= \sum_{|k| = t} {t \choose k} (-2)^{k_2}\{\sum_{1 \leq i,j \leq D} x_i x_j A_{i,j}(w)\}^{k_1} \{\sum_{i=1}^D x_i B_i(w)\}^{k_2} C(w)^{k_3} \\
				&= \sum_{|k|=t} {t \choose k} (-2)^{k_2} C(w)^{k_3}  \sum_{|\ell| = k_1} {k_1 \choose \ell}  \prod_{1 \leq i,j \leq D}(x_i x_j)^{\ell_{i,j}} A_{i,j}^{\ell_{i,j}}(w) \sum_{|m| = k_2} {k_2 \choose m} \prod_{i=1}^D x_i^{m_i} B_i^{m_i}(w) \\
				&= \sum_{(p,\ell,m) \in \cG_t} P_{p,\ell,m}(x) \times \underbrace{C(w)^p \prod_{1 \leq i,j \leq D} A_{i,j}^{\ell_{i,j}}(w) \prod_{1 \leq i \leq D} B_{i}^{m_{i}}(w) }_{ := ~Q_{p,\ell,m}}
			\end{align*}  
			where $P_{p,\ell,m}(x)$ are polynomial functions of $x$, $Q_{p,\ell,m}(w)$ are continuous functions of $w$, and where $\cG_t$ is the set $\{(p,\ell,m)~|~p + |\ell| + |m| = t\} \subset \bbN^{D^2+D+1}$. According to \cite[Lem 3.1]{ghosal2001entropies}, one can always find an atomic probability measure $G_{i,j}$ such that 
			$$
			\int Q_{p,\ell,m}(w) \bar g_{i,j}(w) \diff w = \int Q_{p,\ell,m}(w) G_{i,j}(\diff w) 
			$$
			for all $p,\ell,m$ such that $p + |\ell| + |m| \leq T-1$. Since there are less than $T^{D^2+D+1}$ such triplets, the probability measure $G_{i,j}$ can be taken to have less than $T^{D^2+D+1}$ atoms. Note that then, this measure satisfies that
			\begin{align*} 
				\Big| \int_{\cW_{i,j}} &\varphi_{\bar\Sigma(w)}(x-\bar\Psi(w)) (\bar g_{i,j}(w) \diff w - G_{i,j}(\diff w)) \Big| \\
				&= \frac{1}{(2\pi)^{D/2} \delta^{D-d} \sigma^D} \Big| \int_{\cW_{i,j}} R_T(x,w)(\bar g_{i,j} (w) \diff w - G_{i,j}(\diff w)) \Big| \lesssim \frac{\ve^{-5\Gamma}\log^T(1/\ve)}{\delta^{D-d} \sigma^D T!}.
			\end{align*} 
			All in all, $G_{i,j}$ is such that
			\begin{align*} 
				&\Big| \int_{\cW_{i,j}} \varphi_{\bar\Sigma(w)}(x-\bar\Psi(w)) \bar g_{i,j}(w) \diff w - \int_{\cW_{i,j}} \varphi_{\bar \Sigma(w))}(x-\bar\Psi(w)) G_{i,j}(\diff w) \Big|
				\\
				&\lesssim \begin{cases} \displaystyle\frac{1}{\sigma^D \delta^{D-d}}\ve^H  &\text{if}~~x\notin \cO_{x_i}^1; \vspace{3pt}\\
					\displaystyle\frac{1}{\sigma^D \delta^{D-d}}\ve^{\Gamma/2} &\text{if}~~x\in  \cO_{x_i}^1~\text{and} ~\inf_{w \in J} \|x - \bar \Psi(w) \|^2_{\bar\Sigma^{-1}(w)} \geq \Gamma \log(1/\ve);  \vspace{3pt}\\ 
					\displaystyle\frac{1}{\sigma^D \delta^{D-d}} \frac{1}{T!} \ve^{-5\Gamma} \log^{T} (1/\ve) &\text{if}~~x\in  \cO_{x_i}^1~\text{and} ~\inf_{w \in J} \|x - \bar \Psi(w) \|^2_{\bar\Sigma^{-1}(w)} \leq \Gamma \log(1/\ve).
				\end{cases}
			\end{align*} 
			Taking $H \geq 1$, $\Gamma \geq 2$ and $T \geq 5\Gamma \log(1/\ve)$ yields a bound of order $\ve/\sigma^D \delta^{D-d}$ in every case. Then the probability measure
			$$
			G = \sum_{i \in \cI} \alpha_i  \sum_{j \in \cJ_i}  \bar c_{i,j} (\bar \Psi_{x_i})_{\#} G_{i,j},
			$$
		is a  discrete measure on $M^\delta$ with at most
			$$
		N_\varepsilon = 	T^{D^2+D+1} \sum_{i \in \cI} \Card \cJ_i  \approx \log^{D^2+D+1}(1/\ve) R^D  \sigma^{-D} \leq \sigma^{-D} \log^{D^2+D/(\kappa \wedge 2)}(1/\ve)
			$$
			atoms and such that $\|K_\Sigma G - K_\Sigma g \|_{\infty} \lesssim \ve/(\sigma^D\delta^{D-d})$.
			
			We now turn to bounding the $L^1$ norm. Let again pick $i \in \cI$ and $j \in \cJ_i$.
			We let $x_{i,j} = \bar\Psi_{x_i}(w_{i,j})$ for some $w_{i,j} \in \cW_{i,j}$ and for $T > 0$ we define
			$$
			\cH_A = \{x \in \bbR^D ~|~ \|x - x_{i,j}\|_{\Sigma^{-1}(x_{i,j})} \leq A\}.
			$$
			There holds that $\Leb \cH_A \simeq \delta^{D-d}\sigma^D T^D$. Furthermore, we have, denoting respectively $P_{i,j}$ and $Q_{i,j}$ the push-forwards of $\bar g_{i,j}$ and $G_{i,j}$ through $\bar \Psi_{x_i}$
			\begin{align}\label{KijPQ:1}
				\| K_\Sigma P_{i,j} - K_\Sigma Q_{i,j}\|_1 &=  \int_{\cH_A } | K_\Sigma P_{i,j} - K_\Sigma Q_{i,j} | +   \int_{\cH_A^c} | K_\Sigma P_{i,j} - K_\Sigma Q_{i,j} | \nonumber\\
				&\lesssim  \delta^{D-d} \sigma^D A^D \| K_\Sigma P_{i,j} - K_\Sigma Q_{i,j} \|_\infty +\int_{\cH_A^c} [K_\Sigma Q_{i,j}+K_\Sigma P_{i,j}]  \nonumber\\
				&\lesssim A^D \ve + \int_{\cH_A^c} [K_\Sigma P_{i,j}+K_\Sigma Q_{i,j}] .
			\end{align} 
			Recall that the support of $P_{i,j}$ and $Q_{i,j}$ are in $\bar\Psi_{x_{i}}(\cW_{i,j})$.  Furthermore, since
			$$
			\diam \pr_M \bar\Psi_{x_i}(\cW_{i,j})\lesssim \sigma^{\alpha_0},
			$$
			there holds, using \eqref{lem:angles}, that $\|\pr_{T_y} - \pr_{T_{y'}}\|_{\op} \lesssim \sigma^{\alpha_0}$ for any two $y,y' \in \bar\Psi_{x_i}(\cW_{i,j})$.
			Let $y \in  \bar\Psi_{x_i}(\cW_{i,j})$ and $x \in \cH_A^c$, we now show that 
			$\|x-y\|_{\Sigma^{-1}(y)} \geq A/2 $ if $A$ is chosen large enough.
			\newclem{First notice again that if $x \notin \cW_{x_i}^1$, then $\|x-y\|_{\Sigma^{-1}(y)} \gtrsim \sigma^{-\alpha_0} \wedge \delta^{-1} \sigma^{-\alpha_\perp} \geq A$ for all $A$ so that we can take $x \in \cW_{x_i}^1$. Now using \eqref{eq:xyz} applied to $z = x_{i,j}$, we find that $\|x-y\|_{\Sigma^{-1}(y)} = \|x-x_{i,j}\|_{\Sigma^{-1}(x_{i,j})}+ O(1+\sigma^{2\alpha_0}/(\delta\sigma^{\alpha_\perp})) \geq A/2$,
			where we used that $\sigma^{2\alpha_0} = o(\delta \sigma^{\alpha_\perp})$, and provided that $A$ is chosen big enough.}

			This yields that
			\begin{align*} 
				\int_{\cH_A^c} K_\Sigma P_{i,j} &= \int_{\bar\Psi_{x_i}(\cW_{i,j}} \int_{\cH_A^c} \varphi_{\Sigma(y)}(x-y) \diff x P_{i,j}(\diff y) \\
				&=  \frac{1}{(2\pi)^{D/2} \sigma^D \delta^{D-d}} \int_{\bar\Psi_{x_i}(\cW_{i,j})}\int_{\cH_A^c} \exp\{-\frac1{2} \|x-y\|^2_{\Sigma^{-1}(y)}\} \diff x P(\diff y) \\
				&\leq \frac{1}{(2\pi)^{D/2}} \int_{\|z\| \geq A/2} \exp\{-\frac{1}{2} \|z\|^2\} \diff z \lesssim \exp(-A^2/8),
			\end{align*} 
			where we made the variable change $z= \Sigma^{-1/2}(y)(x-y)$ in the second to last inequality. The same holds for $K_\Sigma Q_{i,j}$. Setting $A = (8\log(1/\ve))^{1/2}$ and combining with \eqref{KijPQ:1}, $\| K_\Sigma P_{i,j} - K_\Sigma Q_{i,j} \|_1  \lesssim \log^{D/2}(1/\ve) \ve$. We finally obtain that 
$$		
\|K_\Sigma g - K_\Sigma G\|_1  \lesssim  \sum_i \alpha_i \sum_{j\in \mathcal J_i} \bar c_{i,j} \log^{D/2}(1/\ve) \ve  \lesssim  \log^{D/2}(1/\ve) \ve.
$$

Also we can choose the atoms of $G$ to be $\sigma^{\alpha_0}\wedge \delta \sigma^{\alpha_\perp} \times \varepsilon $ apart, thanks to \lemref{vj} together with the following  bound on $\Sigma(z)^{-1}\Sigma(y)-\Id $ when $\|z-y\| \leq \sigma^{\alpha_0}\wedge \delta \sigma^{\alpha_\perp} \varepsilon $:
\begin{align*}
\Sigma(z)^{-1}\Sigma(y)-\Id &= O_{z}^\top \Delta^{-2}_{\sigma,\delta} O_{z} O_{y}^\top \Delta_{\sigma,\delta}^2 O_{y} - \Id \\
			&= O_{z}^\top \Delta^{-2}_{\sigma,\delta} (O_{z} O_{y}^\top - \Id) \Delta^2_{\sigma,\delta} O_{y} + O_{z} O_{y}^\top - \Id, 
		\end{align*} 
			so that
			\begin{align*} 
				&\sqrt{\tr(\Sigma(z)^{-1}\Sigma(y)-\Id)^2} \\
				&\leq \sqrt{D} \|\Sigma(z)^{-1}\Sigma(y)-\Id\|_{\op} \lesssim \|\Delta^{-2}_{\sigma,\delta}\|_{\op} \|O_{z} O_{y}^\top - \Id \|_{\op} \| \|\Delta^2_{\sigma,\delta}\|_{\op}
			+ \|O_{z} O_{y}^\top - \Id \|_{\op} \\
			&\lesssim \sigma^{-2\alpha_0} \vee \delta^{-2}\sigma^{-2\alpha_\perp}  \|O_{z} - O_{y}\|_{\op} \lesssim [ \sigma^{-2\alpha_0} \vee \delta^{-2}\sigma^{-2\alpha_\perp}]  \|\pr_{T_{z}} - \pr_{T_{y}}\|_{\op} \\
			 & \quad \lesssim  [\sigma^{-2\alpha_0} \vee \delta^{-2}\sigma^{-2\alpha_\perp} ] \|z-y\| \lesssim \ve,
			\end{align*} 
			where we used \eqref{lem:angles} in the second to last inequality, ending the proof.

		\end{proof}
		
%
%

\section{Appendix to Section \ref{sec:num}} \label{app:num}

\subsection{Additional details on the numerical setting} \label{app:numdetail}
				
				The manifolds used in the numerical experiments are given through the following equations:
	\bitem	
	\item \emph{The two circles}: it is the union of $\cC_1$ and $\cC_2$ with equation $(x-x_i)^2+(y-y_i)^2 = r_i^2$ for $i \in \{1,2\}$. In the experiements, we chose $(x_1,y_1) = (0,0)$, $(x_2,y_2) = (2,0)$ and $r_1 = r_2 = 2$.
	\item \emph{The 2D-spiral}: it is given by the parametric embedding
	$$
	\varphi_{2} : t \in [0,1] \mapsto \begin{cases}
		R (\omega t +\theta_0) \cos(\omega t +\theta_0) \\
		R (\omega t +\theta_0) \sin(\omega t +\theta_0).
	\end{cases} 		
	$$
	In the numerical studies, we chose $R = 1/2\pi$, $\omega = 7\pi/2$ and $\theta_0 = \pi/2$.
	\item \emph{The 3D-spiral}: it is given by the parametric embedding
	$$
	\varphi_{3} : t \in [0,1] \mapsto \begin{cases}
		R (\omega t +\theta_0) \cos(\alpha t +\theta_0) \\
		R (\omega t +\theta_0) \sin(\alpha t +\theta_0) \\
		\nu t. 
	\end{cases} 		
	$$
	In the numerical studies, we chose $R = 1/2\pi$, $\omega = 7\pi/2$, $\theta_0 = \pi/2$ and $\nu = 2$.
	\item \emph{The torus}: it is the set $\cT$ given by the equation $(\sqrt{(x-x_0)^2+(y-y_0)^2}-R)^2+(z-z_0)^2 = r^2$. In the experiements, we chose $(x_0,y_0,z_0) = (0,0,0)$, $R = 3$ and $r = 1$.
	\eitem

\subsection{Details on the Gibbs sampling algorithm} \label{app:gibbssampl}

This first section presents the implementation of the partial location-scale prior \eqref{partialprior} with Gibbs sampling for a particular choice of measure $H$.

In order to simplify the notations we will write indifferently $\Lambda$ as a $D$-dimensional diagonal matrix or a $D$-dimensional vector.

We use the latent cluster representation of the $(\mu_i O_i)$ and denote by $c_i$ the cluster allocation of observation $i$:
\[
c_i = c_j \iff y_i \text{ and } y_j \text{ are in the same cluster}
\] 
and let	$N$ be    the number of clusters, so that $ 1 \leq c_i \leq N$.
We write $\phi_c = (\mu_c^*,O_c^*)_{1 \leq c \leq N}$ for the unique values of the cluster parameters. 

\begin{prp} \label{thm:prior1}
	In the model \eqref{prior:1}, we have :
	\begin{align*}
		p(\Lambda | Y, (\mu_i,O_i),b) = & \bigotimes_{j=1}^D \Invg\left(a_j+n/2, b_j + \frac{1}{2} \sum_{i=1}^n \langle O_i^j | y_i - \mu_i \rangle^2\right) \\
		p(b|Y,(\mu_i,O_i),\Lambda) = & \bigotimes_{j=1}^D \Gamma(a_j+1,\kappa_j + \lambda_j^{-1})
	\end{align*}
	and :
	\begin{align*}
		p(\mu_c^*|O_c^*,Y,(c_i),\Lambda) &=  \mathcal{N}(\cdot | \hat{\mu}_c, A^{-1}) , \, c \leq N\\
		p(O_c^*|\mu_c^*,Y,(c_i),\Lambda) &=  p_{\BMF}(\cdot |S,- \frac{1}{2}\Lambda^{-1},M_0)
	\end{align*}
	where
	\begin{align*}
		A &=  \Sigma_0^{-1} + n_c O_c^* \Lambda^{-1} (O_c^*)^T,~~~~~~
		n_c =  \Card \{i: c_i = c \}, \\
		\hat{\mu}_c &= A^{-1}[ \Sigma_0^{-1}\mu_0 + O_c^*\Lambda^{-1}(O_c^*)^T \sum_{c_i = c}y_i]
		~~\text{and}~~ S =  \sum_{c_i = c} (y_i - \mu_c^*)(y_i - \mu_c^*)^T.
	\end{align*}
\end{prp}
\begin{proof}[Proof of Proposition \ref{thm:prior1}]
	\begin{itemize}
		\item For the variances $\Lambda$ we have 
		\[
		p(\Lambda | Y, (\mu_i,O_i),b) \propto p(\Lambda|b) \prod_i \mathcal{N}(y_i|\mu_i, O_i \Lambda O_i^T) \propto \left[ \prod_{j=1}^D \lambda_j^{-(a+1)} e^{-b_j/\lambda_j} \right] \prod_{i=1}^n \mathcal{N}(y_i|\mu_i, O_i \Lambda O_i^T).
		\]
		But by orthogonality :
		
		\begin{align*}
			\mathcal{N}(y_i|\mu_i, O_i \Lambda O_i^T) \propto & \left[ \prod_{j=1}^D \lambda_j^{-1/2} \right] \exp \left(- \frac{1}{2} \langle y_i - \mu_i | O_i \Lambda^{-1} O_i^T (y_i - \mu_i) \rangle \right) \\
			\propto & \left[ \prod_{j=1}^D \lambda_j^{-1/2} \right] \exp \left(- \frac{1}{2} \langle O_i^T(y_i - \mu_i) | \Lambda^{-1} O_i^T (y_i - \mu_i) \rangle \right) \\
			\propto & \left[ \prod_{j=1}^D \lambda_j^{-1/2} \right] \exp \left(- \frac{1}{2} \sum_{j=1}^D \lambda_j^{-1} \langle O_i^j | y_i - \mu_i \rangle^2 \right) \\
			\propto & \prod_{j=1}^D \left[ \lambda_j^{-1/2} \exp \left(- \frac{1}{2} \lambda_j^{-1} \langle O_i^j | y_i - \mu_i \rangle^2 \right) \right],
		\end{align*}
		where $O_i^j$ is the j-th column of $O_i$. Therefore,
		\begin{align*}
			p(\Lambda | Y, (\mu_i,O_i),b) \propto & \prod_{j=1}^D \left[ \lambda_j^{-(a+1)} e^{-b_j/\lambda_j} \lambda_j^{-n/2} \exp \left(- \frac{1}{2} \lambda_j^{-1} \sum_{i=1}^n \langle O_i^j | y_i - \mu_i \rangle^2 \right)  \right] \\
			\propto & \prod_{j=1}^D \left[ \lambda_j^{-(a+1 + n/2)} \exp \left(- \lambda_j^{-1} \left[b_j + \frac{1}{2} \sum_{i=1}^n \langle O_i^j | y_i - \mu_i \rangle^2 \right] \right)  \right],
		\end{align*}
		i.e conditionally on the rest of the variables, the $(\lambda_j)_{j=1}^D$ are independent with distribution $\Invg(a+n/2, b_j + \frac{1}{2} \sum_{i=1}^n \langle O_i^j | y_i - \mu_i \rangle^2)$.
		
		\item For the hyperparameters $(b_j)_{j=1}^D$ we have
		\[
		p(b|Y,(\mu_i,O_i),\Lambda) \propto p(b)p(\Lambda|b) \propto \bigotimes_{j=1}^D \Gamma (a_j+1,\kappa_j+\lambda_j^{-1})
		\]
		i.e conditionally on the rest of the variables, the $b_j's$ are independent with distribution $\Gamma (a_j+1,\kappa_j + \lambda_j^{-1})$.
		
		\item For the cluster locations $(\mu_c^*)_{1 \leq c \leq K}$ we have 
		\begin{align*}
			p(\mu_c^*|O_c^*,Y,(c_i),\Lambda) &\propto  p(\mu_c^*) \prod_{c_i=c} \mathcal{N}(y_i|\mu_c^*, O_c^* \Lambda (O_c^*)^T) \\
			&\propto  \exp \left(- \frac{1}{2} \{ \langle \mu_c^*-\mu_0 | \Sigma_0^{-1}(\mu_c^*-\mu_0) \rangle + \sum_{c_i=c} \langle y_i - \mu_c^* | O_c^* \Lambda^{-1} (O_c^*)^T(y_i - \mu_c^*) \rangle \}\right),
		\end{align*}
		and since the prior $p(\mu_c^*)$ is Gaussian with mean $\mu_0$ and variance $\Sigma_0$ the above conditional posterior distribution is a Gaussian distribution with mean $\hat{\mu}_c = A^{-1} [\Sigma_0^{-1} \mu_0 + (O_c^*)\Lambda^{-1}(O_c^*)^T \sum_{c_i=c}y_i] $ and variance $A^{-1} = [\Sigma_0^{-1}  + n_c(O_c^*)\Lambda^{-1}(O_c^*)^T]^{-1}$.
		\item For the cluster orientations $(O_c^*)_{1 \leq c \leq K}$ we have
		\begin{align*}
			p(O_c^*|\mu_c^*,Y,(c_i),\Lambda) &\propto  \exp \left( \tr \{ M_0^TO_c^*\} \right) \prod_{c_i = c} \mathcal{N}(y_i|\mu_c^*, O_c^* \Lambda (O_c^*)^T) \\
			&\propto  \exp \left( \tr \{ M_0^TO_c^*\} \right) \exp \left( - \frac{1}{2} \sum_{c_i=c} \langle y_i- \mu_c^* | O_c^* \Lambda^{-1}(O_c^*)^T (y_i - \mu_c^*) \rangle \right) \\
			&\propto  \exp \left( \tr \{ M_0^TO_c^* - \frac{1}{2} \sum_{c_i=c} (y_i-\mu_c^*)(y_i-\mu_c^*)^TO_c^* \Lambda^{-1}(O_c^*)^T \} \right) \\
			&\propto  \exp \left(\tr \{ M_0^TO_c^* - \frac{1}{2} \Lambda^{-1}(O_c^*)^T S O_c^* \},\right).
		\end{align*}
		We thus obtain
		$$
		O_c^* \sim p_{\BMF}(O_c^*|S,- \frac{1}{2}\Lambda^{-1},M_0)~|~\mu_c^*,Y,(c_i),\Lambda.
		$$
	\end{itemize}
\end{proof}

Proposition \ref{thm:prior1} is the key to implement algorithm 8 in \cite{neal2000markov}.

\begin{rem}\label{rem:1}
	Here we can see that when $\sum_{i=1}^n \langle O_i^j | y_i - \mu_i \rangle^2 = o(n^{3/2})$, $\Invg(a+n/2, b_j + \frac{1}{2} \sum_{i=1}^n \langle O_i^j | y_i - \mu_i \rangle^2)$ is tightly concentrated around its mean 
	$$\frac{b_j + \frac{1}{2} \sum_{i=1}^n \langle O_i^j | y_i - \mu_i \rangle^2}{a+n/2-1} \geq \frac{b_j}{a+n/2-1};$$  
	and if we  fix $b_j$ this may induce a rather strong penalization on small values of $\lambda_j$ for finite $n$.
\end{rem}

Even though to the best of our knownledge no direct samplers are available for $p_{\BMF}$, it is still possible to perform a Gibbs sampling update over the columns of $O_c^*$ (cf \cite{hoff2009simulation} and the associated package {\tt rstiefel}). Another more involved (but more efficient) option would be to perform Hamiltonian Monte Carlo via polar expansion as suggested in \cite{jauch2021monte}. With a slight abuse of notation we will write 
$$
O_c^* \sim p_{\BMF,\Gibbs}(\cdot|A,B,C,O_c^*)
$$
for a Gibbs sampling scan on the column of $O_c^*$ starting from $O_c^*$. All in all, this leads us to Algorithm \ref{algo:1} below. It uses the conditional allocation distributions :
\begin{equation}\label{alloc:dis}
	P(c_i=l|c_{-i},y_i,(\phi_l)_{1\leq l \leq h}) \propto
							 \left[
							\begin{array}{cc}
								n_{-i,c} \mathcal{N}(y_i|\mu_l, O_l \Lambda O_l^T) & \forall 1 \leq l \leq k\\
								\frac{\alpha}{m} \mathcal{N}(y_i|\mu_l, O_l \Lambda O_l^T) & \forall k < l \leq h
							\end{array} 
						\right.
\end{equation}

\begin{algorithm}[h!]
    \SetAlgoRefName{1}
    \SetKwInOut{Input}{Input}
    \SetKwInOut{Output}{Output}
    \Input{a dataset $Y = (y_i)_{i=1}^n$, a current partition $c = (c_i)_{i=1}^n$, a current state $\phi = (\mu_l^*,O_l^*)_{l=1}^K$ (with $K$ the current number of clusters), current variances $\Lambda = (\lambda_j)_{j=1}^D$, current hyperparameters $(b_j)_{j=1}^D$, a number of auxilliary parameters $m \in \mathbb{N}^*$, $\alpha > 0,\mu_0 \in \mathbb{R}^D, \Sigma_0 \in \mathcal{S}_D^{++}(\mathbb{R}), M_0 \in \mathcal{M}_D(\mathbb{R})$ the parameters for the base distribution of the Dirichlet process, $(a_j)_{j=1}^D$ the prior parameters for the variances}
    \Output{an updated state $(c,\phi,\Lambda,b)$}
 \For{$i = 1,\ldots,n$}{
							$k \gets \# \{j \neq i : c_j = c_i \}$; $h \gets k + m$; and 							label $(c_j)_{j \neq i}$ with $\{1,\ldots,k\}$\;
							\eIf{$\exists j \neq i : c_i = c_j$}{
								$(\phi_l)_{k < l \leq h} \overset{i.i.d}{\sim} G_0$\;
							}{\If{$\forall j \neq i : c_i \neq c_j$}{
									$c_i \gets k+1$;
									$(\phi_l)_{k+1 < l \leq h} \overset{i.i.d}{\sim} G_0$\;}
							}
							$n_{-i,c} \gets \# \{c_j : j \neq i\}$;
							\text{draw $c_i$ from $\{1,\ldots,h\}$ with probabilitity \eqref{alloc:dis}};
							\text{discard the $(\phi_l)$ not associated with any observations}\;
							\For{$l \in \{c_1,\ldots,c_n\}$}{
								$n_l \gets \# \{i : c_i = l\}$;
								$A=\Sigma_0^{-1} + n_l O_l^* \Lambda^{-1}(O_l^*)^T$ ;
								$\hat{\mu}_l = A^{-1}[\Sigma_0^{-1}\mu_0 + O_l^* \Lambda^{-1}(O_l^*)^T \sum_{c_i=l}y_i$]\;
								$\mu_l \gets \mu_l \sim	 \mathcal N( \hat{\mu}_l, A^{-1})$
					;
								$S \gets \sum_{c_i=l} (y_i-\mu_l^*)(y_i-\mu_l^*)^T$\;
								$O_l^* \gets O_l^* \sim
								p_{BMF,Gibbs}(O_l^* | S,- \frac{1}{2}\Lambda^{-1},M_0,O_l^*)$
								\;
							}\;
							$(\lambda_j)_{j=1}^D \gets (\lambda_j)_{j=1}^D \overset{ind}{\sim} p(\lambda_j|(y_i),c,(\phi_l)) \propto Inv\Gamma(a+n/2,b+\frac{1}{2}\sum_{i=1}^n \langle (O_{c_i}^*)^j | y_i - \mu_{c_i}^* \rangle^2)$\;
							$b \gets (b_j)_{j=1}^D \overset{ind}{\sim} \Gamma(a_j+1,\kappa_j + \lambda_j^{-1})$
							
						}
    \caption{Gibbs sampling algorithm for Partial location scale DP mixture}
    \label{algo:1}
\end{algorithm}

	\end{appendix}

\end{document}